\documentclass[reqno,12p]{amsart}
\usepackage[T1]{fontenc}
\usepackage{geometry}
\usepackage{mathtools,amssymb,amsthm,mathrsfs,color,lineno,paralist,graphicx,float}
\usepackage[colorlinks,
linkcolor=red,
anchorcolor=green,
citecolor=blue,
]{hyperref}
\usepackage{amsmath}
\usepackage{pdfpages}
\usepackage{amstext}
\usepackage{stmaryrd}
\usepackage{bbm}
\usepackage{comment}
\usepackage{faktor,url}
\usepackage{xfrac}
\usepackage{mathabx}
\usepackage{enumitem}
\usepackage{refcount}
\usepackage{amssymb}
\graphicspath{{Dibujos/}}

\numberwithin{equation}{section}
\numberwithin{figure}{section}
\theoremstyle{plain}
\newtheorem{thm}
{\protect\theoremname}[section]
  \theoremstyle{definition}
  \newtheorem{defn}[thm]{\protect\definitionname}
  \theoremstyle{remark}
  \newtheorem{rem}{\protect\remarkname}
  \theoremstyle{plain}
  \newtheorem{lem}[thm]{\protect\lemmaname}
  \theoremstyle{plain}
  \newtheorem{prop}[thm]{\protect\propositionname}
  \theoremstyle{plain}
  
  \theoremstyle{plain}
  \newtheorem{nota}[thm]{\protect\notationname}
    \theoremstyle{plain}

\def\red{\color{red}}
\def\black{\color{black}}

\newtheorem{thmx}{Theorem}

\newcommand{\inn}{\mathrm{in}}
\newcommand{\out}{\mathrm{out}}
\newcommand{\loc}{\mathrm{loc}}

\newcommand{\bs}{\boldsymbol}

\newcommand{\MM}{\mathcal{M}}
\newcommand{\Tj}{T_{1,J}}
\newcommand{\Tphi}{T_{1,\varphi}}
\newcommand{\eps}{\varepsilon}

  \providecommand{\definitionname}{Definition}
  \providecommand{\lemmaname}{Lemma}
  \providecommand{\propositionname}{Proposition}
  \providecommand{\remarkname}{Remark}
\providecommand{\theoremname}{Theorem}
\providecommand{\corollaryname}{Corollary}
\providecommand{\notationname}{Notation}

\title{Partially hyperbolic dynamics in the 3-body problem}

\author[M. Guardia]{Marcel Guardia}
\address[MG]{Departament de Matem\`atiques i Inform\`atica, Universitat de Barcelona, Gran Via, 585, 08007 Barcelona, Spain  \& Centre de Recerca Matemàtica, Edifici C, Campus Bellaterra, 08193 Bellaterra, Spain}
\email{guardia@ub.edu}

\author[J. Paradela]{Jaime Paradela}
\address[JP]{ Department of Mathematics, University of Maryland, 4176 Campus Drive, 20782, MD, US}
\email{paradela@umd.edu}

\begin{document}

\begin{abstract}
We construct \textit{symplectic blenders} for  two classical Hamiltonian systems: the 3-body problem and its restricted version. We use these objects to show that both models exhibit  a robust, strong form of topological instability. We do not assume any smallness conditions on the masses but require only that at least two of them are distinct.

Our construction is based on two abstract results which might be of independent interest. The first one gives an explicit condition under which a given pair of twist maps of the cylinder generates a locally transitive iterated function system. The second one extends this result to certain cylinder skew-products.

\end{abstract}

\maketitle
\tableofcontents

\section{Introduction}\label{sec:introduction}

 In plain words, a stable (unstable) blender is a robustly thick part of a hyperbolic set such that its stable (unstable) manifold meets every unstable manifold that comes near it and, moreover, it does so in a  $C^1$ persistent fashion. Introduced by Bonatti and D\'iaz  in \cite{BonattiDiazblenders}  as a  robust source of  semi-local  transitivity in partially hyperbolic settings, blenders have found numerous applications within smooth dynamics. Among others, these include the construction of $C^1$ robust, not uniformly hyperbolic, transitive diffeomorphisms \cite{BonattiDiazblenders}, $C^1$-robust homoclinic tangencies \cite{Robustc1tangencies}, a $C^1$-dense set of stably ergodic partially hyperbolic diffeomorphisms \cite{MR4198639}, persistence of heterodimensional cycles \cite{BonattiDiazblenders,MR4531665,MR4743518} and persistence of saddle-center homoclinic loops \cite{LiTuraevPreprint}. The related notion of parablenders  was introduced by Berger in \cite{MR3514960} to show that the set of maps  of a smooth manifold displaying infinitely many sinks is typical in the sense of Kolmogorov. Later, these objects also played an important role in his construction of a locally generic set of maps displaying fast growth of the number of periodic points \cite{MR4366414}.

A symplectic version of blenders was developed by Nassiri and Pujals in \cite{NassiriPujalsTransitivity} to produce arcs of smooth Hamiltonian diffeomorphisms which  exhibit  topologically transitive normally hyperbolic  laminations (of codimension-two) in a robust fashion. These arcs are obtained by $C^\infty$ perturbation techniques from \textit{a priori} chaotic Hamiltonians, whose distinguishing property is that they admit a return map which behaves as a partially hyperbolic skew-product. 

However, in all of the works cited above (except for \cite{MR4743518,LiTuraevPreprint}), blenders are created via perturbation techniques only available in the $C^k$ setting ($k=1,\dots,\infty$). In particular, it is difficult  to deduce, using these techniques, the existence of blenders (or their symplectic counterpart) for a given  parametric family of diffeomorphisms (or flows). Results identifying blenders in concrete models are rather scarce and, as far as the authors, up to this paper blenders had not been proven to exist in physical models. To the best of our knowledge, the only known examples include: a) parametric families exhibiting a particular kind of heterodimensional cycle at which the renormalized dynamics is governed by the  center-unstable H\'{e}non family (see  \cite{MR3168258,MR4043213}, or, for a computer-assisted proof see \cite{MR4808810}); b) parametric families of 4-dimensional symplectic diffeomorphisms exhibiting a saddle-center homoclinic loop \cite{LiTuraevPreprint}; c) polynomial automorphisms of $\mathbb C^3$ (see \cite{MR4043881}).

The motivation behind the present work is two sided. First, develop tools to identify blenders in parametric families of Hamiltonian systems, that is provide conditions that can be checked in given systems. Second, investigate how these objects articulate the global dynamics of the system. We  illustrate our approach by applying these tools to two concrete families: the classical 3-body problem and its restricted version, answering an open problem stated in \cite{NassiriPujalsTransitivity}, where the authors mention that ``a major challenge would be to apply the present approach to the context of the restricted
3-body problem''.  
Even if the presented results deal with two Celestial Mechanics models, 
 our techniques, based on  abstract results for Iterated Function Systems and skew-products of the annulus, will have applications to a general class of  nearly-integrable Hamiltonians (i.e. small perturbations of Arnold-Liouville integrable Hamiltonians (see \cite{MR2269239})). The rest of this section is organized as follows.
\medskip

\subsection*{Abstract weak transversality-torsion mechanism:} Section \ref{sec:ifsintro} contains two abstract results which might be of independent interest: Theorem \ref{thm:transitivityIFS} and Theorem \ref{thm:skewproduct}. 

Theorem \ref{thm:transitivityIFS} introduces an explicit condition for a pair of twist maps of the annulus to generate a locally transitive Iterated Function System (IFS from now on), and, in particular, exhibit a symbolic blender (see Section \ref{sec:symbolic}). Theorem \ref{thm:transitivityIFS} is the key ingredient needed to control some  center directions for the flow of the 3-body problem and will find a direct application in the proof of Theorem \ref{thm:Main3bp}. 

Theorem \ref{thm:skewproduct} can be seen as an extension of Theorem \ref{thm:transitivityIFS} to a more general setting: a non-locally-constant skew product over the full shift with fiber maps given by twist maps of the annulus. Theorem \ref{thm:skewproduct} is key to study the dynamics along weakly invariant normally hyperbolic laminations and will find a direct application in the proof of Theorem \ref{thm:MainR3bp}.

We also outline in Section \ref{sec:ifsintro} the main mechanism behind these results: a variant of the transversality-torsion mechanism introduced by Cresson in \cite{MR1949441} in which the transversallity is made arbitrarily small.

\subsection*{The 3-body problem:} In Section \ref{sec:intro3bp} we introduce the 3-body problem and describe its phase space. In particular, we introduce a compactification  (McGehee's compactification \cite{McGeheestablemanifold}), which allows us to study (a particular kind of) unbounded motions in connection with homoclinic bifurcations. We will briefly recall known results on the existence of hyperbolic sets in the planar 3-body problem and their relation to a particularly exotic class of unbounded orbits: the so-called oscillatory motions.

Then, we  outline how one can obtain a locally partially hyperbolic setting within this framework, and introduce our third main result: Theorem \ref{thm:mainblender}, in which we show the existence of a symplectic blender for the (symplectically reduced) planar 3-body problem.  

Finally, we study the implications of the existence of symplectic blenders on the  dynamics of the 3-body problem. Our fourth main result (Theorem \ref{thm:Main3bp}) exploits the existence of a symplectic blender to construct an orbit of the 3-body problem whose projection to a suitable (center) subspace (of codimension two in the reduced five-dimensional phase space) is locally dense. We also present a geometric version of Theorem \ref{thm:Main3bp}, stating the existence of an orbit which accumulates (both forward and backwards) to an open set inside a three-dimensional normally-parabolic invariant manifold (of codimension two in the reduced phase space) with trivial dynamics (i.e. this invariant manifold is foliated by periodic orbits). Notice that accumulating this (degenerate) manifold requires bridging accross two extra dimensions. 

We believe that blenders will prove extremely useful to study other relevant aspects of the dynamics of the 3-body problem (see Section \ref{sec:comparison}).

\subsection*{The restricted three-body problem:} In Section \ref{sec:restricted3bp} we consider the \textit{restricted} 3-body problem, which can be seen as a limit of the 3-body problem along which the mass of one body becomes negligible. The very same ideas used for the 3-body problem can be adapted to this case in order to construct a symplectic blender. In this case we can use the blender to obtain a rather strong form of Arnold diffusion. Provided the eccentricity $\zeta\in[0,1)$ of the massive bodies is small enough,  one can construct a weakly invariant, normally hyperbolic  lamination $\mathcal A_\zeta$ whose three-dimensional leaves have codimension two inside the 5-dimensional phase space.  In our last main result, Theorem \ref{thm:MainR3bp}, we show the existence of an orbit inside this lamination which is almost dense\footnote{We will give a precise meaning to what we mean by ``an orbit which is almost dense in the lamination'' in Theorem \ref{thm:MainR3bp}.} in $\mathcal A_\zeta$.
Moreover, as $\zeta\to 0$, the size of the leaves $\mathcal A_\zeta$ becomes unbounded.

\medskip

\subsection{Iterated function systems and skew-products}\label{sec:ifsintro}
Our first set of  results deal with the topological properties of certain iterated function systems and skew-products.

\subsection*{Iterated function systems}
Fix $\eps_\star>0$, let $\mathbb A=\mathbb T\times[-1,1]$ and consider a pair of real-analytic, exact symplectic maps $T_{0}:\mathbb A\to \mathbb T\times\mathbb R$ and $T_1:B\subset\mathbb A\to\mathbb T\times\mathbb R$ depending on a parameter $\varepsilon\in (0,\eps_\star)$. Given two constants $\rho,\sigma>0$ we denote by 
\begin{equation}\label{def:annuluscomplex}
\mathbb A_{\rho,\sigma}=\{(\varphi,J)\in (\mathbb C/2\pi\mathbb Z)\times\mathbb C\colon |\operatorname{Im}\varphi|<\sigma,\  |\operatorname{Re} J|<1,\ |\operatorname{Im} J|<\rho\}.
\end{equation}
We denote by $|\cdot|_{\rho,\sigma}$ the sup-norm on $\mathbb A_{\rho,\sigma}$ and, for any open set $B\subset \mathbb A$ and $r\in\mathbb N$ we let $|\cdot|_{C^r(B)}$ the $C^r$ norm on $B$. We will suppose that  $T_0,T_1$ satisfy following assumptions for $0\leq \varepsilon<\eps_*$:
\medskip
 
\noindent \textbf{(A0)}  \textit{Twist invariant curve:} For any $(\varphi,J)\in\mathbb A_{\rho,\sigma}$    $T_0$ is of the form
\begin{equation}\label{eq:T0map}
T_0:\begin{pmatrix}\varphi\\J\end{pmatrix}\mapsto \begin{pmatrix}\varphi+\beta+   \tau J+R_\varphi(\varphi,J;\varepsilon)\\ J+R_J(\varphi,J;\varepsilon)\end{pmatrix}
\end{equation}
with $\tau>0$,
\begin{equation}\label{eq:Diophantinetype}
\beta \in \mathcal B_\alpha=\left\{\beta\in\mathbb R\colon |\beta-p/q|\geq \alpha|q|^{-2}\colon \ p\in\mathbb Z,\ q\in\mathbb N\setminus\{0\}\right\}
\end{equation}
for some $\alpha>0$, and $R_\varphi,R_J$ satisfying, for any $\varphi\in\mathbb{T}$, $\partial_J^{n} R_{*}(\varphi,0;\varepsilon)=0$ for $n=0, 1$ if $*=\varphi$ and $ n=0,1, 2$ if $*=J$.
\medskip

\noindent\textbf{(A1)} \textit{Transversality:} Fix an open neighborhood  $B$ of $\{\varphi=J=0\}$ not depending on $\varepsilon$. Then,  for  $(\varphi,J)\in B$, the map $T_1$ is of the form
\begin{equation}\label{eq:T1map}
T_1:\begin{pmatrix}\varphi\\J\end{pmatrix}\mapsto \begin{pmatrix}\varphi+\tilde\beta(J)+\varepsilon \Tphi(\varphi,J;\varepsilon)\\ J+ \varepsilon \Tj(\varphi,J;\varepsilon)\end{pmatrix}
\end{equation}
with   
\begin{equation}\label{eq:transversality}
\Tj(0,0;0)=0\qquad\qquad \partial_\varphi \Tj(0,0;0)=1.
\end{equation}
\medskip

\noindent\textbf{(A2)} \textit{Regularity with respect to parameters:} The maps $T_0,T_1$ depend $C^1$ on $\varepsilon\in (0,\eps_\star)$. 
\medskip

We denote by 
\begin{equation}\label{eq:maxnorms}
K=\sup_{\eps\in (0,\eps_\star)}\max\{|R_\varphi|_{\rho,\sigma}, |R_J|_{\rho,\sigma},|\tilde\beta'|_{C^0(B)},|\Tphi|_{C^2(B)},|\Tj|_{C^2(B)}\}.
\end{equation}
\medskip

\begin{thmx}\label{thm:transitivityIFS}
Let $T_0,T_1$ verify \textbf{(A0)}-\textbf{(A2)} for some $\alpha,\rho,\sigma>0$ and let $K>0$ be as in \eqref{eq:maxnorms}. 
There exists $\varepsilon_0(K,\rho,\sigma)>0$ such that for any $
0<\varepsilon\leq\varepsilon_0 \min\{\tau,\alpha\}$ the IFS generated by $\{T_0,T_1\}$ satisfies the following. For any pair of open balls 
\begin{equation}\label{eq:definitionAgammamainthm}
B_1,B_2\subset \widetilde{\mathbb A}_\alpha:= \mathbb T\times[-\alpha/|\log^3\varepsilon|,\alpha/|\log^3\varepsilon|]
\end{equation}
there exists $M\in\mathbb N$ and $\omega\in\{-1,1\}^M$ such that, $T_{\omega_{M-1}}\circ\cdots \circ T_{\omega_0} (B_1)\cap B_2\neq \emptyset$. In particular, there exists an orbit of the IFS generated by $\{T_0,T_1\}$ for which both its forward and backward closure contains the annulus $\widetilde{\mathbb A}_\alpha$. Moreover, the very same conclusion holds true for any pair $(\widetilde T_0,\widetilde T_1)$ in a open $C^1$-neighborhood of $(T_0,T_1)$.
\end{thmx}

An important observation is that we allow both the Diophantine constant $\alpha$ and the torsion $\tau$ to be arbitrarily small as long as $\varepsilon$ (which measures the transversality between the maps) is much smaller. This will be crucial for applications to degenerate systems and, in particular, for the application to the 3-body problem. Note that Theorem \ref{thm:transitivityIFS}  applies to constant type Diophantine numbers (see \eqref{eq:Diophantinetype}). It can be generalized to all Diophantine numbers but the version above is enough for our purposes.

Another important remark is that we are able to control the dynamics on the annulus $\widetilde {\mathbb A}_\alpha$ whose size is only logarithmically small in $\varepsilon$, and not only on a $O(\varepsilon)$-neighbourhood of the KAM curve $\{J=0\}$ of the map $T_0$.

\begin{rem}\label{rem:c3maps}
    We will also see from the proof of Theorem \ref{thm:transitivityIFS} that if the maps $T_0,T_1$ are only $C^3$, then, there exists $\varepsilon_0(\widetilde K)$ where $\widetilde K=\max\{|R_\varphi|_{C^2(\mathbb A)}, |R_J|_{C^3(\mathbb A)},|\tilde\beta'|_{C^0(B)},|\Tphi|_{C^2(B)},|\Tj|_{C^2(B)}\}$ and $\kappa>0$ (independent of $\varepsilon$ and $\alpha$) such that the same conclusion in Theorem \ref{thm:transitivityIFS} holds on the smaller annulus $\mathbb T\times[-\kappa\varepsilon,\kappa\varepsilon]$.
\end{rem}

\subsubsection*{Appications}
In the context of 4-dimensional symplectic maps the maps $T_0$ and $T_1$ above can be thought of as:
\begin{itemize}
    \item $T_0$: the inner map (i.e. the restriction of the dynamics) on a 2-dimensional normally hyperbolic invariant cylinder; $T_1$: the scattering map to the same cylinder (see \cite{DelaLLavegaps,DelaLLaveScattmap}) or,
    \item $T_0,T_1$: two different scattering maps to the same normally hyperbolic invariant cylinder.
\end{itemize}
Indeed, we strongly believe that, for small perturbations of both \textit{a priori stable} and \textit{a priori unstable} Hamiltonians, it is in general possible to recover the scenario described in the first item with maps $T_0,T_1$ satisfying the assumptions in Theorem \ref{thm:transitivityIFS}. 

For the 3-body problem (and its restricted version), we will study a strongly degenerate scenario in which the inner dynamics is given by the identity map. However, we will see that there exist two different scattering maps which, locally, satisfy the assumptions in Theorem \ref{thm:transitivityIFS}.

\subsection*{Skew-products over the shift}
We now consider a particular class of skew-products and give an extension of Theorem \ref{thm:transitivityIFS}. To do so, we introduce the following notation. Given $N\in\mathbb N$ and \[
\omega=(\dots,\omega_{-1},\omega_0;\omega_1,\omega_2,\dots)\in \{0,1\}^\mathbb Z
\]
let 
\begin{equation}\label{eq:Ncylinder}
C_N(\omega)=\{\omega'\in\{0,1\}^\mathbb Z\colon \omega_k'=\omega_k,\ \text{for all }-N+1\leq k\leq N\}.
\end{equation}
We say that a set $U\subset \{0,1\}^\mathbb Z\times \mathbb A$, where $\mathbb A=\mathbb T\times[-1,1]$, is a $N$-cylinder if $U=C_N(\omega)\times B$ with $C_N(\omega)$ as in \eqref{eq:Ncylinder} and $B\subset  \mathbb A$ an open set.

\begin{defn}\label{eq:countablecoverbyncylinders}
   Fix any $N\in\mathbb N$. We say that a countable covering $\{\mathcal U_k\}_{k\in\mathbb N}$ of $\{0,1\}^\mathbb Z\times\mathbb A$   is a \textit{countable covering by $N$-cylinders} if all the elements $\mathcal U_k$ are $N$-cylinders.
\end{defn}
The following is our second main result.
\begin{thmx}\label{thm:skewproduct}
    Let $T_0,T_1$ verify \textbf{(A0)}-\textbf{(A2)} for some $\rho,\sigma>0$ and let $K>0$ be as in \eqref{eq:maxnorms}. Let $\delta>0$ be arbitrarily small and consider a skew-product map 
    \begin{align*}
    \mathcal F: \{0,1\}^\mathbb Z\times \mathbb A &\to \{0,1\}^\mathbb Z\times (\mathbb T\times\mathbb R)\\
    (\omega,z)&\mapsto (\sigma(\omega),F_\omega(z))
    \end{align*}
    where $\sigma:\{0,1\}^\mathbb Z\to \{0,1\}^\mathbb Z$ is the full-shift and, for any $(\omega,z)\in \{0,1\}^\mathbb Z\times \mathbb A$, the map $z\mapsto F_\omega(z)$ satisfies:
    \begin{itemize}
        \item \textit{($C^1$ approximation)}: $F_\omega(z)=T_{\omega_0}(z)+O_{C^1}(\delta)$.
        \item \textit{(Weak coupling)}: For any $n\in\mathbb{N}$ if $\omega,\omega'$ satisfy that $\omega'\in C_n(\omega)$, uniformly for all $z\in\mathbb A$
        \begin{equation}\label{eq:almostlocallyconstant}
        |F_\omega(z)-F_\omega'(z)|\leq \delta^n.
        \end{equation}
    \end{itemize}
    Fix any $N>0$. Under the above hypotheses, there exist $\varepsilon_0(K,\rho,\sigma, N)>0$ and $\delta_0(\varepsilon)>0$ such that for any 
    \[
 0<\varepsilon\leq \varepsilon_0(K,\rho,\sigma,N)\min\{\tau,\alpha \}\qquad\qquad 0\leq \delta\leq\delta_0(\varepsilon),
    \]
    given any pair of sequences $\omega,\omega'\in \{0,1\}^\mathbb Z$ and any pair of open balls $B,B'\subset \widetilde{\mathbb A}_\alpha$ (see   \eqref{eq:definitionAgammamainthm}) there exists $M\in\mathbb N$ such that
    \[
  \mathcal F^M( C_N(\omega)\times B)\cap C_N(\omega')\times B'\neq \emptyset.
    \]
    In particular, given any countable covering $\{\mathcal U_k\}_{k\in\mathbb N}$ of $\{0,1\}^\mathbb Z\times \widetilde{\mathbb A}_\alpha$ by $N$-cylinders, there exists an orbit of $\mathcal F$ which visits all the $\mathcal U_k$.
\end{thmx}

Before proceeding, this result calls for a few comments. First, we notice that the lamination $\{0,1\}^\mathbb Z\times \mathbb A$ is only weakly invariant for the map $\mathcal F$ (there might exist $(\omega,z)$ such that $F_\omega(z)\notin \mathbb A$). In this situation, one  actually expects that ``most of the orbits'' escape from $\{0,1\}^\mathbb Z\times\mathbb A$ through the boundary of  $\mathbb A$. What we prove in Theorem \ref{thm:skewproduct} is that, still, given any $N\in\mathbb N$, provided $\varepsilon$ is small enough, there exists an orbit which visits any element from any countable covering of $\{0,1\}^\mathbb Z\times\widetilde{\mathbb A}_\alpha$ by $N$-cylinders. Second, it is also worth pointing out that,  for $T_0,T_1$ with finite regularity we also get a similar conclusion but on a smaller annulus (see Remark \ref{rem:c3maps}). This will be clear from the proof.

\subsubsection*{Applications}
The map $\mathcal F$ in Theorem \ref{thm:skewproduct} can be thought of as the restriction of a 4-dimensional map to a weakly invariant  normally hyperbolic lamination which accumulates on two different homoclinic channels.  Weakly invariant laminations appear naturally, close to resonances, in perturbations of non-degenerate integrable Hamiltonians. A particularly interesting feature in Theorem \ref{thm:skewproduct} is that the map $\mathcal F$ is not locally constant. This is important for applications  to laminations arising from real-analytic Hamiltonians. In particular, Theorem \ref{thm:skewproduct} is the main ingredient in the proof of Theorem \ref{thm:MainR3bp} below.

\subsection*{The weak transversality-torsion mechanism}
Originally, the transversality-torsion mechanism was introduced to create isolating blocks for (compositions of ) maps as $T_0,T_1$ above. Roughly speaking the idea can be understood as follows. Let 
\[
L_0=\begin{pmatrix}
    1&\tau\\
    0&1
\end{pmatrix}\qquad\qquad L_1=\begin{pmatrix}
    0&-1\\
    1&0 \end{pmatrix}
\]
The former is a parabolic (shear) matrix while the second one is a rotation. In particular, none of them are hyperbolic. However, an straightforward computation shows that 
\[
L_0^n L_1=\begin{pmatrix}
    n\tau&-1\\
    1&0 \end{pmatrix}
\]
so $L_0^nL_1$ is hyperbolic for $n\tau>2$.
In the setting of maps of the annulus $T_0,T_1$ described above, one can think of $L_0$ as the matrix $D T_0|_{\{J=0\}}$. On the other hand, if one assumes that $\partial_\varphi T_{1,J}|_{\{\varphi=J=0\}}\neq 0$ the map $T_{1}$ acts, close to $\{\varphi=J=0\}$, as a rotation. Indeed, after a affine change of variables (which becomes singular as $\varepsilon\to 0$) one may conjugate $D T_1|_{\{\varphi=J=0\}}$ to the matrix $L_1$. Thus, for fixed $\varepsilon>0$ an isolating block $D$ is obtained for the map $T_0^n\circ T_1$ provided $n$ is taken large enough with respect to $1/\varepsilon$ (see Figure \ref{fig:Fig1}). This argument was used by Cresson in the context of Arnold diffusion \cite{MR1949441} (see also \cite{GideaL06}) and later appeared in \cite{guardia2022hyperbolicdynamicsoscillatorymotions} in order to construct non-trivial hyperbolic sets in the 3-body problem (see Section \ref{sec:intro3bp}).

\begin{figure}
    \centering
    \includegraphics[scale=0.5]{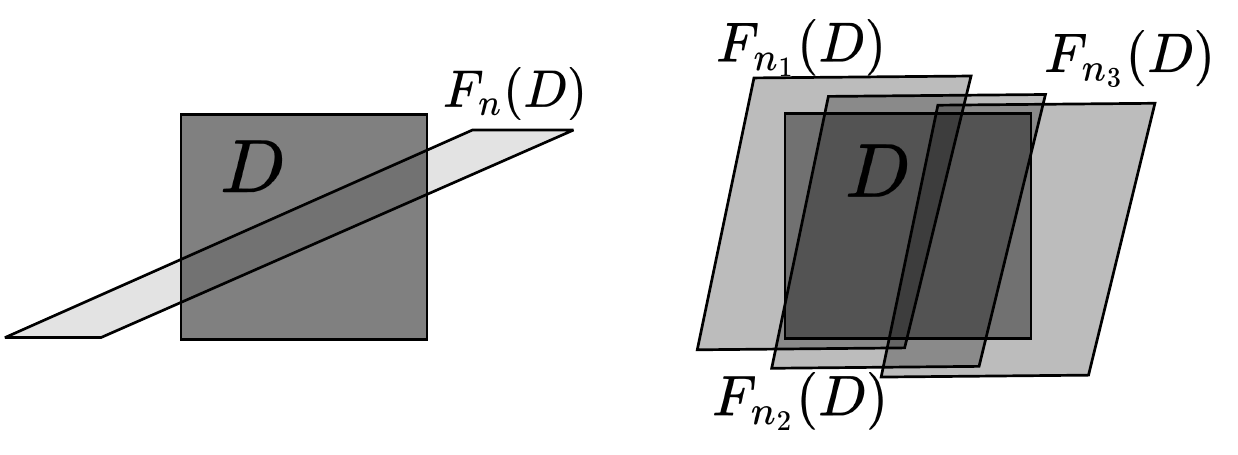}
    \caption{Let $F_n=T_0^n\circ T_1$. In the left we show the image of a small rectangle $D$ under the map $F_n$ for $n\gg 1/\varepsilon$. In the right we show the image of $D$ under $F_{n_i}$, $i=1,2,3$ for  $n_i\sim 1/\varepsilon$. In this regime the expansion/contraction is arbitrarily close to one. Moreover, if $\beta$ is sufficiently irrational it is possible to chose $n_i$ such that the union  $\bigcup_{i}F_{n_i}(D)$ contains $D$.}
    \label{fig:Fig1}
\end{figure}

\subsubsection{Weak transversality}\label{sec:weaktransvintro}
For fixed  $\varepsilon>0$, in the limit $n\to \infty$, the rate of expansion/contraction for the map $T_0^n\circ T_1$ becomes infinite. In particular, the intersection of the isolating block $D$ with itself corresponds to a very thin rectangle within $D$.

The key idea behind the proof of Theorem \ref{thm:transitivityIFS} is to look at a range $n\in\mathcal N\subset\mathbb N$ of iterates for which the rate of expansion/contraction for the maps $T_0^n\circ T_1$ is uniformly (for all $n\in\mathcal N$) close to one. In this setting, the overlapping between the isolating block and its image is large. Moreover, as the dynamics of $T_0$ is driven by a strongly irrational rotation, we will be able to show that there exists a subset $\mathcal N_*\subset \mathcal N$ for which the image of the isolating block under the corresponding maps come very close to the isolating block itself. In particular, we will see that the union of the images of the isolating block under the different maps in $\mathcal N_*$ covers the isolating block (see Figure \ref{fig:Fig1}). 

The idea of working in a framework with weak transversality was first considered by Li and Turaev in \cite{LiTuraevPreprint}, in which the authors analyze the dynamics in a neighborhood of a cubic tangency between the invariant curve of the map $T_0$ and its image by the map $T_1$. 
\medskip

\subsection{The 3-body problem}\label{sec:intro3bp}
We now abandon the abstract setting and introduce a dynamical system which has played a mayor role in the evolution of the modern theory of dynamics, the so-called 3-body problem. 
 For $i=0,1,2$, let $(q_i,p_i)\in  T^*(\mathbb R^2)$ be Cartesian coordinates and let $m_i>0$. The planar \textit{3-body problem} is the Hamiltonian system defined by 
\begin{equation}\label{eq:3bpHam}
H(q,p)=\sum_{0\leq i\leq 2} \frac{|p_i|^2}{2m_i}-\sum_{0\leq i<j\leq 2} \frac{m_im_j}{|q_i-q_j|}
\end{equation}
on the symplectic manifold $(q,p)\in T^*(\mathbb R^6\setminus \Delta)$ equipped with the canonical symplectic form $\mathrm{d}p\wedge \mathrm dq$. Here $|\cdot|:\mathbb R^2\to \mathbb R_+$ is the Euclidean norm and we have defined the collision locus 
$\Delta=\{q=(q_0,q_1,q_2)\in \mathbb R^6\colon q_i=q_j \text{ for some } i\neq j\}$. The flow of \eqref{eq:3bpHam} preserves:
\begin{itemize}
    \item (time translation symmetry): the total energy $H(q,p)$,
    \item (space translation symmetry): the total linear momentum $\bs p(p)=\sum_{0\leq i\leq 2} p_i\in\mathbb R^2$ and,
    \item (rotational symmetry): the total angular momentum $\Theta(q,p)=\sum_{0\leq i\leq 2} \Theta_i\in \mathbb R$ where $\Theta_i=q_{i,x}p_{i,y}-p_{i,x}q_{i,y}$,
\end{itemize}
Hence, for any $H_0,\Theta_0\in\mathbb R$, the Hamiltonian \eqref{eq:3bpHam} induces a complete flow\footnote{To be precise, the flow is only complete for initial conditions which do not lead to collision. This  is a full measure subset of $\mathcal M(H_0,\Theta_0)$ as a consequence of the fact that binary collisions are regularizable and triple collision cannot happen for $\Theta\neq 0$ (see also \cite{MR295648,MR4011691}). Moreover, our analysis will take place on regions very far from the collision locus.} 
\begin{equation}\label{eq:5dflow}
\begin{split}
\phi_H:\mathcal M(H_0,\Theta_0)\times \mathbb R&\to \mathcal M(H_0,\Theta_0)\\
z&\mapsto \phi^t_H(z)
\end{split}
\end{equation}
on the 5-dimensional submanifold 
\begin{equation}\label{eq:energylevel}
\mathcal M(H_0,\Theta_0)=\{(q,p)\in  T^*(\mathbb R^6\setminus\Delta)\colon H(q,p)=H_0,\ \bs p(p)=0,\ \Theta(q,p)=\Theta_0\}/SE(2).
\end{equation}
Here $SE(2)$ is the special Euclidean group on the plane, which acts on $(q,p)\in T^*(\mathbb R^6\setminus \Delta)$ by 
\[
(v,g)\mapsto (v+gq_0,v+g q_1,v+g q_2, gp_0,g p_1,g p_2)
\]
for $v\in\mathbb R^2$ and $g$ a rotation on $\mathbb R^2$.

\subsection*{Homoclinic bifurcations and oscillatory motions}

For centuries, researchers have put considerable effort in understanding the global dynamics of the 3-body problem. A remarkable result in this direction was obtained by Chazy \cite{MR1509241}, who classified, from a qualitative point of view, all the possible asymptitotic behaviors of its complete (i.e. defined for all time) orbits. Among the seven possible behaviors, perhaps the most exotic\footnote{Indeed, all the other 6 cases can already be found in the limit $m_1,m_2\to 0$ for which the system decouples as the (uncoupled) sum of two 2-body problems. On the contrary oscillatory motions are an inherent feature of the 3-body problem.} case corresponds to the so-called \textit{oscillatory orbits}. These are orbits along which at least one of the mutual distances is unbounded while all the bodies come together infinitely many times. We refer to 
\textit{forward} (resp. \textit{backward}) oscillatory orbits if this qualitative behavior happens as $t\to \infty$ (resp. $t\to -\infty$) and we reserve the term oscillatory orbit for those which exhibit this qualitative behavior for both $t\to\pm\infty$. 

The first existence result for this type of oscillatory orbits was obtained by Sitnikov in a simplified (restricted) model \cite{Sitnikov60}. Years later, Alekseev \cite{AlekseevQR1,AlekseevQR2} and Moser \cite{MR442980} put the study of oscillatory orbits on a more geometric framework. Using McGehee's compactification of phase space \cite{McGeheestablemanifold}, they noticed that, again for the simplified model studied by Sitnikov\footnote{Alekseev actually studied Sitnikov's configuration but in the non-restricted case, i.e. with all masses positive. Still the symmetry of the configuration allows one to reduce the system to a three-dimensional flow.}, oscillatory motions are indeed a subset of the \textit{homoclinic class} asociated to a topological saddle at ``infinity''. More precisely, Sitnikov's model, which describes the motion of a massless particle in a certain gravitational field, corresponds to a time-periodic perturbation of a one degree-of-freedom Hamiltonian (i.e. a three-dimensional flow).  McGehee's compactification glues a manifold of periodic orbits at infinity (the massless particle is infinitely far from the sources of the gravitational field). Although degenerate from the dynamical point of view, one of these orbits posseses stable and unstable invariant manifolds: these correspond to the set of initial conditions for which the particle escapes to infinity with zero asymptotic velocity. By showing that these manifolds intersect transversally, they were able to embed the full-shift acting on the space of bi-infinite sequences as a factor for this system. From the construction of the coding, oscillatory orbits could be identified with sequences $\{a_n\}\subset \mathbb N^\mathbb Z$ for which $\limsup a_n\to \infty$ both as $n\to \pm\infty$.

Extending these results to the full 3-body problem is a non-trivial task due to the increase in dimension (unless one considers symmetric configurations, which lower the dimension of the model, as Alekseev did \cite{AlekseevQR1, AlekseevQR2}): as we have seen above, the symplectically reduced flow \eqref{eq:5dflow} takes place on a five-dimensional manifold. The first results on existence of oscillatory motions for the planar 3-body problem \eqref{eq:3bpHam} were obtained by Moeckel in \cite{MR2350333} (a partial proof for one-sided orbits, using a different idea, was given before by Xia in \cite{MR1278373}). These results applied only to non-generic choices of the masses $m_i\in \mathbb R_+$, $i=0,1,2$. Recently, the existence of oscillatory motions on the 3-body problem for any value of the masses (except all of them equal) has been established in \cite{guardia2022hyperbolicdynamicsoscillatorymotions}.
\medskip 

The approach taken in \cite{guardia2022hyperbolicdynamicsoscillatorymotions} can be seen as a generalization to higher dimensions of the ideas of Alekseev and Moser. To describe the main ideas we need to first introduce McGehee's partial compactification in the context of the 3-body problem.

\subsubsection*{McGehee's partial compactification for the 3-body problem}
To reduce the translation invariance we introduce the (symplectic) change of coordinates\footnote{The action of $\mathbb R^2$ here is understood as a diagonal translation in $q=(q_0,q_1,q_2)$.} 
\begin{align*}\Psi:(r,y)\in\mathbb R^8&\mapsto (q,p)\in (T^*(\mathbb R^6) \cap \{\bs p=0\})/\mathbb R^2
\end{align*}
given by $r_{1}=q_1-q_0$, $r_{2}=q_2-q_0$ and $y_i=p_i$ for $i=1,2$. Then, we can define McGehee's map
\begin{align*}
\Phi:B &\mapsto T^*(\mathbb R^4\setminus \Delta) \\
(x, y)&\mapsto (r,y)
\end{align*}
given by $
r_1=x_1$, $r_2=\frac{2}{|x_2|^3} x_2$,  and consider the boundary submanifold 
\[
B_\infty=\{(x,y)\in B\colon |x_2|=0\}\subset \partial B
\]
corresponding to configurations for which the  mutual distance between the third and inner bodies (i.e. $r_{2}$) becomes infinite. On the partially compactified configuration space 
\[
\mathcal B=B\sqcup B_\infty.
\]
the Hamiltonian flow defined by  $\mathcal H=H\circ\Psi\circ\Phi$ and the (singular) non-degenerate symplectic form $\omega=\Phi^*(\mathrm{d}y\wedge \mathrm dr)$:
\begin{itemize}
    \item extends continuously up to $B_\infty$ (on which it defines a non-trivial flow!),
    \item leaves $B_\infty$ invariant.
\end{itemize}
Finally, in order to reduce the invariance by rotations, we observe that on  $\mathcal B$ one may define the action of $SO(2)$ by $
g\mapsto (g x_1,g x_2,gy_1,gy_2)$ with $g\in SO(2)$ (notice that the action extends continuously up to $B_\infty$ and $g (B_\infty)=B_\infty$). Thus, we can fix any $H_0<0$ and any $|\Theta_0|>0$ and define the reduced McGehee flow 
\begin{equation}\label{eq:5dflowextended}
\overline \phi_{\mathcal H}:\overline {\mathcal M}(H_0,\Theta_0)\times\mathbb R\to \overline {\mathcal M}(H_0,\Theta_0)
\end{equation}
where $\overline{\mathcal M}(H_0,\Theta_0)$ is the five dimensional manifold 
\begin{equation}\label{eq:compactifiedmanifold}
\overline {\mathcal M}(H_0,\Theta_0)=\{(x,y)\in\mathcal B\colon \mathcal H(x,y)=H_0,\  \Theta\circ\Psi\circ\Phi(x,y)=\Theta_0\}/SO(2).
\end{equation}

\subsubsection*{A normally-parabolic manifold for McGehee's flow:} For the flow \eqref{eq:5dflowextended}, which extends \eqref{eq:5dflow} continuously (and in a non-trivial way) it is a classical result  \cite{MR744302} (see also \cite{BFM20a,BFM20b}) that the 3-dimensional submanifold 
\begin{equation}\label{eq:ellipticinfinity}
    \mathcal E_\infty(H_0,\Theta_0)= \{(x,y)\in \overline {\mathcal M}(H_0,\Theta_0)\colon |x_2|=0,\ \langle y_2,x_2/|x_2|\rangle =0\}
\end{equation}
possesses four-dimensional stable and unstable invariant manifolds (for the flow \eqref{eq:5dflowextended}): these correspond to orbits along which $q_2$  escapes to infinity with zero asymptotic velocity in the future (stable) or in the past (unstable).
The submanifold $\mathcal E_\infty$ is diffeomorphic\footnote{To be precise, \eqref{eq:ellipticinfinity} is diffeomorphic to $\mathbb S^3$ after regularizing the binary collision between $q_0,q_1$.} to $\mathbb S^3$ (see \cite{MR744302}) and corresponds to configurations of the bodies for which
\begin{itemize}
    \item the two inner bodies $q_0,q_1$ revolve around each other on a Keplerian ellipse of fixed  semimajor axis (determined entirely by $H_0$) and which can be parametrized by its eccentricity $\epsilon\in [0,1)$ and angle of the pericenter $g\in\mathbb T$.
    \item $q_2$ is located infinitely far from the two inner bodies.
\end{itemize}
The submanifold $\mathcal E_\infty$ is \textit{degenerate} in two senses:
\begin{itemize}
    \item it is foliated by periodic orbits of \eqref{eq:5dflowextended} (the two inner bodies revolve around each other following the Kepler laws),
    \item all the eigenvalues of the linearization of \eqref{eq:5dflowextended} at $\mathcal E_\infty$ are zero except along the flow direction.
\end{itemize}
\medskip
 
 \subsubsection*{Embedding the full shift in the 3-body problem}
 A natural strategy to embed the full shift on the flow \eqref{eq:5dflow} is to show that the four-dimensional manifolds $W^{u,s}(\mathcal E_\infty)$ intersect transversally and study the return map $\Psi:U\to U$ to a suitable four-dimensional section $U$ accumulating on $W^{u}(\mathcal E_\infty)\cap W^{s}(\mathcal E_\infty)$. However, in order to follow this route, the authors in \cite{guardia2022hyperbolicdynamicsoscillatorymotions} faced a significant challenge. Namely, as a consequence of the degeneracy of the flow on $\mathcal E_\infty$,   the dynamics of the return map $\Psi$ in the center directions (those tangent to the manifold $\mathcal E_\infty$) is given by a close to identity (symplectic) map. Hence, spotting any sign of hyperbolicity along these directions is a rather complicated task. To overcome this issue and gain hyperbolicity along the center directions, the authors considered a  suitable composition of return maps $\Psi_{i\to j}:U_i\to U_j$ between different sections $U_i\subset U$, $i,j=0,1$ accumulating on different homoclinic channels and which reproduce the so-called \textit{transversality-torsion} mechanism described in  Section \ref{sec:ifsintro}. 

In this way, the authors in \cite{guardia2022hyperbolicdynamicsoscillatorymotions} constructed a non-compact \textit{uniformly hyperbolic} set (for the return map $\Psi$) accumulating on $\mathcal E_\infty$ and on which the dynamics is conjugated to the full shift acting on the space of bi-infinite sequences of countable symbols. From this conjugacy, oscillatory orbits of the 3-body problem can be extracted from the symbolic coding as explained above.

\subsection*{A partially hyperbolic setting and symplectic blenders}\label{sec:intropartiallyhy}
In the present work we will push the ideas above further and show that a \textit{symplectic blender} (see Definition \ref{def:symplecticblender} below) exists close to a homoclinic manifold associated to $\mathcal E_\infty$. Roughly speaking, the blender construction relies on the \textit{weak transversality-torsion} mechanism described in Section \ref{sec:ifsintro}. First, we observe that for a judiciously chosen range of compositions of maps $\Psi_{i\to j}:U_i\to U_j$ as above, one can reproduce the weak transversality-torsion mechanism described in Section \ref{sec:ifsintro}. In particular, one can obtain a family of partially hyperbolic maps $\{\Psi^{(N)}\}$ with arbitrarly weak hyperbolicity in the center directions. Second, by centering the sections $U_i$ on a region where the center dynamics is driven by a strongly irrational rotation, one can guarantee that the images of these sections by the family of maps $\{\Psi^{(N)}\}$ is well-distributed and covers the center directions.

This leads to the third main theorem of the paper.

\begin{thmx}\label{thm:mainblender}
     Fix any value of the masses $m_0,m_1,m_2>0$ (except $m_0=m_1$). Let $H_0<0$ and fix any $|\Theta_0|\gg 1$. Then, there exists a 4-dimensional section $\mathcal X\subset \mathcal M(H_0,\Theta_0)$ such that the return map $\Psi:U\subset\mathcal X\to \mathcal X$ induced on $\mathcal X$ by the flow \eqref{eq:5dflowextended} exhibits a symplectic blender.
\end{thmx}

\subsection*{Blenders and the dynamics of the 3-body problem}\label{sec:introapplicationsblender}

Although already interesting on their own, symplectic blenders can have a global influence on the dynamics of the system. In particular, they allow us to construct an oscillatory orbit whose projection on the center directions (i.e. those tangent to $\mathcal E_\infty$) is ``large''. Before giving a precise statement let us first motivate the study of these orbits.


Consider the planar 2-body problem Hamiltonian 
\begin{equation}\label{eq:2bpHam}
h_{\mathrm{Kep}}(q,p)=\sum_{ i=0,1} \frac{|p_i|^2}{2m_i}- \frac{m_0m_1}{|q_0-q_1|}.
\end{equation}
It is well known that this system is integrable (see \cite{MR2269239}). If we fix any $h\in\mathbb R$ the Hamiltonian \eqref{eq:2bpHam} induces a  flow
\[
\phi_{h_{\mathrm{Kep}}}:\mathcal S(h)\times\mathbb R\to \mathcal S(h)
\]
on the 3-dimensional manifold 
\[
\mathcal S(h)=\{(q,p)\in T^*(\mathbb R^2\setminus \Delta)\colon h_{\mathrm{Kep}}(q,p)=h, p_0+p_1=0\}/\mathbb R^2 
\]
where $\mathbb R^2$ acts on $(q,p)$ by diagonal translation on $q=(q_0,q_1)$. Moreover, for $h<0$ (after regularizing collisions) 
\[
\mathcal S(h)\simeq \mathbb S^3
\]
and $\mathcal S(h)$ is foliated by periodic orbits for the flow $\phi_{h_{\mathrm{Kep}}}$. One may introduce a (local) coordinate system $(\ell,g,\epsilon)\in\mathbb T^2\times[0,1)$ on $\mathcal S(h)$ such that, in the full phase space, the vector $q_1(t)-q_0(t)$ traces a \textit{fixed ellipse} $\mathcal E(g,\epsilon;h)$ which can be parametrized by its semimajor axis (completely determined by $h$), the argument  of its pericenter  $g\in \mathbb T$, (with respect to a fixed direction)  and its eccentricity $\epsilon\in[0,1)$, while the position inside $\mathcal E(g,\epsilon;h)$ is specified by an angle $\ell\in\mathbb T$ (see Figure \ref{fig:Fig2}). In the coordinates $(\ell,g,\epsilon)$ the flow $\phi_{h_{\mathrm{Kep}}}$ is given by a linear (resonant) translation
\begin{equation}\label{eq:linearkeplerflow}
\phi^t_{h_{\mathrm{Kep}}}:(\ell,g,\epsilon)\in \mathcal S(h)\mapsto (\ell+\omega(h)t,g,\epsilon)\in\mathcal S(h).
\end{equation}
for some $\omega(h)\in\mathbb R$.
\medskip 

\begin{figure}
    \centering
    \includegraphics[scale=0.5]{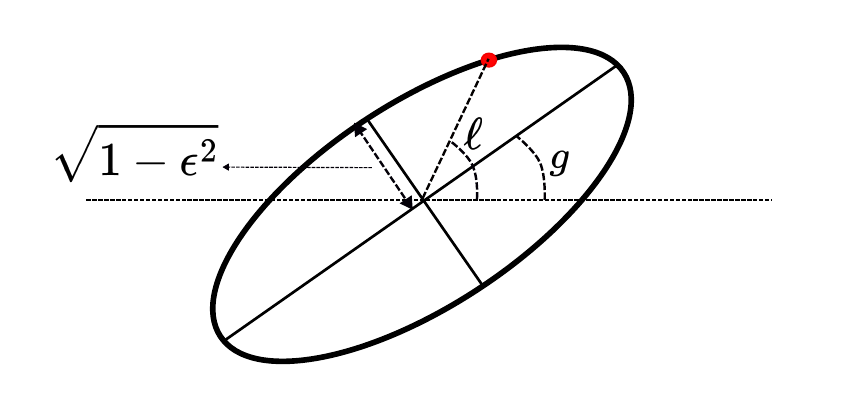}
    \caption{An ellipse with unit semimajor axis, eccentricity $\epsilon\in[0,1)$ and argument of the pericenter $g\in\mathbb T$. The position of the red point inside the ellipse is measured by the angle $\ell\in\mathbb T$.}
    \label{fig:Fig2}
\end{figure}

For the Hamiltonian $H$ in \eqref{eq:3bpHam} the relative motion between $q_0$ and $q_1$ is much more erratic as these bodies now interact with $q_2$. A popular practice to make the problem more tractable is to study the motion in the \textit{hierarchical region} of the phase space  where $|q_2|\gg |q_0|,|q_1|$. The reason is that, in view of the fast decay of Newonian interaction, in this region one still expects that (at least for short time scales) the motion of $q_0,q_1$ is governed by the Hamiltonian $h_{\mathrm{Kep}}$. Hence, in this region of the phase space, we expect that the bodies $q_0,q_1$ still move on  ellipses whose shape \textit{evolve slowly} in time. As we will see below, for fixed total energy $H_0$ and angular momentum $\Theta_0$, one may introduce local coordinate patches on the 5-dimensional manifold \eqref{eq:compactifiedmanifold}
\[
\psi:U\times \mathbb S^3\subset \mathbb R^2\times\mathbb S^3\to \overline{\mathcal M}(H_0,\Theta_0)\cap \{|q_2|\gg |q_0|,|q_1\}
\]
and describe the evolution of $q_1-q_0$ in terms of elliptic elements\footnote{Note that the argument of the pericenter becomes undefined at circular motion. In the paper we use Poincar\'e coordinates which extend to circular motion (see Section \ref{sec:goodcoordinates}).  } $(\ell,g,\epsilon)\in \mathbb T^2\times[0,1)$\footnote{The  reader is probably wondering why we do not include the semimajor axis  in our coordinate system. The reason is that, as we will see below, this quantity is approximately constant on a neighborhood of $\mathcal E_\infty\subset \overline {\mathcal M}(H_0,\Theta_0)$, which is where our analysis will take place. Hence, on this region the semimajor axis can be recovered from the value of $H_0,\Theta_0$ and the other coordinates. 

It is an open problem to construct  orbits of the 3-body problem (far from $\mathcal E_\infty$) along which the semimajor axis presents significant variations.}. In view of the above discussion, we expect that the projection of the flow \eqref{eq:5dflowextended} onto the coordinates $(\ell,g,\epsilon)$ is given by a small perturbation of the linear flow \eqref{eq:linearkeplerflow}.
\medskip

A natural question is:
\begin{align*}
\textit{Describe the set  of elliptical elements $\{(\ell(t),g(t),\epsilon(t))\colon t\in\mathbb R\}$ which can be seen along  orbits of \eqref{eq:3bpHam}.} 
\end{align*}
In view of the above connection between oscillatory motions and chaotic dynamics, these of orbits seem good candidates for displaying a rich set of elliptical elements. Our main result shows that, at least locally, this is indeed the case. The following is a informal version of our second main result.

\begin{thmx}[Informal version]
    Fix any value of the masses $m_0,m_1$ and $m_2$ (expept $m_0=m_1$). There exists an oscillatory orbit of \eqref{eq:3bpHam} whose projection to the shape space of planar oriented ellipses $\mathcal S=\{(\ell,g,\epsilon)\in\mathbb T^2\times[0,1)\}$ contains a locally dense subset. Moreover, the same is true if we just consider the projection of its future or backwards semi-orbit.
\end{thmx}

A more precise and geometric version of this result is given in below.

\medskip

\subsection*{Orbits accumulating a  normally-parabolic invariant manifold}

\setcounter{thmx}{4} 
\addtocounter{thmx}{-1} 
In this section we present a more geometric version of our main result. We fix any $H_0<0$ any $|\Theta|>0$ and let $\mathcal E_\infty(H_0,\Theta_0)$ be the 3-dimensional submanifold in \eqref{eq:ellipticinfinity}, which, we recall is invariant for the extended flow \eqref{eq:5dflowextended}.

\begin{thmx}[Geometric version]\label{thm:Main3bp}
Fix any value of the masses $m_0,m_1$ and $m_2$ (expept $m_0=m_1$). Let $H_0<0$ and fix any $|\Theta_0|\gg 1$. There exists an orbit of the five-dimensional extended flow \eqref{eq:5dflowextended} such that both its forward and backward closure contain an open  subset $\mathcal A_\infty\subset\mathcal E_\infty$ of the three-dimensional normally-parabolic invariant manifold $ \mathcal E_\infty(H_0,\Theta_0)$.
\end{thmx}

This result can be recognized as a strong form of (micro) Arnold diffusion. The classical Arnold diffusion construction (in the context of topological instability of nearly-integrable Hamiltonian systems) considers orbits which shadow a (in general finite) pseudo-orbit connecting a long sequence of partially hyperbolic tori inside a normally-hyperbolic manifold (see \cite{MR163026}). The main features of our result compared to classical results on Arnold diffusion are:
\begin{itemize}
    \item The orbit in Theorem \ref{thm:Main3bp} not only  visits   a finite (or countable) sequence of periodic orbits on $\mathcal E_\infty$, but actually visits a locally dense subset of $\mathcal E_\infty$. Notice that, due to the degeneracy of the flow on $\mathcal E_\infty$ (it is foliated by periodic orbits!), this construction requires bridging across two extra dimensions.
    \item The closure of the orbit in  Theorem \ref{thm:Main3bp} is large in both time directions: i.e. the forward closure contains a locally dense subset of $\mathcal E_\infty$ and the backward closure contains a locally dense subset of $\mathcal E_\infty$. Typical results in Arnold diffusion are only one-sided (and no two-sided conclusions can be extracted from the usual arguments).
\end{itemize}
The geometric object behind our construction is the  symplectic blender in Theorem \ref{thm:mainblender}. This object provides local ``transitivity of the flow \eqref{eq:5dflowextended} along the center directions'' (i.e. those tangent to $\mathcal E_\infty$). Moreover, using this blender we are able to implement a two-sided shadowing argument which allows us to control both the future and past of the orbit.

Finally, before moving on, let us comment on why, at the moment, we are not able to accumulate the whole $\mathcal E_\infty$. The main reason is that, understanding the splitting between the manifolds $W^u(\mathcal E_\infty)$ and $W^s(\mathcal E_\infty)$ is a remarkably complicated task. In particular, in the present paper we rely on a result from \cite{guardia2022hyperbolicdynamicsoscillatorymotions} which shows that these manifolds intersect transversally along at least two (small) homoclinic channels $\Gamma_\pm\subset W^s(\mathcal E_\infty)\cap W^u(\mathcal E_\infty)$.  Hence, with  only this information available on the geometry of the intersection between these manifolds, we can only construct orbits shadowing these small channels.

We do indeed expect that $W^s(\mathcal E_\infty)$ and $W^u(\mathcal E_\infty)$  intersect transversally along two homoclinic channels which are diffeomorphic to $\mathcal E_\infty\setminus \mathcal B$ where $\mathcal B$ is a (finite) set of curves along which tangencies between the manifolds $W^{u,s}(\mathcal E_\infty)$ take place. However, proving this result seems formidably technical and out of the scope of the present work.

\medskip

\subsection{The restricted 3-body problem}\label{sec:restricted3bp}
Let $\mu\in(0,1/2]$, $\zeta\in[0,1)$ and let 
\[
q_0(t)=\mu\varrho(t)(\cos f(t),\sin f(t))\qquad\qquad q_1(t)=-(1-\mu)\varrho(t)(\cos f(t),\sin f(t))
\]
where $m_0=1-\mu$ and $m_1=\mu$ are the masses of the bodies $q_0$ and $q_1$, $\zeta$ is the eccentricity of the ellipses described by these bodies, $\varrho:\mathbb T\to\mathbb R_+$ is the distance between them, defined by 
\[
\varrho(t)=\frac{1-\zeta^2}{1+\zeta\cos f(t)}
\]
and $f:\mathbb T\to\mathbb T$ is the so-called true-anomaly (see \cite{MR0005824}), which satisfies $f(0)=0$ and 
\[
\frac{df}{dt}=\frac{(1+\zeta\cos f(t))^2}{(1-\zeta^2)^{3/2}}.
\]
The \textit{restricted} 3-body problem is the time-periodic Hamiltonian system 
\begin{equation}\label{eq:restricted3bp}
H_{\zeta}(q,p)=\frac{|p|^2}{2}-U(q,t),\qquad\qquad U_{\zeta}(q,t)=\frac{1-\mu}{|q-q_0(t)|}-\frac{\mu}{|q-q_1(t)|}.
\end{equation}
on the symplectic manifold $(q,p)\in T^*(\mathbb R^2)$ equipped with the canonical symplectic form $\mathrm dp\wedge\mathrm dq$. In the \textit{hierarchical region} of the phase space, i.e. for $|q|\gg 1$, 
\[
U_{\zeta}(q,t)=\frac{1}{|q|}+V_{\zeta}(q,t)\qquad\qquad V(q,t)=O(|q|^{-3})
\]
so $H_{\zeta}$ in \eqref{eq:restricted3bp} recasts as a (singular) perturbation of the two-body problem
\[
H_{\zeta}(q,p)=h_{\mathrm{Kep}}(q,p)+V_{\zeta}(q,t)\qquad\qquad h_{\mathrm{Kep}}(q,p)=\frac{|p|^2}{2}-\frac{1}{|q|}.
\]
    In Section \ref{sec:introapplicationsblender} we have already mentioned that  $h_{\mathrm{Kep}}$ possesses periodic dynamics in  each (three-dimensional) negative energy hypersurface. Positive or zero energy levels are not compact but still the invariance by rotation guarantees that, for any $h_*\in\mathbb R$ and any $G_*\in\mathbb R$, the two-dimensional submanifolds 
\begin{equation}\label{eq:angular}
M_{\mathrm{r3bp}}(h_*,G_*)=\{(q,p)\in T^*(\mathbb R^2): h_{\mathrm{Kep}}(q,p)=h_*,\ G(q,p)=G_*\}\qquad\qquad G(q,p)=q_xp_y-q_yp_x
\end{equation}
are left invariant under the flow of $h_{\mathrm{Kep}}$. It is then natural to ask if there exists orbits of \eqref{eq:restricted3bp} along which either $h$ or $G$ exhibit significant variations. We showed in \cite{guardia2023degeneratearnolddiffusionmechanism} that this is the case for any $\mu\in(0,1/2)$ by constructing orbits along which the angular momentum $G$ exhibits arbitrarily large variations provided the eccentricity $\zeta>0$ is sufficiently small. Our last main result is a strengthened version of the main result in \cite{guardia2023degeneratearnolddiffusionmechanism}. To state this result, let us denote by $\phi_{H_\zeta}(t,t_0,q,p)$ the general solution associated to \eqref{eq:restricted3bp} and  introduce the (four-dimensional) time-one map
\begin{align*}
    \Psi_\zeta:\{t=0\}&\to \{t=0\}\\
    (q,p)&\mapsto \phi_{H_\zeta}(2\pi,0,q,p)
\end{align*}
Recall that, in the context of skew-products over the shift, $N$-cylinders were introduced in \eqref{eq:Ncylinder}.

\begin{thmx}\label{thm:MainR3bp}
   Fix any $\mu\in (0,1/2)$, $R>0$, $G_1\gg 1$ and any $N\in\mathbb N$. Then, for any  $\zeta>0$ sufficiently small there exists a subset $\mathcal A_\zeta\subset T^*(\mathbb R^2)$, a homeomorphism $
   \Phi_\zeta:\{0,1\}^\mathbb N\times \mathbb A\to \mathcal A_\zeta$ and a natural number $M\in\mathbb N$ for which the following holds:
   \begin{itemize}
       \item $\mathcal A_\zeta$ is a weakly invariant normally hyperbolic lamination for the map $\Psi^M_\zeta$, where the leaves of the lamination are $C^1$ and are a graph with respect to $(\alpha, G)\in\mathbb{T}\times[G_1,G_1+R]$ with $\alpha$ being the angle of the massless body with respect to the argument of the perihelion (see Figure \ref{fig:Fig8}) of the primaries and $G$ being its angular momentum (see \eqref{eq:angular}).
       \item Given any countable covering of $\mathcal A_\zeta$ by $N$-cylinders there exists an orbit of $\Psi_\zeta^M$ which visits all the elements in the covering,
   \end{itemize}
   In particular, there exist an orbit in the lamination whose  projection onto $(\alpha, G)\in\mathbb{T}\times[G_1,G_1+R]$ is dense.
\end{thmx}

\begin{figure}
    \centering
    \includegraphics[scale=0.4]{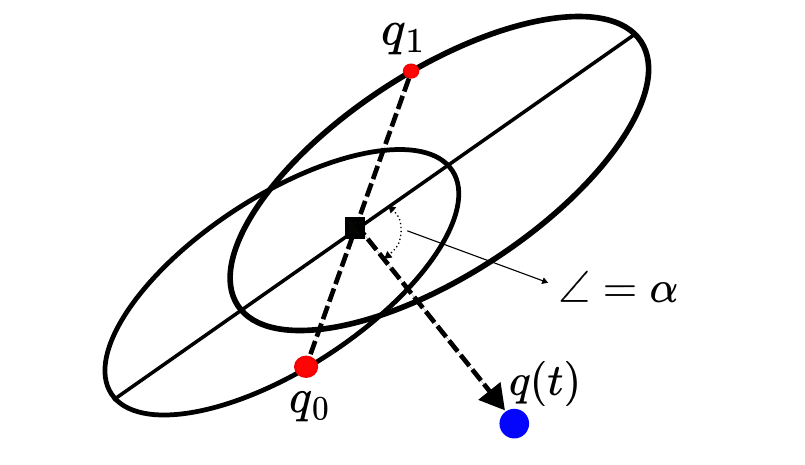}
    \caption{The primaries (i.e. the massive bodies) $q_0,q_1$ orbit around themselves,  each of them describing an ellipse around the center of mass. The massless body $q$ moves influenced by their gravitational field.}
    \label{fig:Fig8}
\end{figure}
 The orbits realizing the construction in Theorem \ref{thm:MainR3bp} not only exhibit large drifts on the $G$-component but are ``dense modulo $N$-cylinders'' on $\mathcal A_\zeta$. In particular, the projection of any of these orbits to the center coordinates contains a dense subset of $\mathbb T\times[G_1,G_1+R]$ and therefore provide a rather strong form of Arnold diffusion. The main ingredient in the proof of Theorem \ref{thm:mainblender} is the abstract result in Theorem \ref{thm:skewproduct}.
\medskip

\subsection{Literature on chaotic dynamics in nearly integrable Hamiltonian systems}\label{sec:comparison}

Before delving into the details of the proofs of our main results, let us compare our results to those previously obtained in the literature. 
\medskip

\paragraph{\textbf{Chaotic dynamics in celestial mechanics:}}

From a historical point of view, Poincar\'e's finding of a homoclinic tangle in the restricted 3-body problem marks celestial mechanics as the starting point of the theory of chaotic dynamics. After the introduction of Smale's horseshoe, the work of Alekseev \cite{AlekseevQR1,AlekseevQR2} and Moser \cite{MR442980} paved the way for the construction of chaotic (i.e non-trivial) \textit{uniformly hyperbolic sets} in models coming from celestial mechanics.

However, in the modern theory of dynamics, it was soon realized that chaos it is often not uniformly hyperbolic. A classical scenario for the \text{robust lack of  uniformity} corresponds to the existence of \textit{homoclinic tangencies}. A number of works have explored the rich dynamics that are associated to this phenomenon and it was natural to ask if these dynamics are also observed in celestial mechanics models. For celestial mechanics models, Newhouse domains (i.e. open parameter sets for which the dynamics exhibits homoclinic tangencies) were first constructed in the unpublished manuscript \cite{GorodetskiK12}). Some other works exploring the dynamics in Newhouse domains associated to the restricted 3-body problem are \cite{Giraltcoorbitalchaos,garrido2024parabolicsaddlesnewhousedomains}.

Switching now to chaotic dynamics associated to \textit{partially hyperbolic} scenarios, Nassiri and Pujals introduced in \cite{NassiriPujalsTransitivity} an abstract construction to produce $C^\infty$ arcs of \textit{a priori chaotic} Hamiltonians exhibiting robustly transitive invariant laminations. By a priori chaotic we mean that the authors perturb systems which already exhibit a normally hyperbolic invariant lamination with regular leaves on which the dynamics is conjugated to a skew-product over the shift. In the list of open problems included in Section 6 of \cite{NassiriPujalsTransitivity} the authors mention that ``a major challenge would be to apply the present approach to the context of the restricted
3-body problem''.  Theorems \ref{thm:mainblender}, \ref{thm:Main3bp} and \ref{thm:MainR3bp} can be seen as  the natural extension of the abstract construction in \cite{NassiriPujalsTransitivity} to the  3-body problem and its restricted version.

There are two main obstacles to adapt their construction to our setting. The first one is that we study the restricted 3-body problem (a similar discussion also applies to the non-restricted problem) as a perturbation of the 2-body problem. The latter is  integrable and, in particular, does not exhibit any non-trivial normally hyperbolic lamination. To construct such a lamination we show that the invariant manifolds of a certain (topological) normally hyperbolic invariant manifold intersect transversally along a homoclinic manifold and then look at the return map to a neighbourhood of this homoclinic manifold. However,  the homoclinic manifold contains \textit{holes} associated to curves of non-transverse intersection (tangencies). We thus have to restrict our study to the region away from these holes. Hence, the corresponding lamination on a neighbourhood of the homoclinic manifold is only \textit{weakly invariant}.

The second difficulty is that in \cite{NassiriPujalsTransitivity} blenders are constructed via perturbation techniques only available in the smooth setting (as they involve the use of compactly supported functions).  More concretely, the authors introduce localized perturbations on different elements of the corresponding Markov partition so that the corresponding skew-product dynamics on the lamination is transitive. Realizing this construction on a real-analytic setting would certainly be rather challenging as the laminations are themselves very localized and the different elements of the Markov partition get extremely close. Our approach to obtain rich dynamics in the center directions is different and does not require any perturbation argument if the system  satisfies a rather explicit condition (encoded in Theorem \ref{thm:transitivityIFS}).
\medskip

\paragraph{\textbf{Laminations in real-analytic Hamiltonians:} }
Weakly invariant normally hyperbolic laminations arise (under suitable conditions) close to \textit{single resonances} in nearly-integrable Hamiltonians. One of the first works in which (a simplified version of) this situation  was considered is \cite{Moeckeldrift}. With a view towards the implementation of the Arnold diffusion mechanism, in this work Moeckel provides a criterion under which the dynamics on the lamination exhibits drifting orbits. This idea was later exploited by Gelfreich and Turaev in \cite{GelfreichTuraev} to prove that a generic real-analytic perturbation of an a priori chaotic Hamiltonian exhibits drifting orbits. Among many  other things, in \cite{MR2104598} Marco and Sauzin construct an explicit Gevrey perturbation of an integrable convex Hamiltonian for which there exists a normally hyperbolic lamination admiting an orbit whose projection to the center subspace (i.e. tangent to the leaves) is dense. In some sense, these results are concerned with \textit{topological aspects} of the dynamics on these laminations. This is also the case of our Theorems \ref{thm:skewproduct} and  \ref{thm:MainR3bp}. We believe that the ideas in Theorem \ref{thm:skewproduct} will be of good use to establish similar conclusions to those in Theorem \ref{thm:MainR3bp} in given parametric families of real-analytic Hamiltonians. Finally, let us also mention that one could also investigate the \textit{statistical properties} of dynamics on these laminations. 
For instance, in \cite{MR2104598} the authors are able to embed a random walk in the nearly integrable dynamics. 
Both  weakly invariant normally hyperbolic laminations and the statistical properties of its dynamics has been also analyzed in Celestial Mechanics models by Capinski and Gidea for the restricted 3-body problem \cite{MR4544807} and also in   the forthcoming papers \cite{guardia2025stochastic1, guardia2025stochastic2},  the authors obtain a similar result for a lamination arising at the so-called \textit{mean motion} resonances in the restricted 3 body problem.
Weakly invariant normally hyperbolic laminations are expected to exist in other Celestial Mechanics contexts such as secular resonances \cite{Clarke25} or along the center manifolds of the Lagrange points. However, for some models the leaves of the laminations have dimension higher to those considered in the present paper.

\medskip

\paragraph{\textbf{Blenders as a tool in Celestial mechanics}:}
We believe that blenders might prove extremely useful to tackle (or at least make some partial progress) on  several longstanding conjectures in celestial mechanics. For instance, one may think of using these objects to address  Alekseev's conjecture on the existence of a locally dense subset of initial conditions leading to collision orbits.
\medskip

\paragraph{\textbf{Blenders in real-analytic Hamiltonians:}}
Another question posed in Section 6 of \cite{NassiriPujalsTransitivity} is if one can introduce real-analytic perturbation techniques to create blenders in nearly-integrable Hamiltonians. This question has recently been considered by Li and Turaev in \cite{LiTuraevPreprint}. They show that for Hamiltonians (not necessarily close to integrable) exhibiting a two-dimensional partially hyperbolic KAM-torus with a homoclinic orbit, a symplectic blender can be created by an arbitrarily small real-analytic perturbation. 

Related to this setting, we believe that our abstract result in Theorem \ref{thm:transitivityIFS} will prove  useful to find blenders in given real-analytic parametric families of Hamiltonians. If we think of the maps $T_0,T_1$ in  Theorem \ref{thm:transitivityIFS} as the inner dynamics and scattering map to a normally hyperbolic cylinder, this theorem gives explicit conditions (which do not require any further perturbation) for the existence of a symbolic blender for this IFS. It seems quite reasonable that, proceeding as  in the proof of Theorem \ref{thm:mainblender} in the context of the 3-body problem, this explicit condition guarantees the existence of a symplectic blender. 
\medskip

\subsection{Organization of the article} \label{sec:outline}
The rest of the article is organized as follows.
In Section \ref{sec:symbolic} we introduce symbolic blenders for IFS on surfaces, describe how these appear in the context of maps satisfying assumptions \textbf{(A0)-(A2)} and relate symbolic double  blenders to the proof of Theorem \ref{thm:transitivityIFS}. We  also collect certain intermediate results which are key to the proof of Theorem \ref{thm:Main3bp} (and are used in Section \ref{sec:partiallyhyp3bp}).
In Section \ref{sec:heuristicblenders} we introduce symplectic blenders in the context of 4-dimensional symplectic maps and explain some of the challenges that we face in our construction of these objects for the 3-body problem.
In Section \ref{sec:IFSlocaltransitive} we present the proof of Theorem \ref{thm:transitivityIFS}. As discussed in Section \ref{sec:ifsintro}, it relies on a variant of the transversality-torsion mechanism introduced by Cresson \cite{MR1949441}.
In Section \ref{sec:skewproduct} we extend the results from Section \ref{sec:IFSlocaltransitive} to cylinder skew-products. The fact that we do not assume the skew-product to be locally constant introduces some technicalities, but the main ideas are those discussed in Section \ref{sec:IFSlocaltransitive} in the context of IFSs.
In Section \ref{sec:partiallyhyp3bp} we  recall results from \cite{guardia2022hyperbolicdynamicsoscillatorymotions}  which show how to  obtain a (locally) partially hyperbolic framework for the 3-body problem. We obtain a four dimensional return map $\Psi$ to a suitable transverse section accumulating on two different homoclinic manifolds. Special emphasis is put in controlling the center dynamics close to each homoclinic manifold and show that, up to first order, are governed by twist maps satisfying the assumptions in Theorem \ref{thm:transitivityIFS}.
In Section \ref{sec:blender3bp} we construct a $cs$-blender for the map $\Psi$. To that end, we look at a  judiciously chosen range of iterates of the map $\Psi$ which, in the center directions reproduces the weak transversality-torsion mechanism behind the proof of Theorem \ref{thm:transitivityIFS}.  We then verify that this family of maps satisfy the so-called covering property and exhibit well-distributed hyperbolic periodic orbits (see Section \ref{sec:heuristicblenders}). In this way we complete the proof of Theorem \ref{thm:mainblender}.
In Section \ref{sec:mainproof} we complete the proof of Theorem \ref{thm:Main3bp}. The main idea is that using a symplectic blender, one can implement a two-sided shadowing argument.
In Section \ref{sec:restricted3bpproof} we introduce a (locally) partially hyperbolic setting for the restricted 3-body problem dynamics. The framework is very similar to that in Section \ref{sec:partiallyhyp3bp} and builds on previous results obtained in \cite{guardia2023degeneratearnolddiffusionmechanism}. We then construct a weakly invariant normally hyperbolic lamination for a suitable return map and  complete the proof of Theorem \ref{thm:MainR3bp} by combining the tools in Section \ref{sec:skewproduct} with those in \cite{guardia2023degeneratearnolddiffusionmechanism}.

\subsection*{Acknowledgements}

The authors want to warmly thank D. Li and D. Turaev for many enlightening discussions about the creation of blenders in Hamiltonian dynamics and for sharing with us a draft of their forthcoming paper \cite{LiTuraevPreprint}.

This work is part of the grant PID-2021-122954NB-100 funded by MCIN/AEI/10.13039/501100011033 and ``ERDF A way of making Europe''.   M.G. is supported by the Catalan Institution for Research and Advanced Studies via an ICREA Academia Prize  2023. This work is also supported by the Spanish State Research Agency, through the Severo Ochoa and Mar\'{i}a de Maeztu Program for Centers and Units of Excellence in R\&D (CEX2020-001084-M).


\section{Symbolic and symplectic blenders}
\subsection{Symbolic blenders for IFS of the cylinder}\label{sec:symbolic}

In this section we introduce symbolic blenders for IFS on surfaces. Then, we  consider an IFS of twist maps of the cylinder and show that it exhibits a symbolic blender provided the maps satisfy assumptions \textbf{(A0)-(A2)}.
    The existence of a symbolic blender is the main ingredient in the proof of Theorem \ref{thm:transitivityIFS}, which will be completed in Section \ref{sec:IFSlocaltransitive}.
 Moreover, the techniques used to construct this symbolic blender will be of high relevance for the construction of symplectic blenders in the context of 4-dimensional symplectic maps (in particular, for a suitable return map in the 3-body problem). 

Consider an IFS $\{T_i\}_{i=1,\dots,k}$ of smooth maps acting on a smooth surface $\MM$. Abusing notation, for any $M\in\mathbb N$ and $\omega\in\{1,\dots,k\}^M$, we denote 
    \[
T_\omega=T_{\omega_M}\circ\ldots\circ T_{\omega_1}.
    \]
Let  $Q\subset \MM$ be a rectangle and denote by $\{\mathtt T_i\}_{i=1,\dots,k}$ the corresponding induced return maps on $Q$. We now suppose that:
\begin{itemize}
    \item for at least one $j\in\{1,\dots,k\}$, the map $\mathtt T_j$ has a hyperbolic fixed point $P_j\in Q$ and denote by $W^u(P_j;\mathtt T_j)$ and $W^s(P_j;\mathtt T_j)$ their local  unstable and stable manifold respectively (for the map $\mathtt T_j$),
    \item there exist families of cone fields $\mathcal C^u,\mathcal C^s$ which are common for all the maps $\mathtt T_i$, $i=1,\dots, k$.
\end{itemize}
 In this setting we say that a $C^1$ curve $\gamma\subset Q$ is a $s$-curve (resp. $u$-curve) if its tangent bundle is contained in the cone $\mathcal C^s$ (resp. $\mathcal C^u$).
\begin{defn}[Symbolic $cs$-blender]\label{defn:symboliccsblender}
    We say that the pair $(P_j,Q)$ is a \textit{symbolic $cs$-blender} for the IFS $\{T_i\}_{i=1,\dots,k}$ if for any $s$-curve $\gamma\subset Q$ there exists $M\in\mathbb N$ and $\omega\in\{1,\dots,k\}^M$ such that 
    \[
    T_\omega^{-1}(\gamma)\pitchfork W^u(P_j;\mathtt T_j)\neq \emptyset.
    \]
\end{defn}

Analogously, the pair $(P_j,Q)$ is a \textit{symbolic $cu$-blender} if for any $u$-curve $\gamma\subset Q$ there exist $M$ and $\omega$ as above such that $T_\omega(\gamma)\pitchfork W^s(P_j;\mathtt T_j)\neq \emptyset$.  

\begin{defn}[Symbolic double blender]
    Let $(P_j,Q)$ be a symbolic $cu$-blender and $(P_i,Q')$ be a symbolic $cs$-blender for the IFS $\{T_i\}_{i=1,\dots,k}$. Together they form a \textit{symbolic double blender}  if 
    \[
    W^s(P_j;\mathtt T_j)\pitchfork W^u(P_i;\mathtt T_i)\neq\emptyset.
    \]
\end{defn}




Next proposition ensures that, under the conditions \textbf{(A0)-(A2)}, the maps $T_0,T_1$ display a symbolic double blender provided the transversality between the maps, measured by $\varepsilon$, is small enough. This proposition will prove also useful for establishing Theorem \ref{thm:Main3bp}.

\begin{prop}\label{prop:symbolicblenderintro}
    Consider two maps $T_0,T_1$ satisfying \textbf{(A0)-(A2)}  for some $\alpha,\rho,\sigma>0$ and let $K>0$ be as in \eqref{eq:maxnorms}. 
There exists $\varepsilon_0(K,\rho,\sigma)>0$ such that for any $
0<\varepsilon\leq\varepsilon_0  \min\{\tau,\alpha\}$  the IFS generated by the pair $\{T_0,T_1\}$ exhibits a symbolic  double blender.
\end{prop}

One of the main ingredients needed to establish Proposition \ref{prop:symbolicblenderintro} is the following normal-form like result for a suitable range of iterates of the maps $T_0$ and $T_1$.

\begin{prop}\label{prop:normalform}
    Fix any value $0<\chi\ll 1$ and consider the setting of Proposition \ref{prop:symbolicblenderintro}. Then, there exists an affine local coordinate system 
    $\phi_{\chi,\eps}:[-2,2]^2\to \mathbb A$ and a subset $\mathcal N_{\chi}\subset\mathbb N$ such that for all $n\in\mathcal N_{\chi}$ and all $(\xi,\eta)\in [-2,2]^2$ 
    \begin{equation}\label{def:Fn}
    \mathcal F_n:=\phi_{\chi,\eps}^{-1}\circ T_0^n\circ T_1\circ\phi_{\chi,\eps}=\binom{b_n}{0}+\begin{pmatrix}
        1-\chi &0\\
        0&1+\chi
    \end{pmatrix}\binom{\xi}{\eta}+O(\chi^2),
    \end{equation}
    for some constant  $b_n\in[-1,1]$. Moreover, the family $\{b_n\}_{n\in\mathcal N_\chi}$ is  $\frac{1}{10}\chi$-dense on $[-10\chi,10\chi]$.
\end{prop}
Proposition \ref{prop:normalform} will be proved in Section \ref{sec:IFSlocaltransitive}. After doing so, we will check that the set of maps $\{\mathcal F_n\}_{n\in\mathcal N_{\chi}}$ satisfy the so-called covering and well-distribution properties (see \cite{NassiriPujalsTransitivity}) in order to establish the existence of $cu$ and $cs$ blenders, leading to the proof of Proposition \ref{prop:symbolicblenderintro}.

The proof of Theorem \ref{thm:transitivityIFS}, also completed in  Section \ref{sec:IFSlocaltransitive}, is based on the existence of the symbolic double blender provided by Proposition \ref{prop:symbolicblenderintro}. Indeed, given a $u$-curve $\gamma\subset Q$ and a $s$-curve $\gamma'\subset Q'$, the transversality condition $W^s(P_j;T_j)\pitchfork W^u(P_i;T_i)\neq\emptyset$ and the classical Lambda lemma (see \cite{PalisMelo})  implies that an iterate of $\gamma$ will intersect transversally $\gamma'$. Hence, local transitivity for the pair of maps $T_0,T_1$ follows (after some minor extra work) from the fact that  we can intersect any $u$-curve with any $s$-curve.

\subsection{Symplectic blenders for  4-dimensional symplectic maps} \label{sec:heuristicblenders}

We introduce now symplectic blenders for 4-dimensional symplectomorphisms. The presentation that we propose here is tailored for the application to the problem at hand and we refer the reader to \cite{NassiriPujalsTransitivity}, \cite{BeyondUH} (and the references therein) for other (more general) constructions.

We start by introducing the scenario in which we will construct symplectic blenders. This setting corresponds to  maps which display (at least locally) partially hyperbolic behavior. Let $\psi$ be a diffeomorphism on a 4-dimensional  manifold $\MM$.  Let $Q\subset \MM$ be a (4-dimensional) rectangle and denote by  $\Psi$ the induced return map on $Q$. We suppose that:

\begin{itemize}
    \item[\textbf{H1}]  There exists $a>b>1$,  $\lambda>\chi>0$ and a (smooth) local coordinate chart on $Q$,
    \[
\phi:(p,\tau,\xi,\eta)\in[-1,1]^4\to Q,
\]
such that at any $Z=(p,\tau,\xi,\eta)$, the cones 
\[
C^{ss}_Z=\{v\in T_Z Q\colon |v_p|\geq a \max\{|v_\tau|,|v_\xi|,|v_\eta|\}\}\quad\quad C^{s}_Z=\{v\in T_Z Q\colon \max\{|v_p|,|v_\xi|\}\geq b \max\{|v_\tau|,|v_\eta|\}\}
\]
and 

\[
C^{uu}_Z=\{v\in T_Z Q\colon |v_\tau|\geq a \max\{|v_p|,|v_\xi|,|v_\eta|\}\}\quad\quad C^{u}_Z=\{v\in T_Z Q\colon \max\{|v_\tau|,|v_\eta|\}\geq b \max\{|v_p|,|v_\xi|\}\}
\]
satisfy that (wherever the return map or its inverse are well defined)
\[
D\Psi^{-1}_Z C_{Z}^{*_s}\subset C^{\star_s}_{\Psi^{-1}(Z)}\qquad\qquad  D\Psi_Z C_{Z}^{*_u}\subset C^{*_u}_{\Psi(Z)}\qquad\qquad*_s=s,ss\qquad *_u=u,uu
\]
and, moreover, 
\begin{align*}
 |(D\Psi^{-1}(Z) v)_p|\geq (1+\lambda) |v_p|\quad \text{if } v\in C^{ss}_Z,& \qquad\qquad
     \max_{*=p,\xi}|(D\Psi^{-1}(Z) v)_*|\geq (1+\chi)  \max_{*=p,\xi}|v_*|\quad \text{if } v\in C^s_Z \\
    |(D\Psi(Z) v)_\tau|\geq (1+\lambda) |v_\tau|\quad \text{if } v\in C^{uu}_Z& \qquad\qquad
     \max_{*=\tau,\eta}|(D\Psi(Z) v)_*|\geq (1+\chi)  \max_{*=\tau,\eta}|v_*|\quad \text{if } v\in C^u_Z. 
\end{align*}
\end{itemize}
    \medskip

Consider now  the larger rectangle 
\begin{equation}\label{eq:largerectangleheuristic}
Q^{\mathrm{ext}}=\phi([-1,1]^2\times[-2,2]^2).
\end{equation}
We say that a $\Delta\subset Q^{\mathrm{ext}}$ is a \textit{horizontal submanifold} if it admits a parametrization of the form 
\begin{equation}\label{eq:horsubfoldintro}
\Delta=\{(p,h_1(p,\xi),\xi,h_2(p,\xi))\colon \gamma_l(p)\leq \xi\leq \gamma_r(p),\ p\in[-1,1]\}
\end{equation}
for any $C^1$ functions  $\gamma_l<\gamma_r$  and $h_1,h_2$ such that   
\[
\partial_p Z_{l}(p),\partial_pZ_r(p)\in \mathcal C^{ss}\qquad \qquad \text{where}\qquad  Z_\star(p)= \phi(p,h_1(p,\gamma_\star(p)),\gamma_\star(p),h_2(p,\gamma_\star(p))).\quad \star=l,r,
\]
and at any $Z\in \phi(\Delta)$ we have $ T_Z\  \phi(\Delta)\subset \mathcal C^s_Z$. We say that a $\Delta\subset Q^{\mathrm{ext}}$ is a \textit{vertical submanifold} if it admits a parametrization of the form 
\[
\Delta=\{(v_1(\tau,\eta),\tau,v_2(\tau,\eta),\eta)\colon \gamma_d(\tau)\leq \eta\leq \gamma_u(\tau),\ \tau\in[-1,1]\}
\]
for any differentiable functions  $\gamma_d<\gamma_u$  and $v_1,v_2$ such that   
\[
\partial_\tau Z_{d}(\tau),\partial_\tau Z_u(\tau)\in \mathcal C^{uu}\qquad \qquad \text{where}\qquad Z_\star(\tau)= \phi(v_1(\tau,\gamma_\star(\tau)),\tau,v_2(\tau,\gamma_\star(\tau)),\gamma_\star(\tau))
\]
and at any $Z\in \phi(\Delta)$ we have $ T_Z  \phi(\Delta)\subset \mathcal C^u_Z$. 

We also assume that on $Q$ the map $\Psi$ satisfies: 
\begin{itemize}
    \item[\textbf{H2}] if for any vertical submanifold $\Delta\subset Q$ the image $\Psi(\Delta)\cap Q^{\mathrm{ext}}$ contains at least $k\geq 2$ disjoint vertical submanifolds  and for any horizontal submanifold $\Delta\subset Q$ the preimage $\Psi^{-1}(\Delta)\cap Q^{\mathrm{ext}}$ contains at least $k\geq 2$ distinct horizontal submanifolds.
    \medskip

\end{itemize}
If we define vertical (resp. horizontal) \textit{rectangles} as 4-dimensional compact subsets which are $C^1$ foliated by two-dimensional vertical (resp. horizontal) submanifolds, \textbf{H2} implies, in particular, that $\Psi(Q)\cap Q^{\mathrm{ext}}$ contains at least $k\geq 2$ vertical rectangles and $\Psi^{-1}(Q)\cap Q^{\mathrm{ext}}$ contains at least $k\geq 2$ horizontal rectangles. 

\begin{rem}
    Observe however that the vertical (resp. horizontal) rectangles contained in $\Psi(Q)\cap Q^{\mathrm{ext}}$ (resp. $\Psi^{-1}(Q)\cap Q^{\mathrm{ext}}$) might not be entirely contained in $Q$. 
\end{rem}

It is in the framework above (Assumptions \textbf{H1} and \textbf{H2}) that we  introduce first $cs$ and $cu$-blenders and  derive sufficient conditions for their existence. We say that a horizontal submanifold $\Delta$ is a \textit{$cs$-strip} if $\Delta\subset Q$.

\begin{defn}
    Let $P\in Q$ be a hyperbolic periodic point of the map $\Psi$. We say that the pair $(P,Q)$ is a \textit{$cs$-blender} for the map $\Psi$ if  any $cs$-strip $\Delta\subset Q$ intersects $W^u(P)$ in a robust fashion.
\end{defn}
The definition of $cu$-blender is completely analogous.
When a $cu$-blender and a $cs$-blender are homoclinically related, these local objects might have a global influence on the dynamics of the system.
\begin{defn}[Symplectic blender]\label{def:symplecticblender}
    Let $(P,Q)$ be a $cu$-blender and $(P',Q')$ be a $cs$-blender for the map $\Psi$. Together, they form a symplectic blender if $P$ and $P'$ are homoclinically related.
\end{defn}

A straightforward application of the lambda-lemma (see \cite{PalisMelo}) implies that some forward iterate of any $cu$-strip in $Q$ intersects any $cs$-strip in $Q'$. Roughly speaking, the presence of a symplectic blender guarantees that the dynamics is locally  transitive ``modulo strongly hyperbolic directions''. 

\subsubsection{The covering property and well distributed periodic orbits}\label{sec:heuristiccovering4d}

Following \cite{BeyondUH}, \cite{NassiriPujalsTransitivity} (but adapted to the present context) we now provide a condition on the map $\Psi$ which guarantees the existence of a $cs$-blender.  We define the width  of a $cs$-strip $\Delta$
as 
\[
\mathrm{width}(\Delta)=\inf_{p\in[-1,1]}|\gamma_r(p)-\gamma_l(p)|.
\]
The proof of the following result is a straightforward application of the Lambda lemma (see \cite{PalisMelo}).
\begin{lem}[Intersection/expansion dichothomy]\label{lem:coveringpluswelldistr}
    Suppose that there exist $k\geq 2$,  pairwise disjoint subsets $\{V_i\}_{i\in\{1,\dots,k\}}\subset Q$, a real number $\chi>0$ and hyperbolic periodic points  $P_i\in V_i$ of the map $\Psi$. Assume that $P_i$ and $P_j$ are homoclinically related for $i\neq j$ and that the following dichotomy holds:  for any $cs$-strip $\Delta\subset Q$ either
    \begin{itemize}
        \item \textit{(Intersection):} there exists $i\in\{1,\dots,k\}$ such that $\Delta\pitchfork W^{u}(P_i)\neq \emptyset$  or
        \item \textit{(Expansion):} there exists $i\in\{1,\dots,k\}$ such that $\bar\Delta=\Psi^{-1}(\Delta\cap V_i)\cap Q$ is a $cs$-strip and 
        \[
        \mathrm{width}(\bar\Delta)>(1+\chi) \mathrm{width}(\Delta).
        \]
    \end{itemize}
    Then, the pair $(P_1,Q)$ is a $cs$-blender for the map $\Psi$.
\end{lem}

Inspired by \cite{NassiriPujalsTransitivity}, we now present an strategy to materialize the above dichotomy in a concrete example. 
Let $\Delta \subset Q$ be a $cs$-strip. By the Assumption \textbf{H2} the image $\Psi^{-1}(\Delta)\cap Q^{\mathrm{ext}}$ contains at least two horizontal submanifolds. Notice however that these submanifolds do not necessarily fit into the definition of $cs$-strips above as they might not be contained in $Q$. 

In order to gain control over the center directions we  proceed as follows. By the assumption \textbf{H2} there exists $k\geq 2$ non-empty vertical subrectangles $\{ \widetilde V_i\}_{i\in\{1,\dots,k\}}\subset Q$ of the form 
\begin{equation}\label{eq:heuristicsetsVi}
\widetilde V_i=\{ v_{1,l}^i(\tau,\eta)\leq p\leq v_{1,r}^i(\tau,\eta),\ v_{2,l}^i(\tau,\eta)\leq \xi\leq v_{2,r}^i(\tau,\eta)\}
\end{equation}
for some  differentiable functions $v_{1,l}^i<v_{1,r}^i$, $v_{2,l}^i<v_{2,r}^i$  and such that $
\bigcup_{i=1}^k  \widetilde V_i= \Psi^{-1}(Q)\cap  Q^{\mathrm{ext}}$. Define now the (possibly empty) vertical subrectangles 
\[
V_i=\widetilde V_i\cap Q.
\]
In this setting we now say that $\Psi$ satisfies the:
\begin{itemize}
    \item \textit{(Covering property):} if for any $cs$-strip $\Delta\subset Q$ as in \eqref{eq:horsubfoldintro}  there exists at least one element $i\in\{1,\dots,k\}$ and differentiable functions $\tilde \gamma_l,\tilde\gamma_r$ with $\gamma_l \leq \tilde\gamma_l<\tilde\gamma_r\leq \gamma_r$ such that the piece $\tilde\Delta_i \subset \Delta$ defined implicitly by\footnote{A sufficient condition for the inequalities 
    \[
    \tilde\gamma_l(p)\leq \xi\leq \tilde\gamma_r(p)\qquad\qquad\text{and}\qquad\qquad\  v_{1,l}^i(h_1(p,\xi),h_2(p,\xi)) \leq p \leq v_{1,r}^i(h_1(p,\xi),h_2(p,\xi))
    \]
    to define a non-empty region of the $(p,\xi)$-plane is that $|\partial_\tau v_{1,\star}^i \partial_p h_1|,|\partial_\tau v_{1,\star}^i\partial_p h_2|\ll 1$ with $\star=l,r$. We assume these conditions are met.}
    \[
    \tilde \Delta_i=\{(p,h_1(p,\xi),\xi, h_2(p,\xi))\colon \tilde\gamma_l(p)\leq \xi\leq \tilde\gamma_r(p),\  v_{1,l}^i(h_1(p,\xi),h_2(p,\xi)) \leq p \leq v_{1,r}^i(h_1(p,\xi),h_2(p,\xi))\}
    \]
    is entirely contained in $V_i$ (see Figure \ref{fig:Fig3}).
\end{itemize}

\begin{figure}
\centering
\includegraphics[scale=0.4]{ 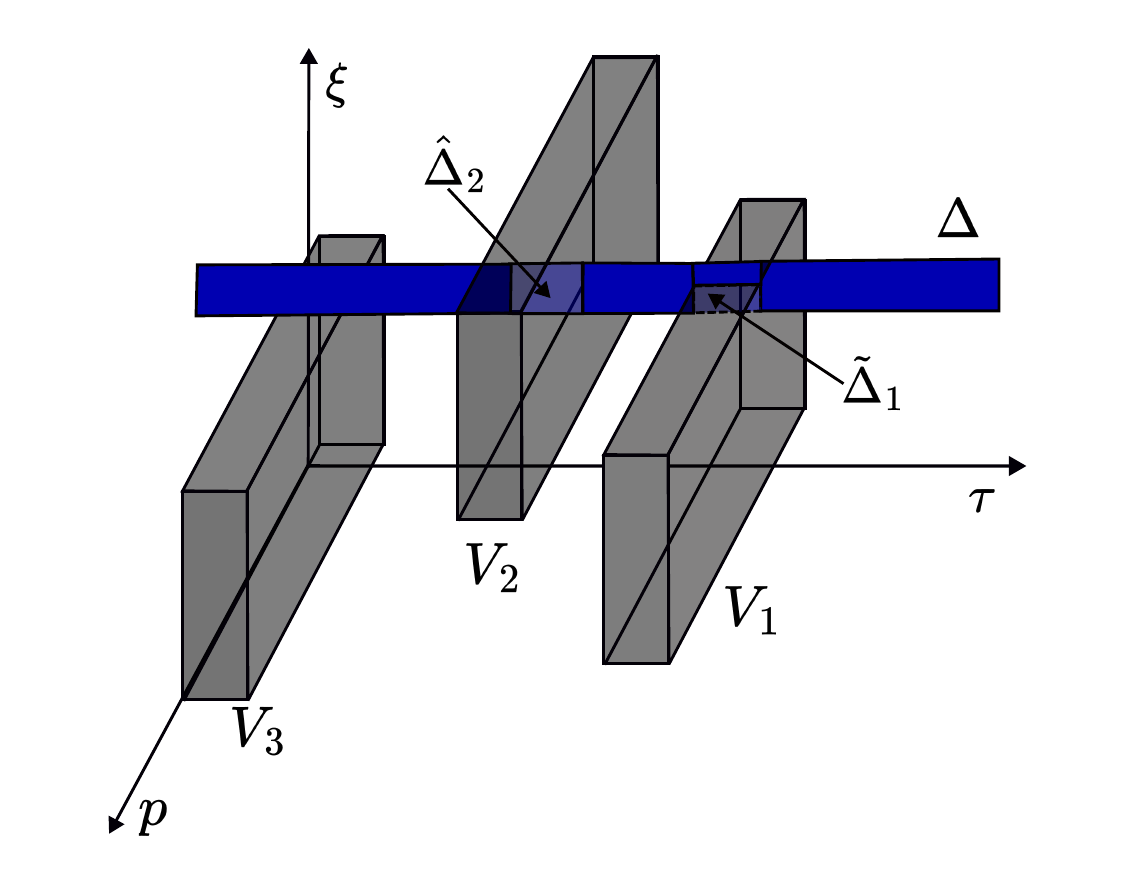}\caption{The strip $\Delta$ intersects only the vertical rectangles $V_1,V_2$. It intersects $V_2$ cleanly, i.e. $\hat \Delta_2\subset V_2$ but the intersection with $V_1$ is not clean, i.e. $\hat\Delta_1\nsubset V_1$. However, the smaller subset $\tilde\Delta_1$ is contained in $V_1$.}\label{fig:Fig3}
\end{figure}

Roughly speaking, the covering property asks that the projections of the $V_i$ onto the center directions $(\xi,\eta)$ are such that their union covers (robustly) the entire square $[-1,1]^2$ (as in the right part of Figure \ref{fig:Fig1}). This property implies that $\Psi^{-1}(\Delta)$ contains at least one $cs$- strip (in particular $\Psi^{-1}(\Delta_i)$ is a $cu$-strip) and hence allows us to produce new $cs$-strips by considering the backward images of $cs$-strips.

The covering property is our first ingredient towards establishing the dichotomy in Lemma \ref{lem:coveringpluswelldistr}. The second ingredient is the following. We say that $\Psi$ satisfies the:

\begin{itemize}
    \item \textit{(Well-distributed periodic orbits property):} if there exists hyperbolic periodic orbits $P_i\in V_i$, which are pairwise homoclinically related, such that, for any $cs$-strip $\Delta\subset Q$ either:
    \begin{itemize}
        \item there exists $i\in\{1,\dots,k\}$ such that $\Delta \cap W^u(P_i)\neq \emptyset $  or,
        \item there exists $i\in\{1,\dots,k\}$ such that the piece $\hat\Delta_i \subset \Delta$ defined implicitly by
    \[
    \hat \Delta_i=\{(p,h_1(p,\xi),\xi, h_2(p,\xi))\colon \gamma_l(p)\leq \xi\leq \gamma_r(p),\  v_l^i(h_1(p,\xi),h_2(p,\xi)) \leq p \leq v_r^i(h_1(p,\xi),h_2(p,\xi))\}
    \]
    is entirely contained in $V_i$.
    \end{itemize}
\end{itemize}
Finally, we observe that if the piece $\hat \Delta_i\subset V_i$ then $\bar\Delta=\Psi^{-1}(\hat \Delta_i)$ is a $cs$-strip and, moreover, since, by definition, it is tangent to the stable cone $\mathcal C^s$,  we will have 
\[
\mathrm{width}(\bar\Delta)\geq (1+\chi) \mathrm{width} (\hat \Delta_i)\geq (1+\chi)\mathrm{width}(\Delta)
\]
for some $\chi>0$.

\section{Local transitivity for cylinder IFS: proof of Theorem \ref{thm:transitivityIFS}}\label{sec:IFSlocaltransitive}

In this section we give the proof of Theorem \ref{thm:transitivityIFS}. We do so in several steps.
\begin{enumerate}
    \item We construct finite sequences $\omega^{(n)}\in \{0,1\}^n$ for some $n\in\mathbb N$ such that the maps
    \[
    F_n=T_{\omega^{(n)}_n}\circ\dots\circ T_{\omega^{(n)}_0}
    \]
    admit  local affine approximations in terms of  weakly hyperbolic affine maps $A_n\binom{\varphi}{J}+([n\beta],0)^\top$ (Section \ref{sec:linearapprox}).
    \item We construct a uniform coordinate system to analyze the maps $F_n$ corresponding to a suitable family of finite sequences $\omega^{(n)}$ and obtain the normal form in Proposition \ref{prop:normalform} (Section \ref{sec:uniformcoordinatesweaklyhyperbolicreg}). 
    \item We verify that the maps $F_n$, thanks to the arithmetic properties of $\beta$, verify the so-called covering and equidistribution  properties considered in \cite{NassiriPujalsTransitivity}  (Section \ref{sec:coveringIFS}).

    \item We study the dynamics of $s$-curves under the family of maps $F_{n}$ and show the existence of a symbolic $cs$-blender (Section \ref{sec:csblenderIFS}). In particular, we prove Proposition \ref{prop:symbolicblenderintro}.
    \item We exploit the almost  reversibility of the system  to show the existence of a $cu$-blender which is homoclinically related to the $cs$-blender and complete the proof of Theorem \ref{thm:transitivityIFS} (Section \ref{sec:localtransIFSproof}).
\end{enumerate}

\

\subsection{Linear approximation of $T_0^n\circ T_1$}\label{sec:linearapprox}
In the following lemma, given a range of  values for $n\in\mathbb N$, we construct a   sufficiently small rectangle on which the map $T_0^n\circ T_1$ (see \eqref{eq:T0map} and \eqref{eq:T1map}) is approximately affine.
\begin{lem}\label{lem:c1control}
Suppose that $\tau\geq\varepsilon$. Then, for any $n\in\mathbb N$ satisfying $n\tau\varepsilon\leq 1$  the map 
\[
F_n:=T_0^n\circ T_1:B\cap\{|J|\leq \varepsilon \}\subset \mathbb A\to \mathbb A
\]
is given by 
\begin{equation}\label{eq:definitionFn}
F_n:\begin{pmatrix}\varphi\\J\end{pmatrix}\mapsto \bs b_n+ A_n\binom{\varphi}{J}+\mathcal E(\varphi,J)
\qquad\text{with}\qquad
A_n= \begin{pmatrix} 1+n\tau\varepsilon&n\tau\\  \varepsilon&1\end{pmatrix},\qquad \bs b_n=\begin{pmatrix}b+[n\beta]\\0\end{pmatrix},
\end{equation}
for some $b\in\mathbb R$ and  $\mathcal E=(\mathcal E_\varphi,\mathcal E_J)^\top$ such that
\[
\mathcal E_\varphi(\varphi,J)=O(\varepsilon,n\tau\varepsilon \varphi^2)\qquad\qquad \mathcal E_J(\varphi,J)=O(n\varepsilon^3,\varepsilon\varphi^2).
\]
Moreover,
\[
D(T_0^n\circ T_1)(\varphi,J)=A_n+\begin{pmatrix}O(n\tau\varepsilon\varphi,n\varepsilon^2)&O(n\varepsilon)\\ O(\varepsilon\varphi,n\varepsilon^3)&O(n\varepsilon^2) \end{pmatrix}.
\]
\end{lem}

The proof of this result is given in Appendix \ref{sec:appendixtechlemmas}. We now study the linear approximation of the map $F_n$.  The proof of the following result is a straightforward computation.

\begin{lem}\label{lem:eigenv}
The symplectic matrix $A_n$ is hyperbolic with eigenvalues $0<\lambda_n<1<1/\lambda_n$ satisfying
\begin{equation}\label{eq:eigenvalues}
\lambda_n=1-\sqrt{n\tau\varepsilon}+O(n\tau\varepsilon)
\end{equation}
and contracting/expanding eigenspaces spanned, respectively, by the vectors
\[
\bs v_n=(1,v_n)^\top\qquad\qquad \bs w_n=(1,w_n)^\top,
\]
where
\begin{equation}\label{eq:eigenspaces}
v_n=-\sqrt{\frac{ \varepsilon}{n\tau}}\ (1+O(\sqrt{n\tau\varepsilon}))\qquad\qquad w_n=\sqrt{\frac{\varepsilon}{n\tau}}\ (1+O(\sqrt{n\tau\varepsilon})).
\end{equation}
\end{lem}

\subsection{Weakly hyperbolic regime: range of iterates and geometry of the domain}\label{sec:uniformcoordinatesweaklyhyperbolicreg}

The crucial feature needed for the construction of a symbolic blender is that the rates of expansion/contraction are much weaker than the speed of equidistribution. For that reason, we focus on the range of iterates
\[
\textbf{(Weakly hyperbolic regime)}\qquad\qquad 0<n\tau\varepsilon\ll 1
\]
in which the rates of expansion/contraction are governed by $0<\chi\ll 1$ (see Lemma \ref{lem:eigenv}). As we will see below the sequence $\{[n\beta]\}_n$ equidistributes much faster (i.e. it needs much less iterations) provided $\alpha\ll \varepsilon$. 

\begin{rem}
Our construction below could also be  easily adapted to the range of iterates  $0<n\tau\varepsilon \lesssim 1$. The only reason why we have chosen to  concentrate in the weakly hyperbolic regime $0<n\tau\varepsilon\ll 1$ is that some parts of the argument simplify slightly in that setting.
\end{rem}

In the following we will introduce two quantifiers which specify: the \textit{range of iterates}, i.e., for which $n\in\mathbb N$ we want to study the maps $F_n$, and the \textit{geometry} of the domain  in which we want to study the maps $F_n$. This is done as follows:
\begin{itemize}
\item We fix any $0<\kappa\ll \chi\ll 1$.
\item We let $N\in\mathbb N$ be given by
\begin{equation}\label{eq:definitionN}
N=\left[\frac{\chi^2}{\varepsilon\tau}\right].
\end{equation}
\item Let $\alpha$ be the constant introduced in \eqref{eq:Diophantinetype}. We define $N_*\in\mathbb N$ as 
\begin{equation}\label{eq:definitionNstar}
N_*=\left[\frac{1}{5\alpha \chi \kappa}\right].
\end{equation}
\end{itemize}
We will consider iterations $T_0^n\circ T_1$ with $n\in\{N,\dots, N+N_*\}$ and describe the dynamics of points in a domain which depends on the quantities $\kappa$ and $\chi$.  We obtain results for $0<\kappa\ll \chi\ll 1 $ small, but fixed, in the regime where 
\[
0<\varepsilon\ll \alpha, \tau
\]
is made arbitrarily small (with respect to our choice of $\kappa$ and $\chi$). Observe that in this weakly hyperbolic regime, the eigenvalues of $A_n$ are approximately given by $\lambda_n=1\pm\chi+O(\chi^2)$. Moreover, the following inhomogeneous version of Dirichlet approximation  theorem provides a density estimate for the sequence $\{[n\beta]\}_{n\in\{N,\dots,N+N_*\}}$.


\begin{thm}[Theorem VI in Chapter 5 of \cite{Cassels72}]\label{thm:Dirichlet}
Let $\beta\in\mathcal B_\alpha$. Then, for $N$ large enough the sequence $\{[n\beta]\}_{n=1,\dots,N}\subset \mathbb T$ is $N^{-1}\alpha^{-1}$-dense in $\mathbb T$.
\end{thm}

Indeed, this theorem implies that  the sequence\footnote{The important observation here is that the domain in which we will work is of size $\kappa\tau$, so the sequence $\{[n\beta]\}$ is $\frac{1}{5}\chi$-dense at that scale, while the hyperbolicity is just slightly stronger: of strength $\chi$.}
\[
\{[n\beta]\}_{n\in\{N,\dots,N+N_*\}} \qquad\text{is } \frac{1}{5}\chi\kappa\text{-dense in }\mathbb T.
\]
In the following section we exploit the fact that, in this regime,
the ratio
\[
\frac{N_*}{N}\lesssim \frac{\varepsilon\tau}{\alpha \kappa\chi^3}
\]
can be made arbitrarily small to construct a uniform (for $n\in\{N,\dots,N+N_*\}$) hyperbolic coordinate system and study the geometry of the images of a suitable rectangle under the map $F_n$.

\begin{figure}
\centering
\includegraphics[scale=0.8]{ 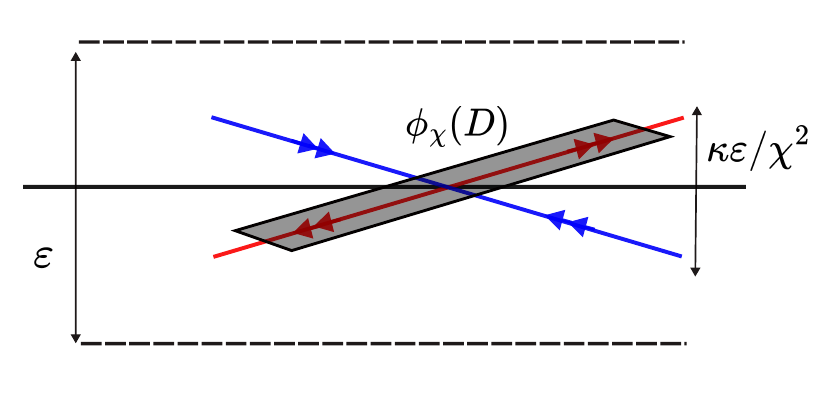}\caption{The expanding (red) and contracting (blue) directions associated to the hyperbolic matrix $A_N$. These directions are approximately symmetric with respect to the $x$-axis and the angle between them is of order $\angle \sim \varepsilon/\chi^2$. The gray rectangle corresponds to the image of $D$ under $\phi_{\chi,\varepsilon}$.}\label{fig:Fig4}
\end{figure}

\subsubsection*{Uniform coordinate system}
We fix any $0<\kappa\ll \chi\ll 1$ and, for $\varepsilon>0$ small enough, we introduce the (local) linear change of coordinates given by (here $v_N$ and $w_N$ are as in Lemma \ref{lem:eigenv})
\[
\begin{pmatrix}\tilde \xi\\ \tilde \eta\end{pmatrix}\mapsto \begin{pmatrix}\varphi\\J\end{pmatrix}=\psi_P(\tilde \xi,\tilde \eta)=\underbrace{\begin{pmatrix} 1&1\\v_N&w_N\end{pmatrix}}_{P}\binom{\tilde \xi}{\tilde \eta},
\]
we define the scaling 
\[
\begin{pmatrix}\xi\\  \eta\end{pmatrix}\mapsto \begin{pmatrix}\tilde \xi\\ \tilde \eta\end{pmatrix}=\psi_S(\xi,\eta)=\underbrace{\begin{pmatrix} \kappa&0\\ 0& \kappa /\chi\end{pmatrix}}_{S}\begin{pmatrix}\xi\\  \eta\end{pmatrix},
\]
and we denote by 
\begin{equation}\label{def:Pmap}
\phi_{\chi,\varepsilon}:=\psi_P\circ \psi_S:D\subset\mathbb R^2 \to\mathbb R^2,\qquad\qquad D=[-1,1]^2.
\end{equation}
We notice that,  in view of the asymptotics in  \eqref{eq:eigenspaces}, for the choice of $N$ in \eqref{eq:definitionN} and for $0<\kappa\ll \chi\ll 1$ sufficiently small  (see Figure \ref{fig:Fig4}) 
\[
\phi_{\chi,\varepsilon}(D)\subset \left[-\frac{2\kappa}{\chi}, \frac{2\kappa}{\chi}\right] \times \left[- \frac{2 \varepsilon \kappa}{\chi^2},\frac{2 \varepsilon \kappa}{\chi^2}\right]\subset  B\times\left[-\varepsilon,\varepsilon\right]
\]
so we can make use of Lemma \ref{lem:c1control} to analyze the dynamics of the family of maps 
\begin{equation}\label{eq:goodmaps}
\mathcal F_n:=\phi_{\chi,\varepsilon}^{-1}\circ F_n\circ \phi_{\chi,\varepsilon}:D\to\mathbb R^2
\end{equation}
with $n\in\{N,\dots, N+N_*\}$ and $F_n$ as in Lemma \ref{lem:c1control}.
\begin{lem}\label{lem:uniformlemma}
Fix any $\chi\ll 1$. Then, there exists $\kappa_0(\chi)>0$ and $\varepsilon_0(\kappa,\chi)>0$ such that for any 
\begin{equation}\label{eq:smallness}
\qquad\qquad 0<\kappa\leq \kappa_0(\chi)\qquad\qquad 0<\varepsilon\leq \varepsilon_0(\kappa,\chi) \min\{\tau,\alpha\}
\end{equation}
the following holds. There exists a subset $\mathcal N_\chi\subset \{N,\dots,N+N_*\}$ for which the maps $\mathcal F_n: D\to\mathbb R^2$  defined in \eqref{eq:goodmaps} with $n\in\mathcal N_\chi$ are of the form 
\begin{equation}\label{eq:normalformifsproof}
\mathcal F_n:\begin{pmatrix}\xi\\ \eta\end{pmatrix}\mapsto \binom{b_n}{0}+\underbrace{\begin{pmatrix} 1-\chi&0\\
0&1+\chi\end{pmatrix}}_{\mathtt A}\binom{\xi}{\eta}+\mathtt E_n(\xi,\eta),
\end{equation}
the sequence $\{b_n\}_{n\in\mathcal N_\chi}$ is $\frac{1}{10}\chi$-dense in $[-10\chi,10\chi]$, and $
|\mathtt E_n|_{C^1}=O(\chi^2)$.
\end{lem}
The proof of Lemma \ref{lem:uniformlemma} is given in Appendix \ref{sec:appendixtechlemmas}. From now on, having fixed $0\ll \chi\ll 1$ and any $0<\kappa\leq \kappa_0(\chi)$ we consider values of $\tau,\alpha$ and $\varepsilon$ such that \eqref{eq:smallness} holds and drop these quantities from the notation.

\subsection{Covering and well distributed periodic orbits}\label{sec:coveringIFS}
We now study the geometry of the images of $D$ under the maps $\mathcal F_n$, $n\in \{N,\dots , N+N_*\}$.

\begin{prop}\label{prop:covering}
Let  $\mathcal N_\chi\subset\{N,\dots , N+N_*\}$ be as in Lemma \ref{lem:uniformlemma}. Then,  
\[
D\subset \bigcup_{n\in\mathcal N_\chi} \mathcal F_n (D).
\]
Moreover, denote by $B_{r,\xi}(z)$ the horizontal segment centered at $z\in D$ of radius $r$. Then, the number
\begin{equation}\label{eq:definitionkappa}
a:=\min\{r\in\mathbb R_+\colon   \text{ there exists } z\in D\text{ such that } B_{r,\xi}(z)\subset D\text{ and } B_{r,\xi}(z)\not\subset \mathcal F_n(D) \text{ for any } n\in\mathcal N_\chi\}
\end{equation}
satisfies
\[
a\geq 1-10\chi.
\]
\end{prop}

\begin{proof}
The first observation is that, in view of the estimates in Lemma \ref{lem:uniformlemma}, in order to prove the first item it is enough to show that
\begin{equation}\label{eq:approxcovering}
\{z\in\mathbb R^2\colon \mathrm{dist}(z,D)\leq \chi/2\}\subset   \bigcup_{n\in \mathcal N_\chi} \mathtt F_n(D)
\end{equation}
where 
\[
\mathtt F_n:(\xi,\eta)\mapsto \binom{b_n}{0}+ \mathtt A\binom{\xi}{\eta},\qquad\qquad \mathtt A=\begin{pmatrix}1-\chi&0\\
0&1+\chi\end{pmatrix}.
\]
Indeed, \eqref{eq:approxcovering} guarantees that a $ \chi/2$-neighborhood of $D$ is contained in the union of its images under $\mathtt F_n$ while the error committed in approximating $\mathcal F_n$ by $\mathtt F_n$ is of order $ O(\chi^2)$, which can be made much smaller than $\chi/2$ by decreasing, if necessary, the value of $\chi$.
\medskip

Recall that   $\{b_n\}_{n\in\mathcal N_\chi}$ forms a $\frac{1}{10}\chi$-dense grid of $[-10\chi,10\chi]$, i.e., for every $\xi$ in this interval there exists $n\in\mathcal{N}_\chi$ such that $|b_n-\xi|\leq \frac{1}{10}\chi$. In particular, there exist $N_\pm\in\mathcal N_\chi$ such that 
\[
b_{N_+}\in(2\chi ,3\chi )\qquad\qquad b_{N_-}\in(-3\chi ,-2\chi ).
\]  
On the other hand
\[
\mathtt F_n(D)\cap\{-1\leq \eta\leq 1\}=[b_n-1+\chi , b_n+ 1-\chi ]\times[-1,1].
\]
so $\{z\in\mathbb R^2\colon \mathrm{dist}(z,D)\leq \chi /2\} \subset F_{N_+}(D)\cup F_{N_-}(D)$.

We now establish the second item. Again it is enough to show that $\tilde a\geq 1-5\chi $ where $\tilde a$ is defined as in \eqref{eq:definitionkappa} but with $\mathtt F_n$ replacing $\mathcal F_n$. Let $\xi\in(-1,0]$ (the case $\xi\in[0,1)$ can be dealt with analogously replacing $N_-$ below with $N_+$) and let 
\[
r\in(0,\min\{1-|\xi|,1-5\chi \})
\]
(notice that for $r>1-|\xi|$ the ball $B_r((\xi,\eta))$ is not contained in $D$). Then,
\[
\xi+r\leq 1-5\chi \leq  b_{N_-}+1-\chi \qquad\text{and}\qquad \xi-r\geq -1\geq  b_{N_-}-1+\chi,
\]
which imply
\[
[\xi-r,\xi+r]\subset [b_{N_-}-1+\chi , b_{N_-}+ 1-\chi ].\qedhere
\]
\end{proof}

In the following proposition we prove the existence of two well-distributed hyperbolic fixed points.
\begin{prop}\label{prop:welldistributed}
 Let $\mathcal N_\chi\subset\mathbb N$ and $b_n$ be as in Lemma \ref{lem:uniformlemma}. There exist $N_l,N_r\subset \mathcal N_\chi$ for which 
\begin{equation}\label{def:Itilde}
b_{N_l}\in\left(-\frac34\chi,-\frac 14\chi \right)\qquad\qquad b_{N_r}\in\left(\frac14\chi,\frac 34\chi \right).
\end{equation}
Each of the maps $\mathcal F_{N_l}$ and $\mathcal F_{N_r}$ has a unique (hyperbolic) fixed point
\[
z_n=\left(\frac{b_n}{\chi }+o(\chi ),\ b_n+ o(\chi )\right)^\top\qquad\qquad n=N_l,N_r
\]
contained in $D$. Moreover, 
\begin{itemize}
\item its unstable manifold is a fully crossing vertical curve, that is, a curve which admits  a graph parametrization of the form 
\[
W^u(z_n;\mathcal F_n)=\{(f_n(\eta),\eta)\colon \eta\in[-1,1]\}
\]
for some differentiable function $f_n$. Moreover, the function $f_n$ satisfies
\[
f_n(\eta)= \frac{b_n}{\chi }+o_{C^1}(\chi ),
\]
\item its stable manifold is a fully crossing horizontal curve, that is admitting a parametrization of the form 
\[
W^s(z_n;\mathcal F_n)=\{(\xi,\tilde f_n(\xi))\colon \xi\in[-1,1]\}\qquad \text{with}\qquad \tilde f_n(\xi)=o_{C^1}(\chi ).
\]
\end{itemize}
\end{prop}
 The proof of this proposition is deferred to Appendix \ref{sec:appendixtechlemmas}.


\begin{figure}
\centering
\includegraphics[scale=0.5]{ 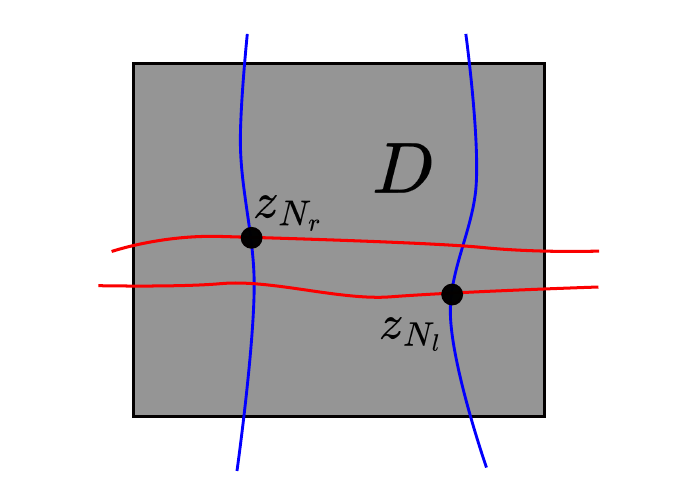}\caption{The hyperbolic fixed points $z_{N_l},z_{N_r}$ together with their stable (red) and unstable (blue) manifolds.}\label{fig:Fig5}
\end{figure}

\subsection{Existence of a symbolic $cs$-blender}\label{sec:csblenderIFS}
We call $s$-curve any curve $\gamma\subset D$ of the form
\begin{equation}\label{eq:scurvedefn}
\gamma=\{(\xi,h(\xi))\colon \xi\in I\}
\end{equation}
for some open interval $I\subset [-1,1]$ and a $C^1$ function $h$ satisfying $|h|_{C^1}\leq 1$. 

\begin{rem}
    Notice that the definition of $s$-curve above is just the coordinate formulation of the definition presented in Section \ref{sec:symbolic} in the context of symbolic blenders.
\end{rem}

We now complete the proof of Proposition \ref{prop:symbolicblenderintro}, i.e we prove the existence of a symbolic $cs$-blender, by showing the following.

\begin{prop}\label{prop:csblender}
Let $\mathcal N_\chi\subset\mathbb N$ be the subset in Lemma \ref{lem:uniformlemma}. Let $\gamma$ be a $s$-curve and for $n\in \{N_l,N_r\}\subset\mathcal N_\chi$ and let $W^{u,s}(z_n;\mathcal F_n)$ be the invariant manifolds of the hyperbolic fixed points constructed in Proposition \ref{prop:welldistributed}. There exists $n\in\{N_l,N_r\}$,  $M\in\mathbb N$ and $\omega\in \mathcal N_\chi^M$ such that 
\[
(\mathcal F_{\omega})^{-1} (\gamma)\pitchfork W^u(z_n;\mathcal F_n)\neq\emptyset
\]
where we have used the notation $\mathcal F_\omega=\mathcal F_{\omega_M-1}\circ\cdots\circ \mathcal F_{\omega_0}$
\end{prop}
\begin{proof}
Let $a$ be the number defined in \eqref{eq:definitionkappa} and introduce
\[
b=\min_{n\in\{N_l,N_r\}} \max_{(\xi,\eta)\in D}\{|\xi-f_n(\eta)|\},
\]
where $f_n$ is the function in the parametrization of $\gamma^u_n$. From Proposition \ref{prop:welldistributed}
we observe that $b\leq 3/4$. Hence, it follows from Proposition \ref{prop:covering} that
\[
b\leq \frac 34<\frac{9}{10}< 1-10\chi \leq a.
\]
Given any $s$-curve $\gamma$,  if $|I|\geq 9/5>2b$ then $\gamma\pitchfork W^u(z_n;\mathcal F_n)\neq\emptyset$ for some $n\in\{N_l,N_r\}$  and the desired conclusion follows. If $|I|< 9/5<2a$ there exists $n'\in\mathcal N$ such that $\gamma\subset \mathcal F_{n'}(D)$. Therefore, $\mathcal F_{n'}^{-1} (\gamma)\subset D$ and it is easy to check (proceeding similarly as in the proof of Proposition \ref{prop:welldistributed}) that
\[
\mathcal F_{n'}^{-1}(\gamma)=\{(\xi,\bar h(\xi))\colon \xi\in\bar I\}
\]
for some open interval $\bar I\subset[-1,1]$ with $|\bar I|\geq(1+\frac12\chi )|I|$ and some $\bar h$ with 
\[
|\bar h'|\leq \frac{1}{1+\frac 12 \chi }|h'|.
\]
Since $1+\frac12\chi >1$, 
by repeating this process at most a finite number of steps  we arrive to the first scenario and we are done.
\end{proof}

We also present the following stronger version of Proposition \ref{prop:csblender} which will prove useful in Section \ref{sec:blender3bp}. 

\begin{prop}\label{prop:blenderforpairs}
     Let $\gamma,\gamma'$ be $s$-curves whose projection onto the horizontal axes overlap and let $W^u(z_n;\mathcal F_n)$ with $n\in\{N_l,N_r\}$ be as in Proposition \ref{prop:csblender}. Then, there exists $n\in\{N_l,N_r\}$, $M\in\mathbb N$ and $\omega\in\mathcal N_\chi^M$ such that 
    \[
    (\mathcal F_\omega)^{-1}(\gamma)\pitchfork W^u(z_n;\mathcal F_n)\neq \emptyset \qquad\qquad \text{and}\qquad\qquad  (\mathcal F_\omega)^{-1}(\gamma')\pitchfork W^u(z_n;\mathcal F_n)\neq \emptyset.
    \]
\end{prop}
The proof of this result can be obtained by   simple inspection of the proof of Proposition \ref{prop:csblender} and it is left to the reader.

\subsection{Existence of a symbolic double blender}\label{sec:symbolicdoubleblender}

We fix any  $0< \chi\ll 1$ and let   $0<\varepsilon\ll 1$  be sufficiently small (depending on $\alpha,\tau$) so that all the results in the preceding sections hold.
\medskip

The results obtained so far can be summarized as follows. Let 
\[
Q^{cs}=\phi(D)\qquad\qquad\qquad P^{cs}=\phi(z_{N_l})
\]
where $D=[-1,1]^2$, $\phi$ is the linear map defined in \eqref{def:Pmap} and $z_{N_l}$ is the  hyperbolic  fixed point for the map $F_{N_l}$ constructed in Proposition \ref{prop:welldistributed} (there is nothing special about choosing $z_l$ instead of $z_r$ and the argument below works in the exact same way with that choice). Then, for any $s$-curve $\gamma^{cs}\subset Q^{cs}$ there exists $M^{cs}\in \mathbb N$ and $\omega^{cs}\in \{0,1\}^{M^{cs}}$ such that $(T_{\omega^{cs}})^{-1}(\gamma^{cs})$ intersects $W^u(P^{cs}; F_{N_l})$. Namely, the pair $(Q^{cs},P^{cs})$ is a symbolic $cs$-blender. 
\medskip

 We now construct a $cu$-blender homoclinically related to the $cs$-blender above. We define the map 
\begin{equation}\label{eq:mapsforcublender}
\widetilde F_n:=T_1\circ T_0^n.
\end{equation}
Then, after some algebraic manipulations, it is not difficult to observe that 
\begin{equation}\label{eq:expressionreversibility}
\widetilde F_n^{-1}:\binom{\varphi}{J}\mapsto \tilde{\bs b}_n+\begin{pmatrix}
    1+n\tau\varepsilon &-n\tau\\
    -\varepsilon & 1
\end{pmatrix}\binom{\varphi}{J}+\widetilde{\mathcal E} (\varphi,J)
\end{equation}
with
\[
\tilde{\bs b}_n=-\bs b_n+\binom{-n\tau \varepsilon\tilde \beta (0)}{\varepsilon \tilde \beta (0)},
\]
where $\bs b_n$ as in \eqref{eq:definitionFn}, $\tilde\beta$ is the function introduced in \eqref{eq:T1map} and $\widetilde {\mathcal E}(\varphi,J)$ satisfying the very same estimates as $\mathcal E$ in Lemma \ref{lem:c1control}.
We distinguish two cases.
\medskip

\subsubsection{Case $\tilde \beta(0)=0$}\label{sec:casebeta0} This is the relevant case for the application to the construction of symplectic blenders in the 3-body problem  in Section \ref{sec:blender3bp}. In this case the maps $T_0$ and $T_1$ are ``almost-reversible'' under the involution
\begin{equation}\label{def:involution}
\psi_R:\begin{pmatrix}\varphi\\J\end{pmatrix}\mapsto \begin{pmatrix}-\varphi\\J\end{pmatrix}
\end{equation}
(that is, reversible up to small errors). One can check that this implies that, from \eqref{eq:definitionFn} and \eqref{eq:expressionreversibility}, when that $\tilde \beta(0)=0$ (note that this implies  $\tilde {\bs b}_n=-\bs b_n$)
\[
\widetilde F_n^{-1}=\psi_R\circ( T_0^n\circ T_1)\circ \psi_R+ \overline{\mathcal E}
\]
with $\overline{\mathcal E}$ satisfying the same estimates as $\mathcal E$ in Lemma \ref{lem:c1control}.
In other words, $\widetilde F_n^{-1}=(T_1\circ T_0^n)^{-1}$ is conjugate (up to small errors) to $T_0^n\circ T_1$.
\black Therefore, verbatim repetition of the discussion in Sections \ref{sec:uniformcoordinatesweaklyhyperbolicreg}, \ref{sec:coveringIFS} and \ref{sec:csblenderIFS} shows that there exists a point 
\[
    P^{cu}=\psi_R (P^{cs})+O(\chi^2)
    \]
which is a hyperbolic fixed point for the map $\widetilde F_{N_l}$ and such that the pair $(P^{cu},Q^{cu})$ where $Q^{cu}=\psi_R(Q^{cs})$
is a $cu$-blender for the IFS generated by $\{T_0,T_1\}$ (see Figure \ref{fig:Fig6}). We now show that $P^{cu}$ and $P^{cs}$ are homoclinically related to conclude the existence of a  symbolic double blender. To alleviate the notation we denote by $W^{u}(P^{cs})=W^{u}(P^{cs};F_{N_l})$ and by $W^{s}(P^{cu})=W^{s}(P^{cu};\widetilde F_{N_l})$.

\begin{lem}
    Let $P^{cu}$ and $P^{cs}$ be as above. Then, we have that  $W^{u}(P^{cs})\pitchfork W^s(P^{cu})\neq 0$
\end{lem}
\begin{proof}
On the coordinate system $(\xi,\eta)\in [-1,1]^2$ on $Q^{cs}$ given by the linear map $\phi$ in \eqref{def:Pmap}, we have shown on Proposition \ref{prop:welldistributed} that the local unstable manifold $W^u(P^{cs})$ is given by a $C^1$ curve which is almost vertical and fully crosses the rectangle $[-1,1]^2$. On the other hand, on the coordinate system $(\tilde \xi,\tilde \eta)\in[-1,1]^2$ on $Q^{cs}$ given by the linear map $\phi\circ \psi_R$, it follows by construction that $W^s(P^{cu})$ is given by a $C^1$ curve which is almost vertical and fully crosses the rectangle $[-1,1]^2$. A straightforward computation shows that on $Q^{cu}\cap Q^{cs}$ the transition map between the two coordinate charts is given by a rotation by $90$ degrees so the proof follows (see  Figure \ref{fig:Fig6}).
\end{proof}

\begin{figure}
    \centering
    \includegraphics[scale=0.7]{ 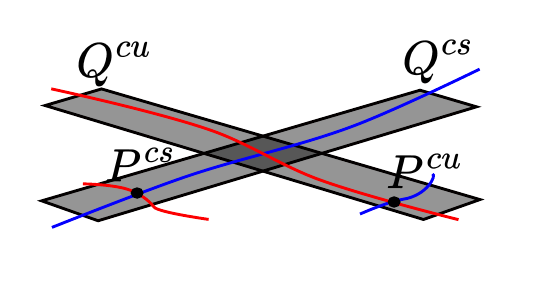}
    \caption{The $cs$-blender formed by the pair $(P^{cs},Q^{cs})$ is homoclinically related to the $cu$-blender $(P^{cu},Q^{cu})$. In blue (resp. red) we depict the local unstable (resp. stable) manifolds.}
    \label{fig:Fig6}
\end{figure}

In particular, we have proven the existence of a symbolic double blender.

\begin{prop}\label{prop:symbolicdoubleblender}
    The pairs $(P^{cu},Q^{cu})$ and $(P^{cs},Q^{cs})$ constructed above form a symbolic double blender for the IFS generated by the maps $\{T_0,T_1\}$.
\end{prop}

\subsubsection{Case $\tilde \beta(0)\neq 0$}   In this case, one can analyze directly the map $\widetilde F_n^{-1}$ in \eqref{eq:expressionreversibility}. Verbatim repetition of the steps in Sections \ref{sec:uniformcoordinatesweaklyhyperbolicreg}-\ref{sec:csblenderIFS} leads to the existence of a 
 $cu$-blender $(P^u,Q^{cu})$ (the only difference is that one has to adjust differently the corresponding set $\mathcal N_{\chi}$). We obtain the very same statement as in Proposition \ref{prop:symbolicdoubleblender}, which we do not repeat here.\black

\subsection{Proof of Theorem \ref{thm:transitivityIFS}}\label{sec:localtransIFSproof}

In this section we complete the proof of Theorem \ref{thm:transitivityIFS}. Given $\kappa>0$ (independent of $\varepsilon,\alpha$) we denote by 
\begin{equation}\label{eq:stripsaroundKAM}
\mathbb A_\varepsilon=\mathbb T\times [-\kappa\varepsilon,\kappa\varepsilon]\qquad\qquad \widetilde{\mathbb A}_\alpha=\mathbb T\times[-\alpha/|\log^3\varepsilon|, \alpha/|\log^3\varepsilon|].
\end{equation}
We divide the proof in three steps:
\begin{itemize}
    \item We first notice that if $B\subset Q^{cu}$ and $B'\subset Q^{cs}$ then, the conclusion follows from the symbolic double blender dynamics.
    \item Second, we show that if $B,B'\in\mathbb A_\varepsilon$ with $\kappa>0$ sufficiently small (independent of $\alpha,\varepsilon$), there exists $n_f,n_b\in\mathbb N$ such that $T^{n_f}(B)\cap Q^{cu}\neq\emptyset$ and $T^{-n_b}(B')\cap Q^{cs}\neq\emptyset$.
    \item Finally, we complete the proof by showing that for any $B,B'\in \widetilde{ \mathbb A}_\alpha $ there exist $n_f',n_b'\in\mathbb N$ such that $T^{n_f'}(B)\cap \mathbb A_\varepsilon \neq\emptyset $ and $T^{-n_b'}(B')\cap \mathbb A_\varepsilon \neq\emptyset$.
\end{itemize}

\noindent\textbf{Step 1:}
$B$ is open so it contains a $u$-curve $\gamma_B$. Hence, since $(P^{cu},Q^{cu})$ is a symbolic $cu$-blender, we must have that $W^s(P^{cu})$ intersects $\gamma_B$ transversally.  On the other hand, $B'$ is open so it contains a  $s$-curve $\gamma_{B'}$ and the fact that $(P^{cs},Q^{cs})$ is a symbolic $cs$-blender impiles that $W^u(P^{cs})$ intersects $\gamma_{B'}$ transversally. The conclusion now follows from a direct application of the lambda-lemma (see \cite{PalisMelo}) and the fact that $P^{cs}$ and $P^{cu}$ are homoclinically related.

\noindent \textbf{Step 2:} We only deal with the existence of $n_f$, the existence of $n_b$ being deduced from the same argument. Let $(\varphi,J)\in \mathbb A_\varepsilon$ and recall that 
\[
Q^{cu}=\psi_R\circ \phi([-1,1]^2),
\]
with $\phi:[-1,1]^2\to\mathbb A$ as in Lemma \ref{lem:uniformlemma} and $\psi_R:(\varphi,J)\mapsto (-\varphi,J)$. We show that, provided $\kappa>0$ is small enough,  there exists $M\in\mathbb N$ such that $T_0^{M}(\varphi,J)\in  Q^{cu}$.
This is done in two steps:
\begin{itemize}
    \item First, we notice that there exists $\ell>0$ independent of $\varepsilon$ such that, for any $\kappa>0$ small and for any $J_*\in[-\kappa\varepsilon,\kappa\varepsilon]$ there exists an interval $I_{J_*}\subset\mathbb T$ of length $|I_{J_*}|\geq \ell$ such that 
    \[
    I_{J_*}\times[J_*-\varepsilon^2,J_*+\varepsilon^2]\subset \psi_R\circ\phi([-1/2,1/2]^2)\subset Q^{cu}.
    \]
    Indeed, if we denote by $\bs\ell_{u},\bs\ell_{d}$ the upper and lower sides of $Q^{cu}$, these correspond to segments of the lines $\bs{\ell}_{u,d}=\{J= a\varphi+b_{u,d}\}$ with $|a|\sim \varepsilon$ and $|b_u-b_d|\sim\varepsilon$. It follows that $\ell\gtrsim |b_u-b_d|/|a|\sim 1$.
\item Second, for any $C>0$,  any $n\leq C/\alpha$, and any
$(\varphi,J)\in \mathbb A_\varepsilon$, we have
\begin{equation}\label{eq:dense}
T_0^n(\varphi,J)=\begin{pmatrix}\varphi+[n\beta]+O(C\kappa\alpha^{-1}\varepsilon)\\ J+O(C\kappa^3\alpha^{-1}\varepsilon^3)\end{pmatrix}.
\end{equation}
By Theorem \ref{thm:Dirichlet} the sequence $\{[n\beta]\}_{n=1,\dots,[C/\alpha]}$ is $\frac{1}{C}$-dense in $\mathbb T$ so, for $C\geq  10\ell^{-1}$,  in view of \eqref{eq:dense} and the fact that $|I_{J_0}|\geq \ell$ we deduce that that, any
$(\varphi,J)\in \mathbb A_\varepsilon$,
\[
\{T_0^n(\varphi,J)\}_{n=1,\dots,[C/\alpha]}\cap (I_{J}\times[J-\varepsilon^2, J+\varepsilon^2])\neq \emptyset.
\]
\end{itemize}

\begin{rem}
    Up to now we have not used at all that the maps $T_0,T_1$ are real-analytic but just $C^2$ estimates. Notice that we have already obtained a proof of the Remark \ref{rem:c3maps}.
\end{rem}

\noindent\textbf{Step 3:} We only deal with the existence of $n_f'$, the existence of $n_b'$ being deduced from the same argument. We rely on the following Birkhoff normal form type result. Although this result is rather standard, we present a proof in Appendix \ref{sec:appendixtechlemmas} to keep track of some quantitative estimates. Given $\rho>0$ we let $\mathbb B_\rho\subset\mathbb C$ the complex ball around the origin of radius $\rho$.
\begin{lem}\label{lem:BNF}
Let $\rho,\sigma>0$ be as in Theorem \ref{thm:transitivityIFS}. Fix any $k\in\mathbb N$ and let 
\[
\rho_0(\alpha,k,\rho,\sigma)= \alpha\frac{\sigma^3\rho^2}{4k^3}.
\]
Then, there exists a real-analytic, exact-symplectic change of variables $\Phi: \mathbb T_{\frac{\sigma}{2}}\times \mathbb B_{\frac 12\rho_0}\to \mathbb A_{\rho,\sigma}$ of the form 
\[
\Phi:\begin{pmatrix}\varphi\\J\end{pmatrix}\mapsto \begin{pmatrix}\varphi+\phi_\varphi(\varphi,J)\\J+\phi_J(\varphi,J)\end{pmatrix},
\]
with $\partial_J^n\phi_*(\varphi,0)=0$ for $n=0,1$ if $*=\varphi$, $n= 0,1,2$ if $*=J$ and such that conjugates the map $T_0$ in \eqref{eq:T0map} to 
\[
\mathtt T_0:=\Phi^{-1}\circ T_0\circ\Phi:\begin{pmatrix}\varphi\\J\end{pmatrix}\mapsto \begin{pmatrix}\varphi+h(J)+\widetilde R_\varphi(\varphi,J)\\ J+\widetilde R_J(\varphi,J)\end{pmatrix}
\]
with $h(J)=\beta+\tau J+O_2(J)$ and  $\partial^n_J \widetilde R_*(\varphi,0)=0$ for $0\leq n\leq k-1$ if $*=\varphi$ and $0\leq n\leq k$ if $*=J$. Moreover, uniformly for $(\varphi,J)\in \mathbb T_{\frac{\sigma}{2}}\times \mathbb B_{\frac12\rho_0}$,
\begin{equation}\label{eq:changeBNF}
|\phi_\varphi(\varphi,J)|\lesssim \left(\frac{|J|}{\rho_0}\right)^{ 2}\qquad\qquad |\phi_J(\varphi,J)|\lesssim\left(\frac{|J|}{\rho_0}\right)^{3}
\end{equation}
and
\begin{equation}\label{eq:smallnessremainderBNF}
|\widetilde R_\varphi(\varphi,J)|\lesssim 2^{-k}\left(\frac{|J|}{\rho_0}\right)^{k-1} \qquad\qquad |\widetilde R_J(\varphi,J)|\lesssim 2^{-k}\left(\frac{|J|}{\rho_0}\right)^k.
\end{equation}
\end{lem}

\begin{rem}
    The smallness asumptions in Lemma \ref{lem:BNF} (i.e. the definition of $\rho_0(\alpha,k,\rho,\sigma)$) are very far from optimal. However, they will be enough for our purposes.
\end{rem}

We now express the map $T_1$ in the new coordinate system.
\begin{lem}
    Let $k\in\mathbb N$ and let $\Phi:\mathbb T_{\frac{\sigma}{2}}\times \mathbb B_{\frac12\rho_0}\to\mathbb A_{\rho,\sigma}$ be as in Lemma \ref{lem:BNF}. Then, 
    \[
    \mathtt T_1:=\Phi^{-1}\circ T_1\circ \Phi:\begin{pmatrix}\varphi\\J\end{pmatrix}\mapsto \begin{pmatrix}\varphi+\varepsilon\mathtt T_{1,\varphi}(\varphi,J;\varepsilon)\\ J+\varepsilon\mathtt T_{1,J}(\varphi,J;\varepsilon)\end{pmatrix}
    \]
and $
\mathtt T_{1,J}(\varphi,J;\varepsilon)=T_{1,J}(\varphi,J;\varepsilon)+\widetilde T_{1,J}(\varphi,J;\varepsilon)$ with 
\[
\partial_J^n\widetilde T_{1,J}(\varphi,0;\varepsilon)=0\qquad\qquad \text{for }n=0,1
\]
and, uniformly, for all $(\varphi,J)\in \mathbb T_{\frac{\sigma}{2}}\times \mathbb B_{\frac12\rho_0}$
\[
|\widetilde T_{1,J}(\varphi,J;\varepsilon)| \lesssim \left( \frac{|J|}{\rho_0}\right)^2
\]
In particular, there exists a smooth curve
$\varphi_*=\varphi_*(J)$ such that 
\begin{equation}\label{eq:curveintersection}
\mathtt T_{1,J}(\varphi_*(J),J;0)=0\qquad\qquad  a(J):=\partial_\varphi\mathtt T_{1,J}(\varphi_*(J),J;0)=1+O(J^2)> 0.
\end{equation}
\end{lem}
\medskip

The proof of this result follows from elementary computations and is left to the reader. We now choose 
\[
k=k_*=\frac{3|\log\varepsilon|}{\log 2}
\]
so, uniformly for $(\varphi,J)\in\widetilde{\mathbb A}_\alpha$ (recall the definition of this (real) annulus in \eqref{eq:stripsaroundKAM})
\[
\mathtt T_0:\begin{pmatrix}\varphi\\J\end{pmatrix}\mapsto \begin{pmatrix}\varphi+\beta+O(\alpha /|\log^3\varepsilon|)\\  J+O(\varepsilon^3)\end{pmatrix}.
\]
Since $\beta\in \mathcal B_\alpha$, there exists $C>0$ (independent of $\alpha$ and $\varepsilon$) such that, if we let $
    \widetilde C(\alpha,\varepsilon)=C\frac{\log\varepsilon}{\alpha}$ the set $\{[n\beta]\}_{n\leq \widetilde C(\alpha,\varepsilon)}$ is $1/|\log\varepsilon|$-dense in $\mathbb T$. On the other hand, for any $(\varphi,J)\in \widetilde{\mathbb A}_{\alpha}$ and $n\leq \widetilde C(\alpha,\varepsilon)$
   \begin{equation}\label{eq:adjustingangle}
T_0^n:\binom{\varphi}{J}\mapsto \begin{pmatrix}\varphi+[n\beta]+O(n\alpha/\log^2\varepsilon)\\ J+O(n \varepsilon^3) \end{pmatrix}=\begin{pmatrix}\varphi+[n\beta]+O(1/\log^2\varepsilon)\\ J+O(\varepsilon^2\log\varepsilon) \end{pmatrix},
   \end{equation}
where we have used that $\varepsilon\lesssim \alpha$. We choose $n_\pm(\varphi,J)\in\mathbb N$ such that 
\[
\varphi+[n_+\beta]-\varphi_*(J)\in (1/\log\varepsilon,4/\log\varepsilon)\qquad\qquad \varphi+[n_-\beta]-\varphi_*(J)\in (-4/\log\varepsilon,-1/\log\varepsilon).
\]
Then, after writing 
\[
\pi_J \mathtt T_1(\varphi,J)=\varepsilon a(J)(\varphi-\varphi_*(J))+O(\varepsilon|\varphi-\varphi_*(J)|^2,\varepsilon^2),
\]
it follows from \eqref{eq:adjustingangle} that 
\begin{align*}
\Delta_+J:=\pi_J (T_1\circ T_0^{n_+})(\varphi,J)-J&\in \left(2a(J)\frac{\varepsilon}{\log\varepsilon},3a(J)\frac{\varepsilon}{\log\varepsilon}\right)\\
\Delta_-J:=\pi_J (T_1\circ T_0^{n_-})(\varphi,J)-J&\in\left(-3a(J)\frac{\varepsilon}{\log\varepsilon},-2a(J)\frac{\varepsilon}{\log\varepsilon}\right) .
\end{align*}
If $J+\Delta_\pm J\in[-\kappa\varepsilon,\kappa\varepsilon]$ (for $\kappa$, independent of $\varepsilon,\alpha$, as in Step 2) we are done. If not, we repeat the argument a finite number of times. The proof of Theorem \ref{thm:transitivityIFS} is completed.

\begin{rem}\label{rem:boundedtimetransport}
  For the applications of these ideas to the skew-product setting in Section \ref{sec:skewproduct} it will be important to bear in mind that, the construction in Steps 2 and 3 actually shows that, for a fixed value of $\alpha,\varepsilon$, there exists a uniform $M$ such that for any $B,B'\in \widetilde{\mathbb A}_\alpha $ there exist $n_f',n_b'\in\{1,\dots,M\}$ such that $T^{n_f'}(B)\cap Q^{cu}\neq\emptyset $ and $T^{-n_b'}(B')\cap Q^{cs} \neq\emptyset$.
\end{rem}


\section{Almost transitivity of cylinder skew-products: proof of Theorem \ref{thm:skewproduct}}\label{sec:skewproduct}
In this section we present the proof of Theorem \ref{thm:skewproduct}. The proof shares many ideas with the proof of Theorem \ref{thm:transitivityIFS} and it is divided in several steps. First, in  Section \ref{sec:techlemmas}, we state two technical lemmas: Lemma \ref{lem:normalformhighiteratesskewprod} (normal form lemma) and Lemma \ref{lem:compositionskewproducts} (comparison of center dynamics along  different base sequences). Then, in Section \ref{sec:blenderskewproduct}, we exploit the fact that for each $\omega\in \{0,1\}^\mathbb Z$ the skew-product fiber dynamics $F_\omega(z)$ is a small perturbation of a map $T_{\omega_0}$ to translate the results in Section \ref{sec:IFSlocaltransitive} (symbolic blender dynamics)  to the skew-product setting. Note that, throughout Section \ref{sec:blenderskewproduct}, we will be using the notation established in Section \ref{sec:IFSlocaltransitive}. In particular,  we will fix any $0<\chi\ll 1$ and  let $0<\varepsilon\ll 1$ be sufficiently small (depending on $\alpha,\tau$) so that all the results from Section \ref{sec:IFSlocaltransitive} hold. 
 Finally, in Section \ref{sec:proofskewprod}, we complete the proof of Theorem \ref{thm:skewproduct}.
%
%
%

\subsection{Technical lemmas}\label{sec:techlemmas}

To prove Theorem \ref{thm:skewproduct} we must  compare the fiber dynamics associated to different sequences $\omega\in\{0,1\}^\mathbb Z$. We denote the iteration of the fiber dynamics as
\[
F^n_\omega(z)= \underbrace{F_{\sigma^{n-1}(\omega)}\circ\cdots\circ F_\omega}_n(z)
\]
and 
\[
(F^n_{\sigma^{-n}(\omega)})^{-1}(z)=\underbrace{F^{-1}_{\sigma^{-n}(\omega)}\circ \cdots \circ F^{-1}_{\sigma^{-1}(\omega)}}_n (z).
\]

\begin{lem}\label{lem:compositionskewproducts}
    Let $\delta>0$ be small enough. Let $n\in\mathbb N$ and $\omega,\omega'\in\{0,1\}^\mathbb Z$ with $\omega_k=\omega'_k$ for all $-n\leq k\leq n$. Then, 
    \[
    |F^n_\omega- F^n_{\omega'}|_{C^0}\lesssim \delta\qquad\qquad |DF^n_\omega- DF^n_{\omega'}|_{C^0}\lesssim \max\{|DT_0|,|DT_1|\}\delta
    \]
   and 
     \[
    |(F^n_{\sigma^{-n}(\omega)})^{-1}- (F^n_{\sigma^{-n}(\omega')})^{-1}|_{C^0}\lesssim \delta\qquad\qquad |(DF^n_{\sigma^{-n}(\omega)})^{-1}- (DF^n_{\sigma^{-n}(\omega')})^{-1}|_{C^0}\lesssim \max\{|DT_0|^{-1},|DT_1|^{-1}\}\delta
    \]
\end{lem}

\begin{proof}
  The proof follows by induction. We only prove the case $n=2$ from which the reader can easily extrapolate  the argument for the general case. We write 
  \[
  F^2_\omega(z)-F^2_{\omega'}(z)=\underbrace{F_{\sigma(\omega)}(F_\omega(z))-F_{\sigma(\omega')}(F_\omega(z))}_{\mathcal E_1}+ \underbrace{F_{\sigma(\omega')}(F_\omega(z))-F_{\sigma(\omega')}(F_{\omega'}(z))}_{\mathcal E_2}
  \]
  On one hand, 
  \[
  (\sigma(\omega))_k=(\sigma(\omega'))_k\qquad\qquad\text{for }k=0,1
  \]
  so it follows from the assumption \eqref{eq:almostlocallyconstant} that $|\mathcal E_1|_{C^0}\leq \delta$. On the other hand, \eqref{eq:almostlocallyconstant} implies as well that $|F_\omega-F_{\omega'}|_{C^0}\leq \delta^2$. Hence, by the mean value theorem $|\mathcal E_2|_{C^0}\leq \max\{|DT_0|,|DT_1|\}\delta^2$. We conclude that 
  \[
  |F^2_\omega(z)-F^2_{\omega'}|_{C^0}\leq \delta(1+\max\{|DT_0|,|DT_1|\}\delta)\lesssim \delta.
  \]
  The estimates for the differential, the inverse, and the differential of the inverse, are obtained in a similar fashion.
\end{proof}

We now obtain normal forms for compositions $F_\omega^n(z)$ associated to sequences $\omega\in\{0,1\}^\mathbb Z$ which reproduce the weak transversality-torsion mechanism. Since for any $\omega\in\{0,1\}^\mathbb Z$,  $F_\omega(z)=T_{\omega_0}+O_{C^1}(\delta)$, verbatim repetition of the arguments in Section \ref{sec:IFSlocaltransitive} shows the following.
\begin{lem}\label{lem:normalformhighiteratesskewprod}
   Fix any $0<\chi\ll 1$. Then, there exists $\varepsilon_0(\chi,\tau,\alpha)$ such that if $0<\varepsilon\leq \varepsilon_0$ there exists a local coordinate system (the one given in Lemma \ref{lem:uniformlemma})
    \[
    \phi:[-2,2]^2\to \mathbb A 
    \]
    and a subset $\mathcal N_{\chi}\subset\mathbb N$ for which the following holds. For any $N\in\mathcal N_{\chi}$ and any $\omega\in\{0,1\}^\mathbb Z$ with 
    \[
    \omega_0=0,\qquad\qquad\omega_1=\cdots=\omega_{N-1}=1
    \]
    the map 
    \[
    \mathcal F_{\omega,N}:= \phi^{-1}\circ F^N_{\omega}\circ\phi
    \]
    satisfies that, uniformly for $(\xi,\eta)\in[-2,2]^2$, provided $\delta$ is small enough,
    \begin{equation}\label{eq:scaleddiagonalcentermaps}
    \mathcal F_{\omega,N}:\binom{\xi}{\eta}\mapsto \binom{b_N}{0}+\begin{pmatrix}
        1-\chi&0\\
        0&1+\chi
    \end{pmatrix}\binom{\xi}{\eta}+O_{C^1}(\chi^2)
    \end{equation}
    for some constant  $b_N\in[-1,1]$. Moreover, the sequence $\{b_N\}_{N\in\mathcal N}$ is $\frac{1}{10}\chi$-dense in $[-10\chi,10\chi]$.
\end{lem}

\subsection{$cs$-blender dynamics}\label{sec:blenderskewproduct}

Recall the definition of the maps $\{\mathcal F_N\}_{N\in\mathcal N}$ in Lemma  \ref{lem:uniformlemma}.

In Section \ref{sec:IFSlocaltransitive} (see Proposition \ref{prop:csblender}) we have shown that the associated iterated function system 
exhibits a symbolic  $cs$-blender. To do so, 
\begin{itemize}
    \item We have proved that (recall that $B_{r,\xi}(z)$ is the horizontal segment of radius $r$ and centered at $z$)
    \[
    a=\min\{r\in\mathbb R_+\colon   \text{ there exists } B_{r,\xi}(z)\subset D\text{ such that } B_{r,\xi}(z)\not\subset \mathcal F_n(D) \text{ for any } n\in\mathcal N\}
\]
satisfies
\[
a\geq 1-10\chi>9/10.
\]
    \item We have established the existence of $\{N_l,N_r\}\in\mathcal N$ and, for $\star=l,r$, a  hyperbolic fixed point $z_{N_\star}$ of the map $\mathcal F_{N_\star}$ with parametrizations of their local unstable manifolds of the form $W^u_{\mathrm{loc}}(z_{N_\star};\mathcal F_{N_\star})=\{(f_{N_{\star}}(\eta),\eta),\ \eta\in[-1,1]\}$. Moreover, we have proved that 
    \[
   b:=\min_{n\in\{N_l,N_r\}} \max_{(\xi,\eta)\in[-1,1]^2}|\xi-f_n(\eta)|< 3/4.
    \]
    \end{itemize}

    Exploiting the fact that (uniformly in $\chi,\varepsilon$)
    \[
  a>9/10>3/4>b,
    \]
    we have proved  that the backwards orbit (with respect to the  iterated function system generated by $\{\mathcal F_N\}_{N\in\mathcal N}$) of any $s$-curve  $\gamma\in [-1,1]^2$ intersects $W^u_{\mathrm{loc}}(z_{N_\star};\mathcal F_{N_\star})$ for some $\star=l,r$.  We now show how to adapt this construction to show the following.

\begin{prop}\label{prop:fundamentalpropskewprod}
Fix any $0<\chi\ll 1$ and let $\varepsilon>0$ be sufficiently small.  There exists $\delta_0(\chi,\varepsilon)>0$ such that for any $0\leq \delta\leq \delta_0$ the following holds.  Let $\gamma\subset[-1,1]^2$ be a $s$-curve (see \eqref{eq:scurvedefn}). Then, there exist $M\in\mathbb N$, $\{N_1,\dots,N_M\}\in \mathcal N_\chi$ and $\omega'\in\{0,1\}^{{\sum_{i=1}^M N_i}}$ such that, for any  $\omega\in \{0,1\}^\mathbb Z$ with 
\[
\omega_k=\omega_k'\qquad\text{for all }\qquad 1\leq k\leq \sum_i^M N_i
\]
the curve
\[
\tilde\gamma=\mathcal F_{\sigma^{-\sum_{i=1}^{M} N_i}(\omega),N_{M}}^{-1}\circ\cdots\circ\mathcal F_{\sigma^{-N_1-N_2}(\omega),N_2}^{-1}\circ \mathcal F_{\sigma^{-N_1}(\omega),N_1}^{-1}(\gamma)
\]
is a fully-crossing $s$-curve, i.e its projection onto the first component $\xi$ covers the interval $[-1,1]$.
\end{prop}

\begin{rem}
    It is worth pointing out that, given a $s$-curve $\gamma$, the sequence $\{N_1,\dots,N_M\}\in\mathcal N_\chi$ obtained as an output of Proposition \ref{prop:fundamentalpropskewprod} is, in general, different from the sequence obtained as an output of Proposition \ref{prop:csblender}. 
\end{rem}

\begin{proof}
    The proof of this result can be obtained using the very same inductive construction in the proof of Proposition \ref{prop:csblender}. It will be important to keep in mind that, having fixed $\chi,\varepsilon$, the subset $\mathcal N_\chi\subset\mathbb N$ in Proposition \ref{prop:covering} is bounded (it is comprised by finitely many elements). Given a sequence $\omega\in\{0,1\}^\mathbb Z$ and a $s$-curve $\gamma$:
    
    \noindent\textbf{Scenario $1$:} if $|I|\geq 9/5>2 b$, we have that $\gamma\pitchfork W^u_{\mathrm{loc}} (z_{N_\star};\mathcal F_{N_\star})\neq \emptyset$ for some $\star=l,r$. Moreover, since $9/5>3/2>2b$, we can suppose that there is a finite piece (of length bounded below) of $\gamma$ to both sides of $W^u_{\mathrm{loc}} (z_{N_\star};\mathcal F_{N_\star})$. Suppose $\star=l$ (the other case being analogous). We can then choose any $\omega\in\{0,1\}^\mathbb Z$ such that  
   \[
    \omega=(\dots,\omega_{-1},\omega_0;\underbrace{0,\dots,0,1}_{N_l},\omega_{N_l+1},\dots).
    \]
   Notice that the choice of $N_l$ depends exclusively on $\gamma$. For any $\omega$ as above, it follows from the assumptions in Theorem \ref{thm:skewproduct} and the definition of $\mathcal F_N$ in Lemma \ref{lem:uniformlemma} that, for $\delta\geq 0$ small enough (since $N_l\in\mathcal N$ and $\mathcal N$ is bounded)
    \[
    |\mathcal F_{\sigma^{-N_l}(\omega),N_l}^{-1}-\mathcal F^{-1}_{N_l}|_{C^1}\lesssim \delta.
    \]
    In particular, it is easy to observe that $\tilde\gamma=\mathcal F^{-1}_{\sigma^{-N_l}(\omega),N_l}(\gamma)$ is again a $s$-curve, $\tilde\gamma\pitchfork W^u_{\mathrm{loc}} (z_{N_\star};\mathcal F_{N_\star})\neq \emptyset$, and the associated $\tilde I\subset[-1,1]$ in its parametrization satisfies $|\tilde I|\geq (1+\chi-O(\chi^2,\delta))|I|$. We can then repeat the above construction with $\sigma^{-N_l}(\omega)$ and $\tilde\gamma$.  After a finite number of iterations the corresponding $s$-curve must be fully crossing.

    \noindent\textbf{Scenario $2$:} If $|I|<9/5<2a$ then there exists $N_\star\in\mathcal N$ such that $\gamma\subset \mathcal F_{N_\star}([-1,1]^2)$. We can then choose $\omega\in\{0,1\}^\mathbb Z$ such that  
     \begin{equation}\label{eq:firstchoiceomega}
    \omega=(\dots,\omega_{-1},\omega_0;\underbrace{0,\dots,0,1}_{N_\star},\omega_{N_\star+1},\dots)
    \end{equation}
   Again, the choice of $N_\star$ depends exclusively on $\gamma$. For any $\omega$ as above we have that 
    \[
    |\mathcal F_{\sigma^{-N_\star}(\omega),N_\star}^{-1}-\mathcal F^{-1}_{N_\star}|_{C^1}\lesssim \delta.
    \]
    In particular,  $\tilde\gamma=\mathcal F^{-1}_{\sigma^{-N_\star}(\omega),N_\star}(\gamma)$ is again a $s$-curve and the corresponding $\widetilde I\subset[-1,1]$ in its parametrization satisfies that $|\tilde I|\geq (1+\chi-O(\chi^2,\delta))|I|$. If $|\tilde I|\geq 9/5$ then we arrive to the first scenario with $\sigma^{-N_\star-1}(\omega)$ and $\tilde\gamma$. If $|\widetilde I|<9/5$ then we repeat the construction of the second scenario with $\sigma^{-N_\star}(\omega)$ and $\tilde\gamma$. The key observation is that, even if $\widetilde \gamma$ depends on $\omega$, for any $\omega$ as in \eqref{eq:firstchoiceomega} we have that the corresponding $\tilde\gamma$ is contained in a $O(\delta)$-neighborhood of $\mathcal F^{-1}_{N_\star}(\gamma)$. Hence  there exists $N_{\star\star}\in\mathcal N$ such that for all $\omega$ as in  \eqref{eq:firstchoiceomega} $\tilde\gamma\subset \mathcal F_{N_{\star\star}}([-1,1]^2)$. After a finite number of iterations we must arrive to the first scenario.
\end{proof}

\subsection*{Double blender dynamics:}
Let 
\[
Q^{cs}:=\phi([-2,2]^2)\subset \mathbb A,
\]
where $\phi$ is as in Lemma \ref{lem:normalformhighiteratesskewprod}. Proposition \ref{prop:fundamentalpropskewprod} can be restated as follows: Given any $s$-curve $\gamma\subset Q^{cs}$ there exists $M\in\mathbb N$ and $\omega'=(\omega_1',\dots,\omega_{N}')\in\mathbb N^{M}$ such that for any $\omega\in\{0,1\}^\mathbb Z$ with  $\omega_k=\omega'_k$ for all $k\in\{1,\dots,M\}$ the curve 
\[
\tilde\gamma=F_{\sigma^{-M}(\omega)}^{-1}\circ\cdots\circ F_{\sigma^{-1}(\omega)}^{-1}(\gamma_s)
\]
is a fully crossing $s$-curve (i.e. in local coordinates $(\xi,\eta)\in[-1,1]^2$ the projection of $\tilde\gamma$ onto the first component covers the interval $[-1,1]$). Let now 
\[
Q^{cu}:=\psi_R\circ\phi([-2,2]^2)\subset \mathbb A
\]
with $\psi_R:(\varphi,J)=(-\varphi,J)$ and $\phi$ as in Lemma \ref{lem:normalformhighiteratesskewprod}. Exploiting the almost reversibility of the maps $T_0,T_1$ under $\psi_R$, the argument in Section \ref{sec:symbolicdoubleblender} plus direct repetition of the proof of Proposition \ref{prop:fundamentalpropskewprod} shows the following. Given any $u$-curve $\gamma\subset Q^{cu}$ there exists $M\in\mathbb N$ and $\omega'=(\omega_{-M+1}',\dots,\omega_0')\in\{0,1\}^{M}$ such that for any $\omega\in\{0,1\}^\mathbb Z$ with  $\omega_k=\omega'_k$ for all $k\in\{-M+1,0\}$ the curve
\[
\tilde\gamma=F_{\sigma^{M-2}(\omega)}\circ\cdots\circ F_{\sigma(\omega)}\circ  F_{\omega}(\gamma_u)
\]
is a fully crossing u-curve. It is then straightforward to prove the following.

\begin{prop}\label{prop:secondfundamentalpropskewprod}
    Fix any $0<\chi\ll 1$ and let $\varepsilon>0$ be sufficiently small. There exists $\delta_0(\chi,\varepsilon)>0$ such that for any $0\leq \delta\leq \delta_0$ the following holds. Let $\gamma_s\subset Q^{cs}$ be a $s$-curve and let $\gamma_u\subset Q^{cu}$ be a u-curve. Then, there exists $M_f,M_b\in\mathbb N$,
    \[
    \omega^f=(\omega_{-M_f+1}^f,\dots,\omega_0^f)\in \{0,1\}^{M_f}\qquad\text{and}\qquad \omega^b=(\omega_{-M_f-M_b+1}^b,\dots,\omega_{-M_f}^b)\in \{0,1\}^{M_b}
    \]
    such that, for any $\omega\in\{0,1\}^\mathbb Z$ with
\[
\omega_k=\omega^f_k\quad\text{for all}\quad k\in\{-M_f+1,\dots,0\}\quad\text{and}\quad\omega_k=\omega^b_k\quad\text{for all}\quad k\in\{-M_f-M_b+1,\dots,-M_f\},
\]
we have
\[
F_{\sigma^{M_f-1}(\omega)}\circ\dots\circ F_\omega(\gamma_u)\pitchfork F_{\sigma^{M_{f}}(\omega)}^{-1}\circ\dots \circ F^{-1}_{\sigma^{M_f+M_b-1}(\omega)}(\gamma_s)\neq\emptyset.
\]
In particular,
\[
F_{\sigma^{M_f+M_b-1}(\omega)}\circ\cdots\circ   F_{(\omega)}(\gamma_s)\pitchfork \gamma_u\neq\emptyset.
\]
\end{prop}

\begin{proof}
    Arguing as above, given $\gamma_s\subset Q^{cs}$ there exists $M_b\in\mathbb N$ and $\omega^b\in  \{0,1\}^{M_b}$ such that, for any $\omega\in\{0,1\}^\mathbb Z$ with $\omega_k=\omega^b_k$ for all $k\in\{1,M_b\}$ the curve
\[
\tilde\gamma_s=F^{-1}_{\sigma^{-M_b}(\omega)}\circ\cdots\circ F^{-1}_{\sigma^{-1}(\omega)}(\gamma_s)
\]
is a fully crossing $s$-curve. Analogously, for any $u$-curve $\gamma_u\subset Q^{cu}$ there exists $M_f\in\mathbb N$ and $\omega^f\in\{0,1\}^{M_f}$ such that for any $\omega\in\{0,1\}^\mathbb Z$ with  $\omega_k=\omega^f_k$ for all $k\in\{-M_f+1,\dots,0\}$ the curve
\[
\tilde\gamma_u=F_{\sigma^{M_f-1}(\omega)}\circ\cdots\circ F_{\sigma(\omega)}\circ  F_{\omega}(\gamma_u)
\]
is a fully crossing u-curve. Hence, we let $M=M_f+M_b$ and let $\omega'\in\{0,1\}^M$ be given by 
\[
\omega'=(\underbrace{\omega'_{-M_f-M_b+1},\dots,\omega'_{-M_f}}_{\omega_f},\underbrace{\omega'_{-M_f+1},\dots,\omega_0'}_{\omega_b}).
\]
By construction, for any $\omega\in\{0,1\}^\mathbb Z$ with $\omega_k=\omega_k'$ for $k\in\{-M+1,\dots,0\}$ we have 
\[
F_{\sigma^{M_f-1}(\omega)}\circ\dots\circ F_\omega(\gamma_u)\pitchfork F_{\sigma^{M_{f}}(\omega)}^{-1}\circ\dots \circ F^{-1}_{\sigma^{M_f+M_b-1}(\omega)}(\gamma_s)\neq\emptyset.
\]
\end{proof}

\subsection{Proof of Theorem \ref{thm:skewproduct}}\label{sec:proofskewprod}
We finally complete the proof of Theorem \ref{thm:skewproduct}. Fix any $N\in\mathbb N$ and let $\varepsilon,\delta>0$ be sufficiently small so that, for any $\omega\in\{0,1\}^\mathbb Z$ we have that 
\[
F^N_\omega(\mathbb T\times[-\alpha/|2\log^3\varepsilon|,\alpha/|2\log^3\varepsilon|])\subset \widetilde{\mathbb A}_\alpha\qquad\qquad (F^N_{\sigma^{-N}(\omega)})^{-1}(\mathbb T\times[-\alpha/|2\log^3\varepsilon|,\alpha/|2\log^3\varepsilon|])\subset \widetilde{\mathbb A}_\alpha
\]
with $\mathbb A_\alpha$ as in \eqref{eq:stripsaroundKAM}. Given $\omega,\omega'\in\{0,1\}^\mathbb Z$  and $B,B'\in\mathbb T\times[-\alpha/|2\log^3\varepsilon|,\alpha/|2\log^3\varepsilon|]$, we let 
\[
\tilde B'= F_{\omega'}^N(B')\qquad\qquad \tilde B=(F_{\sigma^{-N}(\omega)}^N)^{-1}(B).
\]
Observe that, by direct application of Lemma \ref{lem:compositionskewproducts}, for any $\tilde\omega$ with $\tilde\omega_k=\omega_k$ for $|k|\leq N$ and for any $\tilde\omega'$ with $\tilde\omega'_k=\omega_k'$ for $|k|\leq N$ we have that 
\[
\mathrm{dist}( F_{\tilde \omega'}^N(B'),\tilde B')\lesssim \delta\qquad\text{and}\qquad \mathrm{dist}( (F^{N}_{\sigma^{-N}(\tilde\omega)})^{-1}(B),\tilde B)\lesssim  \delta.
\]
In Section \ref{sec:localtransIFSproof} (Step 2) we have shown that there exist $M_{f_1}$ and $\omega^{f_1}\in \{0,1\}^{M_{f_1}}$  such that $T_{\omega^{f_1}}(B')\cap Q^{cu}\neq \emptyset$ and also $M_{b_1}$ and $\omega^{b_1}\in\{0,1\}^{M_{b_1}}$  such that $T^{-1}_{\omega^{b_1}}(B)\cap Q^{cs}\neq \emptyset$. Since $M_{b_1}$ and $M_{f_1}$ are uniformly bounded for any pair $B,B'\in\mathbb A_\alpha$ (see Remark \ref{rem:boundedtimetransport}), provided $\delta$ is chosen sufficiently small, for any $\tilde\omega'\in\{0,1\}^\mathbb Z$ with $\tilde\omega'_k=\omega'_k$ for $|k|\leq N$ and $\tilde\omega'_k=\omega^{f_1}_k$ for $k\in\{-M_{f_1}-N,\dots, -N-1\}$ satisfies that 
\[
F^{N+M_{f_1}}_{\tilde\omega'}(B')\cap Q^{cu}\neq \emptyset.
\]
In particular, $F^{N+M_{f_1}}_{\tilde\omega'}(B')$ contains a $u$-curve $\gamma_u(\tilde\omega')$. 
Analogously, for any and $\tilde\omega$ with $\tilde\omega$ with $\tilde\omega_k=\omega_k$ for $|k|\leq N$ and $\tilde\omega_k=\omega^{b_1}_k$ for $k\in\{N+1,\dots,N+M_{b_1}\}$ satisfies that 
\[
(F^{N+M_{b_1}}_{\sigma^{-N-M_{b_1}}(\tilde\omega)})^{-1}(B)\cap Q^{cu}\neq \emptyset
\]
so, in particular, $(F^{N+M_{b_1}}_{\sigma^{-N-M_{b_1}}(\tilde\omega)})^{-1}(B)$ contains a s-curve $\gamma_s(\tilde\omega)$. Moreover, for any $\tilde\omega,\tilde\omega'$ as above all $\gamma_u(\tilde\omega')$ fit into a ball of radius $O(\delta)$ and all $\gamma_s(\tilde\omega)$ fit into a ball of radius $O(\delta)$. Verbatim repetition of the iterative construction in Proposition \ref{prop:fundamentalpropskewprod} shows that we can find $M_{f_2}\in\mathbb N$ and $\omega^{f_2}\in\{0,1\}^{M_{f_2}}$ such that for any $\tilde\omega'\in\{0,1\}^\mathbb Z$ with $\tilde\omega'_k=\omega'_k$ for $|k|\leq N$, $\tilde\omega'_k=\omega^{f_1}_k$ for $k\in\{-M_{f_1}-N,\dots, -N-1\}$ and $\tilde\omega_k'=\omega_k^{f_2}$ for $k\in\{-M_{f_2}-M_{f_1}-N,\dots,-M_f-N-1\}$ satisfies that  
\[
\hat\gamma_u(\tilde\omega'):=F_{\sigma^{N+M_{f_1}}(\tilde\omega')}^{M_{f_2}}(\gamma_u(\tilde\omega'))
\]
is a fully-crossing u-curve. Analogously, we find $M_{b_2}\in\mathbb N$ and $\omega^{b_2}\in\{0,1\}^{M_{b_2}}$ such that for any $\tilde\omega'\in\{0,1\}^\mathbb Z$ with $\tilde\omega'_k=\omega'_k$ for $|k|\leq N$, $\tilde\omega'_k=\omega^{b_1}_k$ for $k\in\{-N+1,\dots,N+M_{b_1}\}$ and $\tilde\omega_k'=\omega_k^{b_2}$ for $k\in\{M_{b_1}+1,\dots,M_{b_1}+M_{b_2}\}$ satisfies that  
\[
\hat\gamma_s(\tilde\omega):=(F_{\sigma^{-N-M_{b_1}-M_{b_2}}(\tilde\omega')}^{M_{b_2}})^{-1}(\gamma_s(\tilde\omega))
\]
is a fully-crossing $s$-curve. The proof of Theorem \ref{thm:skewproduct} is completed by choosing any $\bar\omega\in\{0,1\}^\mathbb Z$ of the form
\[
\bar\omega=(\dots,\underbrace{\tilde\omega_{-N}',\dots,\tilde\omega_N'}_{2N+1},\omega^{b_1},\omega^{b_2},\omega^{f_2},\omega^{f_1},\underbrace{\tilde\omega_{-N},\dots,0;,\tilde\omega_N}_{2N+1},\dots).
\]

\section{The 3-body problem: a (local) partially hyperbolic setting}\label{sec:partiallyhyp3bp}

In this section we recall the (local) framework introduced in \cite{guardia2022hyperbolicdynamicsoscillatorymotions}. We do not claim to be original in the results presented in this section as these  are just convenient reformulations of those obtained in \cite{guardia2022hyperbolicdynamicsoscillatorymotions}. The section is organized as follows. In Section \ref{sec:goodcoordinates} we perform the symplectic reductions which recast the Hamiltonian \eqref{eq:3bpHam} as a 3 degree-of-freedom Hamiltonian. We also introduce McGehee's partial compactification of the phase space which allows us to study (a particular kind of) unbounded motions. The key point of this compactification is that the extended flow ``at infinity'' is non-trivial.  Then. in Section \ref{sec:infinitymanifold} we introduce parameterizations of the invariant manifold $\mathcal E_\infty$  in \eqref{eq:ellipticinfinity} as well as its invariant manifolds. Moreover, we recall a result from \cite{guardia2022hyperbolicdynamicsoscillatorymotions} which describes two homoclinic channels $\Gamma_{0}, \Gamma_{1}$ contained in the transverse intersection of $W^{u,s}(\mathcal E_\infty)$. In Section \ref{sec:scatteringmaps} we describe the scattering maps associated to these channels. In particular, we show that on a suitably chosen annular region, they satisfy assumptions \textbf{(A0)-(A2)} in Theorem \ref{thm:transitivityIFS} and, moreover, the transversality can be assumed to be arbitrarily small. Finally, in Section \ref{sec:globalmap} we describe the return map to a transverse section which accumulates on both channels. We observe that this map displays a strongly contracting and a strongly expanding direction while the dynamics in the center coordinates (the directions tangent to $\mathcal E_\infty$) are governed by the corresponding scattering maps.
\begin{nota} Throughout this section, we use the notation
$\mathbb D=\{(\xi,\eta)\in \mathbb C^2\colon \ \bar \xi=\eta\}$  (notice that $\mathbb D$ is diffeomorphic to $\mathbb C$) and  $\mathbb D(a)=\{(\xi,\eta)\in \mathbb C^2\colon |\xi|<a,\ \bar \xi=\eta\}$ (notice that for any $a>0$, $\mathbb D(a)$ is diffeomorphic to the unit disk in $\mathbb C$).
We also use the notations  $\mathbb R_+=\{x\in\mathbb R\colon x>0\}$ and  $\mathbb A=\mathbb T\times\{|J|\leq 1\}$.
\end{nota}

\subsection{A good coordinate system and a partial compactification in the PE regime}\label{sec:goodcoordinates}

The first step is to introduce a coordinate system which realizes the symplectic reduction outlined in Section \ref{sec:intro3bp}. We do so by introducing a local coordinate system on a suitable subset $M_{PE}(\Theta_0)\subset M(\Theta_0)$ on which the third body is located very far from the two inner bodies. The notation $PE$ refers to parabolic-elliptic and is explained below.
\begin{lem}\label{lem:sympred}
    Fix any $m_0, m_1,m_2>0$, any $|\Theta_0|>0$ and consider the six-dimensional manifold
    \begin{equation}\label{eq:reduced6-dimn}
    M(\Theta_0)=\{(q,p)\in T^*(\mathbb R^6\setminus\Delta)\colon  \bs p(p)=0,\ \Theta(q,p)=\Theta_0\}/SE(2),
    \end{equation}
where $\bs p$ and $\Theta$ are the total linear and angular momentum introduced in Section \ref{sec:intro3bp}. Choose any pair $0<L_0<L_1$ and let $R_0$ be large enough. Denote by $I_L=(L_0,L_1)$ and $I_R=(R_0,\infty)$. There exists a analytic, local coordinate system 
    \begin{equation}\label{eq:symplecticreductionfinaltext}
        \Phi_{\Theta_0}:\{(\lambda,L,\xi,\eta,r,y)\in \mathbb T\times I_L\times\mathbb D\times I_R\times\mathbb R\colon (\xi,\eta)\in \mathbb D(\sqrt L)\}\to M_{PE}(\Theta_0)\subset M(\Theta_0)
    \end{equation}
on which the projection of the flow of \eqref{eq:3bpHam} to $M(\Theta_0)$ is given by the Hamiltonian vector field generated by 
    \[
    \mathcal H_{\Theta_0}=H_{\mathrm{ell}}(L)+H_{\mathrm{par}}(r,y,\Theta_0-\Gamma(L,\xi,\eta))+ V(\lambda,L,\xi,\eta,r), \quad\quad \omega=\mathrm{d}L\wedge\mathrm d\lambda+i\mathrm{d}\xi\wedge\mathrm d\eta+\mathrm dy\wedge\mathrm dr
    \]
    with \begin{itemize}
        \item $H_{\mathrm{ell}}(L)=-\frac{\nu}{2L^2}$ for some $\nu>0$ which only depends on $m_0,m_1,m_2$.
        \item $H_{\mathrm{par}}(r,y,G)=\frac{y^2}{2}+\frac{G^2}{2r^2}-\frac{1}{r}$ and $\Gamma(L,\xi,\eta)=L-\xi\eta$
        \item For $r\gg L^{4/3}$ we have  $V=O(L^4/r^{3})$.
    \end{itemize}
\end{lem}

The proof of this lemma is achieved by a number of symplectic transformations and is deferred to the Appendix \ref{sec:symplecticred}. 

\begin{rem}
    From now on we fix  $m_0,m_1,m_2>0$ and drop these symbols from the notation.
\end{rem}

Observe that for  $r\gg L^{4/3}$  the Hamiltonian $\mathcal H_{\Theta_0}$ in Lemma \ref{lem:sympred} is given by the sum of two-uncoupled integrable Hamiltonians  $  H_{\mathrm{ell}}$ and $  H_{\mathrm{par}}$ plus a small perturbation. To explain the origin of the  notation  $  H_{\mathrm{ell}}$ and $  H_{\mathrm{par}}$, let us recall that the geometric objects involved in Theorem \ref{thm:Main3bp} (namely, the manifold $\mathcal E_\infty$ introduced in \eqref{eq:ellipticinfinity}) are associated to trajectories along which:
\begin{itemize}
    \item The distance between the two inner bodies remains bounded, i.e. $\sup_{t\in\mathbb R}|q_1(t)-q_0(t)|<\infty$,
    \item The third body escapes (resp. comes from the past) with asymptotic zero velocity, i.e. if we denote by $Q_2$ the relative position of $q_2$ with respect to the center of mass of the inner system, \[
    \text{$\sup_{t\in\mathbb R}|Q_2(t)|=\infty$ and $\lim_{t\to \infty}|\dot Q_2(t)|=0$ (resp. $\lim_{t\to -\infty}|\dot Q_2(t)|=0$)}.
    \]
\end{itemize}

For the two-body problem: a) bounded motions happen exclusively for negative energy levels in which the bodies revolve around each other in Keplerian ellipses; b) unbounded solutions with zero asymptotic velocities happen only on the zero-energy level, for which  the relative position between the bodies describes a parabola. Roughly speaking, 
\begin{itemize}
    \item \textit{(Inner elliptic motion):} $H_{\mathrm{ell}}$ is the expression of the 2-body problem Hamiltonian in the so-called Poincar\'e coordinates (see \cite{MR3146588}). The coordinates $(\lambda,L,\xi,\eta)$ describe the evolution of the relative vector $q_1-q_0$ by specifiying an instantaneous ellipse (parametrized by its semimajor axis  (determined by $L$) and the pair $(\xi,\eta)$, which can be related to the eccentricity and angle of the pericenter, see Appendix \ref{sec:symplecticred}) and the position of this vector inside the ellipse, which is measured by the angle $\lambda$. The flow generated by the (integrable) Hamiltonian $H_{\mathrm{ell}}$ reduces to a linear  translation
\begin{equation}\label{eq:linearresonflow}
\phi_{H_{\mathrm{ell}}}^t:(\lambda,L,\xi,\eta)\mapsto(\lambda+(\nu/L^{3}) t,L,\xi,\eta).
\end{equation}
    In particular, the elliptic elements $(L,\xi,\eta)$ remain constant for this flow.

    \item (\textit{Outer parabolic motion}): $H_{\mathrm{par}}$ is the expression of the Hamiltonian for the two-body problem in polar coordinates (after reduction by rotations). The coordinates $(r,y)$ describe  the evolution of the distance $r$ from $q_2$ to the center of mass of the inner system. We will be interested in motions which happen close to the level set $H_{\mathrm{par}}=0$ for which the outer body describes (approximately) a parabola around the inner system.
\end{itemize}
In our constructions below we  show that the coupling term $V$ although very weak, can alter slightly the trajectory of the parabolic body and make it come back from the parabolic infinity, obtaining motions in which the third body repeatedly approaches the inner bodies and makes far away excursions.


Let $\Theta\in\mathbb R$ and  $\Phi_\Theta$ be as in \eqref{eq:symplecticreductionfinaltext}. We now introduce McGehee's partial compactification of the six-dimensional reduced phase space $M(\Theta)$ in \eqref{eq:reduced6-dimn}, where the tuple $(\lambda,L,\xi,\eta,r,y)$ introduced in Lemma \ref{lem:sympred} forms a global coordinate system. Given any $H_0<0$ we also define the 5-dimensional energy level
\[
\mathcal M(H_0,\Theta)=M(\Theta)\cap\{H=H_0\}.
\]
Aimed at studying motions for which the trajectory of the third body is unbounded (i.e. $r$ is unbounded) we introduce the change of variables 
\[
r=\frac{2}{x^2}
\]
and denote it by $\phi_{\mathrm{MG}}$  (after \cite{McGeheestablemanifold}).
In the new coordinate system we obtain a Hamiltonian 
\[
\overline{\mathcal H}_\Theta:=\mathcal H_\Theta\circ\phi_{\mathrm{MG}}
\]
which extends continuously the flow of \eqref{eq:3bpHam} to the partially compactified manifold 
\begin{equation}\label{eq:compactifiedspace}
\overline M(\Theta)=M(\Theta)\sqcup M_\infty(\Theta),\qquad\qquad M_\infty(\Theta)=\Phi_\Theta\circ\phi_{\mathrm{MG}} (\mathbb T\times\mathbb R_+\times\mathbb D\times\{0\}\times\mathbb R)
\end{equation}
equipped with global coordinates $(\lambda,L,\xi,\eta,x,y)\in\mathbb T\times\mathbb R_+\times\mathbb D\times(\mathbb R_+\cup\{0\})\times\mathbb R$
and  with the (singular) symplectic form
\begin{equation}\label{eq:sympform}
\omega=\mathrm{d}L\wedge\mathrm d\lambda+i\mathrm{d}\xi\wedge\mathrm d\eta+\frac{4}{x^3}\mathrm{d}y\wedge\mathrm dx
\end{equation}

\begin{rem}
From now on we  work  in McGehee's coordinates. Moreover, we identify any object defined in terms of these coordinates with its embedding on the partially compactified space $M(\Theta_0)$ under the coordinate chart 
\[
\Phi_{\Theta}\circ\phi_{\mathrm{MG}}:\mathbb T\times\mathbb R_+\times\mathbb D\times(\mathbb R_+\cup\{0\})\times\mathbb R \to \overline M(\Theta)
\]
constructed above.

\end{rem}

We observe that in McGehee's coordinates the Hamiltonian $\overline {\mathcal H}_{\Theta}$ reads
\begin{equation}\label{eq:McGeheehamiltoniantext}
\overline{\mathcal H}_\Theta=H_{\mathrm{ell}}(L)+\frac{y^2}{2}-\frac{x^2}{2}+\frac{x^4}{8}(\Theta-\Gamma(L,\xi,\eta))^2+V\left(\lambda,L,\xi,\eta,\frac{2}{x^2}\right)
\end{equation}
with $H_{\mathrm{ell}}$, $\Gamma$ and $V$ as in Lemma \ref{lem:sympred}. 
Finally, we fix any value $H_0<0$ and define the  partially compactified energy level 
\begin{equation}\label{eq:compactifiedenergyleveltext}
\overline{\mathcal M}(H_0,\Theta)=\overline M(\Theta)\cap \{H=H_0\}.
\end{equation}
Since, for fixed $L<\infty$ the function $r\mapsto V(\cdot,r)$ decays as $O(r^{-3})$ when $r\to \infty$, we deduce that $V(\cdot,2/x^2)=O(x^{6})$ as $x\to 0$. In particular, on $\mathcal M(H_0,\Theta)$ it is possible to recover $L$ from $H_0,\Theta$ and the remaining coordinates. Hence, for $(x,y)\in U\subset \mathbb R^2$ a small neighborhood of the origin, we may use $
(\lambda,\xi,\eta,x,y)\in\mathbb T\times\mathbb D\times U$ 
as local coordinate system on $\overline{\mathcal M}(H_0,\Theta_0)$. 
In fact, one can make a Poincar\'e-Cartan reduction so that $\lambda$ becomes time (and one gets rid of the conjugate variable). To this end, we define the Hamiltonian, defined implicitly by 
\[
\overline{\mathcal H}_\Theta(\lambda, \mathcal{K}_{\Theta, H_0}(\xi,\eta,x,y,\lambda), \xi,\eta,x,y)=H_0.
\]
Then, the non-autonomous Hamiltonian\footnote{Note that from now on we change the order of the variables and we place last $\lambda$ to emphasize that it is now time.} $\mathcal{K}_{\Theta, H_0}$ and the symplectic form 
\[
\tilde \omega=i\mathrm{d}\xi\wedge\mathrm d\eta+\frac{4}{x^3}\mathrm{d}y\wedge\mathrm dx
\]
generate the flow defined by the system of differential equations:
\begin{equation}\label{eq:systemODEs}
\begin{aligned}
    \dot x=&-\frac{x^3}{4}\partial_y \mathcal{K}_{\Theta, H_0}=-\left( \frac{\nu}{2|H_0|}\right)^{-1/3}\frac{x^3}{4} y(1+O_2(x)\qquad\qquad &\dot \xi=&i\partial_\eta \mathcal{K}_{\Theta, H_0}=O_4(x)\qquad\quad \dot\lambda=1\\
    \dot y=&\frac{x^3}{4}\partial_x \mathcal{K}_{\Theta, H_0}=\left( \frac{\nu}{2|H_0|}\right)^{-1/3}\frac{x^3}{4}(x+O_2(x))\qquad\qquad &\dot \eta=&-i\partial_\xi \mathcal{K}_{\Theta, H_0}=O_4(x).
    \end{aligned}
\end{equation}

\subsection{A normally-parabolic manifold at infinity}\label{sec:infinitymanifold}
It follows easily from \eqref{eq:systemODEs} that  the 3-dimensional manifold
\begin{equation}\label{eq:ellipticmfoldinfinity}
\mathcal E_{\infty}(H_0,\Theta_0)=\left\{(\xi,\eta,0,0,\lambda)\colon \lambda\in\mathbb T,\ (\xi,\eta)\in\mathbb D (\sqrt L_0),\  L_0= \sqrt{\nu/2H_0}\right\}\subset\overline{\mathcal M}(H_0,\Theta_0)
\end{equation}
is invariant for the flow defined by $\overline {\mathcal H}_{\Theta}$.  Moreover, the flow on $\mathcal E_\infty$ is given by (c.f. \eqref{eq:linearresonflow})
\begin{equation}\label{eq:restrictedflow}
\phi_{\overline{\mathcal H}_\Theta}^t|_{\mathcal E_\infty(H_0,\Theta_0)}:(\xi,\eta,\lambda)\mapsto (\xi,\eta,\lambda+t).
\end{equation}
\begin{rem}\label{rem:threesphere}
    The manifold  $\mathcal E_\infty$ consists of configurations in which the third body is ``at infinity'' while the inner bodies revolve around each other on a Keplerian ellipse which is described by its semimajor axis (determined by $H_0$) and the value of the pair $(\xi,\eta)\in \mathbb D(\sqrt L_0)$ which define, implicitly, the corresponding eccentricity $\epsilon\in [0,1)$ and angle of the pericenter $g$ (see Appendix \ref{sec:symplecticred}). 

   In defining the coordinate system in Lemma \ref{lem:sympred} we have implicitly chosen an orientation for the dynamics inside the ellipses. By considering the opposite choice, proceeding as in the proof of Lemma \ref{lem:sympred}, one obtains a local coordinate system on a different subset of $M(\Theta)$ on which the inner bodies rotate with opposite orientation. Analogously, in this region there also exist an invariant manifold $\mathcal E_\infty^{\mathrm{opp}}$ which corresponds to the same set of configurations as those in $\mathcal E_\infty$ but with the inner bodies rotating with opposite orientation. Both manifolds $\mathcal E_\infty$ and $\mathcal E_\infty^{\mathrm{opp}}$ share a common boundary: the 2-torus $\partial \mathcal E_\infty=\{|\xi|=\sqrt L\}$ corresponding to motions on degenerate ellipses with eccentricity one. After regularizing collisions one may glue these manifolds and observe that $\overline {\mathcal E}_\infty=\mathcal E_\infty\sqcup\mathcal E_\infty^{\mathrm{opp}}\sqcup\partial \mathcal E_\infty\simeq \mathbb S^3$. Moreover, it is a classical fact that the flow \eqref{eq:restrictedflow} extends to the Hopf flow on $\overline {\mathcal E}_\infty$ (see \cite{MR744302} for instance).

 For our purposes though, it will be enough to restrict our attention to the solid torus $\mathcal E_\infty (H_0,\Theta)$ and a neighborhood of this manifold inside $M_{PE}(\Theta)\subset M(\Theta)$.
\end{rem}

Denote by $X_\Theta$ the vector field induced by $\overline{\mathcal H}_\Theta$. An straightforward computation shows that the linearization of $X_\Theta$ at $\mathcal E_\infty (H_0,\Theta_0)$ only has one non-zero eigenvalue which corresponds to the flow direction. To get an insight on the dynamics of \eqref{eq:systemODEs} around $\mathcal E_\infty$ we first focus on the behavior of the $(x,y)$ variables alone and neglect higher order terms. We see that the reduced system
\begin{equation}\label{eq:reducedsystemodes}
\dot x=-\left( \frac{\nu}{2|H_0|}\right)^{-1/3}\frac{x^3}{4} y\qquad\qquad \dot y=-\left( \frac{\nu}{2|H_0|}\right)^{-1/3}\frac{x^3}{4} x
\end{equation}
exhibits a parabolic fixed point at $\{x=y=0\}$ with stable and unstable manifolds corresponding, respectively, to the curves $\{x+y=0\}$ and $\{x-y=0\}$. Moreover, after a (singular) time reparametrization we can conjugate \eqref{eq:reducedsystemodes} to $\dot x=y$, $\dot y=x$, for which the origin becomes a hyperbolic fixed point. Hence, one may expect that, at least at a topological level, the flow of \eqref{eq:systemODEs} on a neighborhood of $\mathcal E_\infty$ bears some resemblance with the flow around a normally-hyperbolic invariant manifold.

A classical result by Robinson shows that indeed, and in spite of the strong degeneracy of the flow, the invariant manifold $\mathcal E_\infty$ possesses smooth stable and unstable invariant manifolds.
\begin{rem}
    From now on we fix any value of $H_0<0$ and drop this symbol from the notation.
\end{rem}

 \begin{thm}[\cite{MR744302,BFM20a, BFM20b}]
    Let $\Theta\in\mathbb R$ and let $\mathcal U_\infty$ be a sufficiently small open neighborhood of $\mathcal E_\infty(\Theta)$. The stable and unstable invariant sets 
    \begin{equation}\label{eq:invmfoldseinfty}
    \begin{split}
    W^{s}_{\mathrm{loc}}(\mathcal E_\infty(\Theta))=&\{z\in\mathcal U_\infty\colon  \phi^t_{\overline{\mathcal H}_\Theta}(z)\in\mathcal U_\infty \text{ for all } t>0\}\\
    W^{u}_{\mathrm{loc}}(\mathcal E_\infty(\Theta))=&\{z\in\mathcal U_\infty\colon  \phi^t_{\overline{\mathcal H}_\Theta}(z)\in\mathcal U_\infty \text{ for all } t<0\}
    \end{split}
    \end{equation}
    are 4-dimensional immersed submanifolds, which are $C^\infty$ everywhere and real-analytic on the complement of $\{x=0\}$. Moreover, for any $z_\infty \in\mathcal E_\infty(\Theta)$, the leaves
     \begin{equation}\label{eq:invmfoldseinftyleaves}
    \begin{split}
    W^{s}_{\mathrm{loc}}(z_\infty)=&\{z\in\mathcal U_\infty\colon  |\phi^t_{\overline{\mathcal H}_\Theta}(z)-\phi^t_{\overline{\mathcal H}_\Theta}(z_\infty)|\to 0\text{ as } t\to\infty\}\subset W^{s}_{\mathrm{loc}}(\mathcal E_\infty(\Theta))\\
    W^{u}_{\mathrm{loc}}(z_\infty)=&\{z\in\mathcal U_\infty\colon  |\phi^t_{\overline{\mathcal H}_\Theta}(z)-\phi^t_{\overline{\mathcal H}_\Theta}(z_\infty)|\to 0\text{ as } t\to-\infty\}\subset W^{u}_{\mathrm{loc}}(\mathcal E_\infty(\Theta))
    \end{split}
    \end{equation}
    depend on $z_\infty$ in a real-analytic fashion.
\end{thm}

\begin{rem}
    To be precise, Robinson not only studies the invariant manifolds of $\mathcal E_\infty$ but of the whole three-sphere $\overline{\mathcal E}_\infty$ introduced in Remark \ref{rem:threesphere}.
\end{rem}

Having established the existence of these local manifolds it is natural to wonder wether their globalizations display transverse intersections. This fact was established in \cite{guardia2022hyperbolicdynamicsoscillatorymotions}.

\begin{thm}[Theorem 4.5 in \cite{guardia2022hyperbolicdynamicsoscillatorymotions}]\label{thm:existencehomchannels}
  Let $\Theta\gg 1$. Then, there exists (at least) two different, non-empty, transverse homoclinic manifolds
  \[
 \Gamma_{i}\subset  W^{u}_{\mathrm{loc}}(\mathcal E_\infty(\Theta)\pitchfork W^{s}_{\mathrm{loc}}(\mathcal E_\infty(\Theta)\qquad\qquad i=0,1.
  \]
\end{thm}

Recall that, as we discussed in Section \ref{sec:intro3bp}, our strategy to find a local partially hyperbolic framework, consists on analyzing the return maps to suitable neighborhoods of the homoclinic channels $\Gamma_i$, $i=0,1$. To that end, in the next section we describe the ``outer dynamics'' along the homoclinic channels making use of the scattering map formalism (see \cite{DelaLLavegaps,guardia2022hyperbolicdynamicsoscillatorymotions}).

\subsection{The scattering maps to $\mathcal E_\infty$}\label{sec:scatteringmaps}
We now want to describe the dynamics of orbits along the homoclicnic manifolds $\Gamma_i$. This can be achieved via the construction of the so-called scattering maps first introduced in \cite{DelaLLavegaps}. To construct these maps we first observe that, given a transverse homoclinic manifold $\Gamma\subset  W^{u}_{\mathrm{loc}}(\mathcal E_\infty(\Theta)\pitchfork W^{s}_{\mathrm{loc}}(\mathcal E_\infty(\Theta)$, since the leaves \eqref{eq:invmfoldseinftyleaves} depend regularly on the base point,  the holonomy maps
\begin{equation}\label{eq:wavemaps}
\begin{split}
\Omega^{u,s}:\Gamma &\to \mathcal E_\infty\\
z&\mapsto z_{u,s}
\end{split}
\end{equation}
defined by the rule 
\[
z_{u,s}=\Omega^{u,s}(z) \qquad\qquad \Longleftrightarrow \qquad\qquad z\in W^{u,s}(z^{u,s})
\]
are local diffeomorphisms (see \cite{DelaLLavegaps,DelaLLaveScattmap}). In particular, since for $i=0,1$, the manifolds $W^{u,s}_{\mathrm{loc}}(\mathcal E_\infty(\Theta)$ instersect transversally along $\Gamma_i$ in Theorem \ref{thm:existencehomchannels},  one can define the composition map
\begin{equation}\label{eq:defnscattmap}
\begin{split}
  \widetilde  S_i:\Omega^{u}(\Gamma_i)\subset \mathcal E_\infty(\Theta)&\mapsto \Omega^s(\Gamma_i) \subset \mathcal E_\infty(\Theta)\\
    z&\mapsto \Omega^{s}\circ(\Omega^u)^{-1}(z).
    \end{split}
\end{equation}
The map \eqref{eq:defnscattmap} is the so-called scattering map along the homoclinic channel $\Gamma_i$. 
Note that these maps are a priori only defined locally. The next proposition describes the dynamics of the scattering maps  in suitable domains. 
\begin{prop}[After Proposition 4.6. in \cite{guardia2022hyperbolicdynamicsoscillatorymotions}]\label{prop:firstScattmap}
    Let $r=\Theta^{-3}$, let $\mathbb D_r\subset\mathbb C$ be the disk around the origin of radius $r$ and let $N\in\mathbb N$. Then, for $\Theta\gg 1$ there exists an embedding 
    \[
    \phi:\mathbb T\times\mathbb D_r\to \Omega^u(\Gamma_1)\cap \Omega^u(\Gamma_2)\subset\mathcal E_\infty(\Theta)
    \]
    such that, when expressed in local coordinates\footnote{Recall that on $\mathcal E_\infty$, $(\xi,\eta)\in \mathbb D(\sqrt{L(H_0)})$ and that $\mathbb D(\sqrt{L(H_0)})$ is diffeomorphic to the unit disk in $\mathbb C$.} $(\lambda,z)\in \mathbb T\times \mathbb D_r\subset\mathbb T\times\mathbb C$, the scattering maps $\widetilde S_i$ are of the form 
    \[
    \widetilde S_i:\begin{pmatrix}
        \lambda\\ z
    \end{pmatrix}\to \begin{pmatrix}
        \lambda\\ S_i(z)
    \end{pmatrix}
    \]
    where $S_i$ are  real-analytic, symplectic maps of the form
    \begin{equation}\label{eq:scattmapsoriginal}
    \begin{split}
    S_0:z&\mapsto e^{i\beta_0} z+\sum_{k=2}^N P_k z^k+O(z^{N+ 1})\\
    S_1:z&\mapsto \Delta(z)+S_0(z)+O(|z+\Delta_0|^{N+1})
    \end{split}
    \end{equation}
    with $P_k=O(\Theta^{-3})$, $\Delta(z)=\Delta_0 (1+O(\Theta^{-3/2}))$ and
    \[
    \beta_0(\Theta)\sim \Theta^{-3}\qquad\qquad P_3(\Theta)\sim\Theta^{-3}\qquad\qquad \Delta_0(\Theta)\sim  \Theta^{9/2}\exp(-C\Theta^3)
    \]
    for some constant $C$ independent of $\Theta$.
\end{prop}

\begin{proof}
   In Proposition 4.6 in \cite{guardia2022hyperbolicdynamicsoscillatorymotions} the authors establish the asymptotic formula for $S_0$ above and show that 
   \[
   S_1(z)=\Delta_0+\sum_{k=1}^NP_k^{(1)}(z+\Delta_0)^k+O(|z+\Delta_0|^{N+1})
   \]
   for some coefficients $P_k^{(1)}$. From the proof of Proposition 4.6 it can be easily deduced that 
   \begin{equation}\label{eq:diffscattmaps}
  |P^{(1)}_k-P_k|=O(\Delta_0 \Theta^{3k/2}) \qquad\qquad 1\leq k\leq N,
   \end{equation}
   where $P_1=e^{i\beta_0}$.  We then write
    \[
    S_1(z)=\Delta_0+\sum_{k=1}^N P_k^{(1)} (z+\Delta_0)^k+O(z^{N+1})=\Delta_0+\sum_{n=0}^N \widetilde P_n^{(1)} z^n+O(z^{N+1})
    \]
    for $\widetilde P^{(1)}_n=\sum_{k=\max\{1,n\}}^N P_k^{(1)}\binom{n}{k}\Delta_0^{k-n}$. In particular, for $n=0$
    \[
    \widetilde P_0^{(1)}=\sum_{k=1}^N P_k^{(1)}\Delta_0^{k}=O(\Theta^{-3}\Delta_0)
    \]
    and for $n\geq 1$
    \begin{align*}
    \widetilde P^{(1)}_n=P^{(1)}_k+\Delta_0\sum_{k=n+1}^N P_k^{(1)}\binom{n}{k}\Delta_0^{k-(n+1)}
    =P_k+ (P^{(1)}_k-P_k)+\Delta_0\sum_{k=n+1}^N P_k^{(1)}\binom{n}{k}\Delta_0^{k-(n+1)}.
    \end{align*}
      Hence, 
    \[
    S_1(z)=\Delta_0+\sum_{n=1}^NP_n z^n+O(\Theta^{-3}\Delta_0)+\sup_{k\leq N} O(|P^{(1)}_k-P_k| z^k)+O(\Delta_0 z)+O(|z+\Delta_0|^{N+1})
    \]
    so we are done.
\end{proof}

Note that since the scattering maps act trivially on $\lambda$, all the information of the scattering maps is carried by $S_i$ (which is $\lambda$-independent). By an abuse of language, from now on we refer to $S_i$ as the scattering maps. Observe from the expressions \eqref{eq:scattmapsoriginal} that, up to $O_{N+1}(z)$ corrections, the scattering maps $S_0,S_1$ are exponentially close (in $1/\Theta$). The fact that we have an asymptotic expression for the function $\Delta(z)$ (and not just an upper bound) will be crucial to check that (on a suitable region) the invariant curves of the map $S_0$ are not invariant for the map $S_1$ but instead intersect their images (under $S_1$) transversally. This is the basis for the transversality-torsion mechanism. 

This idea was already exploited in \cite{guardia2022hyperbolicdynamicsoscillatorymotions} to construct hyperbolic basic sets for the 3-body problem. For the purposes of the present paper  we will need a much more delicate control on the relation between the following quantities which can be associated to an invariant curve $\gamma$ of the map $S_0$:
\begin{itemize}
\item the arithmetic properties of its rotation number,
\item the torsion of the map $S_0$ at the curve $\gamma$,
\item the angle at which $S_1(\gamma)$ intersects $\gamma$.
\end{itemize}

In Theorem \ref{thm:transvtorsionScattmaps} below we show that, provided $\Theta$ is large enough, it is possible to  find a KAM curve of the map $S_0$  for which the angle between this curve and its image under $S_1$ can be made arbitrarily small compared to both its Diophantine constant and the torsion coefficient. This is crucial to construct a symbolic blender for the IFS generated by the pair of maps $\{S_0,S_1\}$. 

To state the theorem we define the annulus 
\[
\mathbb A=\mathbb T\times \{|J|\leq 1\},
\]
and, for given  $\rho,\sigma>0$, its complex extension 
\[
\mathbb A_{\rho,\sigma}=\left\{(\varphi,J)\in \mathbb C/\mathbb Z\times\mathbb C: |\Im \varphi|\leq\sigma, \,|\Re J|\leq 1, \,|\Im J|\leq \rho\right\}.
\]

\begin{thm}[After Theorem 4.7. in \cite{guardia2022hyperbolicdynamicsoscillatorymotions}]\label{thm:transvtorsionScattmaps}
   Let $\Theta\gg 1$,  $r=\Theta^{-2}$ and let $S_i:\mathbb D_r\to \mathbb D_{2r}$ $i=0,1$, be the coordinate expression \eqref{eq:scattmapsoriginal} for the scattering maps constructed in \eqref{eq:defnscattmap}.
   Then, there exists constants $\rho,\sigma>0$ independent of $\Theta$ and a real-analytic, conformally symplectic, coordinate transformation $\phi_{\mathrm{KAM}}:(\varphi,J)\in \mathbb A
   _{\rho,\sigma}\to (\xi,\eta)\in \mathbb C^2$ such that $\phi_{\mathrm{KAM}}:(\varphi,J)\in \mathbb A
\to (\xi,\eta)\in \mathbb D_r$  and the maps
    \begin{equation}\label{eq:mathttSimaps}
\mathtt S_i=\phi_{\mathrm{KAM}}^{-1}\circ S_i\circ \phi_{\mathrm{KAM}}
    \end{equation}
    are real-analytic, exact symplectic. Moreover,
   \begin{itemize}
        \item $\mathtt S_0$ is of the form 
 \[\mathtt S_0:\begin{pmatrix}\varphi\\J\end{pmatrix}\mapsto \begin{pmatrix}\varphi+\beta+\tau J+R_\varphi(\varphi,J)\\ J+R_J(\varphi,J)\end{pmatrix}\]
 for some  $R_\varphi,R_J$ satisfying $
\partial_J^{n} R_{*}|(\varphi,0)=0$ for $n=0, 1$ if $*=\varphi$ and $ n=0,1, 2$ if $*=J$. Moreover,  $\beta\in\mathcal B_\alpha$ with
    \[
    \beta(\Theta)\sim \Theta^{-3}\qquad \alpha\gtrsim \Theta^{-3} \exp(-C_1\Theta^{-3})\qquad\qquad \tau(\Theta)\sim \Theta^{-3} \exp(-C_1\Theta^3)
    \]
for certain $C_1>0$ independent of $\Theta$.

    \item $\mathtt S_1$ is of the form 
    \[\mathtt S_1:\begin{pmatrix}\varphi\\J\end{pmatrix}\mapsto \mathtt S_0(\varphi,J)+\begin{pmatrix}O_{C^2}(\varepsilon)\\\varepsilon\sin\varphi+O_{C^2}(\varepsilon^2)\end{pmatrix}\]
    with 
    \[
   \varepsilon(\Theta)\sim \Theta^{3/2}\exp(-C_2\Theta^3)
    \]
  for some $C_2>0$ which does not depend on $\Theta$ and satisfies $C_2> 4 C_1$.  
 
    \end{itemize}
\end{thm}

\begin{proof}
The proof follows from an argument similar to that in the proof of Theorem 4.7. of \cite{guardia2022hyperbolicdynamicsoscillatorymotions} but with some minor modifications. Indeed, the only difference is that we look for a KAM curve of $S_0$ located at a largest distance from the origin (compared to the one in \cite{guardia2022hyperbolicdynamicsoscillatorymotions}). This allows us to obtain smaller angles between this curve and its image by $S_1$.

The proof is divided in four steps. The first three are devoted to putting the map $S_0$ in normal form around a KAM curve with frequency vector of constant type. Then, in the last step we express the map $S_1$ in the new local coordinate system (around the KAM curve for the map $S_0$).

\noindent\textbf{Step 1:} \textit{(Birkhoff normal form)} Observe that, for any $k\in\mathbb N$, provided we take $\Theta$ we can guarantee that
\[
|e^{ik\beta}-1|\gtrsim \Theta^{-3}.
\]
Then, proceeding as in \cite{guardia2022hyperbolicdynamicsoscillatorymotions} one finds a symplectic, real-analytic, coordinate transformation $\Psi_1:\mathbb D_{\rho_1}\to \mathbb D_{2\rho_1}\subset\mathbb D_r$, where 
\[
\rho_1(\Theta)=\Theta^{-4}
\] 
which satisfies $\Psi_1(z)=z+O(z^2)$ such that 
\[
S_0^{(1)}:=\Psi_1^{-1}\circ S_0\circ \Psi_1:z\mapsto z \lambda_1(|z|)+O(z^7)
\]
for $\lambda_1(|z|)=\exp(i(\beta+\mathcal T |z|^2+O(|z|^4)))$ and $\mathcal T\sim \Theta^{-3}$.
\medskip

\noindent\textbf{Step 2:} \textit{(Scaled action-angle coordinates)} We now let $C$ be the constant in Proposition \ref{prop:firstScattmap} and define 
\[
\rho_2(\Theta)=\exp(-C\Theta^3/10).
\]
Fix now any two sufficiently small constants $\rho,\sigma>0$ independent of $\Theta$. Then, we introduce the real-analytic map 
\[
\Psi_2:(\varphi,I)\in \mathbb T_\sigma\times[1,3]_\rho\to \rho_2(\Theta)\sqrt{I}e^{i\theta}\in \mathbb D_{2\rho_2}
\]
(by $[1,3]_\rho$ we mean a $\rho$-complex neighbourhood of the real interval $[1,3]$) so, for any $l\in\mathbb N$ (provided we take $\Theta$ large enough)
\[
S_0^{(2)}:=\Psi_2^{-1}\circ S_0^{(1)}\circ \Psi_2:(\theta,I)^\top\mapsto (\theta+b(I)+ O_{C^l}(\rho_2^6), I+O_{C^l}(\rho_2^6))^\top
\]
with 
\[
b(I)=K_1 \Theta^{-3}+K_2 \rho^2_2(\Theta) \Theta^{-3} I+O(\rho_2^4(\Theta))
\]
for some $K_1,K_2\neq 0$ independent of $\Theta$.
\medskip

\noindent\textbf{Step 3:} \textit{(KAM normal form)} It now follows from a standard application of the KAM theorem for twist maps (see for example Theorem 9.3. in \cite{guardia2022hyperbolicdynamicsoscillatorymotions} which is a simplified version of a theorem of Herman) that there exists a constant
\[
\alpha\geq \rho^2_2(\Theta) \Theta^{-4}
\]
a frequency $\beta\sim \Theta^{-3}$ satisfying $\beta\in \mathcal B_\alpha$, a torsion coefficient $\tau\sim \rho^2_2(\Theta)\Theta^{-3}$, a real number $I_*\in [3/2,5/2]$ and a symplectic transformation $\Psi_2:\mathbb A_{\rho/2,\sigma/2}\to \mathbb T_\sigma\times [1,3]_\rho$ of the form 
\[
\Psi_2:(\varphi,J)^\top\mapsto (\varphi+O_{C^1}(\rho_2^2), J+I_*+O_{C^1}(\rho_2^2))^\top
\]
such that 
\[
\mathtt S_0:=\Psi_3^{-1}\circ S_0^{(2)}\circ \Psi_3:(\varphi,J)^\top\mapsto (\varphi+\beta+\tau J+O(J^2), J+O(J^3))^\top.
\]
\medskip

\noindent\textbf{Step 4:} \textit{(Transforming the map $S_1$)} Recall the expression for the map $S_1$ given in Proposition \ref{prop:firstScattmap} and denote by $\Psi=\Psi_1\circ\Psi_2\circ\Psi_3$. Then, using that $\Psi_1(z)=z+O(z^2)$, the explicit expression for the map $\Psi_2$ and the fact that $\Psi_3=\mathrm{id}+O_{C^1}(\rho^2)$ we arrive to 
\[
\mathtt S_1:=\Psi^{-1}\circ S_1\circ\Psi=\mathtt S_0+\frac{1}{\rho}\Delta(e^{i\beta}z-1)(1+O(\rho))+O_{N-1}(\rho).
\]
Hence, it is a straightforward exercise to check that, for all $(\varphi,J)\in \mathbb A_{\rho/4,\sigma/4}$,
\[
\mathtt S_1:(\varphi,J)^\top\mapsto \mathtt S_0(\varphi,J)+(O(\Delta/\rho),\varepsilon\sin\varphi+O(\Delta,\rho^{N-1}))^\top
\]
with (recall that $C$ is the constant in Proposition \ref{prop:firstScattmap})
\[
\varepsilon(\Theta)\sim \frac{\Delta(\Theta)}{\Theta^3\rho(\Theta)}\sim \Theta^{3/2}\exp(-9C\Theta^3/10).\qedhere
\]
Finally, since $\rho,\sigma$ are constants independent of $\Theta$, the corresponding $C^2$ estimates follow from straightforward Cauchy estimates (after slightly reducing $\rho,\sigma$).
\end{proof}

The main observation now is that in the local coordinate system given by $\phi_{KAM}$, and in the parameter range $\Theta\gg 1$, the scattering maps satisfy the assumptions \textbf{(A0)-(A2)}   introduced in Section \ref{sec:introduction}  and, moreover, 
\[
\frac{\varepsilon(\Theta)}{\alpha(\Theta)},\frac{\varepsilon(\Theta)}{\tau(\Theta)}\to 0\qquad\qquad\text{as}\qquad\Theta\to \infty.
\]
Hence, we can apply Proposition \ref{prop:normalform} to obtain the following.

\begin{prop}\label{prop:scaleddiagonalcentermaps}
    Fix any $0<\chi\ll 1$. For any $\Theta$ large enough there exists an affine local coordinate system 
    \[
    \phi_{\chi,\Theta}:[-2,2]^2\to \mathbb A
    \]
    and a subset $\mathcal N_{\chi,\Theta}\subset\mathbb N$ for which the following holds. Let $\mathtt S_i$ with $i=0,1$, be the maps in \eqref{eq:mathttSimaps}. Then, for any $N\in\mathcal N_{\chi,\Theta}$ the map 
    \begin{equation}\label{eq:scaleddiagonalcentermaps}
    \mathcal F_N:= \phi_{\chi,\Theta}^{-1}\circ\mathtt S_{0}^{N}\circ \mathtt S_1\circ\phi_{\chi,\Theta},
    \end{equation}
    satisfies that, uniformly for $(\xi,\eta)\in[-2,2]^2$,
    \begin{equation*}
    \mathcal F_N:\binom{\xi}{\eta}\mapsto \binom{b_n}{0}+\begin{pmatrix}
        1-\chi&0\\
        0&1+\chi
    \end{pmatrix}\binom{\xi}{\eta}+O_{C^1}(\chi^2)
    \end{equation*}
    for some $b_n\in[-1,1]$. Moreover, the sequence $\{b_N\}_{N\in\mathcal N}$ is $\frac{1}{10}\chi$-dense in $[-10\chi,10\chi]$.
\end{prop}

The following straightforward corollary of Proposition \ref{prop:scaleddiagonalcentermaps} will prove useful for the construction, in Section \ref{sec:blender3bp}, of a $cs$-blender for the return map to a suitable subset located close to $W^s(\mathcal E_\infty)\pitchfork W^u(\mathcal E_\infty)$.

\begin{lem}[Robust covering property for scattering maps]\label{lem:Markovcenter}
    Let $\chi,\Theta$ and $\mathcal F_N$ be as in Proposition \ref{prop:scaleddiagonalcentermaps}. Then, there exists $\delta_0(\chi)>0$ such that for any $0\leq \delta\leq \delta_0(\chi)$
    \[
   [-1,1]^2 \subset \bigcup_{n\in\mathcal N} \mathcal F^{-1}_N([-1+\delta,1-\delta]^2).
    \]
\end{lem}

We also highlight the following fact, which is again a corollary of Proposition \ref{prop:scaleddiagonalcentermaps} and will be useful to construct well-distributed periodic orbits. 

\begin{lem}\label{lem:isolatingbloc}
Let $D_N\subset[-1,1]^2$ be a square centered around the point $(c_N,0)$, where $c_N=b_N/\chi$, and with sides parallel to the coordinates axes and  of size $\ell>0$ with $\chi^2\ll\ell\ll 1$. Then  $D_N$ is an isolating block for $\mathcal F_N$, that is $\mathcal F_N(D_N)$ is ``correctly aligned'' with $D_N$ (see Figure \ref{fig:isolating}). In particular, there exists a hyperbolic periodic orbit of $\mathcal F_N$ in $D_N$.
\end{lem}

\begin{figure}
    \centering
    \includegraphics[scale=0.5]{ 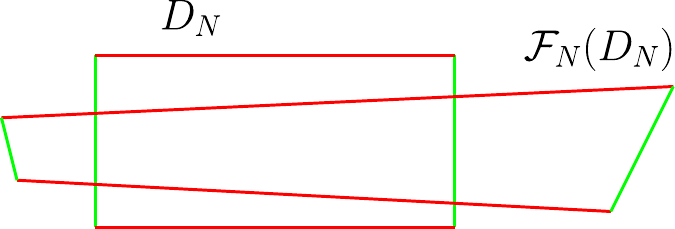}
    \caption{The set $D_N$ is an isolating block (see Lemma \ref{lem:isolatingbloc}): it gets stretched ``horizontally'' and  gets contracted ``vertically'', so that its image is well aligned with the original set.}
    \label{fig:isolating}
\end{figure}

\subsection{The return maps}\label{sec:globalmap}
Having understood the outer dynamics along the homoclinic channels $\Gamma_i$, in this section we construct return maps to suitable sections accumulating in these channels. These are defined as a composition of:
\begin{itemize}
    \item a local map which describes the local dynamics near the normally-parabolic manifold $\mathcal E_\infty$ in \eqref{eq:ellipticmfoldinfinity},
    \item a global map which describes the outer dynamics of orbits which shadow closely the homoclinic channels associated to the manifolds $\Gamma_i\subset W^{u}_{\mathrm{loc}}(\mathcal E_\infty(\Theta)\pitchfork W^{s}_{\mathrm{loc}}(\mathcal E_\infty(\Theta)$ described in Theorem \ref{thm:existencehomchannels}.
\end{itemize}

To analyze these maps, we define
\begin{equation}\label{eq:annulus}
\mathcal A_\infty(H_0,\Theta)=\phi\circ\phi_{\mathrm{KAM}}(\mathbb T\times\mathbb A)\subset \Omega^u(\Gamma_1)\cap \Omega^u(\Gamma_2)\subset \mathcal E_\infty(H_0,\Theta)
\end{equation}
with $\phi$ as in Proposition \ref{prop:firstScattmap} and $\phi_{\mathrm{KAM}}$ as in Theorem \ref{thm:transvtorsionScattmaps}. This is the annulus in which the scattering map dynamics fit into the framework of Theorem \ref{thm:transitivityIFS}. We use coordinates $(\lambda,\varphi,J)\in\mathbb T\times\mathbb A$ on this domain.

In the next lemma we describe the local structure of the flow on a neighborhood of $\mathcal A_\infty$.

\begin{lem}[Theorem 5.2 in \cite{guardia2022hyperbolicdynamicsoscillatorymotions}]\label{lem:dynamicscloseinfty}
    Fix any $k\in\mathbb N$ and let $\mathcal U_\infty \subset \mathcal M(H_0,\Theta_0)$ be a sufficiently small neighborhood of $\mathcal A_\infty (H_0,\Theta_0)\subset\mathcal E_\infty(H_0,\Theta_0)$ with $\mathcal A_\infty$ as in \eqref{eq:annulus}. On $\mathcal U_\infty$ there exists a $C^k$ change of variables $\Phi:(\lambda,\tilde\varphi,\tilde J,q,p)\mapsto (\lambda,\varphi,J,x,y)$ of the form 
    \[
\begin{pmatrix} \varphi\\J\\x\\y\end{pmatrix}= \begin{pmatrix} \mathrm{id}&0\\0&A\end{pmatrix}\begin{pmatrix}\tilde\varphi\\\tilde J\\q\\p\end{pmatrix}+O_2(q,p)
\]
for some matrix $A$, and such that, after a regular time parametrization, conjugates the flow induced by \eqref{eq:systemODEs} on $\overline{\mathcal M}(H_0,\Theta_0)$ to the flow induced by a vector field of the form (we write $\tilde z=(\tilde\varphi,\tilde J)$)
    \begin{equation}\label{eq:straightenedlocalflow}
    \begin{split}
    \dot q=&q((q+p)^3+O_4(q,p))\qquad\qquad \dot {\tilde z}=(qp)^k O_4(q,p)\\
    \dot p=&-p((q+p)^3+O_4(q,p))\qquad\qquad \dot\lambda=1.
    \end{split}
    \end{equation}
\end{lem}
\begin{rem}
    Below, we drop the tilde from the variables $(\tilde\varphi,\tilde J)$ in order to alleviate the notation and we write $z=(\varphi,J)$.
\end{rem}

Before proceeding, a few remarks are in order. First, we note that, in the new coordinate system, the local manifolds $W^{s,u}_{\mathrm{loc}}(\mathcal E_\infty)$ have been straightened, i.e. they are respectively given by $\{q=0\}$ (stable) and $\{p=0\}$ (unstable). Second, by taking $k$ large enough, the dynamics of the $z$ variable are arbitrarily close to the trivial dynamics ($\dot z=0$) as we approach the manifolds $W^{s,u}_{\mathrm{loc}}(\mathcal E_\infty)$. This can be interpreted as a manifestation of the strongly degenerate dynamics on $\mathcal E_\infty$. 

In the coordinate system constructed in Lemma \ref{lem:dynamicscloseinfty}, we define, for $a>0$ small enough, the 4-dimensional transverse sections 
\[
\Sigma_a^{\out}=\{q=a,\ p>0\},\qquad\qquad \Sigma_a^{\inn}=\{p=a,\ q>0\}.
\]
Then, it is not difficult to check that (see for instance Chapter VI in \cite{MR442980}) if $U\subset \Sigma_{a}^{\inn}$ is a small neighborhood of $\Sigma\cap\{q=0\}$  the local map 
\begin{equation}\label{eq:localmap}
    \Phi_{\mathrm{loc}}:U\subset \Sigma_{a}^{\inn}\to \Sigma_{a}^{\out}
\end{equation}
which, to any point in $U$, associates the first point at which the solution of \eqref{eq:straightenedlocalflow} hits $\Sigma_{a}^{\out}$, is well defined.


Let now $U_i^\star\subset \Sigma_{a}^{\star}$ be small neighborhoods of the two-dimensional sets $\Gamma_i\cap \Sigma_a^{\star}$ for $i=0,1$ and $\star=\inn,\out$. We want to describe the dynamics of the maps 
\begin{equation}\label{eq:globalmap}
\Phi_{i,\mathrm{glob}}:U_i^{\out}\subset \Sigma_{a}^{\out}\mapsto \Sigma_{a}^{\inn}\qquad\qquad i=0,1
\end{equation}
defined by following the flow of the Hamiltonian \eqref{eq:McGeheehamiltoniantext}. To describe these maps, for $\delta>0$ sufficiently small, on each of the subsets $U_i^{\out}$ we define a local coordinate system 
\begin{equation}\label{eq:localcoords}
    \phi_i:(p,\tau,\varphi,J)\in [0,\delta]^2\times \mathbb A\to \mathcal Q_{\delta}^i\subset U_i^{\out}
\end{equation}
such that ($\mathcal Q_{\delta}^i$ is simply the image of $[0,\delta]^2\times \mathbb A$ under $\phi_i$ and we use the same labeling both for the coordinate system on $U_1$ and $U_2$ as this will cause no confusion in the future)
\[
W^s(\mathcal E_\infty)\cap U_i^{\out}=\{\tau=0\}.
\]
See Figure \ref{fig:Fig7} for a schematic picture of the new coordinate system (see also Section 6.1. of \cite{guardia2022hyperbolicdynamicsoscillatorymotions} for a more detailed construction). Analogously, we define a local coordinate system \[
    \tilde\phi_i:(q,\sigma,\varphi,J)\in([0,\delta]^2\times \mathbb A)\to  U_i^{\inn}
\]
such that $
W^u(\mathcal E_\infty)\cap U_i^{\inn}=\{\sigma=0\}$.

\begin{figure}
    \centering
    \includegraphics[scale=0.5]{ 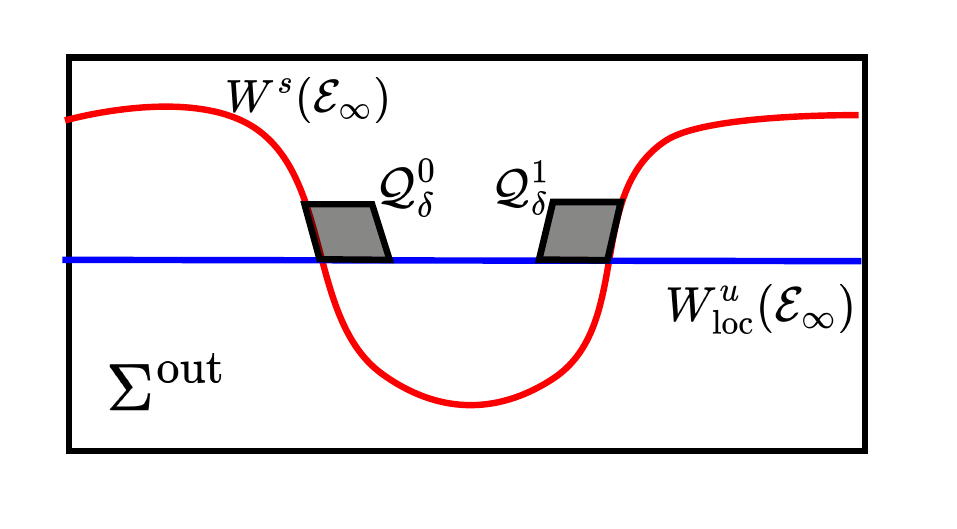}
    \caption{The domains $\mathcal Q_\delta^0,\mathcal Q_\delta^1\subset \Sigma^{\mathrm{out}}$}
    \label{fig:Fig7}
\end{figure}

\begin{lem}\label{lem:globalmap}
Let $(p,\tau,z)$ be the local coordinate system on $\mathcal Q_\delta^i\subset U^i$ defined in \eqref{eq:localcoords} and $(\sigma,q,z)$ be local coordinates on $U_i^{\inn}$. For $i=0,1$ the global maps \eqref{eq:globalmap} are well defined on $(p,\tau,z)\in[0,\delta]^2\times \mathbb A$ and of the form 
    \[
    \Phi_{i,\mathrm{glob}}:\begin{pmatrix}p\\\tau\\z\end{pmatrix}\mapsto \begin{pmatrix}\tau \nu_1(z)(1+O(p,\tau))\\  p\nu_2(z)(1+O(p,\tau))\\  \mathtt S_i(z)+O(p,\tau)\end{pmatrix},
    \]
    where $\nu_1(z)\nu_2(z)\neq 0$ for all $z\in\mathbb{A}$, and  $\mathtt S_i$ are the coordinate expression of the scattering maps given in Theorem \ref{thm:transvtorsionScattmaps}.
\end{lem}
The proof of this result can be found in Section 6.1 of \cite{guardia2022hyperbolicdynamicsoscillatorymotions}. In fact, it follows from a standard argument which simply uses the fact that the manifolds $W^{s,u}(\mathcal E_\infty)$ intersect transversally at $\Gamma_i$. The main observation is that the center dynamics is given by a $C^0$-small perturbation of the scattering map dynamics. This asymptotic formula will be key to control and describe the dynamics of the return maps that we construct below.

Finally, we introduce (whenever it is defined)  the return map
\begin{equation}\label{eq:globalmapfull}
    \Psi:\Sigma_a^{\out}\to \Sigma_a^{\out}.
\end{equation}
For convenience, it will also be convenient to define (whenever they are defined) the maps 
\begin{equation}\label{eq:defnglobalmaps}
    \Psi_{i\to j}=\Psi_{\mathrm{loc}}\circ\Psi_{i,\mathrm{glob}}:\mathcal Q^{i}_\delta\subset \Sigma_a^{\out} \to \mathcal Q_\delta^j\subset  \Sigma_a^{\out}\qquad\qquad i,j=0,1.
\end{equation}
The rest of the section is devoted to analyze these maps $\Psi_{i\mapsto j}$.

\subsection*{Dynamics of partially horizontal and partially vertical strips}

In the following result we study how the global maps $\Psi_{i\to j}$ act on a certain class of two-dimensional submanifolds whose tangent space is close to the strong expanding/contracting  directions of the maps $\Psi_{i\to j}$.
To this end, we define \emph{partially horizontal and partially vertical strips}.

\begin{defn}\label{def:partiallyHstrip}
Let $(p,\tau,\varphi,J)\in[0,\delta]^2\times \mathbb A$ be the local coordinate system in $\mathcal Q^i_{\delta}$, $i=0,1$, introduced in \eqref{eq:localcoords}. We say that a  two-dimensional subset $\Lambda_v$ (respectively  $\Lambda_h$) is a  partially  vertical (resp. horizontal)  strip provided there exist
\begin{itemize}
\item $0<\rho<\delta$ 
\item $C^1$ functions $\gamma_l,\gamma_r:[0,\rho]\to\mathbb{R}$ satisfying $\gamma_l(\tau)<\gamma_r(\tau)$ and $\partial_\tau\gamma_l,\partial_\tau\gamma_r=O(1)$ with respect to $\delta$.
\item $C^1$ functions $F_v$ (resp. $F_h$) of the form
    \[
\begin{split}
    F_v(\tau,s)&=(v_1(\tau,  s),\tau ,v_2^\varphi(\tau,  s), v_2^ J(\tau,  s))\\
    (\text{resp. }     F_h(\tau,s)&=(p,h_1(p,s),h_2^\varphi(p,  s), h_2^ J(p,  s))),
\end{split}
\]
satisfying 
\[
\partial_s (v_2^\varphi(\tau,s),v_2^J(\tau,s))\neq 0 \qquad\qquad (\text{resp.}\quad  (\partial_s h_2^\varphi(\tau,s), \partial_s h_2^J(\tau,s))\neq 0)
\]
and
($\star=\varphi, J$)
    \begin{align*}
    \partial_\tau v_1=O(1),\qquad\partial_  s v_1=&O(\delta),\qquad\partial_\tau v_2^\star=\partial_  s v_2^\star=O(1)\\
     (\text{resp. } \partial_p h_1=O(1),\qquad\partial_  s h_1=&O(\delta),\qquad\partial_p h_2^\star=\partial_  s h_2^\star=O(1))
    \end{align*}
\end{itemize}
such that $\Lambda_v$ (resp. $\Lambda_h$) can be parameterized as 
    \begin{align*}
    \Lambda_v=&\{(p,\tau,\varphi,J)=    F_v(\tau,s)\colon \gamma_l(\tau)\leq s\leq \gamma_r(\tau),\tau\in[0,\rho]\}\subset \mathcal Q^i_{\delta}\\
     (\text{resp. } \Lambda_h=&\{(p,\tau,\varphi,J)= F_h(p,s)\colon \gamma_l(p)\leq s\leq \gamma_r(p),p\in[0,\rho]\}\subset \mathcal Q^i_{\delta}).
    \end{align*}
\end{defn}
Next theorem analyzes how partially horizontal and  vertical strips get mapped under $\Psi_{i\to j}$.

\begin{thm}[After Lemma 12.2 in \cite{guardia2022hyperbolicdynamicsoscillatorymotions}]\label{thm:onestepgraphtransform}
    Let $\delta>0$ be sufficiently small and, for $i=0,1$ let $\mathcal Q_{\delta}^i\subset U^i$ be the subset defined in \eqref{eq:localcoords}. Let $(p,\tau,\varphi,J)\in[0,\delta]^2\times \mathbb A$ be the local coordinate system in $\mathcal Q^i_{\delta}$ introduced in \eqref{eq:localcoords} and consider a  partially  vertical (resp. horizontal)  strip  $\Lambda_v$ (respectively  $\Lambda_h$) as in Definition \ref{def:partiallyHstrip}.
   
 Then,  $\Psi_{i\to j}(\Lambda_v)\cap \mathcal Q^j_\delta$ contains  a countable family of partially vertical strips $\Lambda_v^{(n)}$, $n\geq n_v$ for some sufficiently large $n_v$. Analogously $\Psi_{i\to j}^{-1}(\Lambda_h)\cap \mathcal Q^i_\delta$ contains  a countable family of partially horizontal strips $\Lambda_h^{(n)}$, $n\geq n_h$ for some sufficiently large $n_h$. 
    
 Moreover, these strips can be parameterized as follows. For any  any $n\geq n_v$ (resp. $n\geq n_h$) sufficiently large there exist two differentiable functions  $\gamma_l^{(n)},\gamma_r^{(n)}:[n,n+\delta]\to \mathbb{R}$ (resp. $\tilde\gamma_l^{(n)},\tilde\gamma_r^{(n)}:[n,n+\delta]\to \mathbb{R}$)   such that for any $T\in[n,n+\delta]$ and $ \gamma_l^{(n)}(T)<s<\gamma_r^{(n)
}(T)$ (resp. any $P\in[n,n+\delta]$ and $ \tilde\gamma_l^{(n)}(P)<s<\tilde\gamma_r^{(n)
}(P)$) the equation
    \begin{equation}\label{eq:implicitonetransf}
    T=\pi_{\tilde\tau} (\Psi_{i\to j}\circ F_h)(\tau,  s)\qquad\qquad(\text{resp. } P=\pi_{\tilde p} (\Psi_{i\to j}^{-1}\circ F_v)(p,s))
    \end{equation}
    define functions $\hat{\tau}^n(T,  s)$ (resp. $\hat{p}^n(P,s))$ such that the the strips $\Lambda_v^{(n)}$ (resp. $\Lambda_h^{(n)}$) can be written as 
    \begin{align*}
    \Lambda_v^{(n)}=&\{(p,\tau,\varphi,J)= F_v^{(n)}(T,  s)\colon  \gamma_l^n(T)\leq s\leq \gamma_r^n(T),T\in[n,n+\delta]\}\\
      (\text{resp. } \Lambda_h^{(n)}=&\{(p,\tau,\varphi,J)= F_h^{(n)}(P,  s)\colon  \tilde\gamma_l^n(P)\leq s\leq \tilde\gamma_r^n(P),P\in[n,n+\delta]\})
    \end{align*}
where
\[
F_v^{(n)}(T,  s)=\Psi_{i\to j}\circ F_v(\hat{\tau}^n(T,  s),  s)\qquad (\text{resp.  }F_h^{(n)}(P,  s)=\Psi_{i\to j}^{-1}\circ F_h(\hat{ p}^n(P,  s),  s)
\]
 and satisfy the following:
    \begin{itemize}
        \item \textit{(Accumulation to $W^{u,s}_{\loc}(\mathcal E_\infty)$):}  $\pi_p F_v^{(n)}(T,  s)\to 0$ (resp. $\pi_\tau F_h^{(n)}(P,  s)\to 0$) as $n\to\infty$ in the $C^1$ topology 
        \item \textit{(Asymptotics for central dynamics):}  uniformly in $n\in\mathbb N$
        \begin{equation}\label{eq:onestepcenterdyn}
        \begin{split}
        \pi_z F^{(n)}_v(T,  s)=&\mathtt S_i(v_2^\varphi(\hat\tau^n(T,  s),  s),v_2^ J(\hat\tau^n(T,  s),  s))+O(\delta)\\
           (\text{resp. } \pi_z F^{(n)}_h(P,  s)=&\mathtt S_j^{-1}(h_2^\varphi(\hat p^n(P,  s),  s),h_2^ J(\hat p^n(P,  s),  s))+O(\delta))
        \end{split}
        \end{equation}
        \item \textit{(Action of the differential):} uniformly in $n\in\mathbb N$
        \begin{equation}\label{eq:differentialonestep}
        \partial_T F^{(n)}_v(T,  s)=\begin{pmatrix}
            O(\delta)\\
            1\\
             O(\delta)\\
            O(\delta)
        \end{pmatrix}\qquad\qquad \partial_  s F^{(n)}_v(T,  s)=\begin{pmatrix}
             O(\delta)\\
           0\\
           D\mathtt S_i(v_2^\varphi(\hat\tau^n,  s),v_2^ J(\hat\tau^n,  s))\binom{\partial_  s v_2^\varphi}{\partial_  s v_2^ J}+O(\delta)
        \end{pmatrix}
        \end{equation}
        and 
        \begin{equation}\label{eq:differentialonestepinverse}
        \partial_PF^{(n)}_h(P,  s)=\begin{pmatrix}
           1\\
            O(\delta)\\
             O(1)\\
            O(1)
        \end{pmatrix}\qquad\qquad \partial_  s F^{(n)}_h(P,  s)=\begin{pmatrix}
            0\\
            O(\delta)\\
           D\mathtt S_j^{-1}(h_2^\varphi(\hat p^n,  s),h_2^ J(\hat p^n,  s))\binom{\partial_  s h_2^\varphi}{\partial_  s h_2^ J}+O(\delta)
        \end{pmatrix}.
        \end{equation}
    \end{itemize}
\end{thm}

The proof of this result is entirely contained in the proof of Lemma 12.2 of \cite{guardia2022hyperbolicdynamicsoscillatorymotions}\footnote{To be precise, Lemma 12.2 in \cite{guardia2022hyperbolicdynamicsoscillatorymotions} gives a similar statement to that above but instead of studying one iterate of the map $\Psi_{i\to j}$ considers a  concatenation of iterates of these. Of course, the proof in that paper proceeds by induction so, in particular, the first step is to produce an statement like the one in Theorem \ref{thm:onestepgraphtransform}.}. Roughly speaking, the proof boils down to an application of the graph transform to the manifold $\Lambda_v$ (resp. $\Lambda_h$). Here we just recall briefly the argument and refer the interested reader for the details. We only consider the case of vertical submanifolds (the horizontal case being similar). 

The proof of Theorem \ref{thm:onestepgraphtransform} relies on the following rather technical lemma, a $C^1$ parabolic inclination (Lambda) lemma.


\begin{lem}[Theorem 5.4. in \cite{guardia2022hyperbolicdynamicsoscillatorymotions}]\label{lem:auxiliarlemmahd}
    Fix any $N\in\mathbb N$. Let $\Phi_{\mathrm{loc}}$ be the map in \eqref{eq:localmap}. Then, if $(q,a,\varphi,J)\in \Sigma^{\mathrm{in}}_a$ the image point 
    \[
    (a,p_1,\underbrace{\varphi_1,J_1}_{z_1},\lambda_1)=\Psi_{\mathrm{loc}}(q,a,\underbrace{\varphi,J}_z,\lambda_0)
    \]
    satisfies that 
    \[
    q^{1+Ca}\leq p_1\leq q^{1-Ca},\qquad\qquad |z_1-z|\lesssim a^{N(1+Ca)}q^{N(1-Ca)}\qquad\qquad q^{-3/2-Ca}\leq \lambda_1-\lambda_0\lesssim q^{-3/2+Ca}.
    \]
    Moreover, for any $C^1$ curve $\gamma(q)=(q,a,z(q),\lambda_0(q))$
    \[
    |p_1'(q)|,|z_1'(q)|\lesssim 1\qquad \left|\frac{p_1'(q)}{\lambda_1'(q)}\right|\lesssim q^{1-Ca}\qquad\lambda_1'(q)\gtrsim q^{-3/5+Ca}.
    \]
\end{lem}

\begin{proof}[Proof of Theorem \ref{thm:onestepgraphtransform}]

Let $\Lambda_v$ be the graph of a function $F_v(\tau,s)$ with $(\tau,s)\in \{\gamma_l(\tau)\leq s\leq \gamma_r(\tau),\ \tau\in[0,\rho]\}$ as in Definition \ref{def:partiallyHstrip}. Let  
\[
\widetilde F_v(\tau,s)=\Psi_{i\to j}\circ F_v(\tau,s)=(\tilde h_1(\tau,s),T(\tau,s),Z(\tau,s)).
\]
The authors in \cite{guardia2022hyperbolicdynamicsoscillatorymotions} show in equations (243)-(245) of that paper that (this is a consequence of Lemmas \ref{lem:globalmap} and  \ref{lem:auxiliarlemmahd})
\begin{equation}\label{eq:auxiliaryzero}
\tilde T(\tau,s)\gtrsim \tau^{-3/2+Ca}\qquad\qquad \partial_\tau T(\tau,s)\gtrsim \tau^{-3/5+Ca}
\end{equation}
with $C>0$ being some fixed constant (recall that we may take $a$ to be arbitrarily small). In particular, the equation 
\[
T=\pi_{\tilde\tau} (\Psi_{i\to j}\circ F_h)
\]
defines (for all $n$ large enough) functions $\hat{\tau}^n(T,  s)$ and $\gamma_r^{(n)}>\gamma_l^{(n)}$  such that the  strips $\Lambda_v^{(n)}$  can be written as 
    \begin{align*}
    \Lambda_v^{(n)}=&\{(p,\tau,\varphi,J)= F_v^{(n)}(T,  s)\colon  \gamma_l^n(T)\leq s\leq \gamma_r^n(T), T\in[n,n+\delta]\}
    \end{align*}
where
\[
F_v^{(n)}(T,  s)=\Psi_{i\to j}\circ F_v(\hat{\tau}^n(T,  s),  s)
\]
 and satisfy that $\pi_p F_v^{(n)}(T,s)\to 0$ as $n\to \infty$ in the $C^1$ topology (this is a consequence of the first inequality in Lemma \ref{lem:auxiliarlemmahd}). Moreover, using \eqref{eq:auxiliaryzero} and Lemma \ref{lem:auxiliarlemmahd} one obtains that $\partial_T F_v^{(n)}(T,s)$ is as in \eqref{eq:differentialonestep}. Moreover, by Lemma \ref{lem:globalmap} and the first item in Lemma \ref{lem:auxiliarlemmahd} one deduces \eqref{eq:onestepcenterdyn} (see also expression (185) in \cite{guardia2022hyperbolicdynamicsoscillatorymotions}).

One can proceed analogously to prove the statement for $\partial_s F_v^{(n)}$ by analyzing the  action of the differential $D\Psi_{i\to j}$ on partially vertical (and horizontal) submanifolds (see Lemma 12.1 in \cite{guardia2022hyperbolicdynamicsoscillatorymotions}). We refer the reader to that paper for this part of the argument. The estimates \eqref{eq:differentialonestep} correspond to expression (189) in \cite{guardia2022hyperbolicdynamicsoscillatorymotions}.
\end{proof}

The following complement to Theorem \ref{thm:onestepgraphtransform} will prove key for the construction of blenders in Section \ref{sec:blender3bp}.

\begin{lem}\label{lem:nowidthloss}
   Consider the setting of Theorem \ref{thm:onestepgraphtransform}. Then, for any $n\geq n_v$ $\gamma_r^{(n)}$  (resp. $n\geq n_h$) 
   \begin{equation}\label{eq:controlweaklength}
\begin{split}
\inf_{T\in[n,n+\delta]}|\gamma_r^{(n)}(T)-\gamma_l^{(n)}(T)|>&(1-O(\delta))\inf_{\tau\in[0,\delta]}|\gamma_r(\tau)-\gamma_l(\tau)|,\qquad\qquad \partial_T\gamma_l^{(n)},\partial_T\gamma_r^{(n)}=O(1)\\
(\text{resp. }\inf_{P\in[n,n+\delta]}|\tilde\gamma_r^{(n)}(P)-\tilde\gamma_l^{(n)}(P)|>&(1-O(\delta))\inf_{p\in[0,\delta]}|\tilde\gamma_r(p)-\tilde\gamma_l(p)|,\qquad\qquad \partial_P\tilde\gamma_l^{(n)},\partial_P\tilde\gamma_r^{(n)}=O(1))
\end{split}
\end{equation}
\end{lem}

\begin{proof}
 We already know that $\gamma^{(n)}_r>\gamma^{(n)}_l$. To obtain a quantitative estimate we proceed as follows. Note first that the condition $s-\gamma_l(\tau)>0$ expressed in the new coordinates is 
\begin{equation}\label{def:leftendpoint}
s-\gamma_l(\hat\tau^n(T,s))>0.    
\end{equation}
Observe that from the estimates \eqref{eq:auxiliaryzero} the solution $\hat\tau(T,s)$ to the implicit equation \eqref{eq:implicitonetransf} satisfies (recall that $|\tau|\leq \delta$)
\begin{equation}\label{eq:partialtaus}
\partial_s \hat\tau(T,s)=O(\delta).
\end{equation}
Now, by the fact that $\|\gamma_l\|_{C^1}=O(1)$ and the estimate \eqref{eq:partialtaus}, the implicit function theorem ensures that there exists a function $\gamma_l^{(n)}$ satisfying $\gamma_l(\hat\tau^n(T,\gamma_l^{(n)}(T))=\gamma_l^{(n)}(T)$. Moreover, by construction, \eqref{def:leftendpoint} is equivalent to $s>\gamma_l^{(n)}(T)$. Analogously one can obtain $\gamma_r^{(n)}$. Now we use its implicit definition, as
\[
\gamma_r^{(n)}(T)-\gamma_l^{(n)}(T)=\left(\gamma_l(\hat\tau^n(T,\gamma_l^{(n)}(T))-\gamma_r(\hat\tau^n(T,\gamma_l^{(n)}(T))\right)+\left(\gamma_r(\hat\tau^n(T,\gamma_l^{(n)}(T))-\gamma_r(\hat\tau^n(T,\gamma_r^{(n)}(T))\right).
\]
Now, on the one hand,  since $T\to \hat \tau^n(T,\gamma_l^{(n)}(T)$ is a diffeomorphism,
\[
\inf_{T\in[n,n+\delta]}\left(\gamma_l(\hat\tau^n(T,\gamma_l^{(n)}(T))-\gamma_r(\hat\tau^n(T,\gamma_l^{(n)}(T))\right)=\inf_{\tau\in[0,\delta]}\left(\gamma_l(\tau)-\gamma_r(\tau)\right)
\]
and on the other,  by \eqref{eq:partialtaus} and  $\|\gamma_l\|_{C^1}=O(1)$
\[
\inf_{T\in[n,n+\delta]}\left(\gamma_r(\hat\tau^n(T,\gamma_l^{(n)}(T))-\gamma_r(\hat\tau^n(T,\gamma_r^{(n)}(T))\right)\geq -C\delta \inf_{T\in[n,n+\delta]}\left(\gamma_l^{(n)}(T)-\gamma_r^{(n)}(T)\right),
\]
for some $C>0$ independent of $\delta>0$. This completes the proof.
\end{proof}

\section{A symplectic blender for the 3-body problem: proof of Theorems \ref{thm:mainblender} and \ref{thm:Main3bp}}\label{sec:blender3bp}

Here we show how to construct a symplectic blender for the return map $\Psi$ defined in \eqref{eq:globalmapfull} ( Theorem \ref{thm:mainblender}) and how to construct orbits accumulating densely forward and backward in time to an open subset of the normally-parabolic invariant manifold $\mathcal E_\infty$ (Theorem \ref{thm:Main3bp}). The proof of Theorem  \ref{thm:mainblender} is divided in three steps. First, in Section \ref{sec:transvtorsion} we show how hyperbolicity in the center variables emerges from the transversality-torsion mechanism described in Section \ref{sec:outline}. This  allows us to recover the partially-hyperbolic framework described in Section \ref{sec:heuristicblenders}. Then, in Section \ref{sec:cu-blender3bpprooflemmas} we  prove that the map $\Psi$ satisfies the covering property and, moreover exhibits well-distributed periodic orbits (see Section \ref{sec:heuristicblenders}). From these properties we can conclude that there exists a $cs$-blender. Then, owing to the reversibility of the system, we  see in Section \ref{sec:sympblender3bp} that the $cs$-blender is homoclinically related to a $cu$-blender. Together they form a symplectic blender, completing the proof of Theorem \ref{thm:mainblender}. Finally, in Section \ref{sec:mainproof}, we show how the symplectic blender implies  Theorem \ref{thm:Main3bp}.

\subsection{Transversality-torsion hyperbolicity for the global return map}\label{sec:transvtorsion}

In Theorem \ref{thm:onestepgraphtransform} we have identified strongly hyperbolic behavior along the $(p,\tau)$-directions for any of the global maps $\Psi_{i\to j}$ with $i,j=0,1$. However, the dynamics in the center variables, is given by a $O_{C^0}(\delta)$ perturbation of the scattering maps $\mathtt S_i$. The latter being two-dimensional $\Theta^{-3}$-close to identity twist maps (see Theorem \ref{thm:transvtorsionScattmaps}), makes it very difficult to  detect any sign of hyperbolicity along these directions. Indeed, any of the maps $\mathtt S_i$ is conjugated  to an integrable twist map up to an exponentially small reminder (in $\Theta^{-3})$.

In order to spot hyperbolic behavior along the center directions we reproduce the transversality-torsion mechanism described in Section \ref{sec:outline}. A  similar construction was carried out in \cite{guardia2022hyperbolicdynamicsoscillatorymotions}, where the authors established the existence of basic sets for a (single) map of the form $\Psi^{(N)}=\Psi_{0\to 1}\circ\Psi_{0\to 0}^{N-1}\circ\Psi_{1\to 0}$ with a value of $N\gg1$ such that the expansion/contraction rates are far away from $1$.

Our construction differs from the one in \cite{guardia2022hyperbolicdynamicsoscillatorymotions} in that we will look at a family of return maps $\{\Psi^{(N)}\}_{N\in\mathcal N}$ for a suitable subset $\mathcal N\subset\mathbb N$ (to be specified below) in which:
\begin{itemize}
    \item \textit{Weak-hyperbolicity:} The rate of expansion of $\Psi^{(N)}$ along the (expanding) center direction is of order $1+\chi>1$ with $0<\chi\ll 1$. 
    \item \textit{Well-distributed iterates:} For a suitable subset $Q_{\chi,\delta}\subset\mathcal Q^1_\delta$ the set of images $\bigcup_{N\in\mathcal N}\Psi^{(N)}(Q_{\chi,\delta})$ is well-distributed in a sense similar to that described in Section \ref{sec:heuristicblenders}.
\end{itemize}
The purpose of this section is to define properly (in terms of quantifiers) the aforementioned regime and observe how in this regime hyperbolic behavior emerges in the center directions. We do so by:
\begin{itemize}
    \item studying the action of the maps $\Psi^{(N)}$ along what we call horizontal and vertical strips (see the parametrizations \eqref{eq:horizontalstrip} and \eqref{eq:verticalstrip} below)
    \item establishing the existence of cone fields for the differentials $D\Psi^{(N)}$.
\end{itemize}

The statements and proofs below strongly rely on the analysis carried out  in \cite{guardia2022hyperbolicdynamicsoscillatorymotions}, which indeed follow plainly from Theorem \ref{thm:onestepgraphtransform}.

\begin{rem}\label{rem:uniformcoordsystem}
An important point to notice in the following is that both the definition of the aforementioned horizontal and vertical strips  as well as the directions of the cone fields  are uniform in $N$ for $N\in\mathcal N$. This will be crucial for the construction of a $cs$-blender for the map $\Psi$ in \eqref{eq:globalmapfull}.
\end{rem}


We begin by defining a local coordinate system on a suitable subset of $\mathcal Q_\delta^2$.  Given any value  $0<\chi\ll 1$, we let $\Theta>0$  large enough and let $\phi_{\chi,\Theta}:[-1,1]^2\to \mathbb A$ be the coordinate transformation introduced in Proposition \ref{prop:scaleddiagonalcentermaps}.  Then, we define the subsets (recall that $\phi_i$ is defined in \eqref{eq:localcoords})
\begin{equation}\label{eq:dfnQdelta}
Q_{\delta,\chi,\Theta}=\phi_1\left([0,\delta]^2\times\phi_{\chi,\Theta}([-1,1]^2)\right)\subset \mathcal Q^1_\delta,\qquad\qquad Q^{\mathrm{ext}}_{\delta,\chi,\Theta}=\phi_1\left([0,\delta]^2\times\phi_{\chi,\Theta}([-2,2]^2)\right)\subset \mathcal Q^1_\delta
\end{equation}
and introduce the local coordinate system $(p,\tau,\xi,\eta)$ on $Q^{\mathrm{ext}}_{\delta,\chi,\Theta}$ given by 
\[
(\varphi,J)=\phi_{\chi,\Theta}(\xi,\eta)\qquad\qquad\text{for }(\xi,\eta)\in[-2,2]^2.
\]
We let  $\mathcal N_{\Theta,\chi}$ be the subset defined in Proposition \ref{prop:scaleddiagonalcentermaps}.

For the rest of the paper we fix any value of $\chi,\Theta$ as above and then consider $\delta>0$ small enough. The quantities $\chi,\Theta$ being fixed we also suppress them from the notation and simply write 
\begin{equation}\label{eq:dfnNset}
Q_{\delta,\chi,\Theta}=Q_\delta\qquad\qquad Q^{\mathrm{ext}}_{\delta,\chi,\Theta}=Q^{\mathrm{ext}}_\delta\qquad\qquad \mathcal N=\mathcal N_{\Theta,\chi}.
\end{equation}
Throughout the rest of the paper we will describe submanifolds of $\mathcal Q_\delta^i$ ($i=0,1$) using local parametrizations.

\begin{nota} In this section, given a positive real number $a$, we write 
\[
a=O(\delta)\qquad\Longrightarrow\qquad \text{there exists }C(\Theta,\chi)\text{ such that } a\leq C \delta.
\]
Moreover, when we say that a certain function $h:U\subset \mathbb R^d\to \mathbb R$ (involved in some parametrization) is $C^1$ we will implicitly assume that its partial derivatives are $O(1)$, i.e. there exists $C>$ independent of $\delta$ (but which might depend on $\chi,\Theta$) such that for all $\alpha\in\{1,\dots,d\}$ we have 
\[
\sup_{z\in U} |\partial_{z_\alpha} h(z)|\leq C.
\]
\end{nota}

\subsection*{Dynamics of horizontal and vertical strips and isolating blocks}

We now study the dynamics of a class of submanifolds under the family of maps $\{\Psi^{(N)}\}_{N\in\mathcal N}$.  We say that a two-dimensional submanifold $\Lambda_V$  is a \textit{vertical strip} if it admits a parametrization of the form 
\begin{equation}\label{eq:verticalstrip}
    \Lambda_v=\{(p,\tau,\varphi,J)=F_v(\tau,  \eta)\colon \gamma_d(\tau)\leq \eta \leq \gamma_u(\tau),\ \tau\in[0,\delta]\}\subset Q^{\mathrm{ext}}_{\delta}
 \end{equation}
for some $C^1$ functions $\gamma_d<\gamma_u$  and 
\[
F_v(\tau,  \eta)=(v_1(\tau,  \eta),\tau, v_2(\tau,  \eta),\eta)
\]
with 
    \[
   \partial_  \eta v_1(\tau,  \eta)=O(\delta).
    \]
We say that a two-dimensional submanifold $\Lambda$  is a \textit{horizontal strip} if it admits a parametrization of the form 
\begin{equation}\label{eq:horizontalstrip}
    \Lambda_h=\{(p,\tau,\varphi,J)= F_h(p,\xi)\colon \gamma_l(p)\leq \xi\leq \gamma_r(p),\ p\in[0,\delta]\}\subset Q^{\mathrm{ext}}_{\delta}
 \end{equation}
for some $C^1$ functions $\gamma_d<\gamma_u$ and
  \[
  F_h(p,\xi)=(p,h_1(p, \xi), \xi,h_2(p,\xi))
  \]
 with   
    \[
 \partial_  \xi h_1(p, \xi)=O(\delta).
    \]

The following result is a convenient reformulation of Lemma 12.2 in \cite{guardia2022hyperbolicdynamicsoscillatorymotions} for $N\in\mathcal N$. A subtle but crucial difference with Lemma 12.2 in \cite{guardia2022hyperbolicdynamicsoscillatorymotions} is that our definition of horizontal (resp. vertical) strips allows the strips to be arbitrarily short in the $\eta$-direction (resp. $\xi$-direction). The proof boils down to an iteration of Theorem \ref{thm:onestepgraphtransform} together with Proposition \ref{prop:scaleddiagonalcentermaps}.

\begin{prop}\label{prop:graphtransform}
    Let $\delta>0$ be sufficiently small, let $ Q_{\delta}\subset U^1$ be the subset defined in \eqref{eq:dfnQdelta}  and let $\Lambda\subset Q_\delta$ be a vertical strip (resp. horizontal strip) as in \eqref{eq:horizontalstrip}.  Let $\mathcal N\subset\mathbb N$ be the subset defined in \eqref{eq:dfnNset} and denote by 
    \[
    \Psi^{(N)}=\Psi_{0\to 1}\circ\Psi_{0\to 0}^{N-1}\circ\Psi_{1\to 0}:Q_\delta\to \mathcal Q_\delta^1.
    \]
   
 Then, there exists $n_v\in \mathbb{N}$ (resp. $n_h\in \mathbb{N}$)  such that, for any $N\in\mathcal N$,   $\Psi^{N}(\Lambda_v)\cap \mathcal Q^1_\delta$ contains  a countable family of  vertical strips $\Lambda_{v,N}^{(n)}$, $n\geq n_v$. Analogously $(\Psi^{N})^{-1}(\Lambda_h)\cap \mathcal Q_\delta$ contains  a countable family of  horizontal strips $\Lambda_{h,N}^{(n)}$, $n\geq n_h$ for some sufficiently large $n_h$. 
    
 Moreover, these strips can be parameterized as follows.  Let also $(p,\tau,\varphi,J)$ be the local coordinate system on $Q_\delta$ introduced in \eqref{eq:localcoords}.

    Then, there exists $\tilde\chi>0$ independent of $\delta$ such that for any $N\in\mathcal N$  for any $T>0$ (resp. $P>0$) sufficiently large there exist differentiable functions $\gamma_d^{n,N}$ and $\gamma_u^{n,N}$ for $n\geq n_v$ (resp. $\gamma_l^{n,N}$ and $\gamma_r^{n,N}$ for  $n\geq n_h$) satisfying
    \begin{equation}\label{eq:weakexpansion}
    \begin{split}
    \inf_{T\in[n,n+\delta]} |\gamma_d^{n,N}(T)-\gamma_u^{n,N}(T)|> &(1+\tilde\chi) \inf_{\tau\in[0,\delta]} |\gamma_d(\tau)-\gamma_u(\tau)|\qquad\qquad \partial_T\gamma_d^{n,N},\partial_T \gamma_u^{n,N}=O(1)\\
    (\text{resp. }\inf_{P\in[n,n+\delta]} |\gamma_r^{n,N}(P)-\gamma_d^{n,N}(P)|> &(1+\tilde\chi) \inf_{p\in[0,\delta]} |\gamma_r(p)-\gamma_l(p)|\qquad\qquad \partial_P\gamma_d^{n,N},\partial_P \gamma_u^{n,N}=O(1))
    \end{split}
    \end{equation}
    such that, for any $\upsilon\in (\gamma_d^{n,N}(T),\gamma_u^{n,N}(T))$ (resp. any $\varrho\in(\gamma_l^{n,N}(P),\gamma_r^{n,N}(P))$) the system of equations  
    \begin{align*}
    T&=\pi_\tau (\Psi^{(N)}\circ F_v)(\tau,  \eta)\qquad\qquad &\upsilon&=\pi_\eta(\Psi^{(N)}\circ \Lambda)(\tau,  \eta)\\
     (\text{resp. }P&=\pi_p ((\Psi^{(N)})^{-1}\circ F_h)(p,  \xi)\qquad\qquad &\varrho&=\pi_\xi((\Psi^{(N)})^{-1}\circ \Lambda)(p,  \xi))
    \end{align*}
    defines two functions $\hat\tau_N(T,  \upsilon)$ and $\hat \eta_N(T,\upsilon)$ (resp. $\hat p_N(P,  \varrho)$ and $\hat \xi_N(P,\varrho)$) such that for each $n\geq n_v$ (resp. $n\geq n_h$) the submanifold $\Lambda^{(n)}_{v,N}$ (resp. $\Lambda^{(n)}_{h,N}$)
    admits a parametrization
    \begin{equation}\label{eq:paramcountablegraphtransform}
    \begin{split}
    \Lambda^{(n)}_{v,N}=&\{(v_1^{(n),N}(T,\upsilon),T,v_2^{(n),N}(T,\upsilon),\upsilon) \colon \gamma_d^{n,N}(T)\leq \upsilon \leq \gamma_u^{n,N}(T),T\in[n,n+\delta]\}\\
    (\text{resp. }\Lambda^{(n)}_{h,N}=&\{(P,h_1^{(n),N}(P,\varrho),\varrho,h_2^{(n),N}(P,\varrho),\upsilon) \colon \gamma_l^{n,N}(P)\leq \varrho \leq \gamma_r^{n,N}(P),P\in[n,n+\delta]\})
    \end{split}
    \end{equation}
    in terms of two differentiable functions $v_1^{(n),N}, v_2^{(n),N}$ (resp. $h_1^{(n),N}, h_2^{(n),N}$) such that 
    \[
    \begin{split}
\Psi^{(N)}\circ F_v(\hat\tau(T, \upsilon),  \hat\tau (T,\upsilon))&=(v_1^{(n),N}(T,\upsilon),T,v_2^{(n),N}(T,\upsilon))\\
(\Psi^{(N)})^{-1}\circ F_h(\hat p(P, \varrho),  \hat\xi (P,\varrho))&=(P,h_1^{(n),N}(P,\varrho),\varrho,h_2^{(n),N}(P,\varrho),\upsilon).
\end{split}
    \]        Moreover, these $C^1$ functions $v_1^{(n),N}, v_2^{(n),N}$ (resp. $h_1^{(n),N}, h_2^{(n),N}$) satisfy the  following:
    \begin{itemize}
        \item \textit{Accumulation to $W^{u,s}_{loc}(\mathcal E_\infty)$:}  $v_1^{(n),N}(T,  \upsilon)\to 0$ (resp. $h_1^{(n),N}(P,  \varrho)\to 0$)  as $n\to\infty$ in the $C^1$ topology.
        \item \textit{Asymptotics for central dynamics:}  uniformly in $n\in\mathbb N$
        \begin{equation}\label{eq:approxcenterdyngraphtransf}
        \begin{split}
        \pi_{(\xi,\eta)} \Lambda^{(n)}_{v,N}(T, \upsilon)= &\mathcal F_N\left( v_2\left(\hat\tau(T,\upsilon),\hat\eta(T,\upsilon)\right),\hat \eta(T,\upsilon)\right)+O(\delta)\\
         (\text{resp. } \pi_{(\xi,\eta)} \Lambda^{(n)}_{h,N}(T, \upsilon)= &\mathcal F_N^{-1}\left(\hat\xi(P,\varrho), h_2\left(\hat p(P,\varrho),\hat\xi(P,\varrho)\right)\right)+O(\delta))
        \end{split}
        \end{equation}
        where $\mathcal F_N$ is the map in \eqref{eq:scaleddiagonalcentermaps}.
        \item \textit{Action of the differential:} uniformly in $n\in n_v$  (resp. $n\in n_h$)
        \[
    \partial_T v_1^{(n),N},\partial_\upsilon v_1^{(n),N}=O(\delta)\qquad\qquad (\text{resp. }\partial_P h_1^{(n),N},\partial_\varrho h_1^{(n),N}=O(\delta)))
    \]
    and 
    \begin{align*}
    \partial_T v_2^{(n),N}=&O(\delta)\qquad\qquad|\partial_ {\upsilon} v_2^{(n),N}|\leq (1+\tilde\chi)^{-2} |\partial_\eta v_2|+O(\delta)\\
    (\text{resp. } \partial_P h_2^{(n),N}=&O(1)\qquad\qquad|\partial_ {\varrho} h_2^{(n),N}|\leq (1+\tilde\chi)^{-2} |\partial_\xi h_2|+O(\delta)).
    \end{align*}
    \end{itemize}
    In particular, the image of $\Lambda_v$ (resp. $\Lambda_h$) under $\Psi^{(N)}$ (resp. $(\Psi^{(N)})^{-1}$) contains a countable collection of wider vertical (resp. horizontal) strips.
\end{prop}

\begin{rem}\label{rem:shortmanifolds}
    Notice that, since  Theorem \ref{thm:onestepgraphtransform} allows for arbitrarily short $\Lambda_v$ (resp. $\Lambda_h$) manifolds (as one may take $\rho>0$ arbitrarily small), the statement in Proposition \ref{prop:graphtransform} holds unchanged also for arbitrarily short $\Lambda_v$ (resp. $\Lambda_h$). This will be of importance in the proof of Proposition \ref{prop:twosideshadowing}.
\end{rem}

\begin{proof}
    We only give the proof for vertical strips since the same argument can be applied to horizontal strips. Let $\phi:[-1,1]^2\to \mathbb A$ be the change of coordinates introduced in Lemma \ref{prop:scaleddiagonalcentermaps} and write 
   \[
   (v_2^\varphi(\tau,\eta),v_2^J(\tau,\eta))=\phi(v_2(\tau,\eta),\eta).
   \]
   Then, the submanifold (described now in the coordinate system $(p,\tau,\varphi,J)$)
   \[
   \Lambda_v=\{F_v(\tau,\eta)=(v_1(\tau,  \eta),\tau,v_2^\varphi(\tau,\eta),v_2^J(\tau,\eta))\colon \gamma_d^n(\tau)\leq \eta\leq \gamma_u^n(\tau),\tau\in[0,\delta] \}\subset Q_\delta\subset \mathcal Q^1_{\delta}
   \]
   satisfies the assumptions in Theorem \ref{thm:onestepgraphtransform}. Hence, we can choose any $n_*\in\mathbb N$ sufficiently large and let 
   \[
   \Lambda_1=\Lambda^{(n_*)}=\{\Psi_{1\to 0}\circ F_v(\hat \tau (T,\eta),\eta)\colon \gamma_d^{n_*}(T)\leq\eta \leq \gamma_u^{n_*}(T),T\in[n_*,n_*+\delta]\}
   \]
   where, 
   $\Lambda^{(n_*)}$ and $\gamma_d^{n_*},\gamma_u^{n_*}$ are as in Theorem \ref{thm:onestepgraphtransform}. 
   We now check that $\Lambda_1\subset \mathcal Q^0_\delta$ and that $v_1^1=\pi_p \Lambda_1$ satisfies 
    \[
    \partial_T  v_{1}^1=O(\delta),\qquad\partial_\eta  v_{1}^1=O(\delta),
    \]
    so we can apply again Theorem \ref{thm:onestepgraphtransform} to $\Lambda_1$. Note that, for $n_*\gg 1$ we can guarantee that $0<v_1^1< \delta$. Moreover,  it follows from the asymptotic expansion \eqref{eq:onestepcenterdyn} that  $\pi_z \Lambda_1\subset\mathbb A$ so we can conclude that $\Lambda_1\subset \mathcal Q^0_\delta$. On the other hand, the estimates for the partial derivatives of $ v_1^1$ follow plainly from \eqref{eq:differentialonestep}. By iterating this construction  we deduce that for any $N\in \mathcal N$ and any integer $n\gg 1$ there exist differentiable functions $\gamma_l^{n,N}$ and $\gamma_r^{n,N}$ 
    such that, for any $T\in[n,n+\delta]$ and $\upsilon\in (\gamma_l^{n,N}(T),\gamma_r^{n,N}(T))$ the system of equations  
    \begin{align*}
    T=\pi_\tau (\Psi^{(N)}\circ F_v)(\tau,  \eta)\qquad\qquad \upsilon=\pi_\eta(\Psi^{(N)}\circ F_v)(\tau,  \eta)
    \end{align*}
    defines two functions $\hat\tau_N(T,  \upsilon)$ and $\hat \eta_N(T,\upsilon)$. We thus end up with a countable collection of vertical submanifolds 
    \[
    \Lambda^{(n)}_N=\{F^{(n)}_N(T,  \eta)=\Psi^{(N)}\circ F_v(\hat\tau_N(T,\eta),\eta)\colon\gamma_d^{n,N}(T)\leq\eta \leq \gamma_u^{n,N}(T),T\in[n,n+\delta] \}
    \]
    which satisfy:
    \begin{itemize}
        \item \textit{Accumulation to $W^u_{\mathrm{loc}}(\mathcal E_\infty)$:}  $\pi_p F^{(n)}_N(T,  \eta)\to 0$ as $n\to\infty$
        \item \textit{Asymptotics for central dynamics:}  uniformly in $n\in\mathbb N$
        \begin{equation}\label{eq:onestepcenterdynproof}
        \pi_z F^{(n)}_N(T,  \eta)=\mathtt S_0^N\circ\mathtt S_1(v_2^\varphi(\hat\tau_N(T,  \eta),  \eta),v_2^ J(\hat\eta_N(T,  \eta),  \eta))+O(\delta)
        \end{equation}
        \item \textit{Action of the differential:} uniformly in $n\in\mathbb N$
        \begin{equation}\label{eq:differentialonestepproof}
        \partial_T F^{(n)}_N(T,  \eta)=\begin{pmatrix}
            O(\delta)\\
            1\\
            O(\delta)\\
            O(\delta)
        \end{pmatrix}\qquad\qquad \partial_  \xi F_N^{(n)}(T,  \eta)=\begin{pmatrix}
             O(\delta)\\
            O(\delta)\\
           D(\mathtt S^N_0\circ\mathtt S_1)(v_2^\varphi(\hat\tau_N,  \eta),v_2^ J(\hat\tau_N,  \eta))\binom{\partial_  \eta v_2^\varphi}{\partial_  \eta v_2^ J}+O(\delta)
        \end{pmatrix}
        \end{equation}

        \item \textit{(Control on the center width):} uniformly in $n$
        \begin{equation}\label{eq:weakexpansionintermproof0}
        \inf_{T\in [n,n+\delta]}|\gamma_u^{n,N}(T)- \gamma_d^{n,N}(T)|\geq (1-O(\delta))\  \inf_{\tau\in [0,\delta]}|\gamma_u(\tau)- \gamma_d(\tau)|.
        \end{equation}
    \end{itemize}
    It now only remains to describe the submanifolds $\Lambda^{(n)}_N$ in the $(p,\tau,\xi,\eta)$-local  coordinate system given by the linear map $\phi$ in Proposition \ref{prop:scaleddiagonalcentermaps}
    \[
    (\varphi,J)=\phi(\xi,\eta)\qquad\text{for }(\xi,\eta)\in[-1,1]^2.
    \]
    On one hand, for $\mathcal F_N=\phi^{-1}\circ\mathtt S_0^N\circ\mathtt S_1\circ\phi$, the asymptotic expansion \eqref{eq:onestepcenterdynproof} readily implies that
    \[
    \pi_{(\xi,\eta)} F_N^{(n)}(T,\eta)=\mathcal F_N(v_2^\varphi(v_2(\hat\tau_N(T,  \eta),  \eta)),\eta)+O(\delta).
    \]
    On the other hand, it follows from \eqref{eq:scaleddiagonalcentermaps} that
    \begin{equation}\label{eq:weakexpansionproof}
   \pi_\xi D\mathcal F_N(v_2^\varphi(v_2(\hat\tau_N(T,  \eta),\eta)),  \eta))\binom{\partial_\eta v_2}{1}=1+\chi+O(\chi^2).
    \end{equation}
    Recall moreover that we can assume that $0<\chi\ll 1$ is arbitrarily small. Then, in view of the above expression the equation 
    \[
    \upsilon=\pi_\eta F_N^{(n)}(T,\eta)
    \]
    admits a unique solution $\hat\eta_N(T,\upsilon)$ for all 
    \[
    \upsilon\in \{\pi_\eta F_N^{(n)}(T,\eta)\colon \gamma_d^{n,N}(T)\leq \eta\leq \gamma_u^{n,N}(T)\}.
    \]
    Notice that \eqref{eq:weakexpansionproof} implies that for fixed $T$ 
    \[
    \{\upsilon=\pi_\eta F_N^{(n)}(T,\eta)\colon \gamma_d^{n,N}(T)\leq \eta\leq \gamma_u^{n,N}(T)\}=\{\tilde\gamma_d^{n,N}(T)\leq \upsilon\leq \tilde\gamma_u^{n,N}(T)\}
    \]
    with
    \begin{equation}\label{eq:weakexpansionintermproof}
    |\tilde\gamma_d^{n,N}(T)-\tilde\gamma_u^{n,N}(T)|\geq (1+\chi-O(\chi^2))    |\gamma_d^{n,N}(T)-\gamma_u^{n,N}(T)|.
    \end{equation}
    We have then shown that 
 for any $T>0$ sufficiently large and any $  \upsilon\in  \{\tilde\gamma_d^{n,N}(T)\leq \upsilon\leq \tilde\gamma_u^{n,N}(T)\}$ the system of equations  
    \begin{align*}
    T=\pi_\tau (\Psi^{(N)}\circ F_v)(\tau,  \eta)\qquad\qquad \upsilon=\pi_\eta(\Psi^{(N)}\circ F_v)(\tau,  \eta)
    \end{align*}
    defines two functions $\hat\tau(T,  \upsilon)$ and $\hat \eta(T,\upsilon)$ such that for each $n\geq n_v$ the submanifold $\Lambda_N^{(n)}$ admits a parametrization
    \[
    \Lambda^{(n)}_N=\{(v_1^{(n)}(T,\upsilon),T, v_2^{(n)}(T,\upsilon),\upsilon)\colon \tilde\gamma_d^{n,N}(T)\leq\upsilon \leq \tilde\gamma_u^{n,N}(T),T\in[n,n+\delta]\}
    \]
    in terms of two differentiable functions $v_1^{(n)}, v_2^{(n)}$. The corresponding estimates for $v_1^{(n)}, v_2^{(n)}$ follow from straightforward computations. Finally, the estimate \eqref{eq:weakexpansion} is a trivial consequence of \eqref{eq:weakexpansionintermproof0} and \eqref{eq:weakexpansionintermproof}.
    \end{proof}

We now state a  corollary of Proposition \ref{prop:graphtransform} which will prove useful later. To this end we need the following definition.

\begin{defn}\label{defn:vsetshsets}
We say that a subset $\mathcal H\subset Q^{\mathrm{ext}}_\delta$ is a \textit{h-set} if it can be foliated by horizontal strips: in the local coordinate system given by $(p,\tau,\xi,\eta)\in[0,\delta]^2\times[-2,2]^2$ can be described implicitly as 
\begin{equation}\label{eq:hsetparam}
\mathcal H=\{h_{1,d}(p,\xi,\eta)\leq \tau\leq h_{1,u}(p,\xi,\eta),\  \gamma_l(p,\tau)\leq \xi\leq \gamma_r(p,\tau),\  h_{2,d}(p,\tau,\xi)\leq \eta\leq h_{2,u}(p,\tau,\xi),\ p\in[0,\delta]\}
\end{equation}
for some differentiable functions which satisfy 
\begin{equation}\label{eq:estimateshset}
\partial_p h_{1,\star}=O(\delta),\qquad \partial_\xi h_{1,\star},\partial_\eta h_{1,\star}=O(\delta)\qquad\qquad\star=d,u.
\end{equation}
Analogously,  $\mathcal V\subset Q^{\mathrm{ext}}_\delta$ is a \textit{v-set} if it can be foliated by vertical strips: it admits a parametrization of the form 
\begin{equation}\label{eq:vsetparam}
\mathcal V=\{v_{1,l}(\tau,\xi,\eta)\leq p\leq v_{1,r}(\tau,\xi,\eta),\  \gamma_u(p,\tau)\leq \eta\leq \gamma_d(p,\tau),\  v_{2,l}(p,\tau,\eta)\leq \xi\leq v_{2,r}(p,\tau,\eta),\ \tau\in[0,\delta]\}
\end{equation}
in terms of differentiable functions satisfying
\[
\partial_\tau v_{1,\star}=O(\delta),\qquad \partial_\xi v_{1,\star},\partial_\eta h_{1,\star}=O(\delta)\qquad\qquad \star=l,r.
\]
    
\end{defn}

\begin{prop}\label{prop:imagesofhsetsvsets}
    Let $\mathcal V\subset Q_\delta$ (resp. $\mathcal H$) be a $v$-set (resp. $h$-set). Then, for any $N\in\mathcal N$ the image $\Psi^{(N)}(\mathcal V)\cap Q^{\mathrm{ext}}_\delta$ (resp. $(\Psi^{(N)})^{-1}(\mathcal H)\cap Q^{\mathrm{ext}}_\delta$) contains a countable collection of $v$-sets (resp. $h$-sets). In particular $\Psi^{(N)} (Q_\delta)\cap Q^{\mathrm{ext}}_\delta$ (resp. $(\Psi^{(N)})^{-1} (Q_\delta)\cap Q^{\mathrm{ext}}_\delta)$ contains a countable collection of $v$-sets (resp. $h$-sets).
\end{prop}
\begin{proof}
   We only provide the proof for $h$-sets, the one for $v$-sets being similar. Notice that $\mathcal H$ admits a foliation of the form $\mathcal H=\bigcup_{(r,s)\in[0,1]^2} \Delta_{r,s}$ for 
    \[
    \Delta_{r,s}=\{ \tau= G_\tau(s,p,\xi,\eta),\ \eta=G_\eta(r,p,\tau,\xi),\ 
   \gamma_l(p,\tau(s,p,\xi,\eta))\leq  \xi\leq \gamma_r (p,\tau(s,p,\xi,\eta)),\ p\in[0,\delta]\},
    \]
where
\[
G_\tau(s,p,\xi,\eta)=s h_{1,u}(p,\xi,\eta)+(1-s) h_{1,d}(p,\xi,\eta)\qquad\qquad G_\eta(r,p,\tau,\xi)=r h_{2,u}(p,\tau,\xi)+(1-r) h_{2,d} (p,\tau,\xi).
\]
Making use of the estimates \eqref{eq:estimateshset} it is not difficult to check that the system of equations $\tau= \tau(s,p,\xi,\eta)$ and $\ \eta=\eta(r,p,\tau,\xi)$ defines two functions $h_1(p,\xi;r,s)$ and $h_2(p,\xi;r,s)$ such that 
\[
\Delta_{r,s}=\{(p,h_1(p,\xi),\xi, h_2(p,\xi)\colon \gamma_l(p,h_1(p,\xi))\leq \xi\leq \gamma_r(p,h_2(p,\xi)),\ p\in[0,\delta]\}.
\]
Moreover, one may check that $\partial_p h_1=O(1)$, $\partial_\xi h_1=O(\delta)$ and $\partial_p h_2,\partial_\xi h_2=O(1)$. Hence, $\Delta_{r,s}$ is a horizontal submanifold and the conclusion follows by direct application of Proposition \ref{prop:graphtransform}.
\end{proof}
\medskip

Finally, we conclude this section by constructing an \textit{isolating block} for each of the maps $\Psi^{(N)}$. If $\mathcal H$ is a $h$-set and $\mathcal V$ is a $v$-set (with parametrizations as in \eqref{eq:hsetparam} and \eqref{eq:vsetparam}) we say that $\mathcal V$ fully-crosses $\mathcal H$ if  the subset defined implicitly by
\[
\mathcal V_\mathcal H=\{(p,\tau,\xi,\eta)\in\mathcal V\colon h_{1,d}(p,\xi,\eta)\leq \tau\leq h_{1,u}(p,\xi,\eta),\ h_{2,d}(p,\tau,\xi)\leq \eta\leq h_{2,u}(p,\tau,\xi) \}
\]
is entirely contained in $\mathcal H$. Analogously, we say that $\mathcal H$ fully-crosses $\mathcal V$ if  the subset defined implicitly by
\[
\mathcal H_\mathcal V=\{(p,\tau,\xi,\eta)\in\mathcal H\colon v_{1,l}(\tau,\xi,\eta)\leq p\leq v_{1,r}(\tau,\xi,\eta),\  v_{2,l}(\tau,p,\eta)\leq \xi\leq v_{2,r}(\tau,p,\eta)\}
\]
is entirely contained in $\mathcal V$.

\begin{prop}\label{prop:Markovproperty}
    For each $n\in\mathcal N$ let $Q_{\delta,N}\subset Q_\delta$ be the rectangle given by
    \[
    Q_{\delta,N}=[0,\delta]^2\times D_N
    \]
    with $D_N$ as in Lemma \ref{lem:Markovcenter}. Then, $\Psi^{(N)}(Q_{\delta,N})\cap Q_{\delta}$ fully crosses $Q_{\delta,N}$ and  $(\Psi^{(N)})^{-1}(Q_{\delta,N})\cap Q_\delta$ fully  crosses $Q_{\delta,N}$.
\end{prop}

\begin{proof}
    The proof is a trivial consequence of Proposition \ref{prop:graphtransform} and Lemma \ref{lem:Markovcenter}. In particular, we notice that by \eqref{eq:approxcenterdyngraphtransf}, the center dynamics is given by a $O(\delta)$ perturbation of the maps $\mathcal F_N$ in Proposition \ref{prop:scaleddiagonalcentermaps}. Therefore, the conclusion follows from Lemma \ref{lem:isolatingbloc}.
\end{proof}

\subsection*{Existence of cone fields}

Having analyzed the dynamics of horizontal and vertical strips we now establish the existence of two families of cone fields $\mathcal C^u,\mathcal C^s$  for the family of maps $\{\Psi^{(N)}\}_{N\in\mathcal N}$.  We do so in Proposition \ref{prop:conefields} which, again, is a suitable reformulation of the results in  \cite{guardia2022hyperbolicdynamicsoscillatorymotions}. As pointed out above (see Remark \ref{rem:uniformcoordsystem}) an important observation is that the families of cone fields are uniform in $N$ for $N\in\mathcal N$.

\begin{prop}[After Proposition 6.5 in \cite{guardia2022hyperbolicdynamicsoscillatorymotions}]\label{prop:conefields}
At any $Z\in Q_\delta$ there exists a matrix of the form
\begin{equation}\label{eq:changetangent}
C(Z)=\begin{pmatrix}
1&0&0\\
0&1&0\\
O(1)&0&\mathrm{id}_2
\end{pmatrix}
\end{equation}
and real numbers $0<\alpha'<\alpha<1$ such that the following holds for any $\delta>0$ small enough.  For any $N\in\mathcal N$, let 
\[
\mathcal M_N(Z)=C(\Psi^{(N)}(Z))^{-1} D\Psi^{(N)} (Z)  C(Z)
\]
and, in coordinates $y=(y_1,y_2,y_3,y_4)\in\mathbb R^4$  define the cones 
\[
\mathcal C^s_\alpha=\{ \alpha \max\{|y_1|,|y_3|\} \geq  \max\{|y_2|,|y_4|\} \}\qquad\qquad \mathcal C^u_\alpha=\{ \alpha \max\{|y_2|,|y_4|\} \geq  \max\{|y_1|,|y_3|\} \}.
\]
Then,  at any $Z \in Q_\delta$ we have:
\begin{itemize}
\item \textit{Invariance:}
\[
\mathcal M_N \mathcal C_\alpha^u\subset \mathcal C_{\alpha'}^u\qquad\qquad \mathcal M_N^{-1} \mathcal C_\alpha^s\subset \mathcal C_{\alpha'}^s
\]
\item \textit{Expansion:} There exists $\tilde\chi>1$ such that: if $y\in \mathcal C^u_\alpha$ (resp.  $y\in \mathcal C^s_\alpha$) then 
\[
\max \{ |\mathcal M_N y|_2,  |\mathcal M_N y|_4\}\geq \tilde\chi\max \{ | y|_2,  |y|_4\}\qquad\qquad\text{(resp.  $\max \{ |\mathcal M_N^{-1} y|_1,  |\mathcal M_N^{-1} y|_3\}\geq \tilde\chi\max \{ | y|_1  |y|_3\}$)}.
\]
\end{itemize}
\end{prop}

\begin{proof}
Proposition 12.4 in \cite{guardia2022hyperbolicdynamicsoscillatorymotions} shows that there exist $0<\tilde\alpha'<\tilde\alpha<1$ independent of $\delta$  such that  (in the notation of that paper) for each $N$ there exists  a matrix $C_N(Z)=C_{2,1}(Z)\widetilde C_N(Z)$, with $C_{2,1}$ independent of $N$ and of the form  \eqref{eq:changetangent}, and $\widetilde C_N(Z)$ being a $O(\delta^{1/5})$-perturbation of the identity matrix,  for which
\[
\widetilde {\mathcal M}_N(Z)=C_N(\Psi^{(N)}(Z))^{-1} D\Psi^{(N)} (Z)  C_N(Z) 
\]
satisfies
\[
\widetilde {\mathcal M}_N \mathcal C_\alpha^u\subset \mathcal C_{\alpha'}^u\qquad\qquad\widetilde {\mathcal M}_N^{-1} \mathcal C_\alpha^s\subset \mathcal C_{\alpha'}^s
\]
and
\[
\max \{ |\widetilde{\mathcal M}_N y|_2,  |\widetilde{\mathcal M}_N y|_4\}\geq \tilde\chi\max \{ | y|_2,  |y|_4\}\qquad\qquad \max \{ |\widetilde{\mathcal M}_N^{-1} y|_1,  |\widetilde{\mathcal M}_N^{-1} y|_3\}\geq \tilde\chi\max \{ | y|_1  |y|_3\}.
\]
The result then follows trivially since $\widetilde C_N(Z), \widetilde C_N(Z)^{-1}$ only move the cones by $O(\delta^{1/5})$ (uniformly in $N\in\mathcal N$), so we can take 
\[
\alpha=\tilde\alpha-O(\delta^{1/5})\qquad\qquad \alpha'=\tilde\alpha'+O(\delta^{1/5}).\qedhere
\]
\end{proof}

\subsection{Existence of a $cs$-blender}\label{sec:cu-blender3bpprooflemmas}
In this section we establish the existence of a $cs$-blender for the return map $\Psi$. In this particular setting we define $cs$-strips as follows. 

\begin{defn}[$cs$-strips]
 We say that a horizontal submanifold $\Delta$ is a \textit{$cs$-strip} if $\Delta\subset Q_\delta$. 
\end{defn}
\begin{thm}\label{thm:cublender}
    There exists a hyperbolic periodic point $P\in Q_\delta$ (of the map $\Psi$) such that the pair $(P,Q_\delta)$ is a $cs$-blender for $\Psi$ in \eqref{eq:globalmapfull}. More concretely, any $cs$-strip $\Delta\subset Q_\delta$ intersects $W^u(P)$ robustly.
\end{thm}

Inspired by the discussion in Section \ref{sec:heuristicblenders}, the proof of Theorem \ref{thm:cublender} is split into two parts: the  covering property (Proposition \ref{prop:covering4dmaps}) and  the well-distribution of  hyperbolic periodic orbits (Proposition \ref{prop:hyperbolicpointswelldistr}).

To prove the covering property, we first recall that, by Proposition \ref{prop:imagesofhsetsvsets}, for each $N\in\mathcal N$ there exist a countable collection of $v$-sets   $\mathcal V^{(n)}_N$
    \begin{equation}\label{eq:verticalrectangles}
    \bigcup_{n\geq n_v} \mathcal V^{(n)}_N= \Psi^{(N)}( Q_{\delta})\cap Q^{\mathrm{ext}}_\delta
    \end{equation}
We choose any $n_*$ large enough and let $\mathcal V
    _N=\mathcal V_N^{(n_*)}$.

\begin{prop}\label{prop:covering4dmaps}
   For any $cs$-strip $\Delta$, the image $(\Psi^{(N)})^{-1}(\Delta)\cap Q_\delta$ contains at least one $cs$-strip. More concretely, if $\{\mathcal V_N\}_{N\in\mathcal N}\subset Q_\delta$ is the family of vertical rectangles in \eqref{eq:verticalrectangles}, for any $cs$-strip $\Delta$ there exists a (possibly equal) $cs$-strip $\widetilde \Delta\subset \Delta$ and $N\in\mathcal N$ such that $\bar \Delta=(\Psi^{(N)})^{-1}(\widetilde\Delta\cap \mathcal V_N)$ is a $cs$-strip.   Moreover, if $\mathrm{width}(\tilde\Delta)\leq 7/4$, the conclusion holds with $\widetilde \Delta=\Delta$.
\end{prop}

\begin{proof}
The proof is a rather straightforward  consequence of  Lemma \ref{lem:Markovcenter} and the argument given in Proposition \ref{prop:covering} for the symbolic case. The increased length of the argument is a just a consequence of the small technicalities  inherent to the non-linear setting. 

We consider a $cs$-strip $\Delta$ and let $\gamma_l$, $\gamma_r$ and $h_{1,d}$, $h_{1,u}$, $h_{2,u}$, $h_{2,d}$ be the functions associated to the parametrization of $\Delta$ as in  \eqref{eq:horizontalstrip}. The rectangles $\mathcal V_N$ have a slightly distorted geometry and, before continuing, it will be convenient to first construct suitable subsets of $\mathcal V_N^{(n)}$ which can be written as graphs over $(\xi,\eta)$. Denote by $v_{i,\star}^N$ $i=0,1$, $\star=l,r$ and $\gamma_{\star}^N$, $\star=l,r$ the functions involved in the parametrization \eqref{eq:vsetparam} of the $v$-set $\mathcal V_N$. Since 
    \[
    \partial_\tau v_{2,\star}^N=O(1)\qquad\qquad \partial_\tau \gamma^N_\star=O(1),
    \]
    it is not difficult to see that there exists a constant $C>0$ such that
    \[
  \widetilde {\mathcal V}_N:=\{(p,\tau,\xi,\eta)\colon v_{1,l}^{N}(\tau,\xi,\eta)\leq p\leq v^{N}_{1,r}(\tau,\xi,\eta),\ (\xi,\eta)\in  \mathcal F_N([-1+C\delta,1-C\delta]^2),\  \tau\in[0,\delta]\}\subset \mathcal V_N
    \]
    where $\mathcal F_N$ are the maps in Proposition \ref{prop:scaleddiagonalcentermaps}. Let $b_N$ be as in Proposition \ref{prop:scaleddiagonalcentermaps} and define
    \[
    \tilde\gamma_{l,N}=(1-\chi)(-1+C\delta +b_N),\qquad\qquad \tilde \gamma_{r,N}= (1-\chi)(1-C\delta+b_N)
    \]
    so
    \[
    \mathcal F_N([-1+C\delta,1-C\delta]^2)\cap[-1,1]=[\tilde\gamma_{l,N},\tilde\gamma_{r,N}]\times[-1,1].
    \]
    The result then follows if we show that at least for one $N\in\mathcal N$ we can find 
    \[
    -1\leq\gamma_l\leq \tilde\gamma_l<\tilde\gamma_r\leq \gamma_r\leq 1
    \]
    such that the piece $\tilde\Delta_N$ defined implicitly by 
     \begin{equation}\label{eq:substrip}
      \tilde \Delta_N=\{(p,h_1(p,\xi),\xi, h_2(p,\xi))\colon \tilde\gamma_l\leq \xi\leq \tilde\gamma_r,\  v_{1,l}^N(h_1(p,\xi),\xi,h_2(p,\xi)) \leq p \leq v_{1,r}^N(h_1(p,\xi),\xi,h_2(p,\xi))\}
     \end{equation}
     is contained entirely in $\widetilde{\mathcal V}_N$. Indeed,  Lemma \ref{lem:Markovcenter} implies that for $\delta$ small enough there exist $N_\pm\subset\mathcal N$ such that (it is enough to consider $N_+$ for which $b_{N_+}\in (2\chi,3\chi)$ and $N_-$ for which $b_{N_-}\in (-3\chi,-2\chi)$)
    \[
    [-1,1]\subset \bigcup_{N\in\{N_+,N_-\}} [\tilde \gamma_{l,N},\tilde\gamma_{r,N}].
    \]
 Moreover  a straightfoward computation shows that 
   \[
   \tilde\gamma_{r,N_+}-\tilde\gamma_{l,N_-}\geq 2-O(\chi).
    \]
    so, if the $cu$- strip $\Delta$ is  such that the ``slices'' 
\[
    \Delta_{p_*}=\Delta\cap\{p=p_*\}
    \]
satisfy $\pi_{(\xi,\eta)}\Delta_p=\pi_{(\xi,\eta)}\Delta_{p_*}$ for any pair of values $p,p*\in[0,\delta]$ then the result  follows provided we choose $\tilde\gamma_r-\tilde\gamma_l\leq 7/4$ (any number strictly smaller than two would suffice) in the definition \eqref{eq:substrip} of the pieces $\widetilde\Delta_{N_\pm}$.

Our definition of $cs$-strips only allows these slices to  move slightly as we move $p$. In particular, it follows from the mean value theorem and the fact that $\partial_p h_2=O(1)$ and $\partial_p \gamma_{l},\partial_p \gamma_{r}=O(1)$ that for any pair $p,p_*\in[0,\delta]$
    \[
    \max\{|(\xi,\eta)-(\xi_*,\eta_*)|\colon (\xi,\eta)\in \pi_{(\xi,\eta)}(\Delta_p),\ (\xi_*,\eta_*)\in \pi_{(\xi,\eta)}(\Delta_{p_*}))\}=O(\delta).
    \]
    Hence,  if $\tilde \Delta_{N_+}\nsubseteq \mathcal{\widetilde V}_{N_+}$ we must have $\tilde\Delta_{N_-}\subset\widetilde{\mathcal V}_{N_-}$.
\end{proof}

Now we show that the  hyperbolic periodic orbits in $Q_\delta$ are Well-distributed. 
 Together, Propositions \ref{prop:Markovproperty} and \ref{prop:conefields}  imply that, for each $N\subset \mathcal N$ there exists a hyperbolic set $\mathcal X_N\subset Q_{\delta,N}\subset Q_\delta$ on which $\Psi^{(N)}|_{\mathcal X_N}$ is conjugated to the full-shift acting on the space $\mathbb N^\mathbb Z$ of bi-infinite sequences.  In particular, for each $N\in\mathcal N$ we can extract a hyperbolic fixed point $P_N\subset Q_{\delta,N}$ for the map $\Psi^{(N)}$. 

 \begin{prop}\label{prop:hyperbolicpointswelldistr}
     Let $P_N\in \mathcal V_{N}$ be a hyperbolic periodic point corresponding to a fixed point of the map $\Psi^{(N)}$. Then, 
     \begin{equation}\label{eq:parametrizationlocalunst3bpproof}
     W^u_{\loc}(P_N)=\{(g_{1,N}(\tau,\eta),\tau,g_{2,N}(\tau,\eta),\eta)\colon \tau\in[0,\delta],\ \eta\in[-1,1] \}
     \end{equation}
     for some differentiable functions satisfying
     \[
     \partial_\tau g_{1,N},\partial_\eta g_{1,N}=O(\delta),\qquad\qquad g_{2,N}= \frac{b_N}{\chi}+O_{C^1}(\chi).
     \]
     where $b_N$ is as in Proposition \ref{prop:scaleddiagonalcentermaps}.
 \end{prop}

 \begin{proof}
     The proof of this result follows easily from a graph transform argument and the estimates given in Proposition \ref{prop:graphtransform}. Although establishing the asymptotic expansion for $g_{2,N}$ requires a somewhat delicate analysis, the corresponding estimates can be obtained as in the proof of Proposition \ref{prop:welldistributed} for the symbolic case. The details are left to the reader.
 \end{proof}

Theprem  Propositions \ref{prop:covering4dmaps} and \ref{prop:hyperbolicpointswelldistr}  lead to the proof of Theorem \ref{thm:cublender}.

\begin{proof}[Proof of Theorem \ref{thm:cublender}]
To prove the theorem, we use Propositions \ref{prop:covering4dmaps} and \ref{prop:hyperbolicpointswelldistr} to show that, there exists $\tilde\chi>1$ such that for any $cs$-strip $\Delta\in Q_\delta$ either:
  \begin{itemize}
        \item there exists $N\in\mathcal N$ such that $\Delta\pitchfork W^{u}(P_N)\neq \emptyset$  or
        \item there exists $N\in\mathcal N$ such that $\bar\Delta=(\Psi^{(N)})^{-1}(\Delta\cap \mathcal V_N)\cap Q_\delta$ is a $cs$-strip and $\mathrm{width}(\bar\Delta)>\tilde\chi \mathrm{width}(\Delta)$.
    \end{itemize}
On one hand we notice that (see Proposition \ref{prop:scaleddiagonalcentermaps}) there exist $\{N_l,N_r\}\subset\mathcal N$ such that 
\[
b_{N_l}\in \left(-\frac 34 \chi,-\frac 14\chi\right) \qquad\qquad b_{N_r}\in \left(\frac 14\chi,\frac 34\chi\right).
\]
Hence, for any $cs$-strip $\Delta$ with $\mathrm{width}(\Delta)\geq 7/4$ there exists at least one $N_*\in\{N_l,N_r\}$ such that 
$\Delta\pitchfork W^{u}(P_{N_*})\neq \emptyset$. On the other hand,  if $\mathrm{width}(\Delta)\leq 7/4$ we have shown in Proposition \ref{prop:covering4dmaps} that there exists $N\in\mathcal N$ such that $\Psi^{-1}(\Delta\cap \mathcal V_N)$ is a $cs$-strip. Thus, by Proposition \ref{prop:graphtransform} there exists $\tilde\chi>1$ (independent of $\Delta$) such that 
\[
\mathrm{width}(\Psi^{-1}(\Delta\cap \mathcal V_N))\geq \tilde\chi \mathrm{width}(\Delta).
\]
The proof of Theorem \ref{thm:cublender} is complete.
\end{proof}

\subsection{Existence of a symplectic blender: end of the  proof of Theorem \ref{thm:mainblender}}\label{sec:sympblender3bp}

As we have done for the symbolic case (see Section \ref{sec:IFSlocaltransitive}), we now exploit the almost reversibility of the maps $\mathtt S_i$ constructed in Theorem \ref{thm:transvtorsionScattmaps} under the involution $\psi_R(\varphi,J)\mapsto (-\varphi,J)$ (see \eqref{def:involution}) to construct a $cu$-blender for the map $\Psi$ in \eqref{eq:globalmapfull}. Let us recall that the argument in Section \ref{sec:casebeta0} was based on the fact that ``first orders'' of the maps $T_0,T_1$ in Section \ref{sec:IFSlocaltransitive} (recall that these maps constitute an abstract model in which the scattering maps $\mathtt S_0,\mathtt S_1$ fall) are indeed reversible with respect to $\psi_R$. Therefore, we  have concluded that, up to small errors $T_1\circ T_0^n$ was conjugated by $\psi_R$ to the map $(T_0^n\circ T_1)^{-1}$. Hence, our previous construction of a $cs$-blender using maps of the form $\{T_0^n\circ T_1\}_n$ automatically implied the existence of a $cu$-blender associated to the family of maps $\{T_1\circ T_0^n\}_n$.

The very same argument would be valid in the present context. However, in order to reproduce the sequence $T_1\circ T_0^n$, one needs to consider the map
\[
\Psi_{1\to 0}\circ \Psi_{0\to 1}\circ\Psi_{0\to 0}^{N-1}
\]
which is only defined on a subset of $\mathcal Q_\delta^0$. As a consequence the $cu$-blender obtained by this construction would be associated to a subset of the section $\mathcal Q_\delta^0$, while the $cs$-blender constructed in Theorem \ref{thm:cublender} is associated to a subset $Q_\delta\subset\mathcal Q_\delta^1$. Although one could get around this technical annoyance by later transporting the $cu$-blender to the section $Q_\delta^1$ it is slightly more convenient to take a slightly different approach which we detail below.

\subsection*{Different coordinate systems}
Throughout this section we will make use of several coordinate systems:
\begin{itemize}
    \item We recall that the transverse sections $\mathcal Q_\delta^{i}$ were introduced in \eqref{eq:localcoords} and that we defined local coordinate  charts $\phi_i:(p,\tau,\varphi,J)\in[0,\delta]^2\times \mathbb A\to \mathcal Q_\delta^i$.
 \item In Section \ref{sec:blender3bp}, we have considered a smaller region 
 \[
 Q_\delta=\phi_1\left([0,\delta]^2\times  \phi([-1,1])^2\right)\subset \mathcal Q^1_\delta,
 \]
where $\phi_1$ is the map introduced in \eqref{eq:localcoords}  and   $\phi=\phi_{\chi,\Theta}:[-2,2]^2\to \mathbb A$ is the affine map given in Proposition \ref{prop:scaleddiagonalcentermaps}. These transformations lead to the local coordinate system 
$(p,\tau,\xi,\eta)\in[0,\delta]^2\times[-1,1]^2$.
 
 \item We now define the region 
\begin{equation}\label{eq:domaincublender}
    \widetilde Q_\delta=\phi_0\left([0,\delta]^2\times  \psi_R\circ\phi([-1,1])^2\right)\subset \mathcal Q^0_\delta,
\end{equation}
wher $\psi_R$ is the involution \eqref{def:involution}, and let $(\tilde p,\tilde\tau,\tilde \xi,\tilde\eta)\in[0,\delta]^2\times[-1,1]^2$ be local coordinates on $\widetilde Q_\delta$.
 
\item Finally we define  the region 
\begin{equation}\label{eq:finalblenderregion}
  \widehat Q_\delta=\phi_0([0,\delta]^2\times \phi([-1,1])^2)\subset \mathcal Q_\delta^0.
\end{equation}
We notice that there exist $-1<a<0<b<1$ such that  
\begin{equation}\label{eq:intersectiondomainblenders}
\widehat Q_\delta\cap \widetilde Q_\delta=\phi_0([0,\delta]^2\times \phi([-1,1]\times[a,b]))=\phi_0([0,\delta]^2\times \psi_R\circ\phi([a,b]\times[-1,1]))
\end{equation}
and that the transition map between the two local coordinate patches is given by a rotation by an angle $\pi/2$ in the ``center variables'':
\begin{equation}\label{eq:transchart}
(\tilde p,\tilde\tau,\tilde\xi,\tilde\eta)\mapsto (\hat p,\hat \tau,\hat\xi,\hat\eta)=(\tilde p,\tilde\tau,-\tilde\eta,\tilde\xi).    
\end{equation}
\end{itemize}

\subsection*{A $cs$-blender in $\mathcal Q_\delta^0$}
Our solution to the technical issue highlighted above passes through the construction of a $cs$-blender in $\mathcal Q_\delta^0$. To that end, the main observation is that, for large $n\in\mathbb N$, at the affine level, the maps $T_0^n\circ T_1$ and $T_0^n\circ T_1\circ T_0$ (recall that $T_0,T_1$ are abstract models for the scattering maps $\mathtt S_0,\mathtt S_1$) only differ by a constant horizontal shift. Indeed, the same computation performed in Lemma \ref{lem:c1control} shows that for $n\leq 1/2\varepsilon$ and $\{|J|\leq \varepsilon/8\}$
\[
\widehat F_n=T_1^n\circ T_2\circ T_1:(\varphi,J)\mapsto A_n \binom{\varphi}{J}+\hat b_n+\widehat{\mathcal E}(\varphi,J),
\]
with $A_n$ as in Lemma \ref{lem:c1control}, $\widehat{\mathcal E}$ satisfying the same estimates as $\mathcal E$ in Lemma \ref{lem:c1control} and $\hat b_n=(\hat b+[n\omega],0)^\top$ for some $\hat b\in\mathbb R$. In particular, the very same change of variables $\phi$ in Lemma \ref{lem:uniformlemma} also puts $\widehat F_n$ in normal form \eqref{eq:normalformifsproof} (modulo a horizontal, constant in $n$, shift).

In the context of the scattering maps $\mathtt S_0,\mathtt S_1$, the above discussion implies that the same transformation $\phi$ in Proposition \ref{prop:scaleddiagonalcentermaps} also puts $\mathtt S_0^N\circ \mathtt S_1\circ \mathtt S_0$ in normal form  \eqref{eq:scaleddiagonalcentermaps} and that for a suitable $\widehat{\mathcal N}\subset \mathbb N$ the associated sequence $\{\hat b_N\}_{N\in\widehat{\mathcal N}}$ is $\frac{1}{10}\chi$-dense in $[-10\chi,10\chi]$. We can therefore consider the family of maps 
\[
\widehat{\Psi}^{(N)}=\Psi_{0\to 0}^N\circ \Psi_{1\to 0}\circ\Psi_{0\to 1}:\widehat Q_\delta\subset \mathcal Q_\delta^0 \to \mathcal Q_\delta^0
\]
and repeat all the discussion in Sections \ref{sec:transvtorsion}, \ref{sec:cu-blender3bpprooflemmas} to deduce the following. 
\begin{prop}\label{prop:finalcsblender}
    There exists a hyperbolic periodic point $\widehat P\in \widehat Q_\delta$ (of the map $\Psi$) such that the pair $(\widehat P,\widehat Q_\delta)$ is a $cs$-blender for the map $\Psi$ in \eqref{eq:globalmapfull}: any $cs$-strip $\Delta\subset \widehat Q_\delta$ intersects $W^u (P)$ robustly. Moreover, in local coordinates $(\hat p, \hat \tau,\hat \xi,\hat \eta)$ in $\widehat Q_\delta$ the local manifold $W^u_{\loc}(P)$ admits a parametrization of the form \eqref{eq:parametrizationlocalunst3bpproof}.
\end{prop}
\subsection*{A $cu$-blender in $\mathcal Q_\delta^0$}
For $N\in\widehat{\mathcal N}$ we define on  $\widetilde Q_\delta$ the map 
\begin{equation}\label{eq:mapcublender}
    \widetilde \Psi^{(N)}=\Psi_{1\to 0}\circ \Psi_{0\to 1}\circ\Psi_{0\to 0}^{N-1}
\end{equation}
Then, the very same argument deployed in the proof of Proposition \ref{prop:graphtransform} shows that a similar statement also holds for the family of maps $\{(\widetilde \Psi^{(N)})^{-1}\}_{N\in\widehat{\mathcal N}}$ but with the modifications that we detail below.  Note that this leads to a $cs$-blender for $\{(\widetilde \Psi^{(N)})^{-1}\}_{N\in\widehat{\mathcal N}}$ and, consequently, to a $cu$-blender for $\{\widetilde \Psi^{(N)}\}_{N\in\widehat{\mathcal N}}$.

Let us now explain the modifications we have to make to the argument provided in Sections \ref{sec:transvtorsion} and \ref{sec:cu-blender3bpprooflemmas}. Instead of considering vertical and horizontal submanifolds as the ones in \eqref{eq:verticalstrip} (resp. \eqref{eq:horizontalstrip})  we now consider submanifolds 
\begin{equation}\label{eq:horverticalforcu}
    \begin{split}
    \widetilde\Lambda_v=&\{\widetilde F_v(\tilde\tau,  \tilde\xi)=(v_1(\tilde\tau,  \tilde\xi),\tilde\tau,\tilde\xi, v_2(\tilde\tau,  \tilde\xi))\colon \gamma_d(\tilde\tau)\leq \tilde\xi \leq \gamma_u(\tilde\tau),\tilde\tau\in[0,\delta]\}\subset \widetilde Q_{\delta}\\
   (\text{resp. } \widetilde \Lambda_h=&\{\widetilde F_h(\tilde p, \tilde\eta)=(\tilde p,h_1(\tilde p, \tilde\eta), h_2(\tilde p,\tilde\eta),\tilde\eta)\colon \gamma_l(\tilde p)\leq \tilde\eta\leq \gamma_r(\tilde p),\tilde p\in[0,\delta]\}\subset \widetilde Q_{\delta})
    \end{split}
\end{equation}
 and refer to them as $\widetilde\Psi$-vertical (resp. $\widetilde\Psi$-horizontal).
Each of the images $\widetilde\Psi^{(N)}(\Lambda_v)$ (resp. $(\widetilde\Psi^{(N)})^{-1}(\Lambda_h)$)
contain a countable collection of $\widetilde\Psi$-vertical (resp. $\widetilde\Psi$-horizontal) submanifolds $\widetilde \Lambda_{v,N}^{(n)}$  with parametrizations of the form \eqref{eq:horverticalforcu}  satisfying\footnote{Here $\breve\tau,\breve\xi$ are the corresponding solutions to the system of equations $T=\pi_{\tilde\tau} (\widetilde\Psi^{(N)}\circ F_v)(\tau,\xi)$ and $\upsilon=\pi_{\tilde\xi} (\widetilde\Psi^{(N)}\circ F_v)(\tau,\xi)$.}:

\begin{itemize}
    \item Asymptotics for central dynamics: uniformly in $n\in \mathbb N$ (instead of the expressions \eqref{eq:approxcenterdyngraphtransf}) 
    \begin{equation}\label{eq:approxcenterdyngraphtransfcublender}
        \begin{split}
        \pi_{(\tilde\xi,\tilde\eta)} \widetilde F^{(n)}_{v,N}(T, \upsilon)= &\psi_R\circ\mathcal F^{-1}_N\circ\psi_R( \breve\xi(T,\upsilon),v_2(\breve\tau(T,\upsilon),\breve\xi(T,\upsilon)))+O(\chi^2,\delta)\\
         (\text{resp. } \pi_{(\tilde\xi,\tilde\eta)} \widetilde F^{(n)}_{h,N}(P, \varrho)= &\psi_R\circ\mathcal F_N\circ\psi_R( h_2(\breve p(P, \varrho),\breve\eta(P, \varrho)),\breve\eta(P, \varrho))+O(\chi^2,\delta)).
        \end{split}
        \end{equation}
       \end{itemize}
where $\mathcal F_N$ is the map introduced in \eqref{def:Fn} and $\psi_R$ is the involution in \eqref{def:involution}.   
To see this it is enough to note that, proceeding as in Section \ref{sec:casebeta0},
\[
\mathtt S_1\circ\mathtt S_0^N=\psi_R\circ(\mathtt S_0^N\circ\mathtt S_1 )^{-1}\circ\psi_R+\mathcal{\widetilde E}(\varphi,J)
\]
where  $\mathtt S_i:\mathbb A\to\mathbb A$ are the maps in Theorem  \ref{thm:transvtorsionScattmaps}) and 
 $\mathcal {\widetilde E}$ contain only non-linear errors and  satisfy the same estimates as  the function $\mathcal E$ in Lemma \ref{lem:c1control}). Roughly speaking, the behavior of the center for the map $\widetilde \Psi^{(N)}$ is conjugated by $\psi_R$ to the inverse of the map which gives the center dynamics of the map $\Psi^{(N)}$. The existence of cone fields for the maps $\{\widetilde\Psi^{(N)}\}_{N\in\widehat{\mathcal N}}$ can be deduced in the very same way as for the family of maps  $\{\Psi^{(N)}\}_{N\in\widehat{\mathcal N}}$. The proof of the following result now follows after verbatim repetition of the argument in Section \ref{sec:cu-blender3bpprooflemmas}.

\begin{rem}
    Below, a $cu$-strip $\Delta\subset \mathcal Q_\delta^0$ is a $\widetilde \Psi$-vertical submanifold which is entirely contained in $\widetilde Q_\delta\subset \mathcal Q_\delta^0$.
\end{rem}

\begin{prop}\label{prop:cu-blender}
    There exists a hyperbolic periodic point $\widetilde P\in \widetilde Q_\delta\subset \mathcal Q_\delta^0$  such that the pair $(\widetilde P,\widetilde Q_\delta)$ is a $cu$-blender for the map $\Psi$ in \eqref{eq:globalmapfull}: any $cu$-strip $\Delta\subset \widetilde Q_\delta$ intersects $W^s(\widetilde P)$ robustly. Moreover, in the local coordinate system $(\tilde p,\tilde\tau,\tilde\xi,\tilde\eta)$ defined on $\widetilde Q_\delta$
    \[
     W^s_{\loc}(\widetilde P)=\{(\tilde p,f_{1}(\tilde p,\tilde\eta),f_{2}(\tilde p,\tilde\eta),\tilde\eta)\colon \tilde p\in[0,\delta],\ \tilde\eta\in[-1,1] \}
     \]
     for some $C^1$ functions satisfying
     \[
     \partial_{\tilde p} f_{1},\partial_{\tilde \eta} f_{1}=O(\delta)\qquad\qquad f_{2}= c+O_{C^1}(\chi)
     \]
     for some $c\in(-3/4,3/4)$.
\end{prop}

\subsection*{Homoclinically related blenders yield a symplectic blender} 

It is a straightforward consequence of the parametrizations of $W^u_{\loc}(\widehat P)$ (see Proposition \ref{prop:finalcsblender}) $W^s_{\loc}(\widetilde P)$ (see Proposition \ref{prop:cu-blender}) and  the  expression \eqref{eq:transchart} for the transition map between coordinate charts that the points $\widehat P$ in Proposition \ref{prop:finalcsblender} and $\widetilde P$ in Proposition \ref{prop:cu-blender} are homoclinically related. Thus, the pairs $(\widehat P,\widehat Q_\delta)$ and $(\widetilde P,\widetilde Q_\delta)$ form a symplectic blender for the map $\Psi$. 

\begin{thm}\label{thm:sympblender}
    Let $\widehat P\in \widehat Q_\delta\subset \mathcal Q_\delta^0$ be the hyperbolic periodic point in Proposition \ref{prop:finalcsblender} and let $\widetilde P\in \widetilde Q_\delta\subset\mathcal Q_\delta^0$ be the hyperbolic periodic point in Proposition \ref{prop:cu-blender}. Then, the pairs  $(\widehat P,\widehat Q_\delta)$  and $(\widetilde P,\widetilde Q_\delta)$ form a symplectic blender for the map $\Psi$ in \eqref{eq:globalmapfull}.
\end{thm}
This completes the proof of Theorem \ref{thm:mainblender} .

\subsection{Existence of orbits accumulating $\mathcal E_\infty$: proof of Theorem \ref{thm:Main3bp}}\label{sec:mainproof}

Let 
\[
O_\delta=\widehat Q_\delta \cap \widetilde Q_\delta \subset \mathcal Q_\delta^0
\]
We define  $h$-sets $\mathcal H\subset  O_\delta  $ and $v$-sets $\mathcal V\subset  O_\delta$ according to  Definition \ref{defn:vsetshsets} (but adapted to the local coordinate system $(\hat p,\hat \tau,\hat \xi,\hat\eta)$ introduced in \eqref{eq:transchart}).


\begin{prop}\label{prop:fullycross}
Consider the   map $\Psi$ in \eqref{eq:globalmapfull} and let $\mathcal H\subset  O_\delta $ be a $h$-set and $\mathcal V\subset  O_\delta  $ be a $v$-set. Then, there exists $n_f\in\mathbb N$ (resp. $n_b\in\mathbb N$) and a $v$-set $\mathcal V_{n_f}$ (resp. a $h$-set $\mathcal H_{nb}$) which is a connected component of $\Psi^{n_f}(\mathcal V)\cap  O_\delta$  (resp. $\Psi^{-n_b}(\mathcal H)\cap  O_\delta$) such that $\mathcal V_{n_f}$ (resp. $\mathcal H_{n_b}$) fully crosses $\mathcal H$ (resp. $\mathcal V$).
\end{prop}
\begin{proof}
   Think of $W^u_{\loc}(\widehat P)$ (resp. $W^s_{\loc}(\widetilde P)$)  as a (zero-volume) $v$-set (resp. $h$-set). Then, since $(\widehat P,  \widehat Q_\delta)$ is a $cs$-blender (resp. $(\widetilde P,\widetilde Q_\delta)$ is a $cu$-blender) we claim that for any $h$-set (resp. $v$-set) the local manifold $W^{u}_{\loc}(  \widehat P)$ (resp. $W^s_{\loc}(\widetilde P)$) fully crosses $\mathcal H$ (resp. $\mathcal V$). Since $\widehat P$ and $ \widetilde P$ are homoclinically related we reach the desired conclusion by a direct application of the lambda lemma.

   In order to verify the claim one needs to proceed as follows (we only analyze the case of $h$-sets since the result for $v$-sets follows from the same analysis). Recall that a $h$-set is a union of horizontal submanifolds $\bigcup_{(r,s)\in[0,1]} \Delta_{r,s}$ (see the proof of Proposition \ref{prop:imagesofhsetsvsets}). The claim follows if it is possible to find a single $M\in\mathbb N$ and a single sequence $\omega\in \widehat{\mathcal N}^M$ such that the following holds:  
if we denote by $\mathcal V_\omega$ any of the countable family of vertical $v$-sets contained in $\Psi_{\omega} (\widehat Q_\delta)$ then  for any pair $(r,s)\in[0,1]$ 
\[
W^u_{\loc}(\widehat P)\cap \Psi_{\omega}^{-1}(\Delta_{r,s}\cap\mathcal V_\omega)\neq\emptyset.
\]
However, the existence of such sequence $\omega$ is a direct consequence of Proposition \ref{prop:blenderforpairs}.
   \end{proof} 


Define the countable family of closed  rectangles which, in local coordinates $(p,\tau,\xi,\eta)$ are given by
\[
\{\mathcal U_k\}_{k\in\mathbb N}=\{ [0,r]^2\times[x-r,x+r]\times[y-r,y+r]\colon (x,y)\in \mathbb Q^2\cap [-1,1]^2, r\in\mathbb Q\cap [0,1]\}\subset  O_\delta,
\]
where $k$ runs over the countable set $(x,y,r)\in (\mathbb Q\cap [-1,1])^3$.

We now show the following.
\begin{prop}\label{prop:twosideshadowing}
    There exists $z\in O_\delta$ such that, for any $k\in\mathbb N$
    \[
  \mathcal U_k\cap \bigcup_{n\in\mathbb N}\Psi^{n}(z) \neq\emptyset\qquad\qquad\text{and}\qquad\qquad   \mathcal U_k\cap \bigcup_{n\in\mathbb N}\Psi^{-n}(z) \neq\emptyset.
    \]
\end{prop}
\begin{proof}
Observe first that, for any $k\in\mathbb N$ the image $\Psi^{(N)} (\mathcal U_k)$ (resp. the preimage $(\Psi^{(N)})^{-1}(\mathcal U_k)$)  contains a countable collection of $v$-sets (resp. $h$-sets). Indeed, this is a consequence of Remark \ref{rem:shortmanifolds} after Proposition \ref{prop:graphtransform}. Pick any of them and denote it by $\mathcal V_k$ (resp. $\mathcal H_k$). Then, by Proposition \ref{prop:fullycross} we know that for each $k\in\mathbb N$:
   \begin{itemize}
       \item there exists $n_f\in\mathbb N$ such that $\Psi^{n_f}(\mathcal V_{k+1})$ fully-crosses $\mathcal H_{k}$ and,
       \item there exists $n_b\in\mathbb N$ such that $\Psi^{-n_b}(\mathcal H_{k+1})$ fully-crosses $\mathcal 
       V_{k}$.
   \end{itemize}
   Hence, given $m\in\mathbb N$ the compact sets
   \[
   \mathtt H_m:=\bigcap_{k\leq m}\left(\bigcup_{n\in\mathbb N}\Psi^n (\mathcal U_k) \cap  O_\delta \right)\qquad\qquad \mathtt V_m:=\bigcap_{k\leq m}\left(\bigcup_{n\in\mathbb N}\Psi^{-n} (\mathcal U_k) \cap O_\delta \right)
   \]
   are non-empty and contain, respectively, a $h$-set and a $v$-set. By direct application of Proposition \ref{prop:fullycross} the compact set 
   \[
   \mathtt R_m=\bigcap_{k\leq m}\left( \left(\bigcup_{n\in\mathbb N}\Psi^n (\mathcal U_k) \cap O_\delta \right)\cap \left(\bigcup_{n\in\mathbb N}\Psi^{-n} (\mathcal U_k) \cap O_\delta \right) \right)
   \]
   is non-empty as well. Moreover, by construction $\mathtt R_{m+1}\subset \mathtt R_{m}$ so (recall that the countable intersection of compact non-empty nested sets is non-empty)
   \[
   \mathtt R_\infty:=\bigcap_{m\in\mathbb N} \mathtt R_m\neq \emptyset.
   \]
   Finally, we notice that for $z\in \mathtt R_\infty$ and any $k\in\mathbb N$ there exists $n_1,n_2\in\mathbb N$ such that $\Psi^{n_1}(z)\in \mathcal U_k\neq \emptyset$ and $\Psi^{-n_2}(z)\in \mathcal U_k\neq \emptyset$.
\end{proof}

The proof of Theorem \ref{thm:Main3bp} is complete.

\section{A normally hyperbolic lamination in the restricted problem: proof of Theorem \ref{thm:MainR3bp}}\label{sec:restricted3bpproof}

In this section we construct a weakly invariant normally hyperbolic lamination for the return map to a suitable 4-dimensional section transverse to the flow of \eqref{eq:restricted3bp} in suitable coordinates.
Consider  first polar coordinates 
\[
\phi_{\mathrm{pol}}:(r,\alpha,y,G)\mapsto (q,p)
\]
on $T^*(\mathbb R^2\setminus \Delta)$, where $\Delta$ is the collision set. Then, one can introduce McGehee's partial compactification by the change of coordinates $\phi_{MG}:(x,\alpha,y,G)\mapsto (\frac{2}{x^2},\alpha,y,G)$. On the compactified manifold 
\[
\overline M=M\sqcup M_\infty\qquad\qquad M_\infty=\phi_{\mathrm{pol}}\circ\phi_{MG}(\{0\}\times\mathbb T\times\mathbb R^2)
\]
equipped with local coordinates $(x,\alpha,y,G)\in\mathbb (\mathbb R_+\cup\{0\})\times \mathbb T\times\mathbb R^2$ and the singular symplectic form 
\begin{equation}\label{eq:sympformrestricted}
\varOmega=\frac{4}{x^3}\mathrm dy\wedge\mathrm dx+\mathrm dG\wedge\mathrm d\alpha,
\end{equation}
the Hamiltonian \eqref{eq:restricted3bp} recasts as
\begin{equation}\label{eq:mcgeheerestrictedHam}
\mathcal H(x,\alpha,y,G,t)=\frac{y^2}{2}-\frac{x^2}{2}+\frac{G^2x^4}{8}-V(x,\alpha,t),\qquad\qquad V=U\circ\phi_{\mathrm{pol}}\circ\phi_{MG}-\frac {x^2}{2}.
\end{equation}
The vector field generated by the pair $(\varOmega,\mathcal H)$ reads
\begin{equation}\label{eq:odesrestricted}
    \begin{aligned}
    \dot x=&-\frac{x^3}{4}\partial_y {\mathcal H}=-\frac{x^3}{4} y\qquad\qquad &\dot \alpha=&\partial_G {\mathcal H}=\frac{Gx^4}{4}\qquad\qquad\dot t=1\\
    \dot y=&\frac{x^3}{4}\partial_x {\mathcal H}=-\frac{x^3}{4}(x-\frac{G^2 x^3}{8}-\partial_x V)\qquad\qquad &\dot G=&-\partial_\alpha {\mathcal H}=\partial_\alpha V.
    \end{aligned}
\end{equation}
The following lemma, whose proof boils down to an straightforward computation and the application of Schwarz's lemma, will prove useful.
\begin{lem}\label{lem:expansionpotential}
    Let $V(x,\alpha,t)$ be as in \eqref{eq:mcgeheerestrictedHam} and define $\phi=\alpha-t$. Then, 
    \[
    V(x,\alpha,t;\zeta)=V_0(x,\phi)+\zeta V_1(x,\phi,t;\zeta),
    \]
    with $V_0,V_1=O(x^6)$.
\end{lem}
It follows from the equations \eqref{eq:odesrestricted}, that the cylinder 
\[
\mathcal P_\infty=\{x=y=0,\ (\alpha,t)\in\mathbb T^2,\ G\in\mathbb R\}
\]
is invariant for the flow of \eqref{eq:mcgeheerestrictedHam} and that the flow restricted to $\mathcal P_\infty$ is given by the linear  translation
\[
\phi^s_{\mathcal H}:(\alpha,t,G)\mapsto (\alpha,t+s,G).
\]
The linearized dynamics at $\mathcal P_\infty$ is degenerate (normally-parabolic) but the manifold $\mathcal P_\infty$ admits 4-dimensional local stable and unstable manifolds which we denote by $W^{u,s}_{\loc}(\mathcal P_\infty)$. Moreover, the dependence of the strong stable/unstable leaves on the base point is real-analytic (see for instance \cite{McGeheestablemanifold,BFM20a, BFM20b}). The globalization of these manifolds intersect transversally along ``large homoclinic channels''.

\begin{thm}[Theorem 2.2 in \cite{guardia2023degeneratearnolddiffusionmechanism}]\label{thm:homchannelsrestricted}
  Let $\mu\in(0,1/2)$ and let $\zeta\in(0,1)$ be sufficiently small. Fix any pair 
  \[
  |\log \zeta|\ll G_1< G_2\ll \zeta^{-1/3}.
  \]
  Then, there exists (at least) two different, non-empty, real-analytic, transverse homoclinic manifolds
  \[
 \Gamma_{i}\subset  W^{u}_{\mathrm{loc}}(\mathcal P_\infty)\pitchfork W^{s}_{\mathrm{loc}}(\mathcal P_\infty)\qquad\qquad i=0,1,
  \]
 which satisfy (here the wave map $\Omega^u$ is defined exactly as in \eqref{eq:wavemaps} but for the flow of \eqref{eq:mcgeheerestrictedHam})
 \begin{equation}\label{eq:cutoffcylinder}
 \mathcal P_\infty(G_1,G_2):=\mathcal P_\infty\cap \{G_1\leq G\leq G_2\}\subset \Omega^u(\Gamma_0)\cap \Omega^u(\Gamma_1)
 \end{equation}
\end{thm}
These homoclinic channels allow us to construct two scattering maps as defined in Section \ref{sec:scatteringmaps}. Proceeding as for the 3 body problem in Proposition \ref{prop:firstScattmap}, one can easily see that the scattering maps associated to the homoclinic channels given by Theorem \ref{thm:homchannelsrestricted} are of the form 
\begin{equation}\label{def:scattering:restricted}
    \widetilde S_i:\begin{pmatrix}
        t\\ z
    \end{pmatrix}\to \begin{pmatrix}
        t\\ S_i(z)
    \end{pmatrix}
    \end{equation}
    where $z=(\alpha, G)$ and  $S_i$ are  real-analytic, symplectic maps.
The following result describes the dynamics of these scattering maps.


\begin{thm}[\cite{guardia2023degeneratearnolddiffusionmechanism}]\label{thm:scattmapsrestricted}
  Let $\mu\in(0,1/2)$ and $\zeta\in(0,1)$. Let $\Gamma_i$, $i=0,1$ be the transverse homoclinic channels given by Theorem \ref{thm:homchannelsrestricted} and let $\widetilde S_i$, $i=0,1$ be the corresponding scattering maps. Then:
    \begin{itemize}
        \item There exists an annulus $\mathcal A_\infty$ and a local coordinate system $(\varphi,J)\to  (\alpha, G)=\phi(\varphi,J)\in \mathbb A\to \mathcal  A_\infty$ in which the maps 
        \[
        \mathtt S_i=\phi^{-1}\circ S_i\circ\phi
       \]
       (where $S_i$ are the $(\alpha,J)$ components of the scattering maps in \eqref{def:scattering:restricted}) satisfy the assumptions \textbf{(A0)-(A2)}, for certain $\gamma,\tau,\rho,\sigma\neq 0$ independent of $\zeta$, and $\varepsilon>0$ which satisfies $\varepsilon\to 0$ as $\zeta\to 0$.
       \item There exists $M_*\in\mathbb N$ such that, for any $(\varphi,J)\in \mathbb A$ there exists a natural number $M\leq M_*$ and a finite sequence $\omega\in\{0,1\}^M$ such that 
       \[
       S_{\omega_{M-1}}\circ \cdots \circ S_{\omega_0} (\varphi,J)\in \mathcal A_\infty.
       \]
    \end{itemize}
\end{thm}
Although this result is not stated explicitly as a theorem in \cite{guardia2023degeneratearnolddiffusionmechanism}  it follows plainly from the construction in Section 2.2 of that work. The details are shown in Appendix \ref{sec:appendixrestrictedscattmaps}.


We now reproduce the argument in Section \ref{sec:globalmap} to construct a return map to certain transverse sections which accumulate on the homoclinic manifolds $\Gamma_i$, $i=0,1$  obtained in Theorem \ref{thm:homchannelsrestricted}.  The first step is to show that the normal form Lemma \ref{lem:dynamicscloseinfty} also holds for the vector field \eqref{eq:odesrestricted}. As for Lemma \ref{lem:dynamicscloseinfty}, this is a consequence of the more general result which was obtained in \cite{guardia2022hyperbolicdynamicsoscillatorymotions}.

\begin{thm}[Theorem 5.2 in \cite{guardia2022hyperbolicdynamicsoscillatorymotions}]\label{thm:generalnormalformlemmainfty}
Fix any $k\in\mathbb N$ and let $K\subset \mathbb T\times \mathbb R$ be a compact set. Let $B,C\in\mathbb R$. Let $X$ be any $\mathcal C^\infty$ vector field of the form 
\begin{equation}\label{eq:generalformvectorfield}
\begin{aligned}
\dot x=&-x^3y(1+B(x^2-y^2)+R_1(x,y,z,t))\qquad\qquad &\dot t&=1\\
\dot y=&-x^4(1+(B-C)x^2-By^2+R_2(x,y,z,t))\qquad\qquad &\dot z&=R_3(x,y,z,t)
\end{aligned}
\end{equation}
with $x\mapsto R_i(x,\cdot)$ even for $i=1,2,3$, $R_3=O(x^6)$ and $R_1,R_2=O_2(x^2,y^2)$ and which is defined on $
\mathcal U_\infty=\{(x,y)\in U,\ z=(\varphi,G)\in K,\ t\in\mathbb T\}$ with $U\subset\mathbb R^2$ a sufficiently small open neighborhood of the origin. On $\mathcal U_\infty$ there exists a $C^k$ change of variables $\Phi:(t,\tilde z,q,p)\to (t,z,x,y)$ given by a $O_2(x,y)$ perturbation of a constant linear map $(x,y)=A\binom{q}{p}$ and such that conjugates the vector field $X$ to the vector field (we write $\tilde z=(\tilde\varphi,\tilde J)$)
    \begin{equation}\label{eq:straightenedlocalflow}
    \begin{aligned}
    \dot q=&q((q+p)^3+O_4(q,p))\qquad\qquad& \dot {\tilde z}&=(qp)^k O_4(q,p)\\
    \dot p=&-p((q+p)^3+O_4(q,p))\qquad\qquad& \dot t&=1.
    \end{aligned}
    \end{equation}
\end{thm}
\medskip
A trivial computation shows that the vector field \eqref{eq:odesrestricted} is of the form \eqref{eq:generalformvectorfield} with $B=0$, $C=\frac{G^2}{8}$, $R_1=0$, $R_2=\frac 1x\partial_x V=O_4(x)$ and $R_3=(\frac{Gx^4}{4}, -\partial_\alpha V)^\top=O(x^4,x^6)^\top$. Notice that $R_3=O(x^4)$ so we cannot apply Theorem \ref{thm:generalnormalformlemmainfty} directly. This technical annoyance is bypassed by considering the new variable $\beta=\alpha +yG$. Indeed, an easy computation shows that the component of the vector field in the direction of  $\hat z=(\beta,G)$ is of order $O(x^6$) (see also \cite{OscillatoryEllipticGuardia}).

Hence, fixed any pair $G_1<G_2<\infty$, we can apply Theorem \ref{thm:generalnormalformlemmainfty} to the vector field \eqref{eq:odesrestricted} to reduce it to the normal form  \eqref{eq:straightenedlocalflow} on a neighborhood of $\mathcal P_\infty(G_1,G_2)$. Let now $1\ll G_1<G_2<\infty$ and let $\zeta\in (0,G_2^{-3})$ (see Theorem \ref{thm:homchannelsrestricted}). Denote by 
\[
\mathbb A(G_1,G_2)=\{\varphi\in\mathbb T,\ G_1\leq G\leq G_2\}.
\]
Proceeding in the very same way as in Section \ref{sec:globalmap}, we consider:
\begin{itemize}
    \item Transverse sections $\Sigma^{\out}_a=\{q=a,p>0\}$, $\Sigma^{\inn}_a=\{p=a,q>0\}$. We let $U^{\out}_i,U_i^{\inn}$ be small sets (in $\Sigma^{\out}_{a}$, $\Sigma^{\inn}_{a}$) having the intersections $\Gamma_i\cap \Sigma^{\out}_{a}$,  $\Gamma_i\cap \Sigma^{\out}_{a}$ in their boundaries respectively.
    \item For $\delta$ small enough and $i=0,1$, a local coordinate system on $U^{\out}_i$ given by  
    \[
    \phi_i:(p,\tau,\alpha,G)\in[0,\delta]^2\times \mathbb A (G_1,G_2)\to  \textrm{Im}(\phi_i)=\mathcal Q^i_\delta\subset U_{i}^{\out}
    \]
    and such that $\Gamma_i=\{\tau=0\}$.
    \item A local map $\Phi_{\mathrm{loc}}:U\subset \Sigma^{\inn}_a\to \Sigma^{\out}_a$,
    
    \item Global maps: $\Phi_{i,\mathrm{glob}}:U^{\out}_i\subset \Sigma^{\out}_a\to \Sigma^{\inn}_a$,
    \item Return maps: (wherever they are defined)
    \begin{equation}\label{eq:returnmaprestricted}
    \Psi:\Sigma^{\out}_a\to \Sigma^{\out}_a.
    \end{equation}
    Note that restricting the domain this map gives the family of maps $\Psi_{i\to j}=\Psi_{\mathrm{loc}}\circ\Psi_{i,\mathrm{glob}}:\mathcal Q^{i}_\delta\subset \Sigma^{\out}_a\to \Sigma^{\out}_a$, $i,j=0,1$.
\end{itemize}

\subsection{Existence of a normally hyperbolic lamination}
By construction, the maps $\Psi_{i\to j}$ also satisfy the very same conclusion in Theorem \ref{thm:onestepgraphtransform}. Hence, making use of the second item in Theorem \ref{thm:scattmapsrestricted}, we could, for instance, reproduce the argument in Section \ref{sec:blender3bp} to deduce the analogue statements to Theorems \ref{thm:mainblender} and \ref{thm:Main3bp}.

Instead, we prefer to follow a distinct road and obtain results of slightly different flavor. To that end we notice that, compared to the return maps \eqref{eq:defnglobalmaps} for the 3-body problem, the return maps constructed above for the restricted version are defined on subsets $\mathcal Q^i_\delta\subset \Sigma^{\out}_a$ whose projection onto the center directions $(\varphi,G)$ cover a (arbitrarily large) cylinder. Indeed, given any pair $1\ll G_1<G_2<\infty$, provided $\zeta$ is small enough we can take $(\varphi,G)\in \mathbb A(G_1,G_2)$. 

In our following result we construct a weakly invariant, normally-hyperbolic lamination for the first return map $\Psi:\Sigma^{\out}_a\to \Sigma^{\out}_a$. The leaves of these lamination are compact cylinders which become unbounded as $\zeta\to 0$. More precisely, we let $G_1\gg 1$ be as in Theorem \ref{thm:homchannelsrestricted}, for $\zeta>0$ small enough let $G_2(\zeta)=\zeta^{-1/3}$, and define 
\begin{equation}\label{eq:azeta}
\mathbb A_\zeta=\{(\varphi,G)\colon 2G_1\leq G\leq \frac12 G_2(\zeta),\  \varphi\in\mathbb T\}.
\end{equation}

We prove the following.
\begin{prop}\label{prop:lamination}
    Let $\mu\in (0,1/2)$. Let $\Sigma^{\out}_a$ be the transverse section constructed above and let $\Psi:\Sigma^{\out}_a\to \Sigma^{\out}_a$ be the first return map in \eqref{eq:returnmaprestricted}. Then, for any $\zeta>0$ sufficiently small, provided $\delta$ is sufficiently small,  there exists a subset $\mathcal X\subset (\mathcal Q_\delta^0\cup \mathcal Q_\delta^1)\subset\Sigma^{\out}_a$ such that:
    \begin{itemize}
        \item it is homeomorphic to the product $\mathbb N^\mathbb Z\times \mathbb A_\zeta$. More precisely,  $\mathcal X=\Phi(\mathbb N^\mathbb Z\times \mathbb A_\zeta)$ for a homeomorphism  of the form $\Phi:(\omega,z)\mapsto (p_\omega(z),\tau_\omega(z),z)$ with $p_\omega,\tau_\omega$ differentiable functions which satisfy 
        \[
        \partial_z p_\omega,\partial_z\tau_\omega=O(\delta).
        \]
        For $\omega\in \mathbb N^\mathbb Z$ we refer to  $\mathcal L_\zeta(\omega)=\{(p,\tau)=(p_\omega(z),\tau_\omega(z),z),\ z\in\mathbb A_\zeta\}$ as the leaves of the lamination.
        \item it is weakly invariant for the map $\Psi$ and  the restriction $\Psi|_{\mathcal X}$ (whenever it is defined) is topologically conjugated to the skew-product map 
    \begin{equation}\label{eq:skewproduct}
    \begin{split}
        \mathcal F: \mathbb N^\mathbb Z\times \mathbb A_\zeta&\to Z\times \mathbb A_\zeta\\
        (\omega,z)&\mapsto (\sigma(\omega),\mathcal F_\omega(z)),
    \end{split}
    \end{equation}
  where $\sigma:\mathbb N^\mathbb Z\to \mathbb N^\mathbb Z$ is the full shift and, for any $(\omega,z)\in \mathbb N^\mathbb Z\times \mathbb A_\zeta$,
\begin{equation}\label{eq:centerdynskewprodrestricted}
  \mathcal F_\omega(z)=S_{\mathrm{par}(\omega_0)}(z)+O_{C^1}(\delta).
  \end{equation}
  where $\mathrm{par}(\omega_0)=0$ if $\omega_0\in 2\mathbb N$ and $\mathrm{par}(\omega_0)=1$ otherwise,  and  $S_{i}:\mathbb A_\zeta\to \mathbb A_\zeta$, $i=0,1$ are the scattering maps in Theorem \ref{thm:scattmapsrestricted}.

  \item There exists some $r\in (0,1)$ such that if $\omega,\omega'$ satisfy that $\omega'\in C_n(\omega)$, uniformly for all $z\in\mathbb A$
        \begin{equation}\label{eq:almostlocallyconstantrestricted3bp}
        |\mathcal F(\omega,z)-\mathcal F(\omega',z)|\leq \delta^{rn}.
        \end{equation}
  
    \end{itemize}
\end{prop}

\begin{rem}
    For the case $\zeta=0$ the same result is true with $\mathbb A_{\zeta}$ substituted by $\mathbb A_0=\{(\varphi,G)\colon 2G_1\leq G,\ \varphi\in\mathbb T\}$.
\end{rem}

 \begin{rem} Notice that the lamination constructed above is homeomorphic to the product. This is way more than what we need to establish Theorem \ref{thm:MainR3bp}, which only concerns a lamination in two symbols. However, the construction below does not require any extra effort to handle the infinite symbols case so we present the statement and prove of that result. 
 
 Trivially, from Proposition \ref{prop:lamination}, one can get a lamination on two symbols by restricting to the (weakly) invariant subset $\{0,1\}^\mathbb Z\times\mathbb A_\zeta\subset \mathbb N^\mathbb Z\times\mathbb A_\zeta$.
\end{rem}

\begin{proof}[Proof of Proposition \ref{prop:lamination}]
    The proof is divided in a number of steps.

\noindent\textbf{Step 1:} We first modify the flow as follows. Let $G_1\gg 1$ be as in Theorem \ref{thm:scattmapsrestricted} and, for  $\zeta>0$ small enough let $G_2(\zeta)=\zeta^{-3}$. In view of the result in Lemma \ref{lem:expansionpotential} we introduce the change of variables $\Phi$ given by $\beta=\alpha-t$. In the new set of variables, the dynamics is governed by the Hamiltonian 
\begin{equation}\label{eq:mcgeheerestrictedHam2}
\mathcal H\circ\Phi(x,y,\beta,G,t)-G=\underbrace{\frac{y^2}{2}-\frac{x^2}{2}-G+\frac{G^2x^4}{8}+V_0(x,\beta)}_{\mathcal H_0(x,y,\beta,G)}+\zeta V_1(x,\beta,t),
\end{equation}
that is, a time-periodic perturbation of a system with two degrees of freedom. We let $\psi:\mathbb R\to \mathbb R$ be any $C^\infty$ compactly supported function such that 
\[
\psi(u)=
\begin{cases}
1 & \text{if}\qquad   u\in (2G_1, \frac12 G_2(\zeta)), \\
0 & \text{if}\qquad   u\in \mathbb R\setminus (G_1, G_2(\zeta)).
\end{cases}
\]
Then, we consider the modified Hamiltonian 
\begin{equation}\label{eq:modifiedHam}
\widetilde{\mathcal H}(x,y,\beta,G,t)=\mathcal H_0(x,y,\beta,G)+\zeta \psi(|\mathcal H_0|) V_1(x,\beta,t).
\end{equation}
By construction, the Hamiltonian $\widetilde{\mathcal H}$ leaves invariant the submanifolds $\{{\mathcal H}_0=-G_2(\zeta)\}$ and $\{{\mathcal H}_0=-G_1\}$. In particular, the region 
\[
A(\zeta)=\{(x,y,\beta,G)\in(\mathbb R_+\cup\{0\})\times\mathbb R\times\mathbb T\times\mathbb R,\ -G_2(\zeta)\leq {\mathcal H_0}\leq -G_1\}
\]
is invariant for the flow of $\widetilde{\mathcal H}$. Moreover, the vector field associated to $\widetilde {\mathcal H}$ (with respect to the symplectic form \eqref{eq:sympformrestricted}) verifies the hypotheses in Theorem \ref{thm:generalnormalformlemmainfty}. 

\noindent\textbf{Step 2:} We then define the subsets
\[
\Sigma^{\out}_{a,\zeta}=\Sigma_a^{\out}\cap A(\zeta).
\]
Proceeding in the exact same way as above, we define the modified return maps (wherever they are defined)
\[
\widetilde \Psi:\Sigma^{\out}_{a,\zeta}\to \Sigma^{\out}_{a,\zeta}
\]
and  $\widetilde\Psi_{i\to j}:\mathcal Q^{i}_{\delta,\zeta}\subset \Sigma^{\out}_{a,\zeta}\to \Sigma^{\out}_{a,\zeta}$ where 
\[
{\mathcal Q}^i_{\delta,\zeta}=\mathcal Q^i_\delta\cap A(\zeta).
\]
In local coordinates $(p,\tau,\beta,G)$ we let  $\mathcal G_i(p,\tau,\beta)$, $i=1,2$ be the solutions to the implicit equations
\[
\mathcal H_0\circ\phi_i(p,\tau,\beta,G)=G_1\qquad\qquad \mathcal H_0\circ\phi_i(p,\tau,\beta,G)=G_2(\zeta).
\]
with respect to $G$.

By construction, the boundaries
\[
\Sigma^{\out,+}_{a,\zeta}=\Sigma^{\out}_{a}\cap \{G=\mathcal G_2(p,\tau,\beta)\}
\qquad\qquad \Sigma^{\out,-}_{a,\zeta}=\Sigma^{\out}_{a}\cap \{G=\mathcal G_1(p,\tau,\beta)\}
\]
are invariant under the map $\widetilde \Psi$ (and hence, under the maps $\widetilde \Psi_{i\to j}$).

\noindent\textbf{Step 3:} The set $\mathcal X$ is constructed as a locally maximal invariant set for $\widetilde \Psi$. The details are as follows. We start by noticing that Theorem \ref{thm:onestepgraphtransform} holds for the maps $\widetilde \Psi_{i\to j}$, $i,j=0,1$. During the remaining part of the proof we will exploit this fact in several occasions. We will abuse notation and refer to Theorem \ref{thm:onestepgraphtransform} as if it refers to the maps $\widetilde \Psi_{i\to j}$.

We say that a subset $V\subset \mathcal Q^i_{\delta,\zeta}$ is a \textit{vertical} subset if, in local coordinates $(p,\tau,\beta,G)$, it admits a parametrization of the form 
\[
V=\{(p,\tau,\beta,G):  v_1(\tau,\beta,G)\leq p\leq v_2(\tau,\beta,G),\ \mathcal G_1(p,\tau,\beta)\leq G\leq \mathcal G_2(p,\tau,\beta),\ \beta\in\mathbb T,\ \tau\in[0,\delta]\}
\]
for some $C^1$ regular  functions $v_1$, $v_2$, $\mathcal G_1$, $\mathcal G_2$,  satisfying (write $z=(\beta,G)$)
\[
\partial_z v_i=O(\delta) \qquad\qquad i=1,2.
\]
Analogously $H\subset \mathcal Q^i_{\delta,\zeta}$ is a \textit{horizontal} subset if it admits a parametrization of the form 
\[
H=\{(p,\tau,\beta,G):  h_1(p,\beta,G)\leq \tau\leq h_2(p,\beta,G),\ \mathcal G_1(p,\tau,\beta)\leq G\leq \mathcal G_2(p,\tau,\beta),\ \beta\in\mathbb T,\ p \in[0,\delta]\}
\]
for some $C^1$ regular  functions $h_1$, $h_2$, $\mathcal G_1$, $\mathcal G_2$  satisfying (write $z=(\beta,G)$)
\[
\partial_z h_i=O(\delta),\qquad\qquad i=1,2.
\]
Notice in particular that $\mathcal Q_{\delta,\zeta}^{i}$ are, at the same time, vertical and horizontal subsets. 

It follows from application of Theorem \ref{thm:onestepgraphtransform}  that, for any pair $i,j=0,1$, and any vertical (resp. horizontal)
subset $V$ (resp. $H$) the set $\widetilde\Psi^{-1}_{i\to j} (V)\cap\mathcal Q_{\delta,\zeta}^{i}$ (resp. $\widetilde\Psi_{i\to j} (H)\cap \mathcal Q_{\delta,\zeta}^{j}$) contains infinitely many vertical subsets $V^{(n)}$ (resp. infinitely many horizontal subsets $H^{(n)}$). We let 
\[
\mathcal X=\bigcap_{k\in\mathbb Z} \widetilde\Psi^k (\mathcal Q_{\delta,\zeta}^0\cup \mathcal Q_{\delta,\zeta}^1)\qquad\qquad \mathcal Q_{\delta,\zeta}=Q_{\delta,\zeta}^0\cup \mathcal Q_{\delta,\zeta}^1.
\]

\noindent\textbf{Step 4:} We now introduce a symbolic coding on $\mathcal X$ as follows.  We denote by $\{V^{(n)}_{i,j}\}_{n\in\mathbb N}$ the collection of vertical subsets contained in $\widetilde\Psi_{i\to j}^{-1} (\mathcal Q_{\delta,\zeta}^j)\cap \mathcal Q_{\delta,\zeta}^i$ and let $V^{(n)}_{i}=V^{(n)}_{i,1}\cup V^{(n)}_{i,2}$. By construction $V^{(n)}_{i}\subset \mathcal Q_{\delta,\zeta}^{i}$. Then, to any $x\in\mathcal X$ we can associate a doubly infinite sequence $\omega\in \mathbb N^{\mathbb Z}$ by the rule 
\[
\omega_k=2n-1 \Longleftrightarrow \widetilde\Psi^{-k}(x)\in V_{1}^{(n)}\qquad\qquad\text{and}\qquad\qquad\omega_k=2n \Longleftrightarrow \widetilde\Psi^{-k}(x)\in V_{0}^{(n)}
\]
Notice that, by construction, for any $x\in\mathcal X$ 
\[
\widetilde\Psi^{-k}(\widetilde\Psi^{-1}(x))\in V^{([\omega_k/2])}_{\mathrm{par}(\omega_k)}, 
\]
where $[\cdot]$ denotes the integer part.

 Given any $\omega\in \mathbb N^\mathbb Z$ we let 
\[
L(\omega)=\bigcap_{k\in\mathbb Z} \widetilde\Psi^k (V^{([\omega_k/2])}_{
\mathrm{par}(\omega_k)})
\]
and observe that, by construction, there exist differentiable functions $\tau_\omega(z),p_\omega(z)$ which satisfy 
\begin{equation}\label{eq:smoothlaminae}
\partial_z p_\omega,\partial_z \tau_\omega=O(\delta)
\end{equation}
such that 
\[
L(\omega)=\{p=p_\omega(\beta,G),\ \tau=\tau_\omega(\beta,G)\colon \mathcal G_1(p_\omega(\beta,G),\tau_\omega(\beta,G),\beta)\leq G\leq \mathcal G_2(p_\omega(\beta,G),\tau_\omega(\beta,G),\beta),\ \beta\in\mathbb T\}.
\]
\noindent\textbf{Step 5:} We let 
\[
\mathcal F_z(\omega,z)=\pi_{z} \widetilde\Psi (p_\omega(\beta,G),\tau_\omega(\beta,G),\beta,G).
\]
Clearly, the map $\widetilde\Psi|_{\mathcal X} $ is topologically conjugated to a skew-product  of the form
\begin{align*}
        \mathcal F: \mathbb N^\mathbb Z\times \mathbb A_\zeta&\to \mathbb N^\mathbb Z\times \mathbb A_\zeta\\
        (w,z)&\mapsto (\sigma(\omega),\mathcal F_z(\omega,z)).
\end{align*}
We now show that $\mathcal F_z(\omega,z)=S_{\mathrm{par}(\omega)}(z)+O_{C^1}(\delta))$. At the $C^0$ level this is a plain consequence of the asymptotic expansions \eqref{eq:onestepcenterdyn} in Theorem \ref{thm:onestepgraphtransform} (notice that one should substitute the maps $\mathtt S_i$ in the original statement by the scattering maps $\widetilde S_i$ associated to the modified Hamiltonian \eqref{eq:modifiedHam}). At the $C^1$ level this follows from the chain rule, the estimates \eqref{eq:smoothlaminae} and the estimates \eqref{eq:differentialonestep},\eqref{eq:differentialonestepinverse}  in Theorem \ref{thm:onestepgraphtransform}.

Finally, Normal hyperbolicity follows also from the estimates \eqref{eq:differentialonestep},\eqref{eq:differentialonestepinverse}  in Theorem \ref{thm:onestepgraphtransform},
%
  and the estimate \eqref{eq:almostlocallyconstantrestricted3bp} follows from the construction of the symbolic coding above.
\end{proof}

Before completing the proof of Theorem \ref{thm:MainR3bp} we need to address a few technical points. First, the lamination $\mathcal X$ is only weakly invariant for the map $\Psi$ in \eqref{eq:returnmaprestricted} since ,
for some points $(\omega,z)\in \mathbb N^\mathbb Z\times \mathbb A_\zeta$, there may exist $N\in\mathbb N$ such that $\mathcal F^N(\omega,z)\in \mathbb N^\mathbb Z\times (\mathbb T\times\mathbb R\setminus \mathbb A_{\zeta})$, i.e. the orbit may leave along the center directions. 

The lamination $\mathcal X$ accumulates on the homoclinic channels $\Gamma_i$ in Theorem \ref{thm:homchannelsrestricted}, i.e. for $\omega\in Z$ 
\[
\mathrm{dist}(\mathcal L_\zeta(\omega), \Gamma_i)\to 0\qquad\qquad\text{ as } \omega_0\to \infty.
\]
This is the main reason why we had to give a proof of Proposition \ref{prop:lamination} instead of appealing to classical persistence results for normally hyperbolic laminations  and why, a priori, we can only guarantee that the leaves are $C^1$. We indeed believe that the leaves enjoy much better regularity ($C^\infty$ in particular) but proving such a result would require a  refined version of Lemma \ref{lem:dynamicscloseinfty}. For our purposes $C^1$ regularity is enough. 

\begin{rem} 
If we restrict our attention to a subset of $\mathcal X$ which stays at finite distance from the homoclinic channels $\Gamma_i$ we can easily construct $C^r$ laminations (provided $\zeta$ is small enough). The construction goes as follows. Recall that (trivially) $\mathcal H_0$ in \eqref{eq:mcgeheerestrictedHam2} is a conserved quantity when $\zeta=0$. Denote by $\mathcal G(p,\tau,\beta;E)$ the solution to the equation $\{\mathcal H_0=-E\}$ for any $E\in[G_1,\infty)$, $(p,\tau)\in[0,\delta]^2$ and $\beta\in\mathbb T$. The main observation is that for $\zeta=0$ the corresponding map $\Psi_0:\Sigma^a_{\out}\to \Sigma^a_{\out}$ admits a real-analytic normally hyperbolic lamination $\mathcal X_0$ on which the induced dynamics is given by an integrable twist map (by integrable we mean that the sections $\{\mathcal H_0=\mathrm{const}\}$ are invariant). Since for $(p,\tau)\in[0,\delta]$ we have that $\mathcal G=G+O(\delta)$ we conclude that, for any $\omega\in \mathbb N^\mathbb Z$ the map $z\mapsto \mathcal F_{0,z}(\omega,z)$  is real-analytic and 
\[
\mathcal F_{0,z}(\omega,z)=S_{0,\mathrm{par}(\omega_0)}(z)+O_{C^1}(\delta),
\]
where $S_{0,i}$, $i=0,1$ are the scattering maps in Theorem \ref{thm:homchannelsrestricted} for the case $\zeta=0$. 

If we now fix any $M\in\mathbb N$, consider the subset $\mathcal X_{0,M}=\Phi(\{1,\dots,M\}^\mathbb N\times(\mathbb T\times[G_1,\infty)))$ and let $K\subset \Sigma^a_{\out}$ be a sufficiently small compact neighborhood of this set, the map $\Psi:K\subset \Sigma^a_{\out}\to \Sigma^a_{\out}$ for $\zeta>0$ but small enough, is given by a real-analytic $O(\zeta)$-perturbation of the map $\Psi_0$. The persistence of the lamination $\mathcal X_{0,M}$ for $\zeta>0$ is now a consequence of standard results (see, for instance, \cite{MR501173}). Moreover, having fixed $r\in\mathbb N$, for $\zeta>0$ small enough (depending on $r$), the leaves of the lamination are $C^r$ (this is a consequence of the fact that when $\zeta=0$  the asymptotic rate of contraction/expansion along directions tangent to the lamination is at most polynomial while the normal behavior is hyperbolic). Hence, at any $\omega\in \{1,\dots,M\}^\mathbb N$ the map $z\mapsto \mathcal F_z(\omega,z)$ is a $O_{C^r}(\zeta)$ perturbation of the real-analytic integrable twist map $z\mapsto \mathcal F_{0,z}(\omega,z)$.

\end{rem}

\subsection{Proof of Theorem \ref{thm:MainR3bp}}\label{sec:proofrestrictedmain}

We now give the proof of Theorem \ref{thm:MainR3bp}, which follows from a minor modification of the proof of Theorem \ref{thm:skewproduct} given in Section \ref{sec:skewproduct} (see, in particular, Section \ref{sec:proofskewprod}). Consider the subset $\widehat{\mathcal X}\subset\mathcal X$ defined as $\widehat{\mathcal X}=\Phi:(\{0,1\}^\mathbb Z\times \widehat {\mathbb A}_\zeta)$ where 
\[
\widehat{\mathbb A}_\zeta=\{(\varphi,G)\colon 4G_1\leq G\leq \frac14 G_2(\zeta),\  \varphi\in\mathbb T\}.
\]
Given $\omega\in\{0,1\}^\mathbb Z$ let  $\mathcal F_\omega:\mathbb A_\zeta\to \mathbb A_\zeta$ be as in \eqref{eq:centerdynskewprodrestricted}. We will prove that, given $N\in\mathbb N$, provided $\zeta,\delta>0$ are small enough,  for any $B,B'\in\widehat{\mathbb A}$ and any $\omega,\omega'\in \{0,1\}^\mathbb Z$ there exists $M\in\mathbb N$ such that (here $C_N(\omega)$ is the $N$-cylinder around $\omega$, see \eqref{eq:Ncylinder})
\[
\mathcal F^M(C_N(\omega),B)\cap (C_N(\omega'),B')\neq\emptyset.
\]
The existence of orbits visiting any element of a given countable covering follows by a standard Baire category argument.

We proceed as follows. Fix any $N\in\mathbb N$ and let $\zeta,\delta>0$ be sufficiently small so that, for any $\omega\in\{0,1\}^\mathbb Z$ we have that 
\[
\mathcal F^N_\omega(\widehat {\mathbb A}_\zeta)\subset \mathbb A_\zeta,\qquad\qquad (\mathcal F^N_{\sigma^{-N}(\omega)})^{-1}(\widehat {\mathbb A}_\zeta)\subset \mathbb A_\zeta,
\]
with $\mathbb A_\zeta$ as in \eqref{eq:azeta}. Given $\omega,\omega'\in\{0,1\}^\mathbb Z$  and $B,B'\in \widehat{\mathbb A}_\zeta$ we let 
\[
\tilde B'= \mathcal F_{\omega'}^N(B'),\qquad\qquad \tilde B=(\mathcal F_{\sigma^{-N}(\omega)}^N)^{-1}(B).
\]
Observe that, by direct application of Lemma \ref{lem:compositionskewproducts} (which we can apply in virtue of the estimate \eqref{eq:almostlocallyconstantrestricted3bp}), for any $\tilde\omega$ with $\tilde\omega_k=\omega_k$ for $|k|\leq N$ and for any $\tilde\omega'$ with $\tilde\omega'_k=\omega_k'$ for $|k|\leq N$ we have that 
\[
\mathrm{dist}( \mathcal F_{\tilde \omega'}^N(B'),\tilde B')\lesssim \delta,\qquad\qquad \mathrm{dist}( (\mathcal F^{N}_{\sigma^{-N}(\tilde\omega)})^{-1}(B),\tilde B)\lesssim  \delta.
\]
Let $\mathcal A_\infty\subset\mathbb A_\zeta$ be the annulus in Theorem \ref{thm:scattmapsrestricted}. By the second part of Theorem \ref{thm:scattmapsrestricted},  there exists $M_*<\infty$ (uniform in $B,B'$),  a natural number $M_{f_0}<M_*$ and $\omega^{f_0}\in \{0,1\}^{M_{f_0}}$  such that $S_{\omega^{f_0}}(B')\cap  \mathcal A_\infty\neq \emptyset$ and also $M_{b_0}<M_*$ and $\omega^{b_0}\in \{0,1\}^{M_{b_0}}$  such that $S^{-1}_{\omega^{b_0}}(B)\cap \mathcal A_\infty\neq \emptyset$. Hence,  provided $\delta$ is chosen sufficiently small (but uniformly in $B,B'$), for any $\tilde\omega'\in\{0,1\}^\mathbb Z$ with $\tilde\omega'_k=\omega'_k$ for $|k|\leq N$ and $\tilde\omega'_k=\omega^{f_0}_k$ for $k\in\{-M_{f_0}-N,\dots, -N-1\}$ satisfies that 
\[
\widetilde B'(\tilde\omega')=\mathcal F^{N+M_{f_0}}_{\tilde\omega'}(B')\cap \mathcal A_\infty\neq \emptyset.
\]
Analogously, for any $\tilde\omega$ with $\tilde\omega$ with $\tilde\omega_k=\omega_k$ for $|k|\leq N$ and $\tilde\omega_k=\omega^{b_0}_k$ for $k\in\{N+1,\dots,N+M_{b_0}\}$ satisfies that 
\[
\widetilde B(\tilde\omega)=(\mathcal F^{N+M_{b_0}}_{\sigma^{-N-M_{b_0}}(\tilde\omega)})^{-1}(B)\cap \mathcal A_\infty\neq \emptyset.
\]
Since on $\mathcal A_\infty$ the scattering maps satisfy the assumptions of Theorem \ref{thm:skewproduct}, proceeding as in Section \ref{sec:proofskewprod} we can find $M_{f},M_b\in \mathbb N$ and  $\omega^{f}\in\{0,1\}^{M_{f}}$, $\omega^b\in\{0,1\}^{M_b}$ such that for any:  
\begin{itemize}
    \item $\tilde\omega'\in\{0,1\}^\mathbb Z$ with $\tilde\omega'_k=\omega'_k$ for $|k|\leq N$ and $\tilde\omega'_k=\omega^{f_0}_k$ for $k\in\{-M_{f_0}-N,\dots, -N-1\}$ and $\tilde\omega_k'=\omega_k^{f}$ for $k\in\{-M_{f}-M_{f_0}-N,\dots,-M_f-N-1\}$ 
    
    \item $\tilde\omega\in\{0,1\}^\mathbb Z$ with $\tilde\omega_k=\omega_k$ for $|k|\leq N$ and $\tilde\omega_k=\omega^{b_1}_k$ for $k\in\{-N+1,\dots,N+M_{b_0}\}$ and $\tilde\omega_k=\omega_k^{b_0}$ for $k\in\{M_{b_0}+1,\dots,M_{b_0}+M_{b}\}$ 
\end{itemize}  
the open sets
\[
\hat B'(\tilde\omega'):=\mathcal F_{\sigma^{N+M_{f_0}}(\tilde\omega')}^{M_{f}}(\widetilde  B'(\tilde\omega'))\qquad\qquad 
\widehat  B(\tilde\omega):=(\mathcal F_{\sigma^{-N-M_{b_0}-M_{b}}(\tilde\omega)}^{M_{b}})^{-1}(\widetilde  B(\tilde\omega))
\]
satisfy
\[
\widehat B'(\tilde\omega')\cap \widehat B(\tilde\omega)\neq\emptyset.
\]
The proof of Theorem \ref{thm:MainR3bp} is completed by choosing any $\bar\omega\in\{0,1\}^\mathbb Z$ of the form
\[
\bar\omega=(\dots,\underbrace{\tilde\omega_{-N}',\dots,\tilde\omega_N'}_{2N+1},\omega^{b_0},\omega^{b},\omega^{f},\omega^{f_0},\underbrace{\tilde\omega_{-N},\dots,0;,\tilde\omega_N}_{2N+1},\dots).
\]

\appendix

\section{Technical lemmas from Section \ref{sec:IFSlocaltransitive}}\label{sec:appendixtechlemmas}

In this appendix we give the missing proofs from Section \ref{sec:IFSlocaltransitive}.

\subsection*{Proof of Lemma \ref{lem:c1control}}

 Let $n\in\mathbb N$ with $n\tau\varepsilon\leq 1$ and assume that $\varepsilon\leq \tau$. In the following, given some $f:\mathbb A\to \mathbb R$, and some expression $h(\varepsilon,n)$, when we write $f=O(h(\varepsilon,n))$ we mean that there exists a constant $C$ independent of $\varepsilon,\tau$ and $n$ such that 
 \[
 |f|\leq C h(\varepsilon,n).
 \]
The proof follows by induction. Suppose that for some $n\in\mathbb N$ such that $n\tau\varepsilon\leq 1$
\[
(\varphi_n,J_n):=T_0^n\circ T_1(\varphi,J)
\]
is given by
\begin{align*}
\varphi_n=&\varphi+\tilde\beta(J)+\varepsilon \Tphi(\varphi,J)+n(\beta+\tau J+ \tau\varepsilon\Tj(\varphi,J))+O(n\varepsilon^2)\\
J_n=&J+\varepsilon\Tj(\varphi,J)+O(n\varepsilon^3).
\end{align*}
To verify the inductive claim we now notice that for $|J|\leq \varepsilon$ it follows from the inductive hypothesis that 
\[
|J_n|\leq \varepsilon+O(\varepsilon)+O(n\varepsilon^3)=O(\varepsilon).
\]
Therefore,
\[
J_{n+1}=J_n+O(J_n^3)=J+ \varepsilon\Tj(\varphi,J)+O(n\varepsilon^3)+O(J_n^3)=J+\varepsilon \Tj(\varphi,J)+O((n+1)^3\varepsilon).
\]
and 
\begin{align*}
\varphi_{n+1}=&\varphi_n+\beta+\tau J_n+O(J_n^2)\\
=&\varphi+ \tilde\beta(J)+\varepsilon\Tphi(\varphi,J)+n(\beta+\tau J+ \tau \varepsilon \Tj(\varphi,J))+O(n\varepsilon^2)\\
&+\beta+\tau (J+ \varepsilon\Tj(\varphi,J)+O(n\varepsilon^3))+O(\varepsilon^2)\\
=&\varphi+\tilde\beta(J)+\varepsilon \Tphi(\varphi,J)
+(n+1)(\beta+\tau J+\tau \varepsilon \Tj(\varphi,J))+O((n+1)\varepsilon^2).
\end{align*}
Thus,  we conclude that for $n\in\mathbb N$ verifying that $n\tau\varepsilon\leq 1$
\begin{align*}
T_0^n\circ T_1(\varphi,J)=&(\varphi+\tilde\beta(0)+n(\beta+\tau J+\tau \varepsilon  \Tj(\varphi,J))+O(\varepsilon\varphi,\varepsilon), \ J+\varepsilon\Tj(\varphi,J)+O(n\varepsilon^{3}))\\
=&(\varphi+\tilde\beta(0)+n(\beta+\tau J+\tau  \varepsilon \varphi)+O(\varepsilon\varphi,\varepsilon,n \varepsilon\tau \varphi^2), \ J+ \varepsilon \varphi+O(n\varepsilon^3,\varepsilon\varphi^2)).
\end{align*}
We now show how to obtain $C^1$ estimates. Again we prove this by induction. Suppose that 
\[
D(T_0^n\circ T_1)(\varphi,J)=\begin{pmatrix}1+n \tau \varepsilon \partial_\varphi \Tj(\varphi,J)+O(n\varepsilon^2)&\tau n+O(n\varepsilon)\\  \varepsilon \partial_\varphi \Tj(\varphi,J)+(n\varepsilon^3)&1+O(n\varepsilon^2) \end{pmatrix}.
\]
Then, a straightforward computation shows that 
\begin{align*}
D(T_0^{n+1}\circ T_1)(\varphi,J)=&\begin{pmatrix}1+O(\varepsilon^2)&\tau +O(\varepsilon)\\O(\varepsilon^3)&1+O(\varepsilon^2) \end{pmatrix}\begin{pmatrix}1+n\tau \varepsilon  \partial_\varphi \Tj(\varphi,J)+O(n\varepsilon^2)&\tau n+O(n\varepsilon)\\\varepsilon  \partial_\varphi \Tj(\varphi,J)+(n\varepsilon^3)&1+O(n\varepsilon^2) \end{pmatrix}\\
=&\begin{pmatrix}1+(n+1)\tau\varepsilon  \partial_\varphi \Tj(\varphi,J)+O((n+1)\varepsilon^2)&\tau(n+1)+O((n+1)\varepsilon)\\ \varepsilon \partial_\varphi \Tj(\varphi,J)+((n+1)\varepsilon^3)&1+O((n+1)\varepsilon^2) \end{pmatrix}.\qedhere
\end{align*}
\medskip

\subsection*{Proof of Lemma \ref{lem:uniformlemma}}
 Throughout the proof we write $v,w$ instead of $v_N,w_N$.  Let $P,S$ be the matrices associated to the linear maps $\psi_P,\psi_S$ and observe that
\[
P^{-1}=\frac{1}{w-v}\begin{pmatrix}w&-1\\-v&1\end{pmatrix}.
\]
We let $C=PS$ and 
\[
\mathcal A_n=C^{-1} A_n C,\qquad\qquad \mathtt b_n=C^{-1}\bs b_n, \qquad\qquad \widetilde{\mathcal E}_n=\phi^{-1}\circ \mathcal E_n\circ \phi.
\]
After some algebraic manipulations it is not difficult to show that 
\[
\mathcal A_n=\frac{1}{w-v}\begin{pmatrix}  w-v-\varepsilon+n\tau w(\varepsilon+v)& \frac 1\chi\left(-\varepsilon+n\tau w(\varepsilon+w)\right)\\
\chi\left(\varepsilon-n\tau v(\varepsilon+v)\right)&w-v+\varepsilon-n\tau v(\varepsilon+w) \end{pmatrix}.
\]
We now give an asymptotic expression for the diagonal terms and estimate the off-diagonal terms. To that end we notice that, from the definition of $v,w$ (since $\mathcal A_N$ must be diagonal)
\[
\varepsilon-N\tau v(\varepsilon+v)=0\qquad\qquad \varepsilon-N\tau w(\varepsilon+w)=0.
\]
Therefore, if we introduce 
\[
\delta=\frac {N_*}{N},
\]
for any $n\in\{N,\dots,N+N_*\}$,
\[
|\varepsilon-n\tau v(\varepsilon+v)|=| (N-n)\tau v(\varepsilon+v)|\leq \delta N \tau |v||\varepsilon+v|\qquad\qquad |\varepsilon-nw(\varepsilon+w)|\leq \delta N\tau |w||\varepsilon+w|.
\]
Also, from the asymptotics in \eqref{eq:eigenvalues} and the definition of $v,w$
\[
|w-v|\geq \frac 12v,\qquad \qquad v=O\left(\frac{\varepsilon}{\chi}\right),\qquad\qquad |w|=O\left(\frac{\varepsilon}{\chi}\right).
\]
Thus, we conclude that, for the off-diagonal terms
\[
\left| \frac{\varepsilon-n\tau v(\varepsilon+v)}{w-v} \right|,\left| \frac{\varepsilon-n\tau w(\varepsilon+w)}{w-v} \right|\leq 2\delta N\tau\left(\varepsilon+O\left(\frac{\varepsilon}{\chi}\right)\right)\leq 4\delta N\tau \frac{\varepsilon}{\chi}=2\delta \chi.
\]
Proceeding analogously, for the diagonal terms 
\[
\begin{split}
1-\frac{1}{w-v}(\varepsilon-n\tau w(\varepsilon+v))&=1-\sqrt{N\tau\varepsilon}+O(\delta \chi,\chi^2)\\
1+\frac{1}{w-v}(\varepsilon-n\tau v(\varepsilon+w))&=1+\sqrt{N\tau\varepsilon}+O(\delta \chi,\chi^2),
\end{split}
\]
where we have used that 
\[
w-v=2w(1+O(\chi))=-2v(2+O(\chi))=-2\sqrt{\frac{\varepsilon}{N\tau}}(1+O(\chi)).
\]
Hence, 
\[
\mathcal A_n=\begin{pmatrix} 1-\sqrt{N\varepsilon\tau}+ O(\delta \chi ,\chi^2)& O(\delta)\\
O(\delta \chi^2)&1+\sqrt{N\varepsilon\tau}+O(\delta \chi,\chi^2)\end{pmatrix}.
\]
The expression for $\mathtt b_n$ follows from a simple computation. Indeed, 
\[
\mathtt b_n=C^{-1} \bs b_n=  \frac{1}{\kappa}\frac{[n\beta]}{w-v}\begin{pmatrix}w,-v\chi\end{pmatrix}
\]
so, writing $v=-w(1+O(\chi))$, we obtain
\[
\mathtt b_n=\frac{[n\beta]}{2\kappa}(1+O(\chi))\begin{pmatrix}1\\ O(\chi)\end{pmatrix}.
\]
 Notice that, by Theorem \ref{thm:Dirichlet}, since $\beta\in\mathcal B_\alpha$ and $N_*=\frac{1}{5\alpha \kappa\chi}$, the set 
\[
\{[n\beta]\}_{n\in\{N,\dots,N+N_*\}} 
\]
is $\frac{1}{5}\kappa\chi$-dense on $\mathbb T$. The set $\mathcal N\subset \{N,\dots,N+N_*\}$ is defined as the subset for which 
\[
[n\beta]\in[-20\kappa\chi,20\kappa\chi].
\]

\medskip

Finally,  we  check the estimate for $\widetilde{\mathcal E}_n$.  We recall from Lemma \ref{lem:c1control} that 
\[
\mathcal E_n(\varphi,J)=\begin{pmatrix}O(\varepsilon,n\tau\varepsilon \varphi^2)\\ O(n\varepsilon^3,\varepsilon \varphi^2)\end{pmatrix}=\begin{pmatrix}O(\varepsilon,\chi^2 \varphi^2)\\ O(\chi^2\varepsilon^2/\tau,\varepsilon\varphi^2)\end{pmatrix}.
\]
 We now observe that, for any $(\xi,\eta)\in D$,
\[
|\varphi(\xi,\eta)|=O(\kappa /\chi)
\]
so (since we assume that $\varepsilon\ll \tau$ after having fixed $0<\kappa\ll\chi\ll 1$)
\begin{align*}
\mathcal E_n\circ \phi=&\begin{pmatrix}O(\varepsilon,\kappa^2)\\  O(\chi^2\varepsilon^2/\tau, \varepsilon (\kappa/\chi)^2\end{pmatrix}=\begin{pmatrix}O(\kappa^2)\\ O(\varepsilon (\kappa/\chi)^2)\end{pmatrix}.
\end{align*}
 Using the expression for $C^{-1}$ above it is then easy to check that 
\[
\widetilde{\mathcal E}_n=\phi^{-1}\circ \mathcal E_n\circ \phi=\begin{pmatrix}O(\kappa/\chi)\\ O(\kappa)\end{pmatrix}.
\]
We then obtain $C^0$ estimates for the error.
\[
\mathtt E_n(\xi,\eta)=(\mathcal A_n-\mathtt A)\binom{\xi}{\eta}+\widetilde {\mathcal E}_n(\xi,\eta)=\begin{pmatrix} O(\chi^2, \kappa/\chi)\\\ O(\chi^2, \kappa)\end{pmatrix}=\begin{pmatrix} O(\chi^2)\\ O(\chi^2)\end{pmatrix}.
\]
Finally, we obtain $C^1$ estimates. To that end we notice that, 
\[
D \mathtt E_n=(\mathcal A_n-\mathtt A)+ D\widetilde {\mathcal E}_n=(\mathcal A_n-\mathtt A)+C^{-1}\begin{pmatrix}O(\chi\kappa)&O(\chi^2/\tau)\\ O(\varepsilon \kappa/\chi)&O(\chi^2\varepsilon/\tau) \end{pmatrix} C.
\]
where, in the second equality, we have used the estimates for $D\mathcal E_n$ in Lemma \ref{lem:c1control} and the fact that $n\tau\varepsilon= O(\chi^2)$. Then, a tedious but easy computation yields
\[
D \mathtt E_n=\begin{pmatrix} O(\chi^2)&O(\delta)\\ O(\delta\chi^2)&O(\chi^2) \end{pmatrix}+\begin{pmatrix}O(\kappa)&O(\kappa/\chi) \\O(\chi\kappa)& O(\kappa) \end{pmatrix}.\qedhere
\]

\subsection*{Proof of Proposition \ref{prop:welldistributed}}

In Lemma \ref{lem:uniformlemma} we have seen that $\mathcal F_n-\mathtt F_n=O_{C^1}(\chi^2)$. However, a slightly better affine approximation, can be obtained from that proof in the regime where 
\[
0<\kappa\leq \kappa_0(\chi)\qquad\qquad 0<\varepsilon\leq \varepsilon_0(\kappa)\min\{\tau,\alpha\}.
\]
Indeed, it is easy to check that the proof implies the existence of  
\begin{equation}\label{def:eigenvaluesmodified}
\lambda_n=1-\sqrt{n\varepsilon\tau}+O(\chi^2)\qquad\qquad \nu_n=1+\sqrt{n\varepsilon\tau}+O(\chi^2)
\end{equation}
such that, if we define the affine map 
\[
\widetilde{\mathtt F}_n(\xi,\eta)=\begin{pmatrix}
    \lambda_n&0\\
    0&\nu_n
\end{pmatrix}\binom{\xi}{\eta}+\mathtt b_n,
\]
where
\[
\mathtt b_n=c_n \begin{pmatrix}1\\ \chi\end{pmatrix} \qquad\qquad c_n=\frac{[n\beta]}{2\kappa}(1+O(\chi))
\]
and denote by $g_\xi,g_\eta$ the $\xi$ and $\eta$ components of the difference $\mathcal F_n-\widetilde{\mathtt F}_n$, then
\[
(g_\xi(\xi,\eta),g_\eta(\xi,\eta))=\mathcal F_n(\xi,\eta)-\widetilde{\mathtt F}_n(\xi,\eta)=O_{C^1}(\epsilon),
\]
for a quantifier 
\[
\epsilon(\kappa, \varepsilon,\alpha,\tau)>0
\]
which, can be made arbitrarily small as $\kappa\to 0$ and, a posteriori,  
\[
\frac{\varepsilon}{\varepsilon_0(\kappa)\mathrm{min\{\tau,\alpha\}}}\to 0.
\]
\medskip

The proof of Proposition \ref{prop:welldistributed} is now divided in several steps:

\textit{(I) Result for the map $\widetilde{\mathtt F}_n $:} For the map $\widetilde{\mathtt F}_n $ the $\xi$ and $\eta$ variables are uncoupled. An easy computation shows that the fixed point $z_0^{(n)}=(\xi_0^{(n)},\eta_0^{(n)})$ of this map is  given by 
\[
\xi_0^{(n)}=\frac{c_n}{1-\lambda_n}\qquad \qquad \eta_0^{(n)}=\frac{c_n\chi}{\nu_n-1}(1+O(\chi))
\]
Note that for $n\in\mathcal N$ such that $
|c_n|\lesssim \chi$ we have 
$\xi_0^{(n)}=O(1)$ and  $\eta_0^{(n)}=O(\chi)$. By the implicit function theorem, the map $\mathcal F_n$ has a fixed point $O(\epsilon/\chi)$-close to $z_0^{(n)}$. The rest of the proof is a standard application of the usual graph transform to describe its stable/unstable manifolds. Although this construction is entirely classical very precise quantitative estimates are of importance here so we provide the the details of the construction of the unstable manifold. The construction of the stable manifold follow analogously.
\\

\textit{(II) A space of vertical curves:} We deal with vertical curves of the form $\gamma=\{(f(\eta),\eta)\colon\eta \in[-1,1]\}$. We denote by $C^0$ be the Banach space of continuous functions $f:[-1,1]\to\mathbb R$ and we let $C_{\sigma}\subset C^0$ be the set of Lipschitz functions with Lipschitz constant $\sigma>0$. 
\\

\textit{(III) Graph transform operator:} Let $\sigma>0$. Given a curve $\gamma=\{(f(\eta),\eta)\colon \eta\in[-1,1]\}$ with $f\in C_\sigma$ 
 we define a new curve 
\[
\mathcal F(\gamma)=\{(\mathcal G(f)(\eta),\eta)\colon \eta\in[-1,1]\}
\]
where
\[
\mathcal G(f)(\eta)=-c_n+\lambda_n f(u(\eta))+g_\xi(f(u(\eta)),u(\eta))
\]
and $u(\eta)$ is such that 
\begin{equation}\label{eq:inverseparametrization}
\nu_n u(\eta)+g_\eta(f(u(\eta)), u(\eta))=\eta.
\end{equation}
\begin{lem}\label{lem:technicalinverse}
Let $\mathcal G$ be the operator (formally) defined above and let $\sigma>0$. Then, provided $\epsilon>0$ is small enough, $\mathcal G:C_\sigma\to C_\sigma$ is well defined.
\end{lem}
\begin{proof}
We first show that there exists $u:[-1,1]\to\mathbb R$ satisfying \eqref{eq:inverseparametrization}. To that end we rewrite  \eqref{eq:inverseparametrization} as 
\[
u(\eta)=G(u(\eta)):=\frac{1}{\nu_n} (\eta-g_\eta(f(u(\eta)),u(\eta)))
\]
and notice that, for any fixed $y\in[-1,1]$ and any $u,u_*\in[-1,1]$
\[
|G(u_*)-G(u)|=\frac{1}{\nu_n}|g_\eta(f(u_*),u_*)-g_\eta(f(u),u)|\leq \frac{1}{\nu_n}|g_\eta|_{C^1}(1+\sigma)|u-u_*|\leq \frac{1}{\nu_n}\epsilon(1+\sigma)|u-u_*|.
\]
Then, for $\epsilon>0$ small enough, the existence of a unique continuous $u(\eta)$ follows from the contraction mapping principle. Moreover, it is Lipchitz. Indeed, 
\begin{equation}\label{def:uLip}
\begin{split}
|\eta-\eta_*|=&|\nu_n(u(\eta)-u(\eta_*))-(g_\eta(f(u(\eta)),u(\eta))-g_\eta(f(u(\eta_*)),u(\eta_*)))|\\
\geq & \nu_n|u(\eta)-u(\eta_*)|-\epsilon(1+\sigma)|u(\eta)-u(\eta_*)|=(\nu_n-\epsilon(1+\sigma))|u(\eta)-u(\eta_*)|.
\end{split}
\end{equation}
We now show that $\mathcal G(f):C_\sigma\to C_\sigma$. For $\eta,\eta_*\in[-1,1]$ write 
\begin{align*}
|\mathcal G(f)(\eta_*)-\mathcal G(f)(\eta)|=&|\lambda_n (f(u(\eta))-f(u(\eta_*)) ) + \left( g_\xi(f(u(\eta),u(\eta))-g_\xi(f(u(\eta_*),u(\eta_*)) \right)|\\
\leq & \lambda_n \sigma |u(\eta)-u(\eta_*)|+\epsilon (1+\sigma)|u(\eta)-u(\eta_*)|.
\end{align*}
The conclusion follows now from \eqref{def:uLip}, which implies
\[
|\mathcal G(f)(\eta_*)-\mathcal G(f)(\eta)|\leq \frac{\lambda_n\sigma+\epsilon(1+\sigma)}{\nu_n-\epsilon(1+\sigma)}|\eta_*-\eta|,
\]
and 
\[
\frac{\lambda_n\sigma+\epsilon(1+\sigma)}{\nu_n-\epsilon(1+\sigma)}<\sigma
\]
provided $\epsilon$ is small enough.

\textit{(IV) Fixed point of the graph transform operator:} We now  define inductively a sequence of differentiable curves and  check that this sequence has a limit. Let 
\[
f_0(\eta)=\frac{c_n}{1-\lambda_n}.
\]
We show that 
\begin{equation}\label{eq:firsttransform}
|\mathcal G(f_0)-f_0|_{C^0}\leq \epsilon
\end{equation}
and that, for any pair $f,f_*\in  C_\sigma$ with 
\[
|f_*-f_0|_{C^0}, |f-f_0|_{C^0}\leq \frac{2\epsilon}{\chi\sqrt\tau}
\]
we  have 
\begin{equation}\label{eq:Lipschitztransform}
 |\mathcal G(f)-\mathcal G(f_*)|_{C^0}\leq \frac12(1+\lambda_n) | f-f_*|_{C^0}.
\end{equation}
We claim that \eqref{eq:firsttransform} and \eqref{eq:Lipschitztransform} hold, show how to complete the proof and verify the claim afterwards. To that end we define inductively $f_{i+1}=\mathcal G(f_i)$. Let $\tilde\lambda_n=\frac12(1+\lambda_n)\in (0,1)$. Using  \eqref{eq:firsttransform} and \eqref{eq:Lipschitztransform} one may check by induction that for any $i\in\mathbb N$
\[
|f_i-f_0|_{C^0}\leq \sum_{k=1}^i |f_k-f_{k-1}|_{C^0}\leq |f_1-f_0|_{C^0}\sum_{k=1}^i \tilde\lambda_n^k\leq \frac{|f_1-f_0|_{C^0}}{1-\tilde\lambda_n}\leq \frac{2\epsilon}{\chi\sqrt\tau}.
\]
Thus, using \eqref{eq:Lipschitztransform} we deduce that $\{f_i\}$ is a Cauchy sequence in $C_\sigma$ which converges to $f\in C_\sigma$ satisfying 
\[
|f-f_0|_{C^0}\leq \frac{2\epsilon}{\chi\sqrt\tau}.
\]

\textit{ (V) Claims \eqref{eq:firsttransform}  and \eqref{eq:Lipschitztransform}: } We now verify the claims.  \eqref{eq:firsttransform} is straightforward from the definition of the operator $\mathcal G$ and the fact that $|g|_{C^1}\leq \epsilon$.  We now check \eqref{eq:Lipschitztransform}. To that end we write 
\begin{align*}
\mathcal G(f)-\mathcal G(f_*)=&\lambda_n (f\circ u_f-f_*\circ u_{f_*})+g_\xi(f\circ u_f,u_f)-g_\xi(f_*\circ u_{f_*},u_{f_*})\\
=& \lambda_n(f\circ u_f-f_*\circ u_f)+\lambda_n(f_*\circ u_f-f_*\circ u_{f_*})+(g_\xi(f\circ u_f,u_f)-g_\xi(f_*\circ u_f,u_f))\\
&+g_\xi(f_* \circ u_f,u_f)-g_\xi(f_*\circ u_{f_*}, u_{f_*})\\
=& \mathcal E_1+\mathcal E_2+\mathcal E_3+\mathcal E_4,
\end{align*}
and analyze term by term. For the first term one readily checks that 
\[
|\mathcal E_1|_{C^0}\leq \lambda_n |f-f_*|_{C^0}.
\]
To estimate the second term we notice that
\begin{align*}
|u_f-u_{f_*}|\leq & \frac{1}{\lambda_n}|g|_{C^1}\left(|f\circ u_f-f_*\circ u_f|+|f_*\circ u_f-f_*\circ u_{f_*}|+|u_f-u_{f_*}|\right)\\
\leq & \frac{1}{\nu_n}\epsilon \left(|f-f_*|_{C^0}+(1+\sigma)|u_f-u_{f_*}|\right),
\end{align*}
 so 
\[
|u_f-u_{f_*}|_{C^0}\leq \frac{\epsilon}{\nu_n-(1+\sigma)\epsilon}|f-f_*|_{C^0}\leq 2\epsilon |f-f_*|_{C^0}.
\]
Hence
\[
|\mathcal E_2|_{C^0}\leq \lambda_n\sigma |u_f-u_{f_*}|_{C^0} \leq 2\lambda_n\sigma\epsilon |f-f_*|_{C^0}.
\]
Analogous computations show that 
\[
\begin{split}
|\mathcal E_3|_{C^0}&\leq \epsilon|f-f_*|_{C^0}\\
|\mathcal E_4|_{C^0}&\leq \epsilon (1+\sigma)|u_f-u_{f_*}|_{C^0}\leq 2\epsilon^2(1+\sigma)|f-f_*|_{C^0}.
\end{split}
\]
Thus,
\[
|\mathcal G(f)-\mathcal G(f_*)|_{C^0}\leq\left(\lambda_n\left(1+2\sigma\epsilon\right)+3\epsilon\right)|f-f_*|_{C^0}\leq \frac12 (1+\lambda_n)|f-f_*|_{C^0}.
\]
\end{proof}

\subsection{Proof of Lemma \ref{lem:BNF}}
The proof follows from a standard iteration argument so we only give a sketch of the underlying idea.  Let $\rho_0(\alpha,k,\rho,\sigma)$ be as in the statement of Lemma \ref{lem:BNF}. Given a real-analyic function $h:\mathbb T_\sigma\times\mathbb B_\rho\to \mathbb C$ we consider its Fourier-Taylor series
\[
h(\varphi,J)=\sum_{(j,l)\in\mathbb N\times\mathbb Z}  h^{[l]}_j J^je^{il\varphi},
\]
and the associated  Fourier-Taylor norm. 
\[
\lVert h\rVert_{\rho,\sigma}=\sum_{(j,l)\in\mathbb N\times\mathbb Z} \left|h^{[l]}_j\right| \rho^j e^{|l|\sigma}.
\]
For $n\in\{0,\dots,k\}$, define
\[
\rho^{(n)}=\rho_0\left(1-\frac{n}{2k}\right)\qquad\qquad \sigma^{(n)}=\sigma\left(1-\frac{n}{2k}\right).
\]
Suppose we are given a map $\mathtt T^{(n)}:\mathbb T_{\sigma^{(n)}}\times\mathbb B_{\rho^{(n)}}\to \mathbb A_{\rho,\sigma}$ of the form
\[
\mathtt T^{(n)}:\begin{pmatrix}\varphi\\J\end{pmatrix}\mapsto \begin{pmatrix}\varphi+\beta+h^{(n)}(J)+R^{(n)}_\varphi(\varphi,J)\\ J+R^{(n)}_J(\varphi,J)\end{pmatrix},
\]
with 
\[
\partial_{J^l}^l R^{(n)}_\varphi(\varphi,0)=0\qquad\text{for }1\leq l< n+2,\qquad\qquad\partial_{J^l}^l R^{(n)}_J(\varphi,0)=0\qquad\text{for }1\leq l< n+3,
\]
and satisfying 
\[
\lVert R^{(n)}_\varphi\rVert_{\rho^{(n)},\sigma^{(n)}}\lesssim \frac{k}{\rho_0} \left(\frac{k^3}{\alpha\sigma^3}\right)^n \left(\frac{\rho_0}{\rho}\right)^{n+3}\qquad\qquad \lVert R^{(n)}_J\rVert_{\rho^{(n)},\sigma^{(n)}}\leq \frac{k}{\sigma}  \left(\frac{k^3}{\alpha\sigma^3}\right)^n \left(\frac{\rho_0}{\rho}\right)^{n+3}.
\]
Observe that the estimates are trivially satisfied for the base case $n=0$. We now make use of an inductive argument to show their validity for a general $n\in\{1,\dots,k\}$. We  look for a generating function 
\[
\mathcal S^{(n)}(\varphi,I)=\varphi I+S^{(n)}(\varphi,I),\qquad n=0\ldots k-1,
\] 
such that the associated change of coordinates $\Phi_n:(\theta,I)\mapsto (\varphi,J)$ eliminates the term of order $n+2$ in $R^{(n)}_\varphi$ (by symplectic symmetry this transformation will also eliminate the term of order $n+3$ in $R^{(n)}_J$). Let 
\[
Q^{(n)}(\varphi)=\partial_J^{n+2}  R^{(n)}_\varphi(\varphi,0)
\]
and define 
\[
S^{(n)}(\varphi, I)=s^{(n)}(\varphi) I^{n+3},
\]
with 
\[
s^{(n)}(\varphi)=\frac{1}{(n+3)  !}\sum_{l\in\mathbb Z\setminus\{0\}} \frac{{(Q^{(n)})}^{[l]}}{1-e^{il\beta}}e^{il\varphi}
\]
Since $\beta\in \mathcal B_\alpha$ a standard computation shows that
\[
\left( \partial_{I}  S^{(n)}(\varphi+\beta,I)- \partial_{I} S^{(n)}(\varphi, I)\right)= \frac{1}{(n+2)!} Q^{(n)}(\varphi,I)
\]
and
\[
\lVert S^{(n)}\rVert_{\rho^{(n)},\sigma^{(n)}-\frac{\sigma}{4k}}\lesssim  \frac{\rho_0}{k} \frac{k^2}{\alpha\sigma^2}\lVert R^{(n)}_\varphi\rVert_{\rho^{(n)},\sigma^{(n)}} \lesssim \frac{k^2}{\alpha\sigma^2}  \left(\frac{k^3}{\alpha\sigma^3}\right)^n\left(\frac{\rho_0}{\rho}\right)^{n+3}.
\]
In particular:
\begin{enumerate}
    \item  $S^{(n)}$ generates a change of coordinates $\Phi_n(\theta,I)\mapsto (\varphi,J)$
    \[
    \theta=\varphi+\partial_I S^{(n)}(\varphi,I) \qquad\qquad J=I+\partial_\varphi S^{(n)}(\varphi,I)
    \]
    From the estimates above, it is not difficult to observe that 
    \[
    |\varphi(\theta,I)-\theta|\lesssim \frac{k}{\rho_0}\frac{k^2}{\alpha\sigma^2}\left(\frac{k^3}{\alpha\sigma^3}\right)^n\left(\frac{\rho_0}{\rho}\right)^{n+3}\left(\frac{| I|}{\rho_0}\right)^{n+2}
    \]
    and
    \[
    |J(\theta,I)-I|\lesssim \frac{k}{\sigma}\frac{k^2}{\alpha\sigma^2}\left(\frac{k^3}{\alpha\sigma^3}\right)^n\left(\frac{\rho_0}{\rho}\right)^{n+3}\left(\frac{|I|}{\rho_0}\right)^{n+3}.
    \]
    Since for our choice of $\rho_0$ we have 
\begin{equation}\label{eq:smallnessiterationBNF}
\frac{\rho_0}{\rho}\leq \frac{\alpha\sigma^3\rho}{4k^3}
\end{equation}
we obtain that, for $| I|\leq \rho_0/2$ (this is a very rough estimate)
\begin{equation}\label{eq:smallnesschangeiterationBNF}
    |\varphi(\theta,I)-\theta|\lesssim 2^{-n}  \left(\frac{|I|}{\rho_0}\right)^{n+1} \qquad\qquad |J(\theta,I)-I|\lesssim 2^{-n}\left(\frac{|I|}{\rho_0}\right)^{n+2}.
\end{equation}
\black
\item $\Phi_n$ conjugates $\mathtt T^{(n)}$ to
\[
\Phi^{-1}\circ \mathtt T^{(n)}\circ \Phi:\begin{pmatrix}\varphi\\J\end{pmatrix}\mapsto 
\begin{pmatrix}\varphi+\beta+h^{(n+1)}(J)+R^{(n+1)}_\varphi(\varphi,J)\\ J+R^{(n+1)}_J(\varphi,J)\end{pmatrix}
\]
with 
\[
\partial_{J^l}^l R^{(n+1)}_\varphi(\varphi,0)=0\qquad\text{for }1\leq l< n+2,\qquad\qquad\partial_{J^l}^l R^{(n+1)}_J(\varphi,0)=0\qquad\text{for }1\leq l< n+3
\]
and satisfying 
\begin{align*}
\lVert R^{(n+1)}_\varphi\rVert_{\rho^{(n+1)},\sigma^{(n+1)}}\lesssim& \frac{k}{\rho_0}\left(\frac{k^3}{\alpha\sigma^3}\right)^{n+1}\left(\frac{\rho_0}{\rho}\right)^{(n+1)+3}\\
\lVert R^{(n+1)}_J\rVert_{\rho^{(n+1)},\sigma^{(n+1)}}\lesssim &\frac{k}{\sigma} \left(\frac{k^3}{\alpha\sigma^3}\right)^{n+1}\left(\frac{\rho_0}{\rho}\right)^{(n+1)+3}.
\end{align*}
\end{enumerate}
This completes the inductive step.  Finally, the change of variables $\Phi$ in Lemma \ref{lem:BNF} is obtained as the composition $\Phi:=\Phi_0\circ\cdots\circ\Phi_{k-1}$. The estimate in \eqref{eq:changeBNF} follows from \eqref{eq:smallnesschangeiterationBNF}. Also, using \eqref{eq:smallnessiterationBNF} one deduces the estimates \eqref{eq:smallnessremainderBNF} and the inductive estimates for $R_\varphi^{(n)},R_J^{(n)}$.

\section{Symplectic reduction for the three-body problem}\label{sec:symplecticred}

The reader is referred to \cite{guardia2022hyperbolicdynamicsoscillatorymotions} for explicit expressions of the constants which we do not specify below.

\subsection*{Jacobi reduction}
In order to reduce the invariance by translation (i.e. the invariance of \eqref{eq:3bpHam} by parallel translation of all the bodies) we define the change of coordinates
\[
\Phi_{\mathrm{Jac}}:(Q,P)\mapsto (q,p),
\]
given by the symplectic completion of the change 
\begin{align*}
Q_0=q_0\qquad\qquad Q_1=q_1-q_0\qquad\qquad Q_2=q_2-\frac{m_0q_0+m_1q_1}{m_0+m_1}.
\end{align*}
In the new coordinate system, the Hamiltonian \eqref{eq:3bpHam} reads
\[
H_{\mathrm{Jac}}(Q_1,Q_2,P_1,P_2)=\sum_{i=1}^2\frac{|P_i|^2}{2\mu_i}-\widetilde U(Q_1,Q_2)
\]
for some $\mu_i>0$ and 
\begin{equation}\label{eq:widetildepotential}
\widetilde U(Q_1,Q_2)=\frac{m_0m_1}{|Q_1|}+\frac{m_0m_2}{|Q_2+\sigma_0Q_1|}+\frac{m_1m_2}{|Q_2-\sigma_1Q_1|}-\frac{(m_0+m_1)m_2}{|Q_2|}
\end{equation}
for certain $\sigma_0,\sigma_1\neq 0$ satisfying $m_0\sigma_0+m_1\sigma_1=0$.
We are interested in the hierarchical region of the phase space where $|Q_2|\gg |Q_1|$ and $|Q_1|$ is contained in a bounded region of the plane. Hence, we decompose $H_{\mathrm{Jac}}$ as 
\begin{equation}\label{eq:perturbationuncoupled}
H_{\mathrm{Jac}}(Q,P)=\widehat H_{\mathrm{ell}}(Q_1,P_1)+\widehat H_{\mathrm{par}}(Q_2,P_2)+\widehat V(Q_1,Q_2),
\end{equation}
where 
\[
\widehat H_{\mathrm{ell}}(Q_1,P_1)=\frac{|P_1|^2}{2\mu_1}-\frac{m_0m_1}{|Q_1|}\qquad\qquad \widehat H_{\mathrm{par}}(Q_2,P_2)=\frac{|P_2|^2}{2\mu_2}-\frac{m_2(m_0+m_1)}{|Q_2|}
\]
and
\begin{equation}\label{eq:widehatpotential}
\widehat V(Q_1,Q_2)=\widetilde U(Q_1,Q_2)-\frac{m_0m_1}{|Q_1|}-\frac{m_2(m_0+m_1)}{|Q_2|}.
\end{equation}
 Notice that, in the region $|Q_2|\gg |Q_1|$ and $|Q_1|\sim 1$ the term $\widetilde V$ becomes perturbative since $\widetilde V(Q_1,Q_2)=O_3(|Q_2|/|Q_1|)$ so we can study \eqref{eq:perturbationuncoupled} as a perturbation of two uncoupled two-body problems: $ \widetilde H_{\mathrm{ell}}$ describing the dynamics of the inner system and $\widetilde H_{\mathrm{par}}$ describing the dynamics of the outer body with respect to the inner system. After a conformally symplectic scaling $\Phi_{\bs m}:(\widetilde Q,\widetilde P)\mapsto (Q,P)$  (involving only the masses $m_0,m_1,m_2$) it is possible to recast the system \eqref{eq:perturbationuncoupled} as
\[
\widetilde H_{\mathrm{Jac}}:=H_{\mathrm{Jac}}\circ\Phi_{\bs m}(\widetilde Q,\widetilde P)= \widetilde H_{\mathrm{ell}}(\widetilde Q_1,\widetilde P_1)+
\widetilde H_{\mathrm{par}}(\widetilde Q_2,\widetilde P_2)+\widetilde V(\widetilde Q_1,\widetilde Q_2),
\]
with
\[
\widetilde H_{\mathrm{ell}}(\widetilde Q_1,\widetilde P_1)=\nu\left(\frac{|\widetilde P_1|^2}{2}-\frac{1}{|\widetilde Q_1|}\right)\qquad\qquad \widetilde H_{\mathrm{par}}(\widetilde Q_2, \widetilde P_2)=\frac{|\widetilde P_2|^2}{2}-\frac{1}{|\widetilde Q_2|}
\]
where $\nu\neq 0$ and $ \widetilde V$ is as in \eqref{eq:widehatpotential}
for certain $\tilde\sigma_0,\tilde\sigma_1\neq 0$ which also satisfy $m_0\tilde\sigma_0+m_1\tilde\sigma_1\neq 0$.

\subsection*{Delaunay-polar variables}
We now introduce a change of variables tailored for the description of elliptic motions, the so called Delaunay map (see \cite{MR2269239})
\[
\phi_{\mathrm{Del}}:(\ell,L,g,\Gamma)\mapsto (\widetilde Q_1,\widetilde P_1)
\]
In this coordinate system the instantaneous state of $\widetilde Q_1$ is described in terms of an ``instantaneous'' ellipse, parametrized in terms of $(L,g,\Gamma)$ and an angle $\ell$, the mean anomaly, giving the position of $\widetilde Q_1$ inside this ellipse. The angle $g\in\mathbb T$ measures the angle of the pericenter with respect to some fixed line, $L\in\mathbb R$ is the square root of the semimajor axis and $\Gamma$ is the angular momentum. Then, the eccentricity of the ellipse is given by
\[
\epsilon(L,\Gamma)=\sqrt{1-\frac{\Gamma^2}{L^2}}.
\]
The Delaunay map is real-analytic and symplectic on the region (see \cite{MR2269239,MR3146588})
\[
\mathcal D=\{\widetilde H_{\mathrm{ell}}(\widetilde Q_1,\widetilde  P_1)<0, 0< \Gamma(\widetilde  Q_1,\widetilde P_1)< L(\widetilde Q_1,\widetilde P_1)\},
\]
which corresponds to the set of planar oriented ellipses with strictly positive eccentricity\footnote{Notice, for instance, that $g$ is not well defined for circular motions}. In Delaunay coordinates, the elliptic Hamiltonian reads
\[
H_{\mathrm{ell}}:=\widetilde H_{\mathrm{ell}}\circ\phi_{\mathrm{Del}}(L)=-\frac{\nu}{2L^2},
\]
so the flow induced by $H_{\mathrm{ell}}$ reduces to a linear (resonant) translation
\[
\phi_{H_{\mathrm{ell}}}^t:(\ell,L,g,\Gamma)\mapsto(\ell+(\nu/L^{3}) t,L,g,\Gamma).
\]
On the other hand, to describe the parabolic motion of $\widetilde  Q_2$, it is convenient to introduce polar coordinates 
\[
\phi_{\mathrm{pol}}(r,\alpha,y,G)\mapsto \left(\widetilde Q_2,\widetilde P_2\right)=\left(r\cos\alpha,r\sin\alpha,y\cos\alpha-\frac{G}{r}\sin\alpha, y\sin\alpha+\frac{G}{r}\cos\alpha\right).
\]
In the new coordinate system
\[
 H_{\mathrm{par}}:=\widetilde H_{\mathrm{par}}\circ\phi_{\mathrm{pol}}(r,y,G)=\frac{y^2}{2}+\frac{G^2}{2r^2}-\frac 1r.
\]
For later use we denote by 
\[
\widehat  \Phi=(\phi_{\mathrm{Del}},\phi_{\mathrm{pol}}):(\ell,L,g,\Gamma,r,\alpha,y,G)\mapsto (\widetilde Q,\widetilde P).
\]
and observe that 
\begin{equation}\label{eq:HaminDelPolar}
\widehat H:=H_{\mathrm{Jac}}\circ\widehat \Phi= H_{\mathrm{ell}}(L)+H_{\mathrm{par}}(r,y,G)+\widehat V(\ell,L,g-\alpha,\Gamma,r)\qquad\qquad \widehat V=\widetilde V\circ\widehat \Phi.
\end{equation}

\subsection*{The reduction by rotations}
The Hamiltonian \eqref{eq:HaminDelPolar}  only depends on the difference $g-\alpha$ (a manifestation of the invariance by rotation of the system). To take advantage of this fact and reduce \eqref{eq:HaminDelPolar} to a system with 3 degrees-of-freedom we let 
\[
\Phi_{\mathrm{rot}}:(\ell,L,\phi,\Gamma,r,\alpha,y,\Theta)\mapsto (\ell,L,g,\Gamma,r,\alpha,y,G),
\]
with $\phi=g-\alpha$ and $\Theta=G+\Gamma$. We thus arrive to the Hamiltonian 
\[
H_{\mathrm{rot}}=\widehat H\circ\Phi_{\mathrm{rot}}=H_{\mathrm{ell}}(L)+H_{\mathrm{par}}(r,y,\Theta-\Gamma)+\widehat V(\ell,L,\phi,\Gamma,r)
\]
for which $\alpha$ is a cyclic variable. Hence, the total angular momentum $\Theta$ is conserved along the flow of $H_{\mathrm{rot}}$.

\subsection*{Poincar\'e coordinates}
Finally, we introduce an additional transformation which gives a local chart which also includes circular motions of the inner bodies (recall that Delaunay variables are only well-defined for ellipses with positive eccentricity). This change of coordinates, which we denote by $\Phi_{\mathrm{Poin}}$, is given by 
\[
\lambda=\ell+\phi\qquad\qquad\xi=\sqrt{L-\Gamma}e^{i\phi}\qquad\qquad\eta=\sqrt{L-\Gamma}e^{-i\phi}.
\]
The resulting composition 
\begin{equation}\label{eq:compositionchangesreduction}
\Phi_{\Theta}:=\widehat\Phi\circ\Phi_{\mathrm{rot}}\circ\Phi_{\mathrm{Poin}}:(\lambda,L,\xi,\eta,r,y;\alpha,\Theta)\mapsto (Q_1,P_1,Q_2,P_2)
\end{equation}
is real-analytic in (a complex extension of) $(\lambda,L)\in\mathbb T\times\mathbb R_+$, $(\xi,\eta)\in\mathbb D\subset\mathbb C^2$ and $(r,y)\in\mathbb R_+\times\mathbb R$ and 
\[
\Phi^*_\Theta(\mathrm{d}P\wedge\mathrm{d}Q)=\mathrm{d}L\wedge\mathrm d\lambda+i\mathrm{d}\xi\wedge\mathrm d\eta+\mathrm dy\wedge\mathrm dr+\mathrm d\Theta\wedge\mathrm d\alpha
\]
(see \cite{MR3146588}). For any fixed $\Theta\in\mathbb R$ we thus end up with the real-analytic Hamiltonian  
\begin{equation}\label{eq:finalreduced3bp}
\mathcal H_\Theta:=H\circ\Phi_{\Theta}=H_{\mathrm{ell}}(L)+H_{\mathrm{par}}(r,y,\Theta-\Gamma)+ V(\lambda,L,\eta,\xi,r)\qquad\qquad V=\widehat V\circ\Phi_{\mathrm{Poin}}.
\end{equation}

\section{The scattering maps of the restricted problem}\label{sec:appendixrestrictedscattmaps}
In this section we show how Theorem \ref{thm:scattmapsrestricted} can be extracted from the results in \cite{guardia2023degeneratearnolddiffusionmechanism}. We divide the proof in several steps.

First we obtain asymptotic formulas for the scattering maps.
\begin{thm}[Theorem 2.10 in \cite{guardia2023degeneratearnolddiffusionmechanism}]\label{thm:d1}
Let $1\ll G_1<G_2$ be fixed and let $\mathcal P_\infty(G_1,G_2)$ be as in \eqref{eq:cutoffcylinder}. There exists $\rho>0$ (independent of $G_1,G_2$ and $\zeta$) such that the scattering maps admit a holomorphic extension to a $\rho$-complex neighbourhood of  $\mathcal P_\infty(G_1,G_2)$ and are of the form
\[
S_i:\begin{pmatrix}\varphi\\G\end{pmatrix}=\begin{pmatrix}\varphi+\omega(G)+O(\zeta|G|^{-7})\\
G+\zeta r(\varphi,G)+O(\zeta|G|^{-7})\end{pmatrix}
\]
with 
\begin{equation}\label{eq:freqrbp}
\omega(G)=-\mu(1-\mu)\frac{3\pi}{2G^4}+O(|G|^{-7})\qquad\text{and}\qquad r(\varphi,G)=\mu(1-\mu)(1-2\mu)\frac{15\pi}{8G^5}\sin\varphi.
\end{equation}
\end{thm}

Let now $G_0\in(G_1,G_2)$ and let  $J=G-G_0$. We abuse notation, write 
\[
\omega(J)=\omega(G_0+J)\qquad\qquad r(\varphi,J)=r(\varphi,G_0+J)
\]
and still denote by 
\begin{equation}\label{eq:convenientformscattmaps}
S_i:\begin{pmatrix}\varphi\\J\end{pmatrix}\to \begin{pmatrix}\varphi+\omega(J)+O(\zeta G_0^{-7})\\ J+\zeta r(\varphi,J)+O(\zeta G_0^{-7})\end{pmatrix} \qquad\qquad 
\end{equation}
the expression of the scattering maps in $(\varphi,J)$ coordinates for $(\varphi,J)\in\mathbb A=\mathbb T\times[-1,1]$. In particular, these maps are of the form in Theorem 2.4 of \cite{guardia2023degeneratearnolddiffusionmechanism}. 

\begin{rem}
    Throughout the rest of the section $\rho$ (introduced in Theorem \ref{thm:d1} and $c$ (introduced in Lemma \ref{lem:interpolatinglemma} below) are constants which do not depend on $G,\zeta$. 
\end{rem}

Since  $G_0\gg 1$ so for $J\in[-1,1]$, one has that $\omega(J)\sim G_0^{-4}\ll 1$. This implies that one can performing a high number steps of averaging to $S_0$ to write it by as integrable map plus and exponentially small remainder, as  it is shown the following lemma of \cite{guardia2023degeneratearnolddiffusionmechanism} (in the notation of that paper one should take $\varepsilon=G_0^{-4}$ and $\delta=\zeta G_0^{-5}$).

\begin{lem}[Lemma 6.2 in~\cite{guardia2023degeneratearnolddiffusionmechanism}]\label{lem:interpolatinglemma}
There exists 
\begin{itemize}
    \item a real-analytic, one degree-of-freedom Hamiltonian $\mathcal K$, defined on a $\rho$-complex neighbourhood of $\mathbb A$  and of the form
\begin{equation}\label{eq:interppp}
\mathcal K(\varphi,J)=h(J)+O(\zeta G_0^{-5}) \qquad\qquad h'(J)=\omega(J),
\end{equation}
\item a real-analytic change of variables $\psi$ of the form $\psi=\mathrm{id}+O(\zeta G_0^{-9})$

\end{itemize}
such that the map $\widetilde S_0=\psi^{-1}\circ S_0\circ\psi$ satisfies that 
\begin{equation}\label{eq:maps1inapproxvariables}
\widetilde S_0=\phi_{\mathcal K}+O(\zeta G_0^{-5} \exp(-c G_0^4)).
\end{equation}
\end{lem}

Next step is to introduce action-angle variables  for the Hamiltonian $\mathcal K$ as follows. We let 
\[
L(E)=\frac{1}{2\pi}\int_{\{\mathcal K(\varphi,J)=E+h(0)\}}
J\mathrm d\varphi\]
and define the generating function
\[
W(\varphi,L)=\int_0^\varphi J(\tau, E(L))\mathrm d\tau,
\]
where $J(\varphi,E)$ is the (unique) solution to $\mathcal K(\varphi,J)=E+h(0)$. Making use of \eqref{eq:freqrbp} and \eqref{eq:interppp} it is easy to check that 
\[
L(E)=\frac{1}{\omega(0)}(E +O(G_0^{-1} E^2,\zeta G_0^{-5})).
\]
Then, we define the symplectic change of variables $\phi:(\ell,L)\mapsto (\varphi,J)$ via the implicit relation 
\[
\ell=\partial_LW(\varphi,L)\qquad\qquad J=\partial_\varphi W(\varphi,L).
\]
It is not difficult to check that $\phi$ is $O(\zeta)$-close to the identity on a complex $\rho$-neighbourhood of $\mathbb A$. In the new variables, the map $
\widehat S_0=\phi^{-1}\circ\widetilde S_0\circ\phi$ is a real-analytic perturbation of size $O(\zeta G_0^{-5}\exp(-cG_0^4))$ of the integrable twist map of the annulus given by 
\[
F_0:(\ell,L)^\top\mapsto (\ell+\tilde\omega(L),L)^\top\qquad\qquad \tilde\omega (L)=\frac{\mathrm{d}}{\mathrm d L} (\mathcal K\circ \phi)(L).
\]
Let 
\[
\omega_-:=\tilde\omega(1)<\tilde\omega(-1):=\omega_+.
\]
Since $|\partial_L\widetilde \omega(L)|\sim |\partial_J \omega(J)|\gtrsim G_0^{-5}$,  standard application of the KAM theorem  for real-analytic twist maps shows that for any $\alpha>0$ satisfying 
\[
\alpha\gg \sup_{\mathrm{dist}(z,\mathbb A)\leq \rho''}\lVert \widehat S_0(z)-F_0(z)\lVert^{1/2} =O(\sqrt{\zeta}\exp(-cG_0^4)),
\]
any $\varsigma\geq 2$ and any $\omega_0\in DC(\alpha,\varsigma)\cap[\omega_-,\omega_+]$ there exists an invariant curve $\gamma_{\omega_0}$ on which the motion is conjugated to a rotation by $\omega_0$ and which is given by the graph of a function whose $C^1$ norm is $O(\sqrt\zeta\exp(- cG_0^4))$. We now  write the map $\widetilde S_0$ in Birkhoff normal form (up to order 3) around the curve $\gamma_{\omega_0}$ with suitable $\omega_0$. We restrict ourselves to the case $\omega_0\in \mathcal B_{\alpha_*}:=DC(\alpha_*,2)\cap[\omega_-,\omega_+]$ with $\alpha_*\geq G_0^{-6}$ (clearly this set is non-empty for $G_0$ large enough as $DC(\alpha_*,2)$ is $\alpha_*$-dense on $\mathbb R$ and $|\omega_+-\omega_-|\gtrsim G_0^{-5}$). By standard averaging techniques, for any $\omega_0\in \mathcal B_{\alpha_*}$, one can find a change of variables of the form 
\[
\phi_{\omega_0}:\begin{pmatrix}\theta\\I\end{pmatrix}\mapsto \begin{pmatrix}\ell\\L\end{pmatrix}=\begin{pmatrix}\theta+\phi_\theta(\theta,I)\\I_*+I+\phi_I(\theta,I)\end{pmatrix},\qquad\qquad |\phi_\theta|,|\phi_I|=O(\sqrt{\zeta}\exp(-c G_0^4)),
\]
for some $I_*(\omega_0)\in\mathbb A$ such that 
\[
\mathtt S_0:=\phi_{\omega_0}^{-1}\circ \widehat S_0\circ\phi_{\omega_0}
\]
is of the form \eqref{eq:T0map} with frequency $\omega_0\in \mathcal B_{\alpha_*}$ and  torsion $\tau\in\mathbb R$ satisfying $|\tau|\sim \partial_G \omega(G_0)\sim G_0^{-5}$. Before completing the proof of Theorem \ref{thm:scattmapsrestricted} we introduce some more notation. We let 
\[
\kappa=\sqrt{\zeta}\exp(-cG_0^4)
\]
and define the annulus
\begin{equation}\label{def:annulus:appD}
\mathcal A_\infty=\{(\theta,I)\in\mathbb T\times[-\alpha_*/|\log^3 \kappa|,\alpha_*/|\log^3\kappa|]\}\subset\mathbb A.
\end{equation}
\begin{proof}[Proof of the first claim in Theorem \ref{thm:scattmapsrestricted}]
We start by showing that the invariant curves for the map $S_0$ can be traversed making use of the map $S_1$. Let $\mathcal K$ be the one-degree-of freedom Hamiltonian $\mathcal K$ constructed above, $\widetilde S_0$ be as in \eqref{eq:maps1inapproxvariables} and let $\widetilde S_1=\psi^{-1}\circ S_1\circ\psi$. 

\begin{prop}[Proposition 6.3 in \cite{guardia2023degeneratearnolddiffusionmechanism}] \label{prop:smallgaps}
For any $\gamma$ in the set
\[
A_{\mathrm{ess}}=\{\gamma\subset \mathbb A\colon \gamma\text{ is an essential invariant curve for }\widetilde S_0\},
\]
we have 
\begin{equation}\label{eq:smallgaps}
\max\{|\mathcal K(z_2)-\mathcal K(z_1)|,\ z_2,z_1\in\gamma\}\lesssim \sqrt\zeta \exp(-c G_0^{4})
\end{equation}
\end{prop}
This result is indeed a consequence of the fact that any essential invariant curve is trapped between two KAM curves and the gaps between the latter are  of size $O(\sqrt\zeta \exp(-c G_0^{4}))$. Combining  Theorem 2.14, Lemma 2.15 and Proposition 6.8 in \cite{guardia2023degeneratearnolddiffusionmechanism} (see also Theorem 2.5. in that paper)  it is shown that, for any $(\varphi,J)\in\mathbb A$,
\begin{equation}\label{eq:transvscattmapsrestr}
\mathcal K\circ \widetilde S_1(\varphi,J) -\mathcal K(\varphi,J) = \frac{C\zeta}{(G_0+J)^{5/2}}\exp(-\sigma(G_0,J))(\sin\varphi+O(G_0)^{-1})
\end{equation}
for $C=-\mu^2(1-\mu)^2 9(2\pi)^{3/2}$ and 
\begin{equation}\label{def:sigmaapp}
\sigma(G_0,J)=\frac 13(G_0+J)^3.
\end{equation}
We then claim that, provided 
\[
G_0\gg |\log\zeta|,
\]
for any $(\varphi,J)\in\mathbb A$, there exists $C>0$, $M_\pm<\infty$ and  $\sigma^\pm\in\mathbb N^{M_\pm}$ for which
\begin{equation}\label{eq:upincrement}
\mathcal K\circ \widetilde S_{\sigma^+_{M-1}}\circ\cdots\widetilde S_{\sigma_0^+}(\varphi,J)-\mathcal K(\varphi,J) \in ( C\zeta \exp(-\sigma(G_0,J)), 2C\zeta \exp(-\sigma(G_0,J)))
\end{equation}
and 
\begin{equation}\label{eq:dnincrement}
\mathcal K\circ \widetilde S_{\sigma^-_{M-1}}\circ\cdots\widetilde S_{\sigma_0^-}(\varphi,J)-\mathcal K(\varphi,J) \in ( -2C\zeta \exp(-\sigma(G_0,J)), -C\zeta \exp(-\sigma(G_0,J))
\end{equation}
This claim is verified by noticing the following. First, in view of \eqref{eq:smallgaps} and \eqref{eq:transvscattmapsrestr} there are 2 intervals of $\mathbb T_\pm$ of size of order one on which, making use of the map $\widetilde S_1$, one can either increase or decrease the value of $\mathcal K$ by a quantity bounded below by $C\zeta\exp(-\sigma(G_0,J))$. Second, since $\omega(J)\sim G_0^{-4}$, we can reach these intervals by iterating the map $\widetilde S_0$ no more than $N\sim G_0^{4}$ times. In view of the expression \eqref{eq:maps1inapproxvariables} along this iteration the value of $\mathcal K$ remains almost constant. 

To conclude the proof of the first item in Theorem \ref{thm:scattmapsrestricted} it is enough to notice that the width of $\mathcal A_\infty$ (see \eqref{def:annulus:appD}) is much larger than the increments in \eqref{eq:upincrement}, \eqref{eq:dnincrement}.
\end{proof}

\begin{proof}[Proof of the second claim in Theorem \ref{thm:scattmapsrestricted}]
  We need to obtain an asymptotic expression for the map $S_1$ in the coordinate system given by the transformation $\Phi:=\psi\circ\phi\circ\phi_{\omega_0}$ with $\psi,\phi,\phi_{\omega_0}$ as above. We write 
  \begin{equation}\label{eq:asymptotics2goodform}
  \begin{split}
  \mathtt S_1:=\Phi^{-1}\circ S_1\circ\Phi=&\Phi^{-1} (S_0+(S_1-S_0))\circ\Phi\\
  =&\mathtt S_0+D\Phi^{-1}(S_0\circ\Phi) (S_1-S_0)\circ\Phi+O(\lVert\Phi-\mathrm{id}\rVert^2,\ \lVert S_1-S_0\rVert^2).
  \end{split}
  \end{equation}
  The desired conclusion then plainly follows from i) the fact that $\Phi$ is $O(\zeta)$-close to identity ii) as shown in Theorem 2.14 of \cite{guardia2023degeneratearnolddiffusionmechanism}
  \[
  (S_1-S_0):(\varphi,J)\mapsto \Delta(\varphi,J)+\left(O\left(G_0^{-1/2}\exp(-\sigma(G_0,J)) \right),O\left(\zeta G_0^{-5/2}\exp(-\sigma(G_0,J)) \right)\right)^\top 
  \]
  with 
  \begin{align*}
 \Delta(\varphi,J) =&(\partial_J(\mathcal L_+-\mathcal L_-)(\varphi,J),-\partial_\varphi(\mathcal L_+-\mathcal L_-)(\varphi,J))^\top,
  \end{align*}
  where $\mathcal L_\pm$ being the so-called reduced Melnikov potentials iii) The asymptotic expressions (see Appendix B in \cite{guardia2023degeneratearnolddiffusionmechanism} and expression (44) in \cite{DiffusionEllipticKaloshin})
  \[
  \mathcal L_\pm(\varphi,J)=\pm\mu(1-\mu) \left[2L_{1,1}(J)\cos (s(\varphi)-\varphi)+2L_{1,2}(J)\cos(s(\varphi)-2\varphi)+O\left(\zeta G_0^{-3/2},\zeta^2 G_0^4\right)\right],
  \]
for 
\begin{align*}
L_{1,1}(J)=&O(G_0^{-1/2})\exp(-\sigma(G_0,J))\\
L_{1,2}(J)=&-\left(3\zeta\sqrt{\frac{\pi (G_0+J)^3}{2}}+O\left(\zeta,\zeta^2 G_0^{3/2}\right)\right)\exp(-\sigma(G_0,J)),
\end{align*}
with 
\[
\partial_J \mathcal L_\pm(J)= O\left(G_0^{5/2}\exp(-\sigma(G_0,J)), \zeta G_0^{9/2}\exp(-\sigma(G_0,J))\right) 
\]
and $\sigma$ as in \eqref{def:sigmaapp} and $s(\varphi)$ being of the form $s(\varphi)=\varphi+f(\varphi,J)$ with $|f|,|\partial_\varphi f|=O(\zeta G_0^{3/2})$. We thus find that (recall that we assume  $\zeta\ll G_0^{-4}$)
\begin{equation}\label{eq:melnikov}
\begin{aligned}
\partial_J(\mathcal L_+-\mathcal L_-)(\varphi,J)=&O\left(G_0^{5/2}\exp(-\sigma(G_0,J))\right)\\
\partial_\varphi(\mathcal L_+-\mathcal L_-)(\varphi,J)=&\left(2L_{1,2}(J)\sin\varphi+O(\zeta G_0)\right)\exp(-\sigma(G_0,J)).
\end{aligned}
\end{equation}
Combining the expressions \eqref{eq:asymptotics2goodform} and \eqref{eq:melnikov}  we deduce that the map $\mathtt S_1$ is of the form \eqref{eq:T1map}. By \eqref{eq:melnikov},  for $G_0$ large enough, there exists a unique solution $\varphi_*(J)$  to $\varphi\mapsto \partial_\varphi(\mathcal L_+-\mathcal L_-)$ which is close to zero. Moreover,
\[
\varepsilon:=\sup_{J\in[-1,1]}|\partial_\varphi(\mathcal L_+-\mathcal L_-)(\varphi_*(J),J)|\ll \min\{\alpha_*,\tau\}=O(G_0^{-6}).
\]
Hence, for $G_0$ large enough, the maps $\mathtt S_0,\mathtt S_1$ satisfy the assumptions of Theorem \ref{thm:transitivityIFS} and the proof follows.
\end{proof}

\bibliography{Bibliography}

@book {Cassels72,
    AUTHOR = {Cassels, J. W. S.},
     TITLE = {An introduction to {D}iophantine approximation},
    SERIES = {Cambridge Tracts in Mathematics and Mathematical Physics},
    VOLUME = {No. 45},
      NOTE = {Facsimile reprint of the 1957 edition},
 PUBLISHER = {Hafner Publishing Co., New York},
      YEAR = {1972},
     PAGES = {x+169},
   MRCLASS = {10FXX},
  MRNUMBER = {349591},
}

@article{LiTuraevPreprint,
    AUTHOR = {Li, D. and Turaev, D.},
    TITLE = {Symplectic blenders near whiskered tori and the
persistence of saddle-center homoclinics},
    NOTE= {In preparation}
}

@article {MR3146588,
    AUTHOR = {Fejoz, J.},
     TITLE = {On action-angle coordinates and the {P}oincar\'e{}
              coordinates},
   JOURNAL = {Regul. Chaotic Dyn.},
  FJOURNAL = {Regular and Chaotic Dynamics. International Scientific
              Journal},
    VOLUME = {18},
      YEAR = {2013},
    NUMBER = {6},
     PAGES = {703--718},
      ISSN = {1560-3547,1468-4845},
   MRCLASS = {37-03 (01-01 37C80 37J35 70F05 70G45 70H06)},
  MRNUMBER = {3146588},
       DOI = {10.1134/S1560354713060105},
       URL = {https://doi.org/10.1134/S1560354713060105},
}

@article {MR1278373,
    AUTHOR = {Xia, Z.},
     TITLE = {Arnold diffusion and oscillatory solutions in the
              planar three-body problem},
   JOURNAL = {J. Differential Equations},
  FJOURNAL = {Journal of Differential Equations},
    VOLUME = {110},
      YEAR = {1994},
    NUMBER = {2},
     PAGES = {289--321},
      ISSN = {0022-0396,1090-2732},
   MRCLASS = {70F07 (58F27 70H05 70K50)},
  MRNUMBER = {1278373},
MRREVIEWER = {Alessandra\ Celletti},
       DOI = {10.1006/jdeq.1994.1069},
       URL = {https://doi.org/10.1006/jdeq.1994.1069},
}

@article {MR2350333,
    AUTHOR = {Moeckel, R.},
     TITLE = {Symbolic dynamics in the planar three-body problem},
   JOURNAL = {Regul. Chaotic Dyn.},
  FJOURNAL = {Regular and Chaotic Dynamics. International Scientific
              Journal},
    VOLUME = {12},
      YEAR = {2007},
    NUMBER = {5},
     PAGES = {449--475},
      ISSN = {1560-3547,1468-4845},
   MRCLASS = {37N05 (37B10 70F07)},
  MRNUMBER = {2350333},
MRREVIEWER = {Jes\'us\ F.\ Palaci\'an},
       DOI = {10.1134/S1560354707050012},
       URL = {https://doi.org/10.1134/S1560354707050012},
}

@article {AlekseevQR1,
    AUTHOR = {Alekseev, V. M.},
     TITLE = {Invariant {M}arkov subsets of diffeomorphisms},
   JOURNAL = {Uspehi Mat. Nauk},
  FJOURNAL = {Akademija Nauk SSSR i Moskovskoe Matemati\v ceskoe Ob\v s\v
              cestvo. Uspehi Matemati\v ceskih Nauk},
    VOLUME = {23},
      YEAR = {1968},
    NUMBER = {2(140)},
     PAGES = {209--210},
      ISSN = {0042-1316},
   MRCLASS = {54.82 (28.00)},
  MRNUMBER = {276947},
}

@article {AlekseevQR2,
    AUTHOR = {Alekseev, V. M.},
     TITLE = {Quasirandom dynamical systems. {II}. {O}ne-dimensional
              nonlinear vibrations in a periodically perturbed field},
   JOURNAL = {Mat. Sb. (N.S.)},
  FJOURNAL = {Matematicheski\u i\ Sbornik. Novaya Seriya},
    VOLUME = {77(119)},
      YEAR = {1968},
     PAGES = {545--601},
      ISSN = {0368-8666},
   MRCLASS = {54.82 (28.00)},
  MRNUMBER = {276949},
MRREVIEWER = {A.\ Ver\v sik},
}

@article {Sitnikov60,
    AUTHOR = {Sitnikov, K.},
     TITLE = {The existence of oscillatory motions in the three-body
              problems},
   JOURNAL = {Soviet Physics. Dokl.},
  FJOURNAL = {Soviet Physics. Doklady},
    VOLUME = {5},
      YEAR = {1960},
     PAGES = {647--650},
      ISSN = {0038-5689},
   MRCLASS = {85.34},
  MRNUMBER = {0127389 (23 \#B435)},
MRREVIEWER = {E. Leimanis},
}

@article {MR163026,
    AUTHOR = {Arnold, V. I.},
     TITLE = {Instability of dynamical systems with many degrees of freedom},
   JOURNAL = {Dokl. Akad. Nauk SSSR},
  FJOURNAL = {Doklady Akademii Nauk SSSR},
    VOLUME = {156},
      YEAR = {1964},
     PAGES = {9--12},
      ISSN = {0002-3264},
   MRCLASS = {34.65 (57.48)},
  MRNUMBER = {163026},
MRREVIEWER = {J.\ Moser},
}

@article {MR1949441,
    AUTHOR = {Cresson, J.},
     TITLE = {Symbolic dynamics and {A}rnold diffusion},
   JOURNAL = {J. Differential Equations},
  FJOURNAL = {Journal of Differential Equations},
    VOLUME = {187},
      YEAR = {2003},
    NUMBER = {2},
     PAGES = {269--292},
      ISSN = {0022-0396,1090-2732},
   MRCLASS = {37J40 (37B10 37D05 37J45 70H08)},
  MRNUMBER = {1949441},
MRREVIEWER = {Jean-Pierre\ Marco},
       DOI = {10.1016/S0022-0396(02)00053-0},
       URL = {https://doi.org/10.1016/S0022-0396(02)00053-0},
}

@article {MR744302,
    AUTHOR = {Robinson, C.},
     TITLE = {Homoclinic orbits and oscillation for the planar three-body
              problem},
   JOURNAL = {J. Differential Equations},
  FJOURNAL = {Journal of Differential Equations},
    VOLUME = {52},
      YEAR = {1984},
    NUMBER = {3},
     PAGES = {356--377},
      ISSN = {0022-0396,1090-2732},
   MRCLASS = {58F40 (58F15 70F07)},
  MRNUMBER = {744302},
MRREVIEWER = {S.\ Yu.\ Pilyugin},
       DOI = {10.1016/0022-0396(84)90168-2},
       URL = {https://doi.org/10.1016/0022-0396(84)90168-2},
}

@article {MR295648,
    AUTHOR = {Saari, D. G.},
     TITLE = {Improbability of collisions in {N}ewtonian gravitational
              systems},
   JOURNAL = {Trans. Amer. Math. Soc.},
  FJOURNAL = {Transactions of the American Mathematical Society},
    VOLUME = {162},
      YEAR = {1971},
     PAGES = {267--271; erratum, ibid. 168 (1972), 521},
      ISSN = {0002-9947,1088-6850},
   MRCLASS = {70.34},
  MRNUMBER = {295648},
       DOI = {10.2307/1995752},
       URL = {https://doi.org/10.2307/1995752},
}

@article {MR4011691,
    AUTHOR = {Fleischer, S. and Knauf, A.},
     TITLE = {Improbability of collisions in {$n$}-body systems},
   JOURNAL = {Arch. Ration. Mech. Anal.},
  FJOURNAL = {Archive for Rational Mechanics and Analysis},
    VOLUME = {234},
      YEAR = {2019},
    NUMBER = {3},
     PAGES = {1007--1039},
      ISSN = {0003-9527,1432-0673},
   MRCLASS = {70F10 (70F16)},
  MRNUMBER = {4011691},
MRREVIEWER = {Martha\ Alvarez-Ram\'irez},
       DOI = {10.1007/s00205-019-01406-4},
       URL = {https://doi.org/10.1007/s00205-019-01406-4},
}

@book {MR442980,
    AUTHOR = {Moser, J.},
     TITLE = {Stable and random motions in dynamical systems},
    SERIES = {Annals of Mathematics Studies},
    VOLUME = {No. 77},
      NOTE = {With special emphasis on celestial mechanics,
              Hermann Weyl Lectures, the Institute for Advanced Study,
              Princeton, N. J},
 PUBLISHER = {Princeton University Press, Princeton, NJ; University of Tokyo
              Press, Tokyo},
      YEAR = {1973},
     PAGES = {viii+198},
   MRCLASS = {58FXX (34C35 70.58)},
  MRNUMBER = {442980},
MRREVIEWER = {Clark\ Robinson},
}

@article {MR1509241,
    AUTHOR = {Chazy, J.},
     TITLE = {Sur l'allure du mouvement dans le probl\`eme des trois corps
              quand le temps cro\^it ind\'efiniment},
   JOURNAL = {Ann. Sci. \'Ecole Norm. Sup. (3)},
  FJOURNAL = {Annales Scientifiques de l'\'Ecole Normale Sup\'erieure.
              Troisi\`eme S\'erie},
    VOLUME = {39},
      YEAR = {1922},
     PAGES = {29--130},
      ISSN = {0012-9593},
   MRCLASS = {99-04},
  MRNUMBER = {1509241},
       URL = {http://www.numdam.org/item?id=ASENS_1922_3_39__29_0},
}

@article {Clarke25,
    AUTHOR = {Clarke, A. and Fejoz, J. and Guardia, M.},
     TITLE = {Why are inner planets not inclined?},
   JOURNAL = {Publ. Math. Inst. Hautes \'Etudes Sci.},
  FJOURNAL = {Publications Math\'ematiques. Institut de Hautes \'Etudes
              Scientifiques},
    VOLUME = {141},
      YEAR = {2025},
     PAGES = {1--98},
      ISSN = {0073-8301,1618-1913},
   MRCLASS = {70F15 (37N05 85A05)},
  MRNUMBER = {4916233},
       DOI = {10.1007/s10240-024-00151-z},
       URL = {https://doi-org.sire.ub.edu/10.1007/s10240-024-00151-z},
}

@article {MR4743518,
    AUTHOR = {Li, D. and Turaev, D.},
     TITLE = {Persistence of heterodimensional cycles},
   JOURNAL = {Invent. Math.},
  FJOURNAL = {Inventiones Mathematicae},
    VOLUME = {236},
      YEAR = {2024},
    NUMBER = {3},
     PAGES = {1413--1504},
      ISSN = {0020-9910,1432-1297},
   MRCLASS = {37C27 (37G10)},
  MRNUMBER = {4743518},
       DOI = {10.1007/s00222-024-01255-3},
       URL = {https://doi.org/10.1007/s00222-024-01255-3},
}

@article {MR4531665,
    AUTHOR = {D\'iaz, L. J. and P\'erez, S. A.},
     TITLE = {Nontransverse heterodimensional cycles: stabilisation and
              robust tangencies},
   JOURNAL = {Trans. Amer. Math. Soc.},
  FJOURNAL = {Transactions of the American Mathematical Society},
    VOLUME = {376},
      YEAR = {2023},
    NUMBER = {2},
     PAGES = {891--944},
      ISSN = {0002-9947,1088-6850},
   MRCLASS = {37C20 (37C29 37D20 37D30)},
  MRNUMBER = {4531665},
MRREVIEWER = {Xiao\ Wen},
       DOI = {10.1090/tran/8694},
       URL = {https://doi.org/10.1090/tran/8694},
}

@book {MR2269239,
    AUTHOR = {Arnold, V. I. and Kozlov, V. V. and Neishtadt,
              A. I.},
     TITLE = {Mathematical aspects of classical and celestial mechanics},
    SERIES = {Encyclopaedia of Mathematical Sciences},
    VOLUME = {3},
   EDITION = {Third},
      NOTE = {[Dynamical systems. III],
              Translated from the Russian original by E. Khukhro},
 PUBLISHER = {Springer-Verlag, Berlin},
      YEAR = {2006},
     PAGES = {xiv+518},
      ISBN = {978-3-540-28246-4; 3-540-28246-7},
   MRCLASS = {70-02 (37Jxx 70-01 70H03 70H05)},
  MRNUMBER = {2269239},
MRREVIEWER = {Ernesto\ A.\ Lacomba},
}

@article {MR4808810,
    AUTHOR = {Capi\'nski, M. J. and Krauskopf, B. and Osinga, H.
              M. and Zgliczy\'nski, P.},
     TITLE = {Characterising blenders via covering relations and cone
              conditions},
   JOURNAL = {J. Differential Equations},
  FJOURNAL = {Journal of Differential Equations},
    VOLUME = {416},
      YEAR = {2025},
     PAGES = {768--805},
      ISSN = {0022-0396,1090-2732},
   MRCLASS = {37M21 (37B20 37C29 37D30 65G20)},
  MRNUMBER = {4808810},
MRREVIEWER = {Alexandre\ Artur Pinho Rodrigues},
       DOI = {10.1016/j.jde.2024.10.004},
       URL = {https://doi.org/10.1016/j.jde.2024.10.004},
}

@article {MR3168258,
    AUTHOR = {D\'iaz, L. J. and Kiriki, S. and Shinohara, K.},
     TITLE = {Blenders in centre unstable {H}\'enon-like families: with an
              application to heterodimensional bifurcations},
   JOURNAL = {Nonlinearity},
  FJOURNAL = {Nonlinearity},
    VOLUME = {27},
      YEAR = {2014},
    NUMBER = {3},
     PAGES = {353--378},
      ISSN = {0951-7715,1361-6544},
   MRCLASS = {37C20 (37C25 37C29 37C70)},
  MRNUMBER = {3168258},
MRREVIEWER = {I.\ Dan\ Coroian},
       DOI = {10.1088/0951-7715/27/3/353},
       URL = {https://doi.org/10.1088/0951-7715/27/3/353},
}

@incollection {MR4043213,
    AUTHOR = {D\'iaz, L. J. and P\'erez, S. A.},
     TITLE = {Blender-horseshoes in center-unstable {H}\'enon-like families},
 BOOKTITLE = {New trends in one-dimensional dynamics},
    SERIES = {Springer Proc. Math. Stat.},
    VOLUME = {285},
     PAGES = {137--163},
 PUBLISHER = {Springer, Cham},
      YEAR = {[2019] \copyright 2019},
      ISBN = {978-3-030-16833-9; 978-3-030-16832-2},
   MRCLASS = {37C29 (37C45 37D30)},
  MRNUMBER = {4043213},
MRREVIEWER = {Rui\ Zou},
       DOI = {10.1007/978-3-030-16833-9\_8},
       URL = {https://doi.org/10.1007/978-3-030-16833-9_8},
}

@article {MR4198639,
    AUTHOR = {Avila, A. and Crovisier, S. and Wilkinson, A.},
     TITLE = {{$C^1$} density of stable ergodicity},
   JOURNAL = {Adv. Math.},
  FJOURNAL = {Advances in Mathematics},
    VOLUME = {379},
      YEAR = {2021},
     PAGES = {Paper No. 107496, 68},
      ISSN = {0001-8708,1090-2082},
   MRCLASS = {37C20 (37C40)},
  MRNUMBER = {4198639},
MRREVIEWER = {Luciana\ Silva\ Salgado},
       DOI = {10.1016/j.aim.2020.107496},
       URL = {https://doi.org/10.1016/j.aim.2020.107496},
}

@article {Robustc1tangencies,
    AUTHOR = {Bonatti, C. and D\'iaz, L. J.},
     TITLE = {Abundance of {$C^1$}-robust homoclinic tangencies},
   JOURNAL = {Trans. Amer. Math. Soc.},
  FJOURNAL = {Transactions of the American Mathematical Society},
    VOLUME = {364},
      YEAR = {2012},
    NUMBER = {10},
     PAGES = {5111--5148},
      ISSN = {0002-9947,1088-6850},
   MRCLASS = {37C20 (37C05 37C29 37C70 37D05)},
  MRNUMBER = {2931324},
MRREVIEWER = {Leonardo\ E.\ Mora},
       DOI = {10.1090/S0002-9947-2012-05445-6},
       URL = {https://doi.org/10.1090/S0002-9947-2012-05445-6},
}

@book {BeyondUH,
    AUTHOR = {Bonatti, C. and D\'iaz, L. J. and Viana, M.},
     TITLE = {Dynamics beyond uniform hyperbolicity},
    SERIES = {Encyclopaedia of Mathematical Sciences},
    VOLUME = {102},
      NOTE = {A global geometric and probabilistic perspective,
              Mathematical Physics, III},
 PUBLISHER = {Springer-Verlag, Berlin},
      YEAR = {2005},
     PAGES = {xviii+384},
      ISBN = {3-540-22066-6},
   MRCLASS = {37-02 (37C20 37C29 37D25 37D30)},
  MRNUMBER = {2105774},
MRREVIEWER = {Sheldon\ E.\ Newhouse},
}

@misc{guardia2022hyperbolicdynamicsoscillatorymotions,
      title={Hyperbolic dynamics and oscillatory motions in the 3 Body Problem}, 
      author={M. Guardia and P. Martín and J. Paradela and T. M. Seara},
      year={2022},
      eprint={2207.14351},
      archivePrefix={arXiv},
      primaryClass={math.DS},
      url={https://arxiv.org/abs/2207.14351}, 
}

@article {DelaLLavegaps,
    AUTHOR = {Delshams, A. and de la Llave, R. and Seara, T. M.},
     TITLE = {A geometric mechanism for diffusion in {H}amiltonian systems
              overcoming the large gap problem: heuristics and rigorous
              verification on a model},
   JOURNAL = {Mem. Amer. Math. Soc.},
  FJOURNAL = {Memoirs of the American Mathematical Society},
    VOLUME = {179},
      YEAR = {2006},
    NUMBER = {844},
      ISSN = {0065-9266,1947-6221},
   MRCLASS = {37J40 (70H09)},
  MRNUMBER = {2184276},
MRREVIEWER = {Jean-Pierre\ Marco},
       DOI = {10.1090/memo/0844},
       URL = {https://doi.org/10.1090/memo/0844},
}

@article {MR4366414,
    AUTHOR = {Berger, P.},
     TITLE = {Generic family displaying robustly a fast growth of the number
              of periodic points},
   JOURNAL = {Acta Math.},
  FJOURNAL = {Acta Mathematica},
    VOLUME = {227},
      YEAR = {2021},
    NUMBER = {2},
     PAGES = {205--262},
      ISSN = {0001-5962,1871-2509},
   MRCLASS = {37C35 (37C20 37C25 37E10)},
  MRNUMBER = {4366414},
       DOI = {10.4310/acta.2021.v227.n2.a1},
       URL = {https://doi-org.proxy-um.researchport.umd.edu/10.4310/acta.2021.v227.n2.a1},
}

@article {MR3514960,
    AUTHOR = {Berger, P.},
     TITLE = {Generic family with robustly infinitely many sinks},
   JOURNAL = {Invent. Math.},
  FJOURNAL = {Inventiones Mathematicae},
    VOLUME = {205},
      YEAR = {2016},
    NUMBER = {1},
     PAGES = {121--172},
      ISSN = {0020-9910,1432-1297},
   MRCLASS = {37C20 (26A21 37C25)},
  MRNUMBER = {3514960},
MRREVIEWER = {Raquel\ Ribeiro},
       DOI = {10.1007/s00222-015-0632-6},
       URL = {https://doi-org.proxy-um.researchport.umd.edu/10.1007/s00222-015-0632-6},
}

@article {MR4043881,
    AUTHOR = {Biebler, S.},
     TITLE = {Newhouse phenomenon for automorphisms of low degree in
              {$\Bbb{C}^3$}},
   JOURNAL = {Adv. Math.},
  FJOURNAL = {Advances in Mathematics},
    VOLUME = {361},
      YEAR = {2020},
     PAGES = {106952, 39},
      ISSN = {0001-8708,1090-2082},
   MRCLASS = {37F80 (32H50 37C29 37D10 37F46)},
  MRNUMBER = {4043881},
MRREVIEWER = {Romain\ Dujardin},
       DOI = {10.1016/j.aim.2019.106952},
       URL = {https://doi-org.proxy-um.researchport.umd.edu/10.1016/j.aim.2019.106952},
}

@book {MR0005824,
    AUTHOR = {Wintner, A.},
     TITLE = {The {A}nalytical {F}oundations of {C}elestial {M}echanics},
    SERIES = {Princeton Mathematical Series},
    VOLUME = {vol. 5},
 PUBLISHER = {Princeton University Press, Princeton, NJ},
      YEAR = {1941},
     PAGES = {xii+448},
   MRCLASS = {85.0X},
  MRNUMBER = {5824},
MRREVIEWER = {E.\ J.\ Moulton},
}

@article {garrido2024parabolicsaddlesnewhousedomains,
    AUTHOR = {Garrido, M. and Mart\'in, P. and Paradela, J.},
     TITLE = {Parabolic saddles and {N}ewhouse domains in celestial
              mechanics},
   JOURNAL = {Comm. Math. Phys.},
  FJOURNAL = {Communications in Mathematical Physics},
    VOLUME = {406},
      YEAR = {2025},
    NUMBER = {7},
     PAGES = {Paper No. 173, 67},
      ISSN = {0010-3616,1432-0916},
   MRCLASS = {37N05 (37C29 37D10 37J46 70F15)},
  MRNUMBER = {4922847},
       DOI = {10.1007/s00220-025-05299-1},
       URL = {https://doi-org.proxy-um.researchport.umd.edu/10.1007/s00220-025-05299-1},
}

@article {MR4544807,
    AUTHOR = {Capi\'nski, M. J. and Gidea, M.},
     TITLE = {Arnold diffusion, quantitative estimates, and stochastic
              behavior in the three-body problem},
   JOURNAL = {Comm. Pure Appl. Math.},
  FJOURNAL = {Communications on Pure and Applied Mathematics},
    VOLUME = {76},
      YEAR = {2023},
    NUMBER = {3},
     PAGES = {616--681},
      ISSN = {0010-3640,1097-0312},
   MRCLASS = {70F07 (37N05)},
  MRNUMBER = {4544807},
MRREVIEWER = {Patricia\ Yanguas},
       DOI = {10.1002/cpa.22014},
       URL = {https://doi-org.proxy-um.researchport.umd.edu/10.1002/cpa.22014},
}

@unpublished{GorodetskiK12,
  AUTHOR =   {Gorodetski, A. and Kaloshin, V},
  TITLE =    {Hausdorff dimension of oscillatory motions for restricted three body problems},
  NOTE =     {Preprint, available at http://www.terpconnect.umd.edu/~vkaloshi},
  YEAR =     {2012},
}

@misc{Giraltcoorbitalchaos,
      title={Coorbital homoclinic and chaotic dynamics in the Restricted 3-Body Problem}, 
      author={I. Baldomá and M. Giralt and M. Guardia},
      year={2023},
      eprint={2312.13819},
      archivePrefix={arXiv},
      primaryClass={math.DS},
      url={https://arxiv.org/abs/2312.13819}, 
}

@article {MR2104598,
    AUTHOR = {Marco, J. P. and Sauzin, D.},
     TITLE = {Wandering domains and random walks in {G}evrey near-integrable
              systems},
   JOURNAL = {Ergodic Theory Dynam. Systems},
  FJOURNAL = {Ergodic Theory and Dynamical Systems},
    VOLUME = {24},
      YEAR = {2004},
    NUMBER = {5},
     PAGES = {1619--1666},
      ISSN = {0143-3857,1469-4417},
   MRCLASS = {37J40 (37B10 70H08 70K30)},
  MRNUMBER = {2104598},
MRREVIEWER = {Mikhail\ B.\ Sevryuk},
       DOI = {10.1017/S0143385703000786},
       URL = {https://doi-org.proxy-um.researchport.umd.edu/10.1017/S0143385703000786},
}

@book {MR501173,
    AUTHOR = {Hirsch, M. W. and Pugh, C. C. and Shub, M.},
     TITLE = {Invariant manifolds},
    SERIES = {Lecture Notes in Mathematics},
    VOLUME = {Vol. 583},
 PUBLISHER = {Springer-Verlag, Berlin-New York},
      YEAR = {1977},
     PAGES = {ii+149},
   MRCLASS = {58F15 (58F10)},
  MRNUMBER = {501173},
MRREVIEWER = {M.\ C.\ Irwin},
}

@book {PalisMelo,
    AUTHOR = {Palis, Jr., J. and de Melo, W.},
     TITLE = {Geometric theory of dynamical systems},
      NOTE = {An introduction,
              Translated from the Portuguese by A. K. Manning},
 PUBLISHER = {Springer-Verlag, New York-Berlin},
      YEAR = {1982},
     PAGES = {xii+198},
      ISBN = {0-387-90668-1},
   MRCLASS = {58-01 (34C35 58F09 58Fxx)},
  MRNUMBER = {669541},
MRREVIEWER = {Russell\ B.\ Walker},
}

@article {DelaLLaveScattmap,
    AUTHOR = {Delshams, A. and de la Llave, R. and Seara, T. M.},
     TITLE = {Geometric properties of the scattering map of a normally
              hyperbolic invariant manifold},
   JOURNAL = {Adv. Math.},
  FJOURNAL = {Advances in Mathematics},
    VOLUME = {217},
      YEAR = {2008},
    NUMBER = {3},
     PAGES = {1096--1153},
      ISSN = {0001-8708,1090-2082},
   MRCLASS = {37J40 (37D10 37J99 37N20 81U05)},
  MRNUMBER = {2383896},
MRREVIEWER = {Enrico\ Valdinoci},
       DOI = {10.1016/j.aim.2007.08.014},
       URL = {https://doi.org/10.1016/j.aim.2007.08.014},
}

@misc{guardia2025stochastic1,
      title={Stochastic behavior along  mean motion resonances in 
		the restricted planar 3-body problem}, 
      author={M. Guardia and V. Kaloshin and P. Mart\'in and P. Roldan},
      year={2025},
note={In preparation}
}

@misc{guardia2025stochastic2,
      title={A normally hyperbolic invariant lamination along  mean motion resonances in 
the restricted planar three body problem}, 
author={M. Guardia and V. Kaloshin and P. Mart\'in and P. Roldan},
      year={2025},
note={In preparation}
}

@misc{guardia2023degeneratearnolddiffusionmechanism,
      title={A degenerate Arnold diffusion mechanism in the Restricted 3 Body Problem}, 
      author={M. Guardia and J. Paradela and T. M. Seara},
      year={2023},
      eprint={2302.06973},
      archivePrefix={arXiv},
      primaryClass={math.DS},
      url={https://arxiv.org/abs/2302.06973}, 
}

@article {BFM20a,
    AUTHOR = {Baldom\'a, I. and Fontich, E. and Mart\'in, P.},
     TITLE = {Invariant manifolds of parabolic fixed points ({I}).
              {E}xistence and dependence on parameters},
   JOURNAL = {J. Differential Equations},
  FJOURNAL = {Journal of Differential Equations},
    VOLUME = {268},
      YEAR = {2020},
    NUMBER = {9},
     PAGES = {5516--5573},
      ISSN = {0022-0396,1090-2732},
   MRCLASS = {37D10 (37M21)},
  MRNUMBER = {4066058},
       DOI = {10.1016/j.jde.2019.11.100},
       URL = {https://doi-org.sire.ub.edu/10.1016/j.jde.2019.11.100},
}

@article {BFM20b,
    AUTHOR = {Baldom\'a, I. and Fontich, E. and Mart\'in, P.},
     TITLE = {Invariant manifolds of parabolic fixed points ({II}).
              {A}pproximations by sums of homogeneous functions},
   JOURNAL = {J. Differential Equations},
  FJOURNAL = {Journal of Differential Equations},
    VOLUME = {268},
      YEAR = {2020},
    NUMBER = {9},
     PAGES = {5574--5627},
      ISSN = {0022-0396,1090-2732},
   MRCLASS = {37D10 (37M21)},
  MRNUMBER = {4066057},
       DOI = {10.1016/j.jde.2019.11.099},
       URL = {https://doi-org.sire.ub.edu/10.1016/j.jde.2019.11.099},
}

@article {OscillatoryEllipticGuardia,
    AUTHOR = {Guardia, M. and Seara, T.M. and Mart\'in, P. and
              Sabbagh, L.},
     TITLE = {Oscillatory orbits in the restricted elliptic planar three
              body problem},
   JOURNAL = {Discrete Contin. Dyn. Syst.},
  FJOURNAL = {Discrete and Continuous Dynamical Systems. Series A},
    VOLUME = {37},
      YEAR = {2017},
    NUMBER = {1},
     PAGES = {229--256},
      ISSN = {1078-0947,1553-5231},
   MRCLASS = {70F07 (37J40 70F15)},
  MRNUMBER = {3583476},
MRREVIEWER = {Patricia\ Yanguas},
       DOI = {10.3934/dcds.2017009},
       URL = {https://doi.org/10.3934/dcds.2017009},
}

@article {GideaL06,
    AUTHOR = {Gidea, M. and de la Llave, R.},
     TITLE = {Topological methods in the instability problem of
              {H}amiltonian systems},
   JOURNAL = {Discrete Contin. Dyn. Syst.},
  FJOURNAL = {Discrete and Continuous Dynamical Systems},
    VOLUME = {14},
      YEAR = {2006},
    NUMBER = {2},
     PAGES = {295--328},
      ISSN = {1078-0947,1553-5231},
   MRCLASS = {37J40 (34C37 37C29 37C50 70H08)},
  MRNUMBER = {2163534},
MRREVIEWER = {Ernest\ Fontich},
       DOI = {10.3934/dcds.2006.14.295},
       URL = {https://doi-org.sire.ub.edu/10.3934/dcds.2006.14.295},
}

@article {McGeheestablemanifold,
    AUTHOR = {McGehee, R.},
     TITLE = {A stable manifold theorem for degenerate fixed points with
              applications to celestial mechanics},
   JOURNAL = {J. Differential Equations},
  FJOURNAL = {Journal of Differential Equations},
    VOLUME = {14},
      YEAR = {1973},
     PAGES = {70--88},
      ISSN = {0022-0396,1090-2732},
   MRCLASS = {58F05 (70.34)},
  MRNUMBER = {362403},
MRREVIEWER = {J.\ E.\ Marsden},
       DOI = {10.1016/0022-0396(73)90077-6},
       URL = {https://doi.org/10.1016/0022-0396(73)90077-6},
}

@article {GelfreichTuraev,
    AUTHOR = {Gelfreich, V. and Turaev, D.},
     TITLE = {Arnold diffusion in a priori chaotic symplectic maps},
   JOURNAL = {Comm. Math. Phys.},
  FJOURNAL = {Communications in Mathematical Physics},
    VOLUME = {353},
      YEAR = {2017},
    NUMBER = {2},
     PAGES = {507--547},
      ISSN = {0010-3616,1432-0916},
   MRCLASS = {37J40 (37J10)},
  MRNUMBER = {3649479},
MRREVIEWER = {Dongfeng\ Zhang},
       DOI = {10.1007/s00220-017-2867-0},
       URL = {https://doi.org/10.1007/s00220-017-2867-0},
}

@incollection {Moeckeldrift,
    AUTHOR = {Moeckel, R.},
     TITLE = {Generic drift on {C}antor sets of annuli},
 BOOKTITLE = {Celestial mechanics ({E}vanston, {IL}, 1999)},
    SERIES = {Contemp. Math.},
    VOLUME = {292},
     PAGES = {163--171},
 PUBLISHER = {Amer. Math. Soc., Providence, RI},
      YEAR = {2002},
      ISBN = {0-8218-2902-5},
   MRCLASS = {37J40 (70H08)},
  MRNUMBER = {1884898},
MRREVIEWER = {Pere\ Guti\'errez},
       DOI = {10.1090/conm/292/04922},
       URL = {https://doi.org/10.1090/conm/292/04922},
}

@article {DiffusionEllipticKaloshin,
    AUTHOR = {Delshams, A. and Kaloshin, V. and de la Rosa, A.
              and Seara, T. M.},
     TITLE = {Global instability in the restricted planar elliptic three
              body problem},
   JOURNAL = {Comm. Math. Phys.},
  FJOURNAL = {Communications in Mathematical Physics},
    VOLUME = {366},
      YEAR = {2019},
    NUMBER = {3},
     PAGES = {1173--1228},
      ISSN = {0010-3616,1432-0916},
   MRCLASS = {70F07 (37N05 70F15 70K50)},
  MRNUMBER = {3927089},
MRREVIEWER = {Eduardo\ S. G. Leandro},
       DOI = {10.1007/s00220-018-3248-z},
       URL = {https://doi.org/10.1007/s00220-018-3248-z},
}

@article {BonattiDiazblenders,
    AUTHOR = {Bonatti, C. and D\'iaz, L. J.},
     TITLE = {Persistent nonhyperbolic transitive diffeomorphisms},
   JOURNAL = {Ann. of Math. (2)},
  FJOURNAL = {Annals of Mathematics. Second Series},
    VOLUME = {143},
      YEAR = {1996},
    NUMBER = {2},
     PAGES = {357--396},
      ISSN = {0003-486X,1939-8980},
   MRCLASS = {58F12 (58F10 58F15)},
  MRNUMBER = {1381990},
MRREVIEWER = {Marcelo\ Viana},
       DOI = {10.2307/2118647},
       URL = {https://doi.org/10.2307/2118647},
}

@article {NassiriPujalsTransitivity,
    AUTHOR = {Nassiri, M. and Pujals, E. R.},
     TITLE = {Robust transitivity in {H}amiltonian dynamics},
   JOURNAL = {Ann. Sci. \'Ec. Norm. Sup\'er. (4)},
  FJOURNAL = {Annales Scientifiques de l'\'Ecole Normale Sup\'erieure.
              Quatri\`eme S\'erie},
    VOLUME = {45},
      YEAR = {2012},
    NUMBER = {2},
     PAGES = {191--239},
      ISSN = {0012-9593,1873-2151},
   MRCLASS = {37J10},
  MRNUMBER = {2977619},
MRREVIEWER = {Luigi\ Chierchia},
       DOI = {10.24033/asens.2164},
       URL = {https://doi.org/10.24033/asens.2164},
}
\bibliographystyle{alpha}

\end{document}